\documentclass[aos]{imsart}
\usepackage{amsmath, amssymb,amsfonts,bbm,bm,cases,verbatim,multirow,color,xcolor}
\usepackage[sectionbib]{bibunits}
\usepackage[ruled,vlined]{algorithm2e}
\usepackage[mathscr]{euscript}
\usepackage{epic,eepic,psfrag,epsfig}
\usepackage[sf,bf,SF,footnotesize]{subfigure}
\usepackage{graphicx}
\usepackage{amsthm,breakcites}
\usepackage{amscd}
\usepackage{epsfig}
\usepackage{tikz}
\usepackage{natbib}
\usepackage{enumerate}

\usepackage{array}
\RequirePackage[colorlinks,citecolor={blue!60!black},urlcolor={blue!70!black},linkcolor={red!60!black},breaklinks,hypertexnames=false]{hyperref}

\usetikzlibrary{patterns}

\pgfdeclareradialshading{ring1}{\pgfpoint{0cm}{0cm}}%
{rgb(0cm)=(0.7,0.7,1);
rgb(0.5cm)=(0.7,0.7,1);
rgb(0.51cm)=(0.7,0.7,1);
rgb(0.90cm)=(1,0.7,0.7);
rgb(1.0cm)=(1,0.7,0.7)}

\pgfdeclareradialshading{ring2}{\pgfpoint{0cm}{0cm}}%
{rgb(0cm)=(1,0.7,0.7);
rgb(0.178cm)=(1,0.7,0.7);
rgb(0.36cm)=(0.7,0.7,1);
rgb(0.5cm)=(0.7,0.7,1);
rgb(0.51cm)=(0.7,0.7,1);
rgb(0.90cm)=(1,0.7,0.7);
rgb(1.0cm)=(1,0.7,0.7)}

\usetikzlibrary{patterns}

\pgfdeclarepatternformonly{mydots}{\pgfqpoint{-1pt}{-1pt}}{\pgfqpoint{1pt}{1pt}}{\pgfqpoint{2pt}{2pt}}
{%
  \pgfpathcircle{\pgfqpoint{0pt}{0pt}}{.5pt}%
  \pgfusepath{fill}%
}%

\pgfdeclarepatternformonly{mynewdots}{\pgfqpoint{-1pt}{-1pt}}{\pgfqpoint{1pt}{1pt}}{\pgfqpoint{1pt}{1pt}}
{%
  \pgfpathcircle{\pgfqpoint{0pt}{0pt}}{.5pt}%
  \pgfusepath{fill}%
}%

\usepackage{amsfonts}
\usepackage{xcolor}
\newtheorem{theorem}{Theorem}
\newtheorem{example}{Example}

\newtheorem{lemma}[theorem]{Lemma}

\newtheorem{corollary}[theorem]{Corollary}
\newtheorem{defn}{Definition}
\newtheorem{prop}[theorem]{Proposition}

\newcommand{\holderExponentEst}[1]{\hat{\holderExponent}_{n,#1}}

\newcommand{\E}{\mathbb{E}}
\newcommand{\N}{\mathbb{N}}
\newcommand{\R}{\mathbb{R}}
\newcommand{\Lebesgue}{\mathcal{L}_d}
\newcommand{\Z}{\mathbb{Z}}

\newcommand{\marginalDistribution}{\mu}
\newcommand{\distEuclidean}{\mathrm{dist}_{2}}

\DeclareMathOperator*{\argmin}{argmin}
\DeclareMathOperator*{\argmax}{argmax}
\DeclareMathOperator*{\sargmin}{sargmin}
\DeclareMathOperator*{\sargmax}{sargmax}

\newcommand{\one}{\mathbbm{1}}

\newcommand{\probDistribution}{P}
\newcommand{\Prob}{\mathbb{P}}

\newcommand{\holderConstant}{\lambda}
\newcommand{\holderExponent}{\beta}

\newcommand{\measureClass}{\mathcal{P}}
\newcommand{\sample}{\mathcal{D}}
\newcommand{\sampleX}{\sample_X}

\newcommand{\taylorSeries}{\mathcal{T}}
\newcommand{\totalVariationDistance}{\mathrm{TV}}
\newcommand{\hellingerDistance}{\mathrm{H}}
\newcommand{\borel}{\mathcal{B}}

\newcommand{\classOfRestrictedHolderFunctions}[1]{\mathcal{F}_{\mathrm{H\ddot{o}l}}(\holderExponent,\holderConstant,#1)}

\newcommand{\classOfHolderDistributionsSuperLevelSet}{\measureClass_{\mathrm{H\ddot{o}l}}(\holderExponent,\holderConstant,\tau)}
\newcommand{\classOfHolderDistributionsSuperLevelSetPrime}{\measureClass_{\mathrm{H\ddot{o}l}}(\holderExponent',\holderConstant',\tau)}

\newcommand{\classOfHolderDistributions}{\measureClass_{\mathrm{H\ddot{o}l}}(\holderExponent,\holderConstant)}
\newcommand{\regressionFunction}{\eta}

\newcommand{\diamSup}{\mathrm{diam}_{\infty}}
\newcommand{\supNorm}[1]{\|#1\|_{\infty}}

\newcommand{\distSup}{\mathrm{dist}_{\infty}}

\newcommand{\regularityConstant}{\upsilon}
\newcommand{\muRegularSet}[1][\regularityConstant]{\mathcal{R}_{#1}(\mu)}
\newcommand{\muRegularSetArg}[1]{\mathcal{R}_{\regularityConstant}(#1)}
\newcommand{\kl}{\mathrm{kl}}
\newcommand{\esssup}{\mathrm{ess\hspace{.5mm}sup}}
\newcommand{\powerSet}{\mathrm{Pow}}
\newcommand{\vcDim}{\mathrm{dim}_{\mathrm{VC}}}

\newcommand{\lowerDensity}{\omega}
\newcommand{\etaSuperLevelSet}[1]{\mathcal{X}_{#1}(\regressionFunction)}
\newcommand{\omegaSuperLevelSet}[1]{\mathcal{X}_{#1}({\lowerDensity})}
\newcommand{\approximableDensityExponent}{\kappa}
\newcommand{\approximableMarginExponent}{\gamma}
\newcommand{\approximableSetsConstant}{C_{\mathrm{App}}}
\newcommand{\closedMetricBallSupNorm}[2]{\bar{B}_{ #2}(#1)}
\newcommand{\openMetricBallSupNorm}[2]{B_{#2}(#1)}

\newcommand{\openMetricBallEuclideanNorm}[2]{B_{2,#2}(#1)}
\newcommand{\classOfWellApproximableSets}{\mathcal{P}_{\mathrm{App}}(\mathcal{A},\approximableDensityExponent,\approximableMarginExponent,\tau,\approximableSetsConstant)}
\newcommand{\classOfWellApproximableSetsWithHyperCubes}{\mathcal{P}_{\mathrm{App}}(\mathcal{A}_{\mathrm{hpr}},\approximableDensityExponent,\approximableMarginExponent,\tau,\approximableSetsConstant)}
\newcommand{\classOfWellApproximableSetsWithIntervals}{\mathcal{P}_{\mathrm{App}}(\mathcal{A}_{\mathrm{int}},\approximableDensityExponent,\approximableMarginExponent,\tau,\approximableSetsConstant)}

\newcommand{\classOfLBDistributions}{\mathcal{P}^{\dagger}(\mathcal{A},\holderExponent,\approximableDensityExponent,\approximableMarginExponent,\regularityConstant,\holderConstant,\tau,\approximableSetsConstant)}

\newcommand{\classOfWellApproximableSetsHO}{\mathcal{P}^+_{\mathrm{App}}(\mathcal{A},\approximableDensityExponent,\approximableMarginExponent,\regularityConstant,\tau,\approximableSetsConstant)}

\newcommand{\classOfHolderDistributionsHTE}{\measureClass^{\mathrm{HTE}}_{\mathrm{H\ddot{o}l}}(\holderExponent,\holderConstant,\pi)}

\newcommand{\classOfWellApproximableSetsHTE}{\mathcal{P}^{\mathrm{HTE}}_{\mathrm{App}}(\mathcal{A},t,\approximableDensityExponent,\approximableMarginExponent,\regularityConstant,\approximableSetsConstant)}

\newcommand{\densitySuperLevelSet}[1]{\mathcal{X}_{#1}({f_\mu})}

\newcommand{\support}{\mathrm{supp}}

\newcommand{\numberOfConsideredHyperCubes}{\hat{L}}
\newcommand{\thetaNDelta}{\theta}

\newcommand{\pValueN}{\hat{p}_{n}}
\newcommand{\pValueLong}{\hat{p}_{n,\holderExponent,\holderConstant}}
\newcommand{\empiricalMarginalDistribution}{\hat{\mu}_n}
\newcommand{\subSampleEmpiricalMarginalDistribution}{\hat{\mu}_m}
\newcommand{\empiricalRegressionFunction}{\hat{\eta}_n}
\newcommand{\pValueNHO}{\pValueN^+}
\newcommand{\pValueLongHO}{\pValueLong^+}
\newcommand{\setOfHypercubes}{\mathcal{H}}
\newcommand{\setOfHypercubesHO}{\setOfHypercubes^+}

\newcommand{\baseMeasure}{\mu}
\newcommand{\eigenValueMinimal}{\lambda_{\min}}
\newcommand{\eigenValueMaximal}{\lambda_{\max}}

\newcommand{\holderConstantEstimator}{\hat{\lambda}_{n,\holderExponent,\delta}}

\newcommand{\holderConstantEstimatorArgs}[2]{\hat{\lambda}_{n,#1,#2}}

\newcommand{\holderConstantUpperBound}{{\holderConstant}}

\newcommand{\classOfDistributionsSatisfyingRegularityCondition}{\measureClass_{\mathrm{Reg}}(\tau)}
\newcommand{\classOfHolderDistributionsSuperLevelSetIdentifiable}{\measureClass_{\mathrm{H\ddot{o}l}}^+(\holderExponent,\holderConstantUpperBound,\tau,\boldsymbol{\epsilon})}

\newcommand{\lowerHolderConstantSelfSimilar}{\holderConstant_0}

\newcommand{\lowerDensityConstantSelfSimilar}{c_0}

\newcommand{\regularSetSelfSimilar}{\mathcal{R}_\circ(\mu,\lowerDensityConstantSelfSimilar)}

\newcommand{\minRadiusSelfSimilar}{r_0}

\newcommand{\classOfSelfSimilarHolderDistributions}{\measureClass_{\mathrm{H\ddot{o}l}}^\dagger(\holderExponent,\holderConstant,\lowerHolderConstantSelfSimilar,\lowerDensityConstantSelfSimilar,\minRadiusSelfSimilar)}

\usepackage{xr-hyper}

\usepackage[draft,inline,nomargin,index]{fixme}
\fxsetup{theme=color,mode=multiuser}

\makeatletter
\newcommand*{\addFileDependency}[1]{
  \typeout{(#1)}
  \@addtofilelist{#1}
  \IfFileExists{#1}{}{\typeout{No file #1.}}
}
\makeatother

\startlocaldefs

\endlocaldefs

\begin{document}

\begin{frontmatter}

\title{Optimal subgroup selection}
\runtitle{Optimal subgroup selection}

\begin{aug}
\author[A]{\fnms{Henry} W. J. \snm{Reeve}\ead[label=e1]{henry.reeve@bristol.ac.uk}},
\author[B]{\fnms{Timothy} I. \snm{Cannings}\thanksref{t1}\ead[label=e2]{timothy.cannings@ed.ac.uk}} \\
\thankstext{t1}{Research supported by Engineering and Physical Sciences Research Council (EPSRC) New Investigator Award EP/V002694/1.}
\and
\author[C]{\fnms{Richard} J. \snm{Samworth}\thanksref{t2}\ead[label=e3]{r.samworth@statslab.cam.ac.uk}}
\thankstext{t2}{Research supported by Engineering and Physical Sciences Research Council (EPSRC) Programme grant EP/N031938/1, EPSRC Fellowship EP/P031447/1 and European Research Council Advanced Grant 101019498.}

\runauthor{H. W. J. Reeve, T. I. Cannings and R. J. Samworth}

  \address[A]{School of Mathematics, University of Bristol\\\href{mailto:henry.reeve@bristol.ac.uk}{henry.reeve@bristol.ac.uk} 
}

  \address[B]{School of Mathematics and Maxwell Institute for Mathematical Sciences,\\ The University of Edinburgh\\\href{mailto:timothy.cannings@ed.ac.uk}{timothy.cannings@ed.ac.uk}
}
        
  \address[C]{Statistical Laboratory, University of Cambridge\\
  \href{mailto:r.samworth@statslab.cam.ac.uk}{r.samworth@statslab.cam.ac.uk}
}
\end{aug}

\begin{abstract}
In clinical trials and other applications, we often see regions of the feature space that appear to exhibit interesting behaviour, but it is unclear whether these observed phenomena are reflected at the population level.  Focusing on a regression setting, we consider the subgroup selection challenge of identifying a region of the feature space on which the regression function exceeds a pre-determined threshold.  We formulate the problem as one of constrained optimisation, where we seek a low-complexity, data-dependent selection set on which, with a guaranteed probability, the regression function is uniformly at least as large as the threshold; subject to this constraint, we would like the region to contain as much mass under the marginal feature distribution as possible.  This leads to a natural notion of regret, and our main contribution is to determine the minimax optimal rate for this regret in both the sample size and the Type I error probability.  The rate involves a delicate interplay between parameters that control the smoothness of the regression function, as well as exponents that quantify the extent to which the optimal selection set at the population level can be approximated by families of well-behaved subsets.  Finally, we expand the scope of our previous results by illustrating how they may be generalised to a treatment and control setting, where interest lies in the heterogeneous treatment effect. 
\end{abstract}

\end{frontmatter}

\section{Introduction}
\label{Sec:Introduction}

Consider a clinical trial that assesses the effectiveness of a drug or vaccine.  It will typically be the case that efficacy is heterogeneous across the population, in the sense that the probability of a successful outcome depends on several recorded covariates.  As a consequence, we may be unable to recommend the treatment for all individuals; nevertheless, it may be too conservative to reject it entirely.  It is very tempting to trawl through the data to identify a subset of the population for which the treatment appears to perform well, but statisticians are well-versed in the dangers of this type of data snooping \citep{senn1997wisdom,feinstein1998problem,rothwell2005subgroup,wang2007Statistics,kaufman2013these,altman2015subgroup,zhang2015subgroup,gabler2016no,lipkovich2017tutorial}.

The aim of this paper is to study a \emph{subgroup selection} problem, where we seek to identify a subset of the population for which a regression function exceeds a pre-determined threshold.  In the clinical trial example above, this threshold would represent the level at which the treatment is deemed effective.  Subgroup selection forms an important component of the more general field of \emph{subgroup analysis} \citep{wang2007Statistics,herrera2011overview,ting2020design}, which refers to the problem of understanding the association between a response and subgroups of subjects under study, as defined by one or more subgrouping variables.  The main challenge is to provide valid inference, given that the subgroup will be chosen after seeing the data \citep{lagakos2006challenge}. 

Our first contribution is to formulate subgroup selection as a constrained optimisation problem.  Given independent covariate-response pairs and a family $\mathcal{A}$ of subsets of our feature space, we seek a data-dependent selection set $\hat{A}$ taking values in $\mathcal{A}$ with the Type~I error control property that, with probability at least $1-\alpha$, the regression function is uniformly no smaller than the level $\tau$ on $\hat{A}$; subject to this constraint, we would like the proportion of the population belonging to $\hat{A}$ to be as large as possible.  In practice, $\mathcal{A}$ would typically be chosen to be of relatively low complexity, so as to lead to an interpretable decision rule. 


After introducing this new framework, our first result (Proposition~\ref{prop:lipschitzNotEnoughForConsistency} in Section~\ref{Sec:MainResults}) reveals the extent of the challenge.  We show that if our regression function belongs to a H\"older class, but the corresponding H\"older constant is unknown, then there is a sense in which no algorithm that respects the Type I error guarantee can do better in terms of power than one that ignores the data.   We therefore work initially over H\"older classes of known smoothness~$\holderExponent$, and with a known upper bound $\holderConstant$ on the H\"older constant; see Definition~\ref{def:holder}.   This enables us to define a data-dependent selection set that satisfies our Type I error guarantee.  The idea is to construct, for each hyper-cube $B$ in a suitable collection within our feature space $\mathbb{R}^d$, a $p$-value for testing the null hypothesis that the regression function is not uniformly above the level $\tau$ on~$B$.  The $p$-values are then combined via Holm's procedure \citep{holm1979simple} to identify a finite union of hyper-cubes that satisfy our Type I error control property.  Our final selection set~$\hat{A}_{\mathrm{OSS}}$ maximises the empirical measure among all elements of~$\mathcal{A}$ that lie within this finite union of hyper-cubes.

Next, we define a notion of regret $R_\tau(\hat{A})$ that quantifies the power discrepancy between a particular algorithm $\hat{A}$ and an oracle choice.  Our aim is to study the optimal regret that can be attained while maintaining Type I error control.  We find that the minimax optimal regret is determined by a combination of the smoothness $\holderExponent$ (initially assumed to lie in $(0,1]$) and two further exponents $\approximableDensityExponent, \approximableMarginExponent > 0$ that quantify the extent to which the oracle selection set can be approximated by families of well-behaved subsets in $\mathcal{A}$.  In particular, $\approximableDensityExponent$ and $\approximableMarginExponent$ control respectively the degree of concentration of the marginal measure, and the separation between the regression function and the critical level $\tau$ on these well-behaved subsets.  See Definition~\ref{defn:approximableMeasureClass} for a formal description.

Our main contribution in Section~\ref{Sec:MainResults} is to establish in Theorem~\ref{thm:minimaxRate}, that, with a sample size of~$n$, the minimax optimal rate of convergence of the regret over these distributional classes and over all algorithms that respect the Type I error guarantee at significance level $\alpha \in (0,1/2)$ is of order\footnote{Here, $\log_+ x:=\log x$ when $x \geq  e$ and $\log_+ x := 1$ otherwise. To be fully precise, the upper bound holds when $\mathcal{A}$ is a Vapnik--Chervonenkis class; the lower bound holds when $\holderExponent\approximableMarginExponent (\approximableDensityExponent-1) < d \approximableDensityExponent$ and the class $\mathcal{A}$ consists of convex sets and contains all axis-aligned hyper-rectangles.}
\begin{equation}
\label{Eq:OptimalOrder}
\min \biggl\{ \biggl(\frac{\log_+(n/\alpha)}{n}\biggr)^{\frac{\holderExponent \approximableDensityExponent \approximableMarginExponent}{\approximableDensityExponent(2\holderExponent+d)+\holderExponent\approximableMarginExponent}}+\frac{1}{n^{1/2}} \, , \, 1\biggr\}.
\end{equation}
The second term in the sum reflects the parametric rate, which corresponds to the difficulty of uniformly estimating the population measure of sets in a Vapnik--Chervonenkis class $\mathcal{A}$.  The primary interest, however, is in the first term in the sum, which reveals an intricate interplay between the distributional parameters, the sample size and the significance level.

As mentioned above, the algorithm that achieves the upper bound in Theorem~\ref{thm:minimaxRate} takes $\holderConstant$ and $\holderExponent$ as inputs.  In Section~\ref{Sec:Adaptation}, therefore, we describe how these parameters can be chosen in a data-driven manner.  Since Proposition~\ref{prop:lipschitzNotEnoughForConsistency} reveals the impossibility of adaptation in full generality, in Section~\ref{SubSec:lambda}, we impose mild additional regularity conditions on our classes, and show that under a sample size condition, we can with high probability estimate the H\"older constant of the regression function to within a factor of 2.  Moreover, in Section~\ref{SubSec:beta}, we show that under a self-similarity condition, we can also estimate $\holderExponent$ accurately, so that under a sample size condition, our fully data-driven algorithm maintains Type I error control, and has the same regret as the original algorithm up to a sub-logarithmic factor.

A limitation of our constructions for the upper bounds in Sections~\ref{Sec:MainResults} and~\ref{Sec:Adaptation} are that they are unable to take advantage of higher orders of smoothness beyond $\holderExponent = 1$.  To overcome this, in Section~\ref{Sec:HigherOrder}, we introduce a modified algorithm based on a local polynomial approximation of the regression function, and prove in Theorem~\ref{Thm:minimaxRateHOS} that this new construction both respects the Type I error at significance level $\alpha$ and has a regret of optimal order~\eqref{Eq:OptimalOrder} for general smoothness $\beta \in (0,\infty)$.  The price we pay for this is a stronger assumption on the marginal feature distribution: we now ask for it to have a well-behaved density with respect to Lebesgue measure (though we do not require this density to be bounded away from zero on its support). 

The lower bound constructions for Theorems~\ref{thm:minimaxRate} and~\ref{Thm:minimaxRateHOS} are addressed in Section~\ref{Sec:LowerBound}.  They involve three different finite collections of distributions within our classes, each designed to highlight different aspects of the challenge.  The first is a two-point construction, with both distributions having regression functions that are close to $\tau$ on disconnected regions, but with each such function only being uniformly above $\tau$ on one of these regions; this identifies the dependence of the lower bound on $\alpha$.  The second extends this construction to many distributions, each having its own region where the regression function is uniformly above~$\tau$, which underlines the necessity of the logarithmic factor in $n$ in~\eqref{Eq:OptimalOrder}.  Finally, the third family, which identifies the parametric rate, is another two-point construction with a shared regression function, but whose marginal feature distributions assign slightly different masses to the different connected components of the $\tau$-super level set of this regression function.  

Finally, in Section~\ref{sec:hteApplication}, we consider the more general setting where individuals may belong to either a treatment or control group, and where interest lies in the heterogeneous treatment effect.  We show that this heterogeneous treatment effect plays a very similar role to that of the regression function in earlier sections, so that our results generalise almost immediately.  Proofs of all of our results, as well as auxiliary results and their proofs, are deferred to the appended supplementary material.

One of the interesting messages of our work from an applied perspective is that, when carefully formulated, it is possible to make formally-justified, post-hoc observations concerning subgroup analyses from clinical studies.  When attempted without due care, such observations have been rightly criticised in the medical literature; for example:

\begin{quotation}
Analyses  must  be  predefined,  carefully justified,  and  limited  to  a  few  clinically  important  questions,  and  post-hoc  observations  should  be  treated  with scepticism irrespective of their statistical significance. \hfill \citep{rothwell2005subgroup}
\end{quotation}
\begin{quotation}
 The  statisticians  are  right  in denouncing subgroups that are formed post hoc from exercises in pure data dredging. \hfill \citep{feinstein1998problem}
\end{quotation}
A standard approach to handle subgroup analysis is via statistical tests of interaction \citep{brookes2001subgroup,brookes2004subgroup,kehl2006responder}.  \cite{zhang2017subgroup} propose a procedure to select a subgroup defined by a halfspace that seeks to maximise the expected difference in treatment effect in the context of an adaptive signature design trial.  Several other methods have been proposed for studying subgroups defined through heterogeneous treatment effects.  For instance, \citet{foster2011subgroup} propose an approach to identify subgroups having enhanced treatment effect via the construction of `virtual twins', while \citet{ballarini2018subgroup} consider maximum likelihood and Lasso-type approaches for estimating a difference in treatment effect in a parametric linear model setting.  \citet{su2009subgroup}, \citet{dusseldorp2010combining}, \citet{lipkovich2011subgroup} and \citet{seibold2016model} propose tree-based procedures to explore the heterogeneity structure of a treatment effect across subgroups that are defined after seeing the data; \citet{huber2019comparison} provide a simulation comparison of the relative performance of these methods, as well as the algorithm for adaptive refinement by directed peeling proposed by \citet{patel2016identifying}.  \cite{crump2008nonparametric} and \cite{watson2020machine} introduce tests of the global null hypothesis of no treatment effect heterogeneity (no subgroups).  

One can think of subgroup selection in our context as a super-level set estimation problem, with a key feature being the asymmetry of the way in which we handle cases where~$\hat{A}$ contains regions where the regression function is below $\tau$, and where it misses regions where the regression function is at least at level $\tau$.  This is motivated by applications such as clinical trials, where the primary concern is the retention of Type I error control despite the post-selection inference.  In this respect, our framework has some similarities with that of Neyman--Pearson classification \citep{cannon2002learning,scott2005neyman,tong2016survey,xia2021intentional}.  There, our covariate-response pairs $(X,Y)$ take values in $\mathbb{R}^d \times \{0,1\}$, and we seek a classifier $C:\mathbb{R}^d \rightarrow \{0,1\}$ that minimises $\mathbb{P}\bigl(C(X) = 0|Y=1\bigr)$ subject to an upper bound on $\mathbb{P}\bigl(C(X) = 1|Y=0\bigr)$.  Thus, as in our setting, the way in which the two types of error are handled is asymmetric.  On the other hand, as well as allowing continuous responses, our notions of loss are very different.  In particular, in our context, we incur a Type I error whenever our selected set $\hat{A}$ contains a single point that does not belong to the $\tau$-super-level set of the regression function.  In other words, our framework provides guarantees at an individual level, instead of on average over sub-populations.  This may well be ethically and practically advantageous, e.g.~in medical contexts, as discussed above.


Related work on the estimation of super-level sets of a regression function includes \citet{cavalier1997nonparametric}, \citet{scott2007regression}, \citet{willett2007minimax}, \citet{gotovos2013active}, \citet{laloe2013estimation}, \citet{zanette2018robust} and \citet{dau2020exact}; likewise, in a density estimation context, there is a large literature on highest density region estimation \citep{polonik1995measuring,hyndman1996computing,tsybakov1997nonparametric,mason2009asymptotic,samworth2010asymptotics,chen2017density,doss2018bandwidth,qiao2019nonparametric,rodriguez2019minimax,qiao2020asymptotics}.  The formulations of the problems studied in these works are rather different from ours, tending to focus on measures of the set difference or Hausdorff distance between the estimated and true sets of interest.  \citet{mammen2013confidence} study bootstrap confidence regions for level sets of nonparametric functions, with a particular emphasis on kernel estimation of density level sets.

We conclude this introduction with some notation used throughout the paper.  We adopt the convention that $\inf \emptyset := \infty$, and write $[n] := \{1,\ldots,n\}$ for $n \in \mathbb{N} \cup \{0\}$, with $[0] := \emptyset$.  Given a set $S$, we denote its power set by $\powerSet(S)$ and its cardinality and complement by $|S|$ and $S^{\mathrm{c}}$ respectively.  If $S$ has a strict total ordering, and $g:S \rightarrow \R$ is a function that attains its maximum, then we write $\sargmax\{ g(s) : s \in S\}$ for the smallest element of $\argmax\{g(s):s \in S\}$.  The $\sigma$-algebra of Borel measurable subsets of $\R^d$ is denoted by $\borel(\R^d)$.  We write $\vcDim(\mathcal{A})$ for the Vapnik--Chervonenkis dimension of a class of sets $\mathcal{A}$ \citep[e.g.][Chapter~2]{vershynin2018high}.  We also let $\mathcal{A}_{\mathrm{hpr}}$ and $\mathcal{A}_{\mathrm{conv}}$ denote the class of compact axis-aligned hyper-rectangles in $\R^d$ (i.e.~sets of the form $\prod_{j=1}^d [a_j,b_j]$ for some $a_j \leq b_j$ and $j \in [d]$), and the set of convex subsets of $\R^d$ respectively.

Given $x \in \mathbb{R}$, we write $x_+:=\max(x,0)$ and $\log_+ x:=\log x$ when $x \geq  e$ and $\log_+ x := 1$ otherwise.  Let $\supNorm{\cdot}$ and $\|\cdot\|_2$ denote the supremum and Euclidean norms on $\R^d$ respectively.  We denote the $d$-dimensional Lebesgue on $\R^d$ by $\Lebesgue$, and let $V_d :=\Lebesgue(\{x\in \R^d:\|x\|_2\leq 1\}) = \pi^{d/2}/\Gamma(1+d/2)$.  Given a set $S \subseteq \R^d$, we denote its $\ell_\infty$-norm \emph{diameter} by $\diamSup(S) := \sup_{x,y \in S} \supNorm{x-y}$ and write $\distSup(x,S) := \inf_{y \in S} \supNorm{x-y}$.  For $r > 0$, let $\openMetricBallSupNorm{x}{r}$ and $\closedMetricBallSupNorm{x}{r}$ denote the open and closed $\ell_\infty$ balls of radius $r$ about $x \in \R^d$ respectively.  For a function $f:\R^d\rightarrow \R$, and for $\xi \in \mathbb{R}$, we also let $\mathcal{X}_\xi(f):=\{x \in \R^d: f(x) \geq \xi \}$ denote its super-level set at level $\xi$.  Given a symmetric matrix $A \in \R^{q\times q}$, we write $A^+$ for its Moore--Penrose pseudo-inverse and $\eigenValueMinimal(A)$ and $\eigenValueMaximal(A)$  for its minimal and maximal eigenvalues, respectively.

For $p \in [0,1]$, we let $\mathrm{Bern}(p)$ denote the Bernoulli distribution on $\{0,1\}$ with mean $p$.  Given a Borel probability measure $\mu$ on $\R^d$, we write $\support(\mu)$ for its \emph{support}, i.e.~the intersection of all closed sets $C \subseteq \R^d$ with $\mu(C) = 1$.  Given Borel subsets $B_0$, $B_1 \subseteq \R^d$ and a measure $\mu$ on $\R^d$, we write $B_0 \subseteq B_1$ if $\mu(B_0 \setminus B_1) = 0$ and $B_0 \not\subseteq B_1$ if $\mu(B_0 \setminus B_1) > 0$; the dependence on $\mu$ in our notation here is left implicit since it will be clear from context, and we are thus equating sets whose symmetric difference has $\mu$-measure zero.  For probability measures $P,Q$ on a measurable space $(\Omega,\mathcal{F})$, we denote their total variation distance by $\mathrm{TV}(P,Q) := \sup_{B \in \mathcal{F}}|P(B) - Q(B)|$.  If these measures are absolutely continuous with respect to a $\sigma$-finite measure $\baseMeasure$, with Radon--Nikodym derivatives $f$ and $g$ respectively, then we write $\hellingerDistance(P,Q) := \bigl\{\int_{\Omega} (f^{1/2} - g^{1/2})^2 \, d\baseMeasure\bigr\}^{1/2}$ for their Hellinger distance, and $\chi^2(P,Q) := \int_{\Omega} f^2/g \, d\baseMeasure - 1$ for their $\chi^2$-divergence.  For $a \in [0,1]$, $b \in (0,1)$, we define $\kl(a,b)$ to be the Kullback--Leibler divergence between the $\mathrm{Bern}(a)$ and $\mathrm{Bern}(b)$ distributions; i.e., for $a \in (0,1)$,
\begin{align*}
\kl(a,b):=  a\log\biggl(\frac{a}{b}\biggr) +(1-a)\log\biggl(\frac{1-a}{1-b}\biggr),
\end{align*}
with $\kl(0,b) := -\log(1-b)$ and $\kl(1,b) := -\log b$.  

\section{Subset selection framework and minimax rates}
\label{Sec:MainResults}

Suppose that the covariate-response pair $(X,Y)$ has joint Borel probability distribution $\probDistribution$ on $\R^d \times [0,1]$.  Let $\mu \equiv \mu_{\probDistribution}$ denote the marginal distribution of $X$.  We say that $\regressionFunction \equiv \regressionFunction_{\probDistribution}:\R^d \rightarrow [0,1]$ is a \emph{regression function} for $P$ if $\eta$ is a version of the conditional expectation $\E(Y|X)$.  In other words, $\eta:\R^d \rightarrow [0,1]$ is a Borel measurable function such that $\int_B \eta(x) \, d\mu(x) = \int_{B\times [0,1]} y \, dP(x,y)$ for all $B \in \borel(\R^d)$. We let $\mathcal{A} \subseteq \borel(\R^d)$ denote our class of candidate selection sets, and assume that $\emptyset \in \mathcal{A}$.  Given a threshold $\tau \in (0,1)$, and recalling the notation $\etaSuperLevelSet{\tau} := \{x \in \R^d:\regressionFunction(x) \geq \tau\}$ for the $\tau$-super level set of $\regressionFunction$, an ideal output set in our class would have measure
\begin{align*}
M_\tau \equiv M_\tau(\probDistribution,\mathcal{A}):= \sup\bigl\{ \mu(A): A \in \mathcal{A} \cap \mathrm{Pow}\bigl(\mathcal{X}_\tau(\regressionFunction)\bigr)\bigr\}.
\end{align*}
Since $P$ is unknown, it will typically not be possible to output such an ideal subset.  Instead, we will assume that the practitioner has access to a sample $\sample \equiv \bigl((X_1,Y_1),\ldots, (X_n,Y_n)\bigr)$ of independent copies of $(X,Y)$.   We define the class of \emph{data-dependent selection sets}, denoted~$\hat{\mathcal{A}}_n$, to be the set of functions $\hat{A}:(\R^d \times [0,1])^n \rightarrow \mathcal{A}$ such that $(x,D) \mapsto \mathbbm{1}_{\hat{A}(D)}(x)$ is a Borel measurable function on $\R^d \times (\R^d \times [0,1])^n$.  Given a family $\mathcal{P}$ of distributions on $\R^d \times [0,1]$ and a significance level $\alpha \in (0,1)$, we relax the hard requirement that our output set should be a subset of $\mathcal{X}_\tau(\regressionFunction)$ by seeking a data-dependent selection set $\hat{A} \in \hat{\mathcal{A}}_n$, with 
\begin{equation}
\label{Eq:TypeIError}
\inf_{P \in \mathcal{P}} \Prob_P\bigl( \hat{A}(\sample) \subseteq{} \etaSuperLevelSet{\tau}\bigr) \geq 1-\alpha.
\end{equation}
Note that the condition $A \subseteq \etaSuperLevelSet{\tau}$ is independent of our choice of regression function (Lemma~\ref{lemma:tauSuperLevelSetConditionWellDefined}). When~\eqref{Eq:TypeIError} holds, we will say that $\hat{A}$ controls the Type~I error at level $\alpha$ over the class $\mathcal{P}$, and denote the set of data-dependent selection sets that satisfy this requirement as $\hat{\mathcal{A}}_n(\alpha,\mathcal{P})$.  For $\hat{A} \in \hat{\mathcal{A}}_n(\alpha,\mathcal{P})$, we would also like that for each $P \in \mathcal{P}$, the random quantity $\mu\bigl(\hat{A}(\sample)\bigr)$ should be close to $M_\tau$, i.e.~we will seek upper bounds for the regret 
\[
R_\tau(\hat{A}) \equiv R_\tau(\hat{A},P,\mathcal{A}) := M_\tau - \E_P\big\{\mu\bigl(\hat{A}(\sample)\bigr)\bigm|\hat{A}(\sample) \subseteq{} \etaSuperLevelSet{\tau}\big\}.
\]
In several places below, we abbreviate $\hat{A}(\sample)$ as $\hat{A}$ where the argument is clear from context.  

Our first result reveals that even Lipschitz restrictions on the regression function in our class~$\mathcal{P}$ do not suffice to obtain a data-dependent selection set $\hat{A}$ that satisfies both~\eqref{Eq:TypeIError} and $\Prob_P\bigl(\mu(\hat{A}) > 0\bigr) > \alpha$ for some $P \in \mathcal{P}$.  The negative implication is that, regardless of smoothness properties of the true regression function, the regret of any~$\hat{A}$ satisfying~\eqref{Eq:TypeIError} can be no smaller than the infimum of the regrets of all selection sets that ignore the data while still controlling the Type I error over our Lipschitz class.  

Given a probability measure~$\mu$ on $\R^d$, we let $\mathcal{P}_{\mathrm{Lip}}(\mu)$ denote the set of all Borel probability distributions on $\R^d \times [0,1]$ with marginal $\mu$ on $\R^d$, and for which the corresponding regression function~$\regressionFunction$ is Lipschitz. 
We say that $\bar{A} \in \hat{\mathcal{A}}_n$ is \emph{data independent} if $\one_{\{\bar{A}(\sample)=A\}}$ and $\sample$ are independent for all $A \in \mathcal{A}$, and write $\bar{\mathcal{A}}$ for the set of data-independent selection sets.  
\begin{prop}\label{prop:lipschitzNotEnoughForConsistency} 
Let $\mu$ be a distribution on $\R^d$ without atoms and take $\mathcal{A} \subseteq \borel(\R^d)$ with $\vcDim(\mathcal{A}) < \infty$.  Further, let $\hat{A} \in \hat{\mathcal{A}}_n\bigl(\alpha,\mathcal{P}_{\mathrm{Lip}}(\mu)\bigr)$. Then for all $\probDistribution \in \mathcal{P}_{\mathrm{Lip}}(\mu)$, we have 
\begin{align}
\label{Eq:FirstConc}
\Prob_P\bigl(\mu(\hat{A}) = 0 \mid \hat{A} \subseteq{} \etaSuperLevelSet{\tau}\bigr) \geq \Prob_P\big(\bigl\{\mu(\hat{A}) = 0\bigr\} \cap \bigl\{\hat{A} \subseteq{} \etaSuperLevelSet{\tau}\bigr\} \big)\geq 1-\alpha.
\end{align}
Hence
\begin{align}\label{eq:conclusionPowerWithoutHolderKnoweldge}
R_\tau(\hat{A}) \geq M_\tau \cdot (1-\alpha) =  \inf\bigl\{R_\tau(\bar{A}): \bar{A} \in \mathcal{\bar{A}} \cap \hat{\mathcal{A}}_n\bigl(\alpha,\mathcal{P}_{\mathrm{Lip}}(\mu)\bigr)\bigr\}.   
\end{align}
\end{prop}
In the light of Proposition~\ref{prop:lipschitzNotEnoughForConsistency}, we will assume initially that our regression function belongs to a H\"older class for which both the H\"older exponent and the associated constant are known.  In this section, we will work with smoothness exponents that are at most 1.
\begin{defn}[H\"{o}lder class]\label{def:holder} Given $\holderExponent \in (0,1]$, $\holderConstant \in (0,\infty)$ and $A\subseteq \R^d$, we let $\classOfRestrictedHolderFunctions{A}$ denote the set of all continuous functions $\eta:\R^d \rightarrow [0,1]$ such that 
\begin{align*}
|\eta(x')-\eta(x)|\leq \holderConstant \cdot \supNorm{x'-x}^{\holderExponent},
\end{align*}
for all $x, x' \in A$. We then let $\classOfHolderDistributionsSuperLevelSet$ denote the class of all distributions $\probDistribution$ on $\R^d \times [0,1]$ with a regression function $\regressionFunction \in \mathcal{F}_{\mathrm{H\ddot ol}}\bigl(\holderExponent,\holderConstant,\etaSuperLevelSet{\tau}\bigr)$.
\end{defn}
Observe that in this definition, the H\"older smoothness condition is only required to hold on the restriction of $\eta$ to $\mathcal{X}_\tau(\eta)$.  While our algorithms will control the Type I error over H\"older classes, we will see that the optimal regret for a data-dependent selection set depends on further aspects of the underlying data generating mechanism.  To describe the relevant classes, we first define a function $\lowerDensity\equiv \lowerDensity_{\mu,d}:\R^d \rightarrow [0,1]$ by 
\begin{align*}
\lowerDensity(x):=\inf_{r \in (0,1)}\frac{\mu\bigl(\closedMetricBallSupNorm{x}{r}\bigr)}{r^d}.
\end{align*}
Borrowing the terminology of \citet{reeve2021adaptive}, we will refer to $\omega$ as a \emph{lower density}, even though our definition is slightly different as we work with an $\ell_\infty$-ball instead of a Euclidean ball.  A nice feature of this definition is that it allows us to avoid assuming that $\mu$ is absolutely continuous with respect to Lebesgue measure; see \citet{reeve2021adaptive} for several basic properties of lower density functions.  

We are now in a position to define what we refer to as an \emph{approximable} class of distributions; these are ones for which we can approximate $M_\tau$ well by $\mu(A)$, where $A \in \mathcal{A}$ is both such that the lower density on $A$ is not too small and such that the regression function on~$A$ is bounded away from the critical threshold $\tau$.   
\begin{defn}[Approximable class]\label{defn:approximableMeasureClass} Given $\mathcal{A} \subseteq \borel(\R^d)$, $\approximableDensityExponent, \approximableMarginExponent > 0$, $\tau \in (0,1)$ and $\approximableSetsConstant \geq 1$, let $\classOfWellApproximableSets$ denote the class of all distributions $\probDistribution$ on $\R^d \times [0,1]$ with marginal~$\mu$ on $\R^d$ and a regression function $\regressionFunction:\R^d \rightarrow [0,1]$ such that
\begin{align*}
\sup\big\{ \mu(A):A \in \mathcal{A}\cap \powerSet\bigl(\omegaSuperLevelSet{\xi}\cap \etaSuperLevelSet{\tau+\Delta}\bigr)  \big\} \geq M_\tau - \approximableSetsConstant \cdot (\xi^{\approximableDensityExponent}+\Delta^{\approximableMarginExponent}),
\end{align*}
for all $\xi, \Delta > 0$.
\end{defn}
We now provide several examples of distributions belonging to appropriate approximable classes.  The proofs of the claims in these examples are given in Section~\ref{Sec:Examples}.

\begin{example}
\label{ex:1}
Let $\Prob(Y=0) = \Prob(Y=1) = 1/2$, and $X|Y=r \sim N\bigl((-1)^{r-1}\nu,1\bigr)$ for $r \in \{0,1\}$ and some $\nu > 0$.  Fix some $\tau \in (0,1)$, and let $\mathcal{A}_{\mathrm{int}}$ denote the set of all closed intervals in $\mathbb{R}$.  Then the distribution $P$ of $(X,Y)$ belongs to $\classOfWellApproximableSetsWithIntervals$ with $\kappa = \gamma = 1$, for a suitably large choice of $\approximableSetsConstant$, depending only on $\nu$ and $\tau$. 
\end{example}

\begin{example}
\label{ex:2}
Let $\mu$ denote the uniform distribution on $[0,1]^d$ and fix $\tau \in (0,1)$.  Suppose that $\regressionFunction:\R^d \rightarrow [0,1]$ is coordinate-wise increasing, that $\mathcal{S}_\tau := \{x \in [0,1]^d:\regressionFunction(x) = \tau\} \neq \emptyset$, and that there exist $\delta, \epsilon \in (0,1]$ and $\gamma > 0$ such that $\regressionFunction(x) - \tau \geq \epsilon \cdot \distSup(x,\mathcal{S}_\tau)^{1/\gamma}$, for every $x \in \etaSuperLevelSet{\tau} \cap [0,1]^d$ with $\distSup(x,\mathcal{S}_\tau) \leq \delta$.  If $P$ denotes a distribution on $\R^d \times [0,1]$ with marginal $\mu$ on $\R^d$ and regression function $\regressionFunction$, then $P \in \classOfWellApproximableSetsWithHyperCubes$ for arbitrarily large $\kappa > 0$, provided that $\approximableSetsConstant \geq 2d/(\epsilon^{\approximableMarginExponent}\delta)$. 
\end{example}

\begin{example}
\label{ex:3}
Consider the family of distributions $\{\mu_\kappa : \kappa \in (0,\infty)\}$ on $\R^d$ with densities of the form $x \mapsto g_\kappa(\supNorm{x})$, where $g_\kappa:[0,\infty) \rightarrow [0,\infty)$ is given by
\begin{align*}
  g_{\kappa}(y) := \begin{cases} (\kappa/2^d) \cdot \bigl\{1 + (1-\kappa)y^d\bigr\}^{-1/(1-\kappa)} &\text{ if } \kappa \in (0,1) \\
(1/2^d) \cdot e^{-y^d} &\text{ if } \kappa=1\\
(\kappa/2^d) \cdot \bigl\{1 - (\kappa-1)y^d\bigr\}^{1/(\kappa-1)}\one_{\{y \leq 1/(\kappa-1)^{1/d}\}} &\text{ if }\kappa \in (1,\infty).
\end{cases}
\end{align*}
Now, for $\gamma > 0$ and $\tau \in (0,1)$, define the regression function $\eta_\gamma:\R^d \rightarrow [0,1]$ by
\[
\eta_{\gamma}(x) := 0 \vee \bigl\{\tau + \holderConstant \cdot \mathrm{sgn}(x_1)|x_1|^{1/\gamma}\bigr\} \wedge 1,
\] 
Writing $P_{\kappa,\gamma}$ for the distribution on $\R^d \times \{0,1\}$ with marginal $\mu_\kappa$ on $\R^d$ and regression function~$\eta_\gamma$, we have that $P_{\approximableDensityExponent,\approximableMarginExponent} \in \classOfWellApproximableSetsWithHyperCubes$ for $\approximableSetsConstant \equiv \approximableSetsConstant(d,\kappa,\gamma,\holderConstant) > 0$ sufficiently large.
\end{example}
Example~\ref{ex:1} is designed to be a simple setting of our problem, where we can take $\gamma = \kappa = 1$.  Example~\ref{ex:2} illustrates the way that the growth of $\regressionFunction$ as we move away from $\eta^{-1}(\tau)$ affects the parameter $\gamma$ of our class, while Example~\ref{ex:3} shows the effect of the tail behaviour of the marginal density on $\mathbb{R}^d$ on the parameter $\kappa$.  Proposition~\ref{Prop:GeneralExample} in the online supplement provides general conditions under which our joint distribution belongs to $\mathcal{P}_{\mathrm{App}}(\mathcal{A},\kappa,\gamma,\tau,C_{\mathrm{App}})$ with $\gamma=1$.  In essence, the $\gamma=1$ setting occurs when the gradient of the regression function never vanishes on the boundary $\eta^{-1}(\tau)$ of $\mathcal{X}_\tau(\eta)$.

We can now state the main theorem of this section, which reveals the minimax optimal rate of convergence for the regret over $\classOfHolderDistributionsSuperLevelSet \cap \classOfWellApproximableSets$ for a data-dependent selection set in $\hat{\mathcal{A}}_n\bigl(\alpha,\classOfHolderDistributionsSuperLevelSet\bigr)$.
\begin{theorem}
\label{thm:minimaxRate} 
Take $\holderExponent \in (0,1]$, $\holderConstant \geq 1$, $\approximableDensityExponent, \approximableMarginExponent > 0$, $\tau \in (0,1)$ and $\approximableSetsConstant \geq 1$.  

\medskip

\noindent\emph{\textbf{(i) Upper bound:}} Let $\mathcal{A} \subseteq \borel(\R^d)$ satisfy $\vcDim(\mathcal{A}) < \infty$ and $\emptyset \in \mathcal{A}$.  Then there exists $C\geq 1$, depending only on $d$, $\approximableDensityExponent$, $\approximableMarginExponent$, $\tau$, $\approximableSetsConstant$ and $\vcDim(\mathcal{A})$, such that for all $n \in \N$ and $\alpha \in (0,1/2]$, we have
\begin{equation}
    \label{Eq:MinimaxUpperBoundRegret}
\inf_{\hat{A}} \, \sup_{\probDistribution} \,  R_\tau(\hat{A}) \leq {C} \cdot \min\biggl\{\biggl({\frac{\holderConstant^{d/\holderExponent}\cdot \log_+( n/ \alpha)}{n}}\biggr)^{\frac{\holderExponent \approximableDensityExponent \approximableMarginExponent}{\approximableDensityExponent(2\holderExponent+d)+\holderExponent\approximableMarginExponent}}+\frac{1}{n^{1/2}},1\biggr\},
\end{equation}
where the infimum in~\eqref{Eq:MinimaxUpperBoundRegret} is taken over $\hat{\mathcal{A}}_n\bigl(\tau,\alpha,\classOfHolderDistributionsSuperLevelSet\bigr)$ and the supremum is taken over $\classOfHolderDistributionsSuperLevelSet \cap \classOfWellApproximableSets$.

\medskip

\noindent \emph{\textbf{(ii) Lower bound:}} Now suppose that $\holderExponent \approximableMarginExponent(\approximableDensityExponent-1) < d\approximableDensityExponent$, $\epsilon_0 \in (0,1/2)$, $\tau \in (\epsilon_0,1-\epsilon_0)$ and $\alpha \in (0,1/2-\epsilon_0]$.  Then there exists $c > 0$, depending only on $d$, $\holderExponent$, $\approximableDensityExponent$, $\approximableMarginExponent$, $\approximableSetsConstant$ and $\epsilon_0$, such that for any $\mathcal{A} \subseteq \borel(\R^d)$ satisfying $\mathcal{A}_{\mathrm{hpr}}  \subseteq \mathcal{A} \subseteq \mathcal{A}_{\mathrm{conv}}$ and any $n \in \N$, we have 
\begin{equation}
    \label{Eq:MinimaxLowerBound}
\inf_{\hat{A}} \, \sup_{\probDistribution} \,  R_\tau(\hat{A}) \geq {c} \cdot \min\biggl\{ \biggl({\frac{\holderConstant^{d/\holderExponent}\cdot \log_+\bigl( n/ (\holderConstant^{d/\holderExponent} \alpha)\bigr)}{n}}\biggr)^{\frac{\holderExponent \approximableDensityExponent \approximableMarginExponent}{\approximableDensityExponent(2\holderExponent+d)+\holderExponent\approximableMarginExponent}}+\frac{1}{n^{1/2}},1\biggr\},
\end{equation}
where, again, the infimum in~\eqref{Eq:MinimaxLowerBound} is taken over $\hat{\mathcal{A}}_n\bigl(\tau,\alpha,\classOfHolderDistributionsSuperLevelSet\bigr)$ and the supremum is taken over $\classOfHolderDistributionsSuperLevelSet \cap \classOfWellApproximableSets$.
\end{theorem}
Since $\vcDim(\mathcal{A}_{\mathrm{hpr}} ) < \infty$ \citep[e.g.][Exercise~5 in Section~6.8]{shalev2014understanding}, the choice $\mathcal{A} = \mathcal{A}_{\mathrm{hpr}} $ provides a natural example satisfying both the lower and upper bounds in Theorem~\ref{thm:minimaxRate}. 

In order to introduce the algorithm that achieves the upper bound, we define the \emph{empirical marginal distribution} $\empiricalMarginalDistribution$ and \emph{empirical regression function} $\empiricalRegressionFunction$, for $B \subseteq \R^d$, by
\begin{align}
\empiricalMarginalDistribution(B)&:=\frac{1}{n}\sum_{i=1}^n\one_{\{X_i \in B\}} \label{eq:empiricalMarginalDef}\\ 
\empiricalRegressionFunction(B)&:= \frac{1}{n\cdot \empiricalMarginalDistribution(B)}\sum_{i=1}^n Y_i\cdot\one_{\{X_i \in B\}} \label{eq:empiricalRegFuncDef}
\end{align}
whenever $\empiricalMarginalDistribution(B) > 0$, and $\empiricalRegressionFunction(B) := 1/2$ otherwise.  The main idea is to associate, to each $B \subseteq \R^d$, a $p$-value for a test of the hypothesis that the regression function is uniformly above the level $\tau$ on $B$.  More precisely, for $B \subseteq \R^d$, we define 
\begin{align}\label{eq:pValueDef}
\pValueN(B) \equiv \pValueLong(B) := \exp\Bigl\{ -n \cdot \empiricalMarginalDistribution(B)\cdot \kl\bigl( \empiricalRegressionFunction(B),\tau+\holderConstant\cdot \diamSup(B)^{\holderExponent}\bigr) \Bigr\},
\end{align}
whenever $\empiricalRegressionFunction(B)>\tau+\holderConstant\cdot \diamSup(B)^{\holderExponent}$, and $\pValueN(B):=1$ otherwise.  Lemma~\ref{lemma:pValue} below confirms that $\pValueN(B)$ is indeed a $p$-value (even conditionally on $\sampleX \equiv (X_i)_{i \in [n]}$).
\begin{lemma}
\label{lemma:pValue} 
Fix $\holderExponent \in (0,1]$, $\holderConstant \in [1,\infty)$ and $\probDistribution \in \classOfHolderDistributionsSuperLevelSet$ with a regression function  $\regressionFunction \in \mathcal{F}_{\mathrm{H\ddot ol}}\bigl(\holderExponent,\holderConstant,\etaSuperLevelSet{\tau}\bigr)$. Then given $B \in \borel(\R^d)$ with $B\not \subseteq \etaSuperLevelSet{\tau}$, and any $\alpha \in (0,1)$, we have
\begin{align*}
{\Prob}_P\big( \pValueN(B) \leq \alpha \mid \sampleX \big) \leq \alpha.
\end{align*}
\end{lemma}

We now exploit these $p$-values to specify a data-dependent selection set $\hat{A}$ that controls the Type I error over $\classOfHolderDistributionsSuperLevelSet$.  First, define a set of hyper-cubes 
\begin{align*}
\setOfHypercubes:=\biggl\{ 2^{-q} \prod_{j=1}^d [{2a_j-1},{2a_j+3}) : (a_1,\ldots,a_d)\in \Z^d,~q\in \N\biggr\}. 
\end{align*}
Now, given $n \in \N$ and $\bm{x}_{1:n}=(x_i)_{i \in [n]} \in (\R^d)^n$, we define 
\begin{align*}
\setOfHypercubes(\bm{x}_{1:n}):=\bigl\{ B \in \setOfHypercubes : \{x_1,\ldots,x_n\}\cap B \neq \emptyset \text{ and }\diamSup(B)\geq 1/n\bigr\},
\end{align*}
so that $|\setOfHypercubes(\bm{x}_{1:n})| \leq 2^dn(2+\log_2 n)$.  The overall algorithm, which applies Holm's procedure \citep{holm1979simple} to the $p$-values in~\eqref{eq:pValueDef} and is denoted by $\hat{A}_{\mathrm{OSS}} \in \hat{\mathcal{A}}_n$, is given in Algorithm~\ref{subsetSelectionAlgo}.  Note that in this algorithm, there is no loss of generality in assuming that $\mathcal{A}$ is ordered, by the well-ordering theorem, which is equivalent to the axiom of choice.  Of course, in most practical settings, there would be a natural ordering on $\mathcal{A}$ induced by an injective map from $\mathcal{A}$ to a Euclidean space with lexicographic ordering; for example, if $\mathcal{A}$ is the set of hyper-rectangles in $\mathbb{R}^d$, then this map could take a given hyper-rectangle $A = \prod_{j=1}^d [a_j,b_j]$ to $(a_1,b_1,\ldots,a_d,b_d) \in \mathbb{R}^{2d}$.  We remark also that in Algorithm~\ref{subsetSelectionAlgo}, it may be the case that there exist $B_1,B_2 \in \setOfHypercubes(\sampleX)$ with $B_1 \subseteq B_2$ and $B_2 \in \{B_{(1)},\ldots,B_{(\ell_\alpha)}\}$, while $B_1 \notin \{B_{(1)},\ldots,B_{(\ell_\alpha)}\}$.  This causes no difficulties for our theory.

\LinesNumbered
\begin{algorithm}[H]
\SetAlgoLined
\textbf{Input: } Data $\sample = \bigl((X_1,Y_1),\ldots,(X_n,Y_n)\bigr) \in (\R^d \times [0,1])^n$, an ordered set $\mathcal{A}$ of subsets of $\R^d$ with $\emptyset \in \mathcal{A}$, $(\tau,\alpha) \in (0,1)^2$, H\"older parameters $(\holderExponent,\holderConstant) \in (0,1] \times [1,\infty)$\;
Compute $\pValueN(B)$ for each $B \in \setOfHypercubes(\sampleX)$ using \eqref{eq:pValueDef} and let $\numberOfConsideredHyperCubes:=|\setOfHypercubes(\sampleX)|$;\\
Enumerate $\setOfHypercubes(\sampleX)$ as $(B_{(\ell)})_{\ell\in [\numberOfConsideredHyperCubes]}$, in such a way that $\hat{p}_n(B_{(\ell)})\leq \hat{p}_n(B_{(\ell')})$ for $\ell \leq \ell'$;\\
\uIf{$\numberOfConsideredHyperCubes \cdot \hat{p}_n(B_{(1)})\leq \alpha$}{

    Compute~$\ell_{\alpha}:=\max\bigl\{\ell \in [\numberOfConsideredHyperCubes]: (\numberOfConsideredHyperCubes+1-\ell')\cdot \hat{p}_n(B_{(\ell')}) \leq \alpha \ \forall \ell' \leq \ell\bigr\}$\;
    Choose $\hat{A}_{\mathrm{OSS}}(\sample) := \sargmax\big\{ \empiricalMarginalDistribution(A):A\in \mathcal{A}\cap \powerSet\bigl(\bigcup_{\ell \in [\ell_{\alpha}]}B_{(\ell)}\bigr) \big\} $\; 
  }
  \Else{
    Set $\hat{A}_{\mathrm{OSS}}(\sample)=\emptyset$\;
  }
\KwResult{The selected set $\hat{A}_{\mathrm{OSS}}(\sample)$.}

\caption{\label{subsetSelectionAlgo} The data-dependent selection set $\hat{A}_{\mathrm{OSS}}$}
\end{algorithm}

Proposition~\ref{thm:typeIControl} below provides part of the proof of the upper bound in Theorem~\ref{thm:minimaxRate}.
\begin{prop}
\label{thm:typeIControl} 
Let $\alpha \in (0,1)$ and $(\holderExponent,\holderConstant) \in (0,1]\times [1,\infty)$. Then the  data-dependent selection set $\hat{A}_{\mathrm{OSS}}$ controls the Type I error at level $\alpha$ over $\classOfHolderDistributionsSuperLevelSet$; in other words, $\hat{A}_{\mathrm{OSS}} \in \hat{\mathcal{A}}_n\bigl(\tau,\alpha,\classOfHolderDistributionsSuperLevelSet\bigr)$.
\end{prop}
Proposition~\ref{thm:powerBound} complements Proposition~\ref{thm:typeIControl} by bounding the regret $R_{\tau}(\hat{A}_{\mathrm{OSS}})$, and together these results prove the upper bound in Theorem~\ref{thm:minimaxRate}.  In fact, we provide a high-probability bound as well as an expectation bound.
\begin{prop}
\label{thm:powerBound} 
\sloppy Take $\tau, \alpha \in (0,1)$, $\holderExponent \in (0,1]$, $\holderConstant \geq 1$, $\approximableDensityExponent,\approximableMarginExponent > 0$, $\approximableSetsConstant \geq 1$ and $\mathcal{A} \subseteq \borel(\R^d)$ with $\vcDim(\mathcal{A})<\infty$ and $\emptyset \in \mathcal{A}$. There exists $\tilde{C}\geq 1$, depending only on $d$, $\approximableDensityExponent$, $\approximableMarginExponent$, $\tau$, $\approximableSetsConstant$ and $\vcDim(\mathcal{A})$, such that for all $\probDistribution \in \classOfHolderDistributionsSuperLevelSet \cap \classOfWellApproximableSets$, $n \in \N$ and $\delta \in (0,1)$, we have  
\begin{align*}
\mathbb{P}_P\biggl[M_{\tau}-\mu(\hat{A}_{\mathrm{OSS}}) >\tilde{C}\biggl\{ \biggl({\frac{\holderConstant^{d/\holderExponent}}{n} \cdot \log_+\Bigl(\frac{n}{ \alpha\wedge \delta }\Bigr)\biggr)^{\frac{\holderExponent \approximableDensityExponent \approximableMarginExponent}{\approximableDensityExponent(2\holderExponent+d) +\holderExponent\approximableMarginExponent}}}+\biggl(\frac{\log_+(1/\delta)}{n}\biggr)^{1/2}\biggr\}\biggr] \leq \delta. 
\end{align*}
As a consequence, for $\alpha \in (0,1/2]$,
\[
R_\tau(\hat{A}_{\mathrm{OSS}}) \leq C\biggl\{ \biggl({\frac{\holderConstant^{d/\holderExponent}}{n} \cdot \log_+\Bigl(\frac{n}{\alpha}\Bigr)\biggr)^{\frac{\holderExponent \approximableDensityExponent \approximableMarginExponent}{\approximableDensityExponent(2\holderExponent+d) +\holderExponent\approximableMarginExponent}}}+\frac{1}{n^{1/2}}\biggr\},
\]
where $C > 0$ depends only on $\tilde{C}$.
\end{prop}


In the lower bound part of Theorem~\ref{thm:minimaxRate}, we have the condition $\holderExponent \approximableMarginExponent(\approximableDensityExponent-1) < d\approximableDensityExponent$.  A constraint of this form is natural in light of the tension between $\holderExponent$, $\approximableDensityExponent$ and $\approximableMarginExponent$.  In particular, large values of $\approximableDensityExponent$ and $\approximableMarginExponent$ mean that little $\mu$-mass in $\etaSuperLevelSet{\tau}$ is lost by restricting to sets $A \in \mathcal{A}$ for which the lower density of $A$ is not too small, and the regression function on $A$ is uniformly well above~$\tau$; but the smoothness of the regression function constrains the rate of change of $\regressionFunction$, and therefore the extent to which this is possible.  This intuition is formalised in Lemma~\ref{lemma:parameterConstraints}, where we prove that $\holderExponent \approximableMarginExponent(\approximableDensityExponent-1) \leq d\approximableDensityExponent$ provided there exists a distribution in our class for which the pre-image of $\tau$ under $\regressionFunction$ is non-empty, and~$\mu$ is sufficiently well-behaved.  Since the lower bound construction for the proof of Theorem~\ref{thm:minimaxRate}(ii) is common to both the setting of this section and that of the upcoming Section~\ref{Sec:HigherOrder} on higher-order smoothness, we will defer discussion of this construction until Section~\ref{Sec:LowerBound}.

\section{Choice of \texorpdfstring{$\holderConstant$}{lambda} and \texorpdfstring{$\holderExponent$}{beta}}
\label{Sec:Adaptation}

Our original algorithm for computing $\hat{A}_{\mathrm{OSS}}$ takes $\holderConstant$ and $\holderExponent$ as inputs.  In cases where a practitioner is unable to make informed default choices, it is natural to seek to understand the effect of overspecification and underspecification of these parameters, as well as to seek data-driven estimators.  To study the first of these questions, fix $\holderExponent\in (0,1]$ and $\holderConstant \geq 1$, as well as $\holderExponent' \in (0,\holderExponent]$ and $\holderConstant' \geq \holderConstant$.  Since $\classOfHolderDistributionsSuperLevelSetPrime \supseteq \classOfHolderDistributionsSuperLevelSet$, if we apply Algorithm~\ref{subsetSelectionAlgo} with inputs $\holderExponent'$ and $\holderConstant'$, then $\hat{A}_{\mathrm{OSS}} \in \hat{\mathcal{A}}_n\bigl(\tau,\alpha,\classOfHolderDistributionsSuperLevelSetPrime\bigr) \subseteq \hat{\mathcal{A}}_n\bigl(\tau,\alpha,\classOfHolderDistributionsSuperLevelSet\bigr)$; in other words, if we underspecify the smoothness, then we continue to control the Type I error.  On the other hand, we may pay a price in terms of a suboptimal rate of convergence for the regret: in the bounds in Proposition~\ref{thm:powerBound}, we would need to replace $\holderConstant$ and $\holderExponent$ with $\holderConstant'$ and $\holderExponent'$ respectively.  Conversely, if we overspecify the smoothness, then we no longer have guaranteed Type I error control.

In the remainder of this section, we tackle in turn the problems of estimating $\holderConstant$ and $\holderExponent$ from the data.

\subsection{Choice of \texorpdfstring{$\holderConstant$}{lambda}}
\label{SubSec:lambda}

In view of Proposition~\ref{prop:lipschitzNotEnoughForConsistency}, we will need to impose some additional (mild) restrictions on our classes of distributions, as well as a sample size condition, in order to estimate $\holderConstant$ effectively.   First, however, we describe our algorithm to estimate $\holderConstant$ for fixed $\holderExponent$, and demonstrate a sense in which it provides a slightly conservative estimate (Corollary~\ref{Cor:HolderConstantEstimation}).  It is convenient, for $u \in \R \setminus \{0\}$, to define $u/0 := \infty$ if $u>0$ and $u/0 := -\infty$ if $u < 0$.  For each $i$, $k \in [n]$ we let $r_{i,k}\equiv r_{i,k}(\sampleX) := \inf \bigl\{ r \in (0,\infty): |\{X_j\}_{j \in [n]} \cap \closedMetricBallSupNorm{X_i}{r}|\geq k\bigr\}$ denote the $k$th nearest neighbour distance from $X_i$ within $\sampleX$.  Recalling the definition of the empirical regression function $\empiricalRegressionFunction$ from~\eqref{eq:empiricalRegFuncDef}, given $\holderExponent \in (0,1]$, $\delta \in (0,1)$, and $i,j,k,\ell \in [n]$, we define
\begin{align*}
\hat{\phi}_{n,\holderExponent,\delta}(i,j,k,\ell)&:=\frac{\empiricalRegressionFunction\bigl(\closedMetricBallSupNorm{X_i}{r_{i,k}}\bigr) -\empiricalRegressionFunction\bigl(\closedMetricBallSupNorm{X_j}{r_{j,\ell}}\bigr)-\sqrt{2\log(4n^2/\delta)/(k\wedge \ell)}}{ \|X_i-X_j\|_{\infty}^\holderExponent+r_{i,k}^\holderExponent+r_{j,\ell}^\holderExponent }, \\
\hat{\psi}_{n,\holderExponent,\delta}(i,k) &:= \frac{1}{(2r_{i,k})^{\holderExponent}} \cdot \Bigl\{ \empiricalRegressionFunction\bigl(\closedMetricBallSupNorm{X_i}{r_{i,k}}\bigr) - \tau-\sqrt{\log(4n^2/\delta)/(2k)}\Bigr\}.
\end{align*}
We then set
\begin{align*}
\holderConstantEstimator \equiv \holderConstantEstimator(\sample) := 1 \vee\max_{(i,j,k,\ell) \in [n]^4} \min\bigl\{\hat{\phi}_{n,\holderExponent,\delta}(i,j,k,\ell),\hat{\psi}_{n,\holderExponent,\delta}(i,k),\hat{\psi}_{n,\holderExponent,\delta}(j,\ell)\bigr\}.
\end{align*}
To explain the idea behind this construction, suppose that $\probDistribution$ has continuous regression function~$\regressionFunction$ with H\"older constant $\holderConstant$ on $\etaSuperLevelSet{\tau}$, and marginal distribution $\marginalDistribution$ on $\R^d$. Note that $\empiricalRegressionFunction\bigl(\closedMetricBallSupNorm{X_i}{r_{i,k}}\bigr)$ is the $k$-nearest neighbour regression estimate of $\regressionFunction(X_i)$ when the nearest neighbour distances of $X_i$ are distinct.  Thus $\empiricalRegressionFunction\bigl(\closedMetricBallSupNorm{X_i}{r_{i,k}}\bigr) - \sqrt{\log(4n^2/\delta)/(2k)} - \holderConstant \cdot r_{i,k}^\holderExponent$ is a lower confidence bound for $\regressionFunction(X_i)$ and $\empiricalRegressionFunction\bigl(\closedMetricBallSupNorm{X_j}{r_{j,\ell}}\bigr) + \sqrt{\log(4n^2/\delta)/(2\ell)} + \holderConstant \cdot r_{j,\ell}^\holderExponent$ is an upper confidence bound for $\regressionFunction(X_j)$.  It follows that when $X_i,X_j \in \etaSuperLevelSet{\tau}$, we have with high probability that $\hat{\phi}_{n,\holderExponent,\delta}(i,j,k,\ell)$ does not exceed $\holderConstant$ for any $k,\ell \in [n]$.  On the other hand, if $X_i \notin \etaSuperLevelSet{\tau}$,  then with high probability, $\hat{\psi}_{n,\holderExponent,\delta}(i,k) \leq \holderConstant$, and similarly with $j,\ell$ in place of $i,k$.  Thus, with high probability $\holderConstantEstimator \leq \holderConstant$.

Now suppose that there exist well-separated points $x_0$, $x_1 \in \etaSuperLevelSet{\tau}$ such that $\mu$ is well behaved near both $x_0$ and $x_1$, and $\regressionFunction(x_0)-\regressionFunction(x_1) = \tilde{\holderConstant} \cdot \supNorm{x_0-x_1}^\holderExponent$ for some $\tilde{\holderConstant} \leq \holderConstant$, then with high probability, for a sufficiently large sample size, there will also exist data points $X_i$ near $x_0$ and $X_j$ near $x_1$, along with $k$, $\ell \in [n]$ for which $\hat{\phi}_{n,\holderExponent,\delta}(i,j,k,\ell)$ is not too much less than~$\tilde{\holderConstant}$. If, in addition, $\regressionFunction(x_0)$ is not too close to $\tau$ then with high probability, $\hat{\psi}_{n,\holderExponent,\delta}(i,k)$ will also not be much less than $\tilde{\holderConstant}$ for an appropriately chosen $k \in [n]$, and similarly for $x_1$, $X_j$ and $\ell$ in place of $x_0$, $X_i$ and $k$, respectively.  Overall, this ensures that $\holderConstantEstimator$ will be nearly as large as $\tilde{\holderConstant}$ with high probability, for a sufficiently large sample size.

\begin{theorem}\label{thm:lipEst} Let $(\holderExponent,\holderConstant) \in (0,1] \times [1,\infty)$ and $\probDistribution \in \classOfHolderDistributionsSuperLevelSet$.  Then, for $\delta \in (0,1)$,
\begin{align*}
\Prob_P\Biggl( \sup_{x_0,x_1 \in \mathcal{X}_\tau(\eta-\Delta_n)} \biggl\{ \frac{|\eta(x_0)-\eta(x_1)| -\Delta_n(x_0)\vee\Delta_n(x_1)}{\|x_{0}-x_1\|_{\infty}^\holderExponent} \biggr\} \leq \holderConstantEstimator\leq \holderConstant\Biggr) \geq 1-\delta,
\end{align*}
where $\Delta_n \equiv \Delta_{n,\holderExponent,\holderConstant}: \R^d \rightarrow [0,\infty]$ is defined by 
\begin{equation}
\label{Eq:Deltanbeta}
 \Delta_{n}(x) := 192 \cdot \holderConstant^{d/(2\holderExponent+d)}\cdot \biggl(\frac{\log(2n/\delta)}{n\cdot \omega(x)}\biggr)^{\holderExponent/(2\holderExponent+d)},
\end{equation}
with the convention that $\Delta_{n}(x):= \infty$ if $\omega(x) = 0$.
\end{theorem}
An attraction of Theorem~\ref{thm:lipEst} is that it makes no assumptions on $P$ apart from the corresponding regression function satisfying the H\"older condition of Definition~\ref{def:holder}.  We now show that, under further conditions on our class, it is possible to control the lower confidence bound for $\holderConstantEstimator$ to guarantee that with high probability it is within a factor of 2 of the appropriate H\"older constant.  More precisely, for a distribution $P$ on $\R^d \times [0,1]$ having marginal distribution $\mu$ on $\R^d$ and continuous regression function $\eta$, and for $\holderExponent \in (0,1]$, we define the $\holderExponent$-H\"older constant of $P$ to be
\[
\underline{\holderConstant}_{\holderExponent}(P) := \sup \biggl\{\frac{|\eta(x_0)-\eta(x_1)|}{\|x_{0}-x_1\|_{\infty}^\holderExponent} : x_0, x_1 \in \mathrm{supp}(\mu) \cap \etaSuperLevelSet{\tau}, x_0 \neq x_1 \biggr\} \vee 1.
\]
We write $P \in \classOfDistributionsSatisfyingRegularityCondition$ if  $\regressionFunction$ is continuous and satisfies $\eta^{-1}\bigl([\tau,\tau+\varepsilon)\bigr) \subseteq \mathrm{supp}(\mu)$ for some $\varepsilon>0$.  Further, for $\holderConstantUpperBound \in [1,\infty)$ and $\boldsymbol{\epsilon}=(\epsilon_0,\epsilon_1, \epsilon_2) \in (0,1]^3$, we write $P \in \classOfHolderDistributionsSuperLevelSetIdentifiable$ if $\underline{\holderConstant}_{\holderExponent}(P) \leq \holderConstantUpperBound$ and either $\underline{\holderConstant}_{\holderExponent}(P)=1$, or there exist $x_0, x_1 \in \etaSuperLevelSet{\tau+\epsilon_0}$ with $\|x_0-x_1\|_{\infty} \geq \epsilon_1$, as well as $\min\{\omega(x_0),\omega(x_1)\} \geq \epsilon_2$ and 
\begin{align*}
|\eta(x_0)-\eta(x_1)| \geq \frac{3}{4}\cdot  \underline{\holderConstant}_{\holderExponent}(P) \cdot \|x_{0}-x_1\|_{\infty}^\holderExponent.
\end{align*}
The idea here is that if $P \in \classOfHolderDistributionsSuperLevelSetIdentifiable$ and $\underline{\holderConstant}_{\holderExponent}(P) > 1$, then we can find a well-separated pair of points that nearly attains the supremum in the definition of $\underline{\holderConstant}_{\holderExponent}(P)$, as well as belonging comfortably to the $\tau$-super level set of $\regressionFunction$ and having $\mu$ assign  sufficient mass in small neighbourhoods of each of the points.  In Lemma~\ref{Lemma:Inclusion} we show that 
\[
\classOfDistributionsSatisfyingRegularityCondition \cap  \classOfHolderDistributionsSuperLevelSet \subseteq  \bigcup_{\boldsymbol{\epsilon}\in (0,\infty)^3}\classOfHolderDistributionsSuperLevelSetIdentifiable,
\]
so that, under the mild condition that $P \in \classOfDistributionsSatisfyingRegularityCondition$, the additional restriction enforced by $P \in \classOfHolderDistributionsSuperLevelSetIdentifiable$ for small $\epsilon_0, \epsilon_1, \epsilon_2 > 0$ amounts to very little more than asking for $P$ to satisfy the H\"older condition of Definition~\ref{def:holder} for some $\holderConstant \geq 1$.  
\begin{corollary}
\label{Cor:HolderConstantEstimation}
Fix $\holderExponent \in (0,1]$, $\holderConstantUpperBound \in [1,\infty)$, $\boldsymbol{\epsilon}=(\epsilon_0,\epsilon_1, \epsilon_2) \in (0,1]^3$ and take $P \in \classOfDistributionsSatisfyingRegularityCondition\cap  \classOfHolderDistributionsSuperLevelSetIdentifiable$.  Let $n \in \mathbb{N}$ and $\delta \in (0,1)$ be such that 
\begin{equation}
\label{Eq:SampleSizeCondition}
\frac{n}{\log(2n/\delta)} \geq \frac{1}{\epsilon_2} \cdot \max\biggl\{ \Bigl(\frac{192}{\epsilon_0}\Bigr)^{(2\holderExponent+d)/\holderExponent}\holderConstantUpperBound^{d/\holderExponent}, \Bigl(\frac{768}{ \epsilon_1} \Bigr)^{(2\holderExponent+d)/\holderExponent}\biggr\}.
\end{equation}
Then
\[
\Prob_P\biggl(\frac{\underline{\holderConstant}_{\holderExponent}(P)}{2} \leq \holderConstantEstimator \leq \underline{\holderConstant}_{\holderExponent}(P)\biggr) \geq 1 - \delta.
\]
\end{corollary}
Corollary~\ref{Cor:HolderConstantEstimation} reveals that when $P \in \classOfDistributionsSatisfyingRegularityCondition\cap  \classOfHolderDistributionsSuperLevelSetIdentifiable$, the estimator $\holderConstantEstimator$ is reliable in the sense that with high probability, it is within a factor of 2 of the desired $\underline{\holderConstant}_{\holderExponent}(P)$ for a sufficiently large sample size.  However, in order to control Type I error we will in fact use the estimator $2\holderConstantEstimator$, which has the benefit of being slightly conservative, i.e.~it tends to overestimate $\underline{\holderConstant}_{\holderExponent}(P)$, while still being within a factor of 2 of the desired $\underline{\holderConstant}_{\holderExponent}(P)$.  Theorem~\ref{Thm:BetaKnown} summarises our overall guarantees when applying Algorithm~\ref{subsetSelectionAlgo} in this context.  
\begin{theorem}\label{Thm:BetaKnown} Fix $\alpha \in (0,1)$, $d \in \N$, $\holderExponent \in (0,1]$, $\holderConstantUpperBound \in [1,\infty)$ and $\boldsymbol{\epsilon}=(\epsilon_0,\epsilon_1, \epsilon_2) \in (0,1]^3$. Suppose that $n \in \mathbb{N}$ satisfies~\eqref{Eq:SampleSizeCondition} with $\alpha_n:=(\alpha/2)\wedge (1/n)$ in place of $\delta$.  Let $\hat{A}_{\mathrm{OSS}}'$ denote the output of Algorithm~\ref{subsetSelectionAlgo} with inputs $\mathcal{D}$, $\mathcal{A}$, $\tau$,  $\alpha_n$, $\holderExponent$ and $2\holderConstantEstimatorArgs{\holderExponent}{\alpha_n}$.  Then

\noindent\emph{\textbf{(i) Type I error:}} $\hat{A}_{\mathrm{OSS}}' \in \hat{\mathcal{A}}_n\bigl(\tau,\alpha,\classOfDistributionsSatisfyingRegularityCondition\cap  \classOfHolderDistributionsSuperLevelSetIdentifiable\bigr)$.

\noindent\emph{\textbf{(ii) Regret:}} Now suppose further that $\mathcal{A}$ satisfies $\vcDim(\mathcal{A})<\infty$ and $\emptyset \in \mathcal{A}$, that $\alpha \in (0,1/2]$, and fix $\approximableDensityExponent$, $\approximableMarginExponent$ and $\approximableSetsConstant$.  There exists $C \geq 1$, depending only on $d$, $\approximableDensityExponent$, $\approximableMarginExponent$, $\tau$, $\approximableSetsConstant$ and $\vcDim(\mathcal{A})$, such that for any  $P \in \classOfDistributionsSatisfyingRegularityCondition\cap  \classOfHolderDistributionsSuperLevelSetIdentifiable\cap \classOfWellApproximableSets$, we have
\begin{align}\label{eq:regretKnownBetaThm}
 R_\tau(\hat{A}_{\mathrm{OSS}}') \leq C\biggl\{ \biggl({\frac{\underline{\holderConstant}_{\holderExponent}(P) ^{d/\holderExponent}}{n} \cdot \log_+\Bigl(\frac{n}{\alpha}\Bigr)\biggr)^{\frac{\holderExponent \approximableDensityExponent \approximableMarginExponent}{\approximableDensityExponent(2\holderExponent+d) +\holderExponent\approximableMarginExponent}}}+\frac{1}{n^{1/2}}\biggr\}.
\end{align}
\end{theorem}
The main message of Theorem~\ref{Thm:BetaKnown} is that replacing $\holderConstant$ with $2\holderConstantEstimatorArgs{\holderExponent}{\alpha/2}$ in Algorithm~\ref{subsetSelectionAlgo} retains both the Type~I error validity and the regret guarantees when the sample size is sufficiently large and provided that we make the slight restrictions to our classes of distributions mentioned above.  An immediate consequence of this result together with Lemma~\ref{Lemma:Inclusion} is that for any  $P \in \classOfDistributionsSatisfyingRegularityCondition \cap \bigcup_{\holderConstant \in [1,\infty)} \classOfHolderDistributionsSuperLevelSet$, there exists  $N_\alpha(P)\in \N$ such that for all $n \geq N_\alpha(P)$ we have both Type I error control, i.e.~$\Prob_P\bigl( \hat{A}(\sample) \subseteq{} \etaSuperLevelSet{\tau}\bigr)\geq 1-\alpha$, and the regret guarantee~\eqref{eq:regretKnownBetaThm}.  An attraction of this approach to choosing the input $\holderConstant$ is that we avoid sample splitting.

\subsection{Choice of \texorpdfstring{$\holderExponent$}{beta}}
\label{SubSec:beta}

\newcommand{\slowlyIncreasingFunctionHolderExponent}{f}
\newcommand{\holderExponentEstJoint}[1]{\hat{\holderExponent}_{n,#1}}
\newcommand{\holderConstantEstJoint}[1]{\hat{\holderConstant}_{n,\holderExponentEstJoint{#1},#1}}

Since the algorithm described in Section~\ref{SubSec:lambda} takes $\beta \in (0,1]$ as an input, we now describe a data-driven algorithm for estimating this parameter.  In contrast to identifying the H\"older constant, identifying the H\"older exponent requires an analysis of the behaviour of the regression function at multiple scales.  This motivates the introduction of distributional classes for which the corresponding regression functions exhibit `self-similar' behaviour; see Definition~\ref{def:holderSelfSimilar} below.  Related ideas have appeared in the adaptive confidence band literature \citep{picard2000adaptive,GineNicklAOS738,bull2012honest,gur2022smoothness}. First, given a Borel measure $\mu$ on $\R^d$ and $\lowerDensityConstantSelfSimilar \in (0,\infty)$, we let 
\begin{align*}
\regularSetSelfSimilar:= \bigcap_{r \in (0,1)} \bigl\{ x \in \R^d :  \mu\bigl( \closedMetricBallSupNorm{x}{r}\bigr) \geq \lowerDensityConstantSelfSimilar \cdot r^d \bigr\}.
\end{align*}
Thus $\regularSetSelfSimilar$ denotes the set of points $x$ for which the $\mu$-measure of small balls centred at~$x$ can be bounded below by their volumes, up to constants.
\begin{defn}[Self-similar H\"{o}lder class]\label{def:holderSelfSimilar} Given $\holderExponent \in (0,1]$, $\holderConstant \in [1,\infty)$, $\lowerHolderConstantSelfSimilar\in (0,\infty)$, $\lowerDensityConstantSelfSimilar, \minRadiusSelfSimilar \in (0,1]$, we let $\classOfSelfSimilarHolderDistributions$ denote the class of all distributions $\probDistribution$ on $\R^d \times [0,1]$ with regression function $\regressionFunction \in \classOfRestrictedHolderFunctions{\R^d}$ such that for all $r \in (0,r_0]$ there exist $x_0, x_1 \in \regularSetSelfSimilar$ such that $\supNorm{x_0-x_1} \leq r$ and $|\regressionFunction(x_0)-\regressionFunction(x_1)| \geq \lowerHolderConstantSelfSimilar \cdot r^\holderExponent$.
\end{defn}
We view the regression functions $\regressionFunction$ of distributions $P \in \classOfSelfSimilarHolderDistributions$ as self-similar since the fluctuations permitted by the $\holderExponent$-H\"older constraint are exhibited at every sufficiently small scale.  For $\delta \in (0,1)$, we aim to define estimators $(\holderExponentEstJoint{\delta},\holderConstantEstJoint{\delta})$ of $(\holderExponent,\holderConstant)$.  To this end, for $i,j,k,\ell \in [n]$, let
\begin{align*}
\hat{\varepsilon}^\dagger_{n,\holderExponent,\delta}(i,j,k,\ell) &:= -\log\bigl( \|X_i-X_j\|_{\infty}+r_{i,k}+r_{j,\ell}\bigr),\\
\hat{\varphi}^\dagger_{n,\holderExponent,\delta}(i,j,k,\ell)&:=-\log\Bigl(\bigr|\empiricalRegressionFunction\bigl(\closedMetricBallSupNorm{X_i}{r_{i,k}}\bigr) -\empiricalRegressionFunction\bigl(\closedMetricBallSupNorm{X_j}{r_{j,\ell}}\bigr)\bigr|-\sqrt{2\log(4n^2/\delta)/(k\wedge \ell)}\Bigr),
\end{align*}
with the convention that $-\log z := \infty$ for $z \leq 0$. We then define 
\begin{align*}
\hat{\Gamma}_{n,\delta}^\dagger:= \bigl\{ \bigl(\hat{\varepsilon}^\dagger_{n,\holderExponent,\delta}(i,j,k,\ell), \hat{\varphi}^\dagger_{n,\holderExponent,\delta}(i,j,k,\ell) \bigr) ~:~ (i,j,k,\ell) \in [n]^4,~\hat{\varphi}^\dagger_{n,\holderExponent,\delta}(i,j,k,\ell)<\infty\bigr\}.
\end{align*}
Finally, letting $\slowlyIncreasingFunctionHolderExponent:\N \rightarrow [1,\infty)$ denote any increasing function satisfying $f(n) \rightarrow \infty$ as $n \rightarrow \infty$, we define
\begin{align*}
\holderExponentEstJoint{\delta}\equiv\holderExponentEstJoint{\delta}(\sample):= 0 \vee  ~\max_{(u_0,v_0) \in \hat{\Gamma}_{n,\delta}^\dagger~:~u_0 \leq \frac{\log n}{6+2d}}~\min_{(u_1,v_1) \in \hat{\Gamma}_{n,\delta}^\dagger~:~u_1 \geq 2u_0} \frac{v_1-v_0-\log \slowlyIncreasingFunctionHolderExponent(n)}{u_1-u_0},
\end{align*}
with the conventions that $\max \emptyset := -\infty$ and $\min \emptyset:= \infty$. Theorem~\ref{thm:holderExponentEstimationThm} below provides a high-probability guarantee on the performance of~$\holderExponentEstJoint{\delta}$.

\begin{theorem}\label{thm:holderExponentEstimationThm} Fix $\holderExponent \in (0,1]$, $d \in \N$, $\holderConstant \in [1,\infty)$ as well as $\lowerHolderConstantSelfSimilar \in (0,\holderConstant]$, $\lowerDensityConstantSelfSimilar, \minRadiusSelfSimilar \in (0,1]$ and $P \in \classOfSelfSimilarHolderDistributions$. Then for $n \in \N$ such that $\slowlyIncreasingFunctionHolderExponent(n) \geq 14\holderConstant/\lowerHolderConstantSelfSimilar$ and
\begin{align}\label{eq:sampleSizeConditionFromThmholderExponentEstimationThm}
 n \geq \max\biggl\{\frac{1}{r_0^{(7+2d)}},\frac{8(16\holderConstant)^{d/\holderExponent}\log(4n^2/\delta)}{2^{7d/2\holderExponent} \lowerDensityConstantSelfSimilar},\biggl(\frac{2^{10}7^{2\holderExponent+d}(16\holderConstant/\lowerHolderConstantSelfSimilar)^{d/\holderExponent}\log(4n^2/\delta)}{\lowerDensityConstantSelfSimilar\cdot \lowerHolderConstantSelfSimilar^2} \biggr)^{3+d} \biggr\},
\end{align}
we have 
\begin{align*}
\Prob_P\biggl(\holderExponent -\frac{2(7+2d)\log \slowlyIncreasingFunctionHolderExponent(n)}{\log n}\leq \holderExponentEstJoint{\delta} \leq \holderExponent\biggr)\geq 1-\delta.
\end{align*}
\end{theorem}

Theorem~\ref{thm:holderExponentEstimationThm} ensures that $\holderExponentEstJoint{\delta}$ is a uniformly consistent estimator over each of our self-similar classes.  Moreover, analogously to Corollary~\ref{Cor:HolderConstantEstimation}, it tends to slightly underestimate the true H\"older exponent, which is again advantageous for establishing our overall guarantees on the performance of Algorithm~\ref{subsetSelectionAlgo} with this data-driven choice of $\holderExponent$.
\begin{theorem}\label{Thm:BetaUnknown} Fix $\alpha \in (0,1)$, $\holderExponent \in (0,1]$, $d \in \N$, $\holderConstant \in [1,\infty)$, $\lowerHolderConstantSelfSimilar \in (0,\holderConstant]$, $\lowerDensityConstantSelfSimilar, \minRadiusSelfSimilar \in (0,1]$ and $\boldsymbol{\epsilon}=(\epsilon_0,\epsilon_1, \epsilon_2) \in (0,1]^3$. Let $\tilde{\alpha}_n := (\alpha/3) \wedge (1/n)$. Suppose that $n \in \N$ satisfies $ 14\holderConstant/\lowerHolderConstantSelfSimilar \leq \slowlyIncreasingFunctionHolderExponent(n) \leq n^{\frac{\log(9/7)}{2(7+2d)\log(1/\epsilon_2)}}$ and 
\begin{align*}
n \geq \max\biggl\{\frac{1}{r_0^{(7+2d)}},\frac{8(16\holderConstant)^{d/\holderExponent}\log(12n^3/\alpha)}{2^{7d/2\holderExponent} \lowerDensityConstantSelfSimilar},\biggl(\frac{2^{10}7^{2\holderExponent+d}(16\holderConstant/\lowerHolderConstantSelfSimilar)^{d/\holderExponent}\log(12n^3/\alpha)}{\lowerDensityConstantSelfSimilar\cdot \lowerHolderConstantSelfSimilar^2} \biggr)^{3+d},\\
 \frac{\log(6n^2/\alpha)}{\epsilon_2}\cdot \biggl[  \slowlyIncreasingFunctionHolderExponent(n)^{2(7+2d)}\cdot \Bigl( 192 \cdot \max\Bigl\{ \frac{\holderConstant}{\epsilon_0},\frac{12}{\epsilon_1}\Bigr\}\Bigr)^{2+d}\biggr]^{1/\holderExponent} \biggr\}.
\end{align*}
Let $\hat{A}_{\mathrm{OSS}}''$ denote the output of Algorithm~\ref{subsetSelectionAlgo} with input $\tilde{\alpha}_n$ in place of $\alpha$, $\hat{\holderExponent}_{n,\tilde{\alpha}_n}$ in place of $\holderExponent$ and $2\hat{\holderConstant}_{n,\hat{\holderExponent}_{n,\tilde{\alpha}_n},\tilde{\alpha}_n}$ in place of $\holderConstant$. 
Then

\noindent\emph{\textbf{(i) Type I error:}} $\hat{A}_{\mathrm{OSS}}'' \in \hat{\mathcal{A}}_n\bigl(\tau,\alpha, \classOfDistributionsSatisfyingRegularityCondition\cap  \classOfHolderDistributionsSuperLevelSetIdentifiable \cap \classOfSelfSimilarHolderDistributions\bigr)$.

\noindent\emph{\textbf{(ii) Regret:}} Now suppose further that $\mathcal{A}$ satisfies $\vcDim(\mathcal{A})<\infty$ and $\emptyset \in \mathcal{A}$, that $\alpha \in (0,1/2]$, and fix $\approximableDensityExponent$, $\approximableMarginExponent$ and $\approximableSetsConstant$.  There exists $C \geq 1$, depending only on $d$, $\approximableDensityExponent$, $\approximableMarginExponent$, $\tau$, $\approximableSetsConstant$ and $\vcDim(\mathcal{A})$, such that for $P \in  \classOfDistributionsSatisfyingRegularityCondition\cap  \classOfHolderDistributionsSuperLevelSetIdentifiable \cap \classOfSelfSimilarHolderDistributions\cap \classOfWellApproximableSets$, we have
\begin{align}\label{eq:regretUnnownBetaThm}
 R_\tau(\hat{A}_{\mathrm{OSS}}'') \leq C\biggl\{\slowlyIncreasingFunctionHolderExponent(n)^{4\approximableMarginExponent(7+2d)/d} \biggl({\frac{\underline{\holderConstant}_{\holderExponent}(P) ^{d/\holderExponent}}{n} \cdot \log_+\Bigl(\frac{n}{\alpha}\Bigr)\biggr)^{\frac{\holderExponent \approximableDensityExponent \approximableMarginExponent}{\approximableDensityExponent(2\holderExponent+d) +\holderExponent\approximableMarginExponent}}}+\frac{1}{n^{1/2}}\biggr\}. 
\end{align}
\end{theorem}
Theorem~\ref{Thm:BetaUnknown} confirms that when applying Algorithm~\ref{subsetSelectionAlgo} with our data-driven choices of $\holderExponent$ and $\holderConstant$, we maintain large-sample Type I error control over appropriate classes, and only lose a sub-logarithmic factor in $n$ in terms of regret.


\section{Higher-order smoothness}\label{Sec:HigherOrder}

In this section, we explain how the procedure and analysis of Section~\ref{Sec:MainResults} can be modified and extended to cover a general smoothness level $\holderExponent > 0$ for the regression function. Given $\nu=(\nu_1,\ldots,\nu_d)^\top \in \N_0^d$ and $x=(x_1,\ldots,x_d)^\top \in \R^d$, we define $\|\nu\|_1:=\sum_{j=1}^d \nu_j$, $\nu!:=\prod_{j=1}^d \nu_j!$ and $x^\nu:=\prod_{j=1}^d x_j^{\nu_j}$.  For an $\|\nu\|_1$-times differentiable function $g: \R^d \rightarrow \R$, define $\partial^\nu_x (g):= \frac{\partial^{\|\nu\|_1}g}{\partial x_1^{\nu_1}\ldots \partial x_d^{\nu_d}}(x)$.  Given $\holderExponent \in (0,\infty)$ we let $\mathcal{V}(\holderExponent):=\{\nu \in \N_0^d:\|\nu\|_1\leq \lceil \holderExponent \rceil -1\}$, so that $|\mathcal{V}(\holderExponent)| = \binom{\lceil \holderExponent \rceil + d - 1}{d}$, and for a $(\lceil \holderExponent \rceil -1)$-times differentiable function $g:\R^d \rightarrow \R$, let $\taylorSeries_{x}^{\holderExponent}[g]:\R^d \rightarrow \R$ denote the associated Taylor polynomial at $x \in \mathbb{R}^d$, defined by
\begin{align*}
\taylorSeries_{x}^{\holderExponent}[g](x'):= \sum_{\nu \in \mathcal{V}(\holderExponent)}\frac{(x'-x)^\nu}{\nu!}\cdot \partial^\nu_x(g),
\end{align*}
for $x' \in \R^d$.
\begin{defn}[General H\"{o}lder class]\label{def:holderHigherOrder} Given $(\holderExponent,\holderConstant) \in (0,\infty)\times [1,\infty)$, let $\tilde{\measureClass}_{\mathrm{H\ddot{o}l}}(\holderExponent,\holderConstant)$ denote the class of all distributions $\probDistribution$ on $\R^d \times [0,1]$ such that the associated regression function $\regressionFunction:\R^d \rightarrow [0,1]$ is $(\lceil \holderExponent \rceil-1)$-differentiable and satisfies
\begin{align*}
\bigl|\regressionFunction(x') - \taylorSeries_{x}^{\holderExponent}[\regressionFunction](x')\bigr|\leq \holderConstant \cdot \supNorm{x'-x}^{\holderExponent},
\end{align*}
for all $x, x' \in \R^d$. Moreover, we let $ \classOfHolderDistributions:= \bigcap_{\holderExponent' \in (0, \holderExponent]}\tilde{\measureClass}_{\mathrm{H\ddot{o}l}}(\holderExponent',\holderConstant)$.
\end{defn}

Throughout this section, and in contrast to Section~\ref{Sec:MainResults}, we will require that the marginal distribution $\mu$ is absolutely continuous with respect to Lebesgue measure, and write $f_\mu$ for its density. Given a probability measure $\mu$ on $\R^d$ and some $\regularityConstant \in (0,1)$, we define
\begin{align}\label{eq:defRegularSet}
\muRegularSet:=\bigcap_{r \in (0,1)}\biggl\{ x \in \R^d: \mu\bigl(\openMetricBallSupNorm{x}{r}\bigr) \geq \regularityConstant \cdot r^d \cdot \sup_{x' \in \openMetricBallSupNorm{x}{(1+\regularityConstant)r}}f_\mu(x')\biggr\}.
\end{align}
To provide some intuition about $\muRegularSet$, consider first a simple example where $\mu$ denotes the $N(0,\sigma^2)$ distribution.  In that case, the point $x=0$ belongs to $\mathcal{R}_{\upsilon}(\mu)$ if and only if $\upsilon \leq \sqrt{2\pi} \sigma\bigl\{2\Phi(1/\sigma)-1\bigr\}$.  In particular, we must have $\upsilon \leq \sqrt{2\pi} \sigma$, and (since we only consider $\upsilon < 1$), it suffices that $\upsilon \leq 2\sigma$.  More generally, if $S \subseteq \mathrm{supp}(\mu)$ is a $(c_0,r_0)$-regular set \citep[][Equation~(2.1)]{audibert2007fast} and if $\mu$ is absolutely continuous with respect to~$\Lebesgue$ with corresponding density $f_\mu$ satisfying the condition that $K_\mu := \sup_{x \in \mathrm{supp}(\mu)} f_\mu(x)/\inf_{x \in S} f_\mu(x) < \infty$, then $S \subseteq \muRegularSet$ for $\regularityConstant \leq c_0 \cdot V_d \cdot (r_0 \wedge 1)^d\cdot K_{\mu}^{-1}$.  Moreover, we can still have $\muRegularSet = \mathrm{supp}(\mu)$ even when $\mu$ is not compactly supported and there is no uniform positive lower bound for~$f_\mu$ on its support.  For instance, the family of probability measures $\{\mu_\kappa: \kappa \in (0,1)\}$ on $\R^d$ considered in Example~\ref{ex:3} satisfies $\mathcal{R}_{\regularityConstant}(\mu_\kappa) = \R^d$ for $\regularityConstant \leq 2^d \cdot \bigl\{1 + 3^d(1-\kappa)\bigr\}^{-1/(1-\kappa)}$.  Finally, if $\log f_\mu$ if $L$-Lipschitz with respect to the supremum norm, then $\mathcal{R}_{\upsilon}(\mu) = \R^d$ provided that $\upsilon \leq 2^d e^{-3L}$. 

We are now in a position to define an appropriate definition of approximable classes for regression functions with higher-order smoothness.
\begin{defn}[Approximable classes for higher-order smoothness]\label{defn:approximableMeasureClassHO} Given $\mathcal{A} \subseteq \borel(\R^d)$ and $(\approximableDensityExponent,\approximableMarginExponent,\regularityConstant,\approximableSetsConstant)\in (0,\infty)^2\times (0,1)\times [1,\infty)$, we let $\classOfWellApproximableSetsHO$ denote the class of all distributions $\probDistribution$ on $\R^d \times [0,1]$ with marginal~$\mu$ on $\R^d$ and a continuous regression function $\regressionFunction:\R^d \rightarrow [0,1]$ such that
\begin{align}\label{eq:FromDefApproximableMeasureClassHO}
\sup\bigl\{ \mu(A):A \in \mathcal{A}\cap  \powerSet\bigl(\muRegularSet \cap \densitySuperLevelSet{\xi}\cap \etaSuperLevelSet{\tau+\Delta}\bigr)  \bigr\} \geq M_\tau - \approximableSetsConstant \cdot (\xi^{\approximableDensityExponent}+\Delta^{\approximableMarginExponent}),
\end{align}
for all $(\xi,\Delta) \in (0,\infty)^2$.
\end{defn}
Finally, then, we can state the main theorem of this section.
\begin{theorem}
\label{Thm:minimaxRateHOS}
Take $(\tau,\alpha) \in (0,1)^2$, $(\holderExponent,\holderConstant) \in (0,\infty) \times [1,\infty)$ and $(\approximableDensityExponent,\approximableMarginExponent,\regularityConstant,\approximableSetsConstant)\in (0,\infty)^2 \times (0,1) \times [1,\infty)$.  

\medskip

\noindent\emph{\textbf{(i) Upper bound:}} Let $\mathcal{A} \subseteq \borel(\R^d)$ satisfy $\vcDim(\mathcal{A}) < \infty$ and $\emptyset \in \mathcal{A}$.  Then there exists $C\geq 1$, depending only on $d$, $\holderExponent$, $\approximableDensityExponent$, $\approximableMarginExponent$, $\regularityConstant$, $\approximableSetsConstant$ and $\vcDim(\mathcal{A})$, such that for all $n \in \N$ and $\alpha \in (0,1/2]$, we have
\begin{equation}\label{Eq:MinimaxUpperBoundHO}
\inf_{\hat{A}} \, \sup_{\probDistribution} \,  R_\tau(\hat{A}) \leq {C} \cdot \min\biggl\{\biggl({\frac{\holderConstant^{d/\holderExponent}\cdot \log_+(n/ \alpha)}{n}}\biggr)^{\frac{\holderExponent \approximableDensityExponent \approximableMarginExponent}{\approximableDensityExponent(2\holderExponent+d)+\holderExponent\approximableMarginExponent}}+\frac{1}{n^{1/2}},1\biggr\},
\end{equation}
where the infimum in~\eqref{Eq:MinimaxUpperBoundHO} is taken over $\hat{\mathcal{A}}_n\bigl(\tau,\alpha,\classOfHolderDistributions\bigr)$ and the supremum is taken over $\classOfHolderDistributions \cap \classOfWellApproximableSetsHO$.

\medskip

\noindent \emph{\textbf{(ii) Lower bound:}} Now suppose that $\holderExponent \approximableMarginExponent(\approximableDensityExponent-1) < d\approximableDensityExponent$, $\epsilon_0 \in (0,1/2)$, $\tau \in (\epsilon_0,1-\epsilon_0)$, $\alpha \in (0,1/2-\epsilon_0]$ and $\regularityConstant \in \bigl(0,(4d^{1/2})^{-d}\bigr]$.  Then there exists $c > 0$, depending only on $d$, $\holderExponent$, $\approximableDensityExponent$, $\approximableMarginExponent$, $\approximableSetsConstant$ and $\epsilon_0$, such that for any $\mathcal{A} \subseteq \borel(\R^d)$ satisfying $\mathcal{A}_{\mathrm{hpr}}  \subseteq \mathcal{A} \subseteq \mathcal{A}_{\mathrm{conv}}$ and any $n \in \N$, we have 
\begin{equation}
    \label{Eq:MinimaxLowerBoundHO}
\inf_{\hat{A}} \, \sup_{\probDistribution} \,  R_\tau(\hat{A}) \geq {c} \cdot \min\biggl\{ \biggl({\frac{\holderConstant^{d/\holderExponent}\cdot \log_+\{ n/ (\holderConstant^{d/\holderExponent} \alpha)\}}{n}}\biggr)^{\frac{\holderExponent \approximableDensityExponent \approximableMarginExponent}{\approximableDensityExponent(2\holderExponent+d)+\holderExponent\approximableMarginExponent}}+\frac{1}{n^{1/2}},1\biggr\},
\end{equation}
where, again, the infimum in~\eqref{Eq:MinimaxLowerBoundHO} is taken over $\hat{\mathcal{A}}_n\bigl(\tau,\alpha,\classOfHolderDistributions\bigr)$ and the supremum is taken over $\classOfHolderDistributions \cap \classOfWellApproximableSetsHO$.
\end{theorem}
In order to prove the upper bound in Theorem~\ref{Thm:minimaxRateHOS}, we will introduce a modified algorithm.  The key alteration is a different choice of $p$-values that now makes use of data points outside (as well as within) our hyper-cube of interest to test whether or not the regression function is uniformly above $\tau$ on the hyper-cube. Given $\holderExponent \in (0,\infty)$, $x,x' \in \R^d$, $h \in (0,1]$ and $P \in \classOfHolderDistributions$ with regression function $\regressionFunction$, we let 
\begin{align*}
\Phi_{x,h}^\holderExponent(x'):= \biggl(\Bigl(\frac{x'-x}{h}\Bigr)^\nu\biggr)_{\nu \in \mathcal{V}(\holderExponent)} \in \R^{\mathcal{V}(\holderExponent)}\ \ \text{and} \ \ 
w_{x,h}^\holderExponent:=\biggl(
\frac{h^{\|\nu\|_1}}{\nu!}\cdot \partial^\nu_x(\regressionFunction)
\biggr)_{\nu \in \mathcal{V}(\holderExponent)}\in \R^{\mathcal{V}(\holderExponent)},
\end{align*}
so that $\taylorSeries_{x}^{\holderExponent}[g](x') = \langle w_{x,h}^\holderExponent, \Phi_{x,h}^\holderExponent(x')\rangle$ for all $x,x' \in \R^d$ and $h \in (0,1]$, where $\langle \cdot, \cdot \rangle$ denotes the Euclidean inner product.  Moreover, if we let $e_0:=\bigl(\one_{\{\nu = (0,\ldots,0)^\top\}}\bigr)_{\nu \in \mathcal{V}(\holderExponent)} \in \R^{\mathcal{V}(\holderExponent)}$, then $\eta(x)=\langle e_0,w_{x,h}^\holderExponent\rangle$. A natural estimator of $w_{x,h}^\holderExponent$ is the local polynomial estimator obtained by taking $\mathcal{N}_{x,h}:=\{i \in [n]~:~X_i \in \closedMetricBallSupNorm{x}{h}\}$ and letting
\begin{align*}
\hat{w}_{x,h}^\holderExponent \in \argmin_{w \in \R^{\mathcal{V}(\holderExponent)}} \sum_{i \in \mathcal{N}_{x,h}} \bigl( Y_i-\big\langle w, \Phi_{x,h}^\holderExponent(X_i)\big\rangle\bigr)^2.
\end{align*}
In fact, it will be convenient to choose a particular element of this $\argmin$: if we define
\begin{align*}
V_{x,h}^{\holderExponent}&:=\sum_{i \in \mathcal{N}_{x,h}}  Y_i  \cdot \Phi_{x,h}^\holderExponent(X_i)\in \R^{\mathcal{V}(\holderExponent)}\\
Q_{x,h}^{\holderExponent}&:=\sum_{i \in \mathcal{N}_{x,h}} \Phi_{x,h}^\holderExponent(X_i) \Phi_{x,h}^\holderExponent(X_i)^\top\in \R^{\mathcal{V}(\holderExponent) \times \mathcal{V}(\holderExponent)},
\end{align*}
then we will take $\hat{w}_{x,h}^\holderExponent :=\bigl(Q_{x,h}^{\holderExponent}\bigr)^+ V_{x,h}^{\holderExponent}$.  Thus, $\hat{\regressionFunction}(x) := 0 \vee \bigl(1\wedge\langle e_0, \hat{w}_{x,h}^\holderExponent \rangle\bigr)$ is an estimator of $\regressionFunction(x)$. Next we associate a $p$-value to closed hyper-cubes $B\subseteq \R^d$ with $\diamSup(B) \leq 1$ as follows.  Let $x \in \R^d$ and $r \in [0,1/2]$ denote the centre and $\ell_\infty$-radius of $B$, so that $B=\closedMetricBallSupNorm{x}{r}$. Let $h :=(2r)^{1 \wedge \frac{1}{\holderExponent}} \in [0,1]$, and define 
\begin{align}
\pValueNHO&(B) \phantom{:}\equiv \pValueLongHO(B) \label{eq:pValueDefHigherOrderSmoothness}\\
&:= \exp\biggl\{-\frac{2}{e_0^\top \bigl(Q_{x,h}^{\holderExponent}\bigr)^{-1} e_0}\biggl( \hat{\regressionFunction}(x) - \tau -  \holderConstant \Bigl(1+2\sqrt{e_0^\top \bigl(Q_{x,h}^{\holderExponent}\bigr)^{-1} e_0 \cdot |\mathcal{N}_{x,h}|}\Bigr)  r^{\holderExponent \wedge 1}\biggr)^2\biggr\},\nonumber
\end{align}
whenever $Q_{x,h}^{\holderExponent}$ is invertible and $\hat{\regressionFunction}(x) \geq \tau + \holderConstant \Bigl(1+2\sqrt{e_0^\top (Q_{x,h}^{\holderExponent})^{-1} e_0 \cdot |\mathcal{N}_{x,h}|}\Bigr)  r^{\holderExponent \wedge 1}$, and $\pValueNHO(B):=1$ otherwise. Lemma \ref{lemma:pValueHigherOrderSmoothness} in Section~\ref{Sec:minimaxRateHOSProof} shows that these are indeed $p$-values. 

We will also make use of an alternative set of hyper-cubes 
\begin{align*}
\setOfHypercubesHO:=\biggl\{ 2^{-q} \prod_{j=1}^d [a_j,a_j+1] : (a_1,\ldots,a_d)\in \Z^d,~q\in \N\biggr\}.
\end{align*}
Now, given $n \in \N$, $\bm{x}_{1:n}=(x_i)_{i \in [n]} \in (\R^d)^n$, we define 
\begin{align}\label{eq:updatedHyperCubesDataDependentHO}
\setOfHypercubesHO(\bm{x}_{1:n}):=\bigl\{ B \in \setOfHypercubesHO : \{x_1,\ldots,x_n\}\cap B \neq \emptyset \text{ and }\diamSup(B)\geq 1/n\bigr\},
\end{align}
so that $|\setOfHypercubesHO(\bm{x}_{1:n})| \leq 2^dn\log_2 n$, and $|\setOfHypercubesHO(\sampleX)| \leq n\log_2 n$ with probability $1$ when $\mu$ is absolutely continuous with respect to Lebesgue measure. We denote our modified procedure for general smoothness, obtained by applying Algorithm~\ref{subsetSelectionAlgo} with the $p$-values \eqref{eq:pValueDefHigherOrderSmoothness} and the hyper-cubes given by~\eqref{eq:updatedHyperCubesDataDependentHO}, as $\hat{A}_{\mathrm{OSS}}^+$.

Propositions~\ref{thm:typeIControlHigherOrderSmoothness} and~\ref{thm:powerBoundHO} are the analogues of Propositions~\ref{thm:typeIControl} and~\ref{thm:powerBound} for $\hat{A}_{\mathrm{OSS}}^+$ and, in combination, prove the upper bound in Theorem~\ref{Thm:minimaxRateHOS}.

\begin{prop}\label{thm:typeIControlHigherOrderSmoothness} Let $\tau \in (0,1)$, $\alpha \in (0,1)$ and $(\holderExponent,\holderConstant) \in (0,\infty)\times [1,\infty)$. Then $\hat{A}_{\mathrm{OSS}}^+ \in \hat{\mathcal{A}}_n\bigl(\tau,\alpha,\classOfHolderDistributions\bigr)$.
\end{prop}

\begin{prop}\label{thm:powerBoundHO}\sloppy Take $\alpha \in (0,1)$, $(\holderExponent,\holderConstant) \in (0,\infty)\times [1,\infty)$, $(\approximableDensityExponent,\approximableMarginExponent,\regularityConstant,\approximableSetsConstant)\in (0,\infty)^2\times (0,1)\times [1,\infty)$ and $\mathcal{A} \subseteq \borel(\R^d)$ with $\vcDim(\mathcal{A})<\infty$ and $\emptyset \in \mathcal{A}$. There exists $\tilde{C}\geq 1$, depending only on $d$, $\holderExponent$,  $\approximableDensityExponent$, $\approximableMarginExponent$, $\regularityConstant$, $\approximableSetsConstant$ and $\vcDim(\mathcal{A})$, such that for all $\probDistribution \in \classOfHolderDistributions \cap \classOfWellApproximableSetsHO$, $n \in \N$ and $\delta \in (0,1)$, we have 
\begin{align*}
\Prob_P\biggl[M_{\tau}-\mu(\hat{A}_{\mathrm{OSS}}^+) > \tilde{C}\biggl\{ \biggl({\frac{\holderConstant^{d/\holderExponent}}{n} \cdot \log_+\Bigl(\frac{n}{\alpha\wedge \delta }\Bigr)\biggr)^{\frac{\holderExponent \approximableDensityExponent \approximableMarginExponent}{\approximableDensityExponent(2\holderExponent+d) +\holderExponent\approximableMarginExponent}}}+\biggl(\frac{\log_+(1/\delta)}{n}\biggr)^{1/2}\biggr\}\biggr] \leq \delta. 
\end{align*}
As a consequence, for $\alpha \in (0,1/2]$,
\[
R_\tau(\hat{A}_{\mathrm{OSS}}^+) \leq C \biggl\{ \biggl({\frac{\holderConstant^{d/\holderExponent}}{n} \cdot \log_+\Bigl(\frac{n}{\alpha}\Bigr)\biggr)^{\frac{\holderExponent \approximableDensityExponent \approximableMarginExponent}{\approximableDensityExponent(2\holderExponent+d) +\holderExponent\approximableMarginExponent}}}+\frac{1}{n^{1/2}}\biggr\},
\]
where $C > 0$ depends only on $\tilde{C}$.
\end{prop}

\section{Lower bound constructions}
\label{Sec:LowerBound}

As mentioned at the end of Section~\ref{Sec:MainResults}, our lower bound constructions are common to both Theorem~\ref{thm:minimaxRate} and Theorem~\ref{Thm:minimaxRateHOS}.  In fact, both lower bounds will follow from Propositions~\ref{lemma:constrainedRiskApplicationfanoApplication} and~\ref{lemma:paramatricLB} below.  As shorthand, given $\mathcal{A}\subseteq \borel(\R^d)$, $\tau \in (0,1)$, $(\holderExponent, \approximableDensityExponent, \approximableMarginExponent) \in (0,\infty)^3$, $\regularityConstant \in (0,1)$ and $(\holderConstant,\approximableSetsConstant) \in [1,\infty)^2$, we write
\begin{align*}
&\classOfLBDistributions:= \\
&\hspace{2cm}\classOfHolderDistributions\cap \classOfWellApproximableSets\cap \classOfWellApproximableSetsHO.
\end{align*}
\begin{prop}\label{lemma:constrainedRiskApplicationfanoApplication} 
Take $\epsilon_0 \in (0,1/2)$, $\tau \in (\epsilon_0,1-\epsilon_0)$,  $\holderExponent > 0$, $\holderConstant \geq 1$, $\approximableDensityExponent, \approximableMarginExponent > 0$, $\approximableSetsConstant \geq 1$ with $\holderExponent\approximableMarginExponent (\approximableDensityExponent-1) < d \approximableDensityExponent$, and $\regularityConstant \in \bigl(0,(4d^{1/2})^{-d}\bigr]$.  Suppose that $\mathcal{A} \subseteq \borel(\R^d)$ satisfies $\mathcal{A}_{\mathrm{hpr}}  \subseteq \mathcal{A} \subseteq \mathcal{A}_{\mathrm{conv}}$.

\medskip

\noindent \textbf{\emph{(i)}} There exists $c_0 > 0$, depending only on $d, \holderExponent,  \approximableMarginExponent,\approximableDensityExponent,  \approximableSetsConstant$ and $\epsilon_0$, such that, for every $\alpha \in (0,1/8]$, $n \in \N$ and $\hat{A}\in \hat{\mathcal{A}}_n\bigl(\tau,\alpha,\classOfLBDistributions\bigr)$, we can find $\probDistribution \in \classOfLBDistributions$ with regression function $\regressionFunction:\R^d \rightarrow [\tau-\epsilon_0/2,\tau+\epsilon_0/2]$ and marginal distribution $\mu$ on $\R^d$, satisfying 
\begin{align*}
\E_{\probDistribution} \bigl[\bigl\{ M_\tau(P,\mathcal{A})-\mu(\hat{A})\bigr\} \cdot \one_{\{\hat{A} \subseteq{}  \etaSuperLevelSet{\tau} \}}\bigr] \geq c_0 \cdot \biggl( \frac{{\holderConstant^{d/\holderExponent}}\log\bigl(1/(4\alpha)\bigr)}{n}\wedge 1\biggr)^{\frac{\holderExponent  \approximableDensityExponent \approximableMarginExponent}{\approximableDensityExponent(2\holderExponent+d)+\holderExponent\approximableMarginExponent}}.
\end{align*}

\medskip 

\noindent \textbf{\emph{(ii)}} There exists $c_1 > 0$, depending only on $d, \holderExponent, \approximableDensityExponent, \approximableMarginExponent,  \approximableSetsConstant$ and $\epsilon_0$, such that, given $\alpha \in \bigl(0,\frac{1}{2} - \epsilon_0\bigr]$, $n \in \N$ and $\hat{A}\in \hat{\mathcal{A}}_n\bigl(\tau,\alpha,\classOfLBDistributions\bigr)$, we can find $\probDistribution \in \classOfLBDistributions$ with regression function $\regressionFunction:\R^d \rightarrow [\tau-\epsilon_0/2,\tau+\epsilon_0/2]$ and marginal distribution $\mu$ on $\R^d$, satisfying 
\begin{align*}
\E_{\probDistribution} \bigl[\bigl\{ M_\tau\bigl(P,\mathcal{A}\bigr)-\mu(\hat{A})\bigr\} \cdot \one_{\{\hat{A} \subseteq{}  \etaSuperLevelSet{\tau} \}}\bigr] \geq c_1 \cdot \biggl( \frac{{\holderConstant^{d/\holderExponent}}\log_+( {n/\holderConstant^{d/\holderExponent}})}{n}\wedge 1\biggr)^{\frac{\holderExponent  \approximableDensityExponent \approximableMarginExponent}{\approximableDensityExponent(2\holderExponent+d)+\holderExponent\approximableMarginExponent}}.
\end{align*}
\end{prop}
Two remarks are in order.  First, note that for any $\hat{A} \in  \hat{\mathcal{A}}_n$, we have
\begin{align*}
R_\tau(\hat{A}) &= \frac{\E_{\probDistribution} \bigl[\bigl\{ M_\tau\bigl(P,\mathcal{A}\bigr)-\mu(\hat{A})\bigr\} \cdot \one_{\{\hat{A} \subseteq{}  \etaSuperLevelSet{\tau} \}}\bigr]}{\Prob_P\bigl(\hat{A} \subseteq{} \etaSuperLevelSet{\tau}\bigr)} \\
&\geq \E_{\probDistribution} \bigl[\bigl\{ M_\tau\bigl(P,\mathcal{A}\bigr)-\mu(\hat{A})\bigr\} \cdot \one_{\{\hat{A} \subseteq{}  \etaSuperLevelSet{\tau} \}}\bigr],
\end{align*}
so Proposition~\ref{lemma:constrainedRiskApplicationfanoApplication} does indeed yield lower bounds on the worst-case regret.  Second, 
\[
\hat{\mathcal{A}}_n\bigl(\tau,\alpha,\classOfLBDistributions\bigr) \supseteq \hat{\mathcal{A}}_n\bigl(\tau,\alpha,\classOfHolderDistributions\bigr),
\]
so it suffices to provide a lower bound for the regret when $\hat{A}$ belongs to the larger set; note also that $\classOfHolderDistributions \subseteq \mathcal{P}_{\mathrm{H\ddot ol}}(\holderExponent,\holderConstant,\tau)$ for $\holderExponent \in (0,1]$.

To lay the groundwork for the proof of Proposition~\ref{lemma:constrainedRiskApplicationfanoApplication}, let $L \in \N$, $r \in (0,\infty)$, $w \in \bigl(0,(2r)^{-d}\wedge 1\bigr)$, $s \in (0,1\wedge(r/2)]$ and $\theta \in (0,\epsilon_0/2]$.  Our goal is to define a collection of probability distributions $\bigl\{\probDistribution^\ell \equiv \probDistribution_{L,r,w,s,\theta}^{\ell}:\ell \in [L]\bigr\}$ on $\R^d \times [0,1]$ as illustrated in Figure~\ref{Fig:LowerBoundConstruction}; we will show that these distributions belong to $\classOfLBDistributions$ for appropriate choices of $L$, $r$, $w$, $s$ and $\theta$, and are such that any data-dependent selection set $\hat{A}\in \hat{\mathcal{A}}_n\bigl(\tau,\alpha,\classOfLBDistributions\bigr)$ must satisfy the lower bound in Proposition~\ref{lemma:constrainedRiskApplicationfanoApplication}.   To this end, let $r_\sharp(w):=\frac{1}{2}\bigl( \{(4\sqrt{d})^d-2^d\}r^d+w^{-1}\bigr)^{1/d}$ and choose $\{{z}_1,\ldots, {z}_L\} \subseteq \R^d$ such that $\supNorm{{z}_{\ell}-{z}_{\ell'}}>2\bigl(r_\sharp(w)+1\bigr)$ for all distinct $\ell, \ell' \in [L]$.  We introduce sets $K^0_{r}(1), \ldots, K^0_{r}(L) \subseteq \R^d$, $K_{r}^1(1), \ldots, K_{r}^1(L) \subseteq \R^d$ defined by $K^0_{r}(\ell):=\closedMetricBallSupNorm{{z}_{\ell}}{r}$ and $K_{r}^1(\ell):=\closedMetricBallSupNorm{{z}_{\ell}}{r_\sharp(w)}\setminus \openMetricBallSupNorm{{z}_{\ell}}{2{d}^{1/2}r}$ for $\ell \in [L]$.  We also define the probability measure $\mu_{L,r,w}$ on $\R^d$ to be the uniform distribution on $J_{L,r,w}:= \bigcup_{(\ell,j) \in  [L]\times\{0,1\}} K_r^j(\ell)$; since $\Lebesgue\bigl(K_r^0(\ell) \cup K_r^1(\ell)\bigr) =  (2r)^d+(2r_\sharp(w))^d-(4d^{1/2}r)^d=w^{-1}$ for all $\ell \in [L]$, it follows that the density of $\mu_{L,r,w}$ with respect to $\Lebesgue$ takes the constant value $w/L$ on $J_{L,r,w}$.


Now define a function $h:[0,1] \rightarrow [0,1]$ by $h(z) := e^{-z^2/(1-z^2)}$ for $z \in [0,1)$ and $h(1) := 0$, so that $h(0)=1$, $\max_{k \in \N} \max_{z \in \{0,1\}} |h^{(k)}(z)|=0$ and 
\begin{equation}
\label{Eq:Am}
A_m := \max_{k \in [m]} \sup_{z \in [0,1]} |h^{(k)}(z)| \in (0,\infty)
\end{equation}
for each $m \in \N$.  This allows us to define regression functions $\eta_{L,r,w,s,\theta}^{\ell}:\R^d \rightarrow [0,1]$ for $\ell \in [L]$ by
\begin{align}\label{def:etaFuncFirstPartLBExtension}
\eta_{L,r,w,s,\theta}^{\ell}(x)\!:=\! \begin{cases} \tau+\theta  &\!\!\text{if } \|x-{z}_\ell\|_2 \leq d^{1/2}r  \\
\tau - \theta &\!\!\text{if } \|x-{z}_{\ell'}\|_2 \leq s \text{ with }\ell' \in [L]\setminus \{\ell\} \\
\tau + \theta - 2\theta h\bigl(\frac{\|x - {z}_{\ell'}\|_2}{s} \!-\! 1\bigr) &\!\!\text{if } s < \|x-{z}_{\ell'}\|_2 \leq 2s \text{ with }\ell' \in [L]\setminus \{\ell\} \\
\tau + \theta &\!\!\text{if } 2s < \|x-{z}_{\ell'}\|_2 \leq d^{1/2}r \text{ with }\ell' \in [L]\setminus \{\ell\} \\
\tau - \theta + 2\theta h\bigl(\frac{\|x - {z}_{\ell'}\|_2 }{d^{1/2}r}\!-\!1\bigr) &\!\!\text{if } d^{1/2}r < \|x-{z}_{\ell'}\|_2 < 2d^{1/2}r \text{ with }\ell' \in [L] \\
\tau-\theta &\text{otherwise.}
\end{cases}
\end{align}

Thus, $\eta_{L,r,w,s,\theta}^{\ell}$ is infinitely differentiable and uniformly above the level $\tau$ on $K_r^0(\ell)$, but both takes a value below $\tau$ at the centre ${z}_{\ell'}$ of each $K_r^0(\ell')$ with $\ell' \in [L] \setminus \{\ell\}$, and is uniformly below the level $\tau$ on each $K_r^1(\ell')$ with $\ell' \in [L]$.  Finally, then, for $\ell \in [L]$, we can let $\probDistribution_{L,r,w,s,\theta}^{\ell}$ denote the unique Borel probability distribution on $\R^d \times \{0,1\}$ with marginal $\mu_{L,r,w}$ on $\R^d$ and regression function $\eta_{L,r,w,s,\theta}^{\ell}$.  Figure~\ref{Fig:LowerBoundConstruction} provides an illustration of the regression functions used in this construction.

\begin{figure}[htbp!]
 \centering
 \begin{tikzpicture}[scale=0.6]
\fill[fill = red!30,] (-9,-9) rectangle (11,11);
\fill[shading = ring1, draw = red!30] (-4,6) circle (2*1.4cm); 
\draw[draw = red!30] (-4,6) circle (2*1.4cm); 
\fill[shading = ring2, draw = red!30] (6,6) circle (2*1.4cm); 
\draw[draw = red!30] (6,6) circle (2*1.4cm); 
\fill[shading = ring2, draw = red!30] (2,-4) circle (2*1.4cm); 
\draw[draw = red!30] (2,-4) circle (2*1.4cm); 
\draw[draw=black] (-4-4,6-4) rectangle (-4+4,6+4); 
\draw[draw=black] (-4-2*1.4,6-2*1.4) rectangle (-4+2*1.4,6+2*1.4); 
\draw[draw=black] (-4-1,6-1) rectangle (-4+1,6+1);
\filldraw[black] (-4,7) circle (0pt) node[anchor=north] {{\tiny $K_{r}^0(\ell)$}};
\filldraw[black] (-4,6) circle (1pt) node[anchor=north] {$z_{\ell}$};
\filldraw[black] (-4,9.8) circle (0pt) node[anchor=north] {\tiny $K_{r}^1(\ell)$};
\draw[draw=black] (6-4,6-4) rectangle (6+4,6+4);
\draw[draw=black] (6-2*1.4,6-2*1.4) rectangle (6+2*1.4,6+2*1.4); 
\draw[draw=black] (6-1,6-1) rectangle (6+1,6+1);
\filldraw[black] (6,7) circle (0pt) node[anchor=north] {\tiny $K_{r}^0(\ell')$};
\filldraw[black] (6,6) circle (1pt) node[anchor=north] {$z_{\ell'}$};
\filldraw[black] (6,9.8) circle (0pt) node[anchor=north] {\tiny $K_{r}^1(\ell')$};
\draw[draw=black] (2-4,-4-4) rectangle (2+4,-4+4);
\draw[draw=black] (2-2*1.4,-4-2*1.4) rectangle (2+2*1.4,-4+2*1.4); 
\draw[draw=black] (2-1,-4-1) rectangle (2+1,-4+1);
\filldraw[black] (2,-3) circle (0pt) node[anchor=north] {\tiny $K_{r}^0(\ell'')$};
\filldraw[black] (2,-4) circle (1pt) node[anchor=north] {$z_{\ell''}$};
\filldraw[black] (2,-0.2) circle (0pt) node[anchor=north] {\tiny $K_{r}^1(\ell'')$};
\end{tikzpicture}
 \caption{Illustration of the lower bound construction of $P_{L,r,w,s,\theta}^{\ell}$ in the proof of  Proposition~\ref{lemma:constrainedRiskApplicationfanoApplication}.  Blue and red regions correspond to the regression function $\eta_{L,r,w,s,\theta}^{\ell}$ being above and below $\tau$ respectively.  Note the different behaviour in the $\ell$th region $K_r^0(\ell)$ from the others.  The marginal measure $\mu_{L,r,w,s,\theta}^{\ell}$ on~$\R^d$ is uniformly distributed on $\bigcup_{\ell \in [L]} \bigl(K_r^0(\ell) \cup K_r^1(\ell)\bigr)$; the boundaries of these regions are denoted with black lines.}
 \label{Fig:LowerBoundConstruction}
\end{figure}
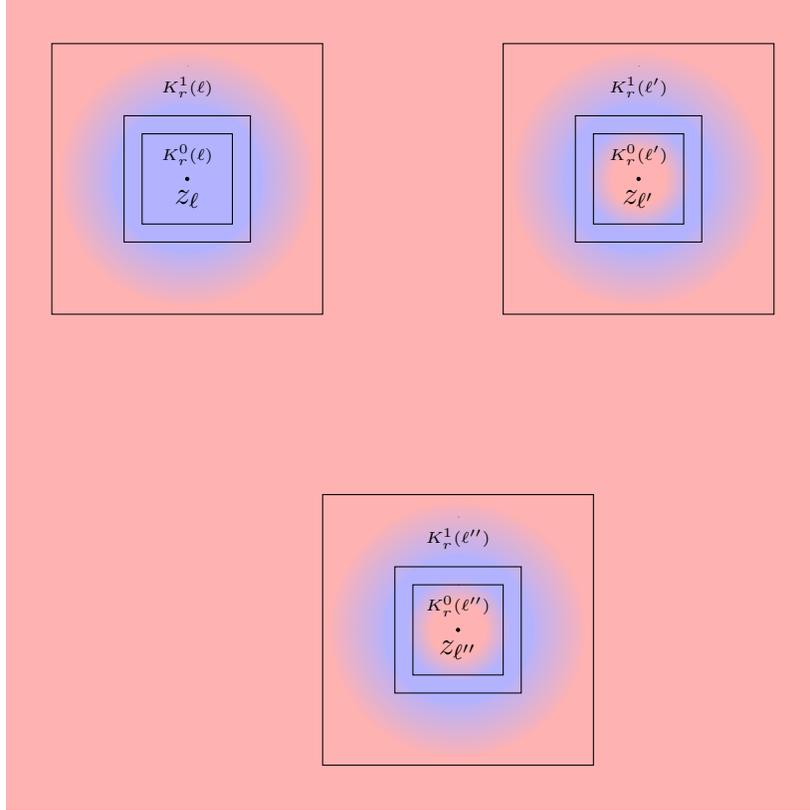 
Lemmas~\ref{lemma:LB1HolderRegFunctions} and~\ref{lemma:approximableConditionForLB1} verify that $\bigl\{\probDistribution_{L,r,w,s,\theta}^{\ell}:\ell \in [L]\bigr\} \subseteq \classOfLBDistributions$ for appropriate choices of $r$, $w$, $s$ and $\theta$.  Moreover, Proposition~\ref{lemma:completingFirstLBConstructions} reveals both that the chi-squared divergences between pairs of distributions in our construction are small, and yet that the distributions are sufficiently different that any set $A \in \mathcal{A} \cap \powerSet\bigl( \mathcal{X}_{\tau}(\eta_{L,r,w,s,\theta}^{\ell})\cap \mathcal{X}_{\tau}(\eta_{L,r,w,s,\theta}^{\ell'})\bigr)$ for distinct $\ell,\ell' \in [L]$ must have much smaller $\mu$-measure than $M_\tau$.  To conclude, we apply a constrained risk inequality due to \citet{brown1996constrained} in the proof of Proposition~\ref{lemma:constrainedRiskApplicationfanoApplication}(i) and a version of Fano's lemma in the proof of Proposition~\ref{lemma:constrainedRiskApplicationfanoApplication}(ii).

Proposition~\ref{lemma:paramatricLB} provides the final (parametric) part of the lower bounds in Theorems~\ref{thm:minimaxRate} and~\ref{Thm:minimaxRateHOS}.
\begin{prop}\label{lemma:paramatricLB} \sloppy  Take $\epsilon_0 \in (0,1/2]$, $\tau \in (\epsilon_0,1-\epsilon_0)$ $(\holderExponent,\holderConstant) \in (0,\infty)\times [1,\infty)$, $(\approximableDensityExponent, \approximableMarginExponent, \regularityConstant, \approximableSetsConstant)\in (0,\infty)^2\times (0,1)\times [1,\infty)$ with $ \holderExponent\approximableMarginExponent(\approximableDensityExponent-1) < d \approximableDensityExponent$ and $\regularityConstant  \in (0,4^{-d}]$. Suppose that $\mathcal{A} \subseteq \borel(\R^d)$ satisfies $\mathcal{A}_{\mathrm{hpr}}  \subseteq \mathcal{A} \subseteq \mathcal{A}_{\mathrm{conv}}$. Then there exists $c_2 > 0$, depending only on $\epsilon_0$, $d$, $\holderExponent$, $\approximableDensityExponent$, $\approximableMarginExponent$, $\holderConstant$ and $\approximableSetsConstant$, such that for any $n \in \N$, $\alpha \in (0,1/2-\epsilon_0]$ and $\hat{A}\in \hat{\mathcal{A}}_n\bigl(\tau,\alpha,\classOfLBDistributions\bigr)$, we can find $\probDistribution \in \classOfLBDistributions$ with regression function $\regressionFunction:\R^d \rightarrow [\tau-\epsilon_0/2,\tau+\epsilon_0/2]$ and marginal distribution $\mu$ on $\R^d$ that satisfies
\begin{align*}
\E_{\probDistribution} \Bigl[\bigl\{ M_\tau(P,\mathcal{A})-\mu(\hat{A})\bigr\} \cdot \one_{\{\hat{A} \subseteq{}  \etaSuperLevelSet{\tau} \}}\Bigr] \geq \frac{c_2}{\sqrt{n}}.
\end{align*}
\end{prop}
The construction for the proof of Proposition~\ref{lemma:paramatricLB} is somewhat different from those in the proof of Proposition~\ref{lemma:constrainedRiskApplicationfanoApplication} and is illustrated in Figure~\ref{Fig:lowerBoundConstruction2}: it hinges on the difficulty of estimating $\mu(A)$ for $A \in \mathcal{A}$.  To formalise this idea, given $t \in [1,\infty)$, $\theta \in (0,\epsilon_0/2]$, $s \in (0,1]$ and $\zeta \in \bigl[0,\frac{s^d}{2\{(2t)^d+2s^d\}} \bigr]$, we first define a pair of distributions $\{P_{t,\theta,s,\zeta}^{\ell}\}_{\ell \in \{-1,1\}}$ on $\R^d \times [0,1]$. Define $\eta\equiv \eta_{t,\theta,s}:\R^d \rightarrow [0,1]$ by
\begin{align*}
\regressionFunction(x) \equiv \eta_{t,\theta,s}(x_1,\ldots,x_d) :=\begin{cases}\tau+{\theta} &\text{ for }x_1 \leq -t-s\\
\tau+\theta\bigl\{1-2h\bigl(\frac{-x_1-t}{s}\bigr)\bigr\} &\text{ for }-t-s< x_1 \leq -t\\
\tau-{\theta} &\text{ for }-t<x_1 \leq t \\
\tau+\theta\bigl\{1-2h\bigl(\frac{x_1-t}{s}\bigr)\bigr\} &\text{ for }t<x_1 \leq t+s\\
\tau+\theta &\text{ for }x_1 \geq t+s.
\end{cases}
\end{align*}
Define $A_0:=[-t,t]^d$, $A_{-1}:=[-t-2s,-t-s]\times [-\frac{s}{2},\frac{s}{2}]^{d-1}$ and $A_{1}:=[t+s,t+2s]\times [-\frac{s}{2},\frac{s}{2}]^{d-1}$. For $\ell \in \{-1,1\}$, let $\mu^\ell_\zeta \equiv \mu_{t,s,\zeta}^{\ell}$ be the Lebesgue-absolutely continuous measure supported on $A_{-1}\cup A_0 \cup A_1\subseteq \R^d$ with piecewise constant density $f_{\mu_{\zeta}^{\ell}}:\R^d \rightarrow [0,\infty)$ given by
\begin{align*}
f_{\mu_{\zeta}^{\ell}}(x):=\begin{cases} \frac{1}{(2t)^{d}+2s^d}+\frac{\zeta \cdot j \cdot \ell}{s^d}&\text{ for }x \in A_j \text{ with } j \in \{-1,0,1\},\\
0&\text{ for }x \notin A_{-1}\cup A_0 \cup A_1.
\end{cases}
\end{align*}
Now for $\ell \in \{-1,1\}$, let $P_\zeta^\ell \equiv P_{t,\theta,s,\zeta}^{\ell}$ denote the unique distribution on $\R^d \times \{0,1\}$ with marginal~$\mu_{\zeta}^{\ell}$ on~$\R^d$ and regression function $\regressionFunction$.  Figure~\ref{Fig:lowerBoundConstruction2} illustrates this construction.
\begin{figure}[htbp!]
    \centering
\begin{tikzpicture}[scale = 1.4]
\fill[fill = red!30] (-4.5,-2.5) rectangle (4.5,2.5);
\fill[fill = red!30] (-2,-2.5) rectangle (2,2.5);
\fill[fill = blue!30] (2+1,-2.5) rectangle (4.5,2.5);
\fill[fill = blue!30] (-4.5,-2.5) rectangle (-2-1,2.5);
\shade[left color = blue!30, right color = red!30] (-2-1,-2.5) rectangle (-2,2.5);
\shade[left color = red!30, right color = blue!30] (2,-2.5) rectangle (2+1,2.5);
\draw[draw = blue!30] (-4.5,-2.5) rectangle (-2-1,2.5);
\draw[draw = blue!30] (2+1,-2.5) rectangle (4.5,2.5);
\draw[draw = black] (-2,-2) rectangle (2,2);
\draw[pattern=mydots] (-2,-2) rectangle (2,2);
\filldraw[black] (0,1/4) circle (0pt) node[anchor=north] {$A_0$};
\draw[draw = black] (-2-2,-1/2) rectangle (-2-1,1/2);
\draw[pattern=dots] (-2-2,-1/2) rectangle (-2-1,1/2);
\filldraw[black] (-7/2,1/4) circle (0pt) node[anchor=north] {$A_{-1}$};
\draw[draw = black] (2+1,-1/2) rectangle (2+2,1/2);
\draw[pattern=mynewdots, pattern color=black] (2+1,-1/2) rectangle (2+2,1/2);
\filldraw[black] (7/2,1/4) circle (0pt) node[anchor=north] {$A_{1}$};
\end{tikzpicture}
\caption{Illustration of the lower bound construction of $P_\zeta^1 \equiv P_{t,\theta,s,\zeta}^1$ in the proof of Proposition~\ref{lemma:paramatricLB}.  Blue and red regions correspond to the regression function $\eta_{t,\theta,s}$ being above and below $\tau$ respectively.  The density of dots is greatest on $A_1$ and smallest on $A_{-1}$, reflecting the greater marginal density of $\mu_{\zeta}^1$ on $A_1$; for $P_\zeta^{-1}$, the density of dots would be reversed.}
    \label{Fig:lowerBoundConstruction2}
\end{figure}
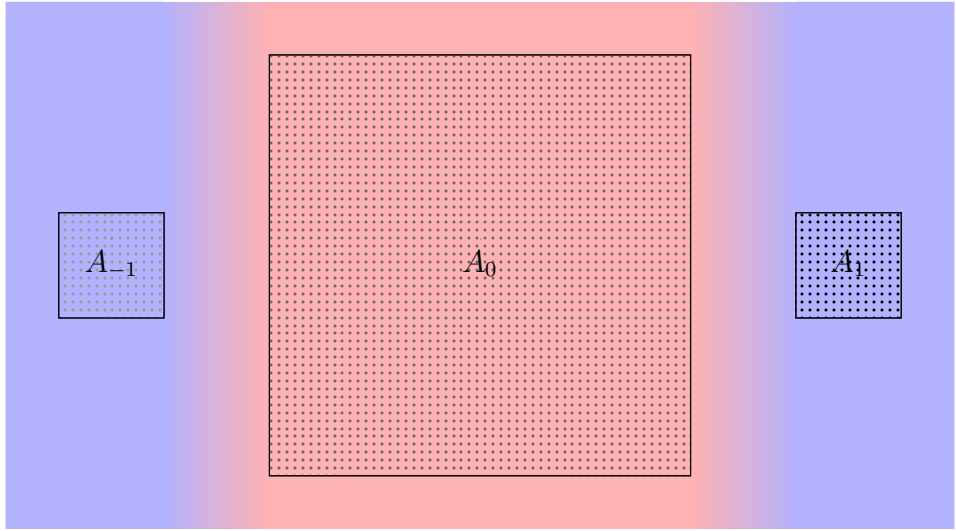

\sloppy In the proof of Proposition~\ref{lemma:paramatricLB}, we will show that $\bigl\{P_{t,\theta,s,\zeta}^{\ell}: \ell \in \{-1,1\}\bigr\} \subseteq \classOfLBDistributions$ for suitable $t$, $\theta$, $s$ and $\zeta$.  Moreover, $P_{t,\theta,s,\zeta}^{-1}$ and $P_{t,\theta,s,\zeta}^1$ are close in chi-squared divergence, but nevertheless we cannot have both $\mu_\zeta^{-1}(A)$ and $\mu_\zeta^{1}(A)$ close to $M_\tau(P_\zeta^{-1},\mathcal{A}) = M_\tau(P_\zeta^1,\mathcal{A})$ for $A \in \mathcal{A} \cap \powerSet\bigl(\etaSuperLevelSet{\tau}\bigr)$. Hence, any data-dependent selection set~$\hat{A}$ that satisfies our Type I error guarantee must incur large regret for at least one of these distributions.

\section{Application to study of heterogeneous treatment effects}\label{sec:hteApplication}

The aim of this section is to show how our previous results may be applied to the two-arm setting with a treatment and control, where we are interested in regions of substantial treatment effect.  To this end, let~$\tilde{P}$ denote the distribution of a random triple $(X,T,\tilde{Y})$ taking values in $\R^d \times \{0,1\} \times [0,1]$, where $X$ represents covariates, $T$ is a treatment indicator and $\tilde{Y}$ denotes the corresponding response.  Assume that $X \sim \mu$, and that the function $\pi:\R^d \rightarrow [0,1]$ given by $\pi(x) := \mathbb{P}(T=1|X=x)$ is known.  For $\ell \in \{0,1\}$, let $\tilde{\eta}^\ell(x) := \E(\tilde{Y}|X=x,T=\ell)$.  The \emph{heterogeneous treatment effect} is the function $\varphi:\R^d \rightarrow [-1,1]$ defined by $\varphi(x):=\tilde{\eta}^1(x)-\tilde{\eta}^0(x)$ for $x \in \R^d$. Given $t \in [-1,1]$ and a class of sets $\mathcal{A} \subseteq \borel(\R^d)$, our primary interest is in identifying subsets $A \in \mathcal{A}$ that are contained in $\mathcal{X}_t(\varphi):= \{x \in \R^d:\varphi(x) \geq t\}$ based on data $\tilde{\sample} := \bigl((X_1,T_1,\tilde{Y}_1),\ldots, (X_n,T_n,\tilde{Y}_n)\bigr) \sim \tilde{P}^{\otimes n}$.  

Given a family ${\mathcal{P}}$ of distributions on $\R^d \times \{0,1\} \times [0,1]$ and a significance level $\alpha \in (0,1)$, we let $\hat{\mathcal{A}}_n^{\mathrm{HTE}}(t,\alpha,\mathcal{P})$ denote the set of functions $\hat{A}:(\R^d \times \{0,1\}\times [0,1])^n \rightarrow \mathcal{A}$ such that $(x,\tilde{D}) \mapsto \mathbbm{1}_{\hat{A}(\tilde{D})}(x)$ is a Borel measurable function on $\R^d \times (\R^d \times \{0,1\}\times [0,1])^n$ and we have the Type I error guarantee that
\begin{equation*}
\inf_{\tilde{P} \in {\mathcal{P}}} \Prob_{\tilde{P}}\bigl( \hat{A}(\tilde{\sample}) \subseteq{} \mathcal{X}_t(\varphi)\bigr) \geq 1-\alpha.
\end{equation*}
Similarly to our formulation in Section~\ref{Sec:MainResults}, we seek $\hat{A} \in \hat{\mathcal{A}}_n^{\mathrm{HTE}}(t,\alpha,{\mathcal{P}})$ with low regret
\[
R_{t}^\varphi(\hat{A})  :=\sup\bigl\{ \mu(A): A \in \mathcal{A} \cap \mathrm{Pow}\bigl(\mathcal{X}_t(\varphi)\bigr)\bigr\}-\E_{\tilde{P}}\big\{\mu\bigl(\hat{A}(\tilde{\sample})\bigr)\bigm|\hat{A}(\tilde{\sample}) \subseteq{} \mathcal{X}_t(\varphi)\big\}.
\]
Given $(\holderExponent,\holderConstant) \in (0,\infty)\times[1,\infty)$ and a Borel measurable function $\pi:\R^d \rightarrow [0,1]$, we let $\classOfHolderDistributionsHTE$ denote the class of distributions $\tilde{P}$ on $\R^d \times \{0,1\} \times [0,1]$ such that $\varphi$ is $(\holderExponent,\holderConstant)$-H\"{o}lder (see Definition~\ref{def:holderHigherOrder}), and such that $\pi(x) = \mathbb{P}_{\tilde{P}}(T=1|X=x)$ for all $x \in \R^d$. Similarly, given $(\approximableDensityExponent,\approximableMarginExponent,\regularityConstant,\approximableSetsConstant)\in (0,\infty)^2\times (0,1)\times [1,\infty)$, we let $\classOfWellApproximableSetsHTE$ denote the class of all distributions $\tilde{\probDistribution}$ such that $\mu$ is absolutely continuous, with Lebesgue density $f_\mu$, and such that~\eqref{eq:FromDefApproximableMeasureClassHO} holds with $\varphi$ in place of $\regressionFunction$, with $t$ in place of $\tau$, and with $\sup\bigl\{ \mu(A): A \in \mathcal{A} \cap \mathrm{Pow}\bigl(\mathcal{X}_t(\varphi)\bigr)\bigr\}$ in place of $M_\tau$.

The following result on the minimax rate of regret in this heterogeneous treatment effect context is an almost immediate corollary of Theorem~\ref{Thm:minimaxRateHOS}.
\begin{corollary}
\label{cor:HTE}
Take $\zeta_0 \in (0,1/2)$,  $t \in \bigl[-(1-{\zeta_0}),1-{\zeta_0}\bigr]$, $(\holderExponent,\holderConstant) \in (0,\infty) \times [1,\infty)$,  $(\approximableDensityExponent,\approximableMarginExponent,\regularityConstant,\approximableSetsConstant)\in (0,\infty)^2 \times (0,1) \times [1,\infty)$ with $\holderExponent \approximableMarginExponent(\approximableDensityExponent-1) < d\approximableDensityExponent$ and $\regularityConstant \in \bigl(0,(4d^{1/2})^{-d}\bigr]$, and let $\pi:\R^d \rightarrow [\zeta_0,1-\zeta_0]$ be Borel measurable. Let $\mathcal{A} \subseteq \borel(\R^d)$ satisfy $\mathcal{A}_{\mathrm{hpr}}  \subseteq \mathcal{A} \subseteq \mathcal{A}_{\mathrm{conv}}$, $\vcDim(\mathcal{A}) < \infty$ and $\emptyset \in \mathcal{A}$. Given $n \in \N$ and $\alpha \in (0,1/2-\zeta_0]$, we have
\begin{equation}\label{Eq:MinimaxHTE}
\inf_{\hat{A}} \, \sup_{\tilde{P}} \,  R_t^\varphi(\hat{A}) \asymp \min\biggl\{\biggl(\frac{\log_+(n/\alpha)}{n}\biggr)^{\frac{\holderExponent \approximableDensityExponent \approximableMarginExponent}{\approximableDensityExponent(2\holderExponent+d)+\holderExponent\approximableMarginExponent}}+\frac{1}{n^{1/2}},1\biggr\},
\end{equation}
where the infimum in~\eqref{Eq:MinimaxHTE} is taken over $\hat{\mathcal{A}}_n^{\mathrm{HTE}}\bigl(t,\alpha,\classOfHolderDistributionsHTE\bigr)$, the supremum is taken over $\classOfHolderDistributionsHTE \cap \classOfWellApproximableSetsHTE$.  In~\eqref{Eq:MinimaxHTE}, $\asymp$ indicates that the ratio of the left- and right-hand sides is bounded above and below by positive quantities depending only on $d$, $\holderExponent$, $\holderConstant$, $\approximableDensityExponent$, $\approximableMarginExponent$, $\regularityConstant$,  $\approximableSetsConstant$, $\zeta_0$ and $\vcDim(\mathcal{A})$.
\end{corollary}
To establish the upper bound in Corollary~\ref{cor:HTE}, we reduce the problem to the setting of Section \ref{Sec:HigherOrder} by letting $\rho_{\min} := \min\{\inf_{x \in \R^d} \pi(x),1 - \sup_{x \in \R^d} \pi(x)\}$ and introducing proxy labels
\begin{align*}
Y := \frac{1}{2}\biggl\{1 + \frac{\rho_{\min}}{\pi(X)(1-\pi(X))}\cdot\bigl(T-\pi(X)\bigr)\tilde{Y}\biggr\},
\end{align*}
so that $Y$ takes values in $[0,1]$ and satisfies both $\regressionFunction(x) :=\E(Y|X=x) = \frac{1}{2}\bigl(1+\rho_{\min} \cdot \varphi(x)\bigr)$ and $\mathcal{X}_t(\varphi)=\etaSuperLevelSet{\tau}$ with $\tau:=\frac{1}{2}(1+ \rho_{\min} \cdot t)$. The upper bound then follows from Theorem \ref{Thm:minimaxRateHOS}(i).

To deduce the lower bound, we convert distributions $P$ of random pairs $(X,Y)$ into distributions $\tilde{P}$ of random triples $(X,T,\tilde{Y})$ for which $\Prob_{\tilde{P}}(T=1|X=x,Y=y)=\pi(x)$ and $\tilde{Y}:=T \cdot Y+(1-T)\cdot (1-Y)$, so that the corresponding heterogeneous treatment effect satisfies $\varphi(x)=2\regressionFunction(x)-1$. We may therefore deduce the lower bound in Corollary~\ref{cor:HTE} from Theorem~\ref{Thm:minimaxRateHOS}(ii), applied with $\tau :=(1+t)/2$.

\vspace{5mm}

\textbf{Acknowledgements:} We thank the anonymous reviewers for constructive feedback that helped to improve the paper.  


\bibliographystyle{imsart-nameyear}
\bibliography{mybib}

\begin{thebibliography}{68}

\bibitem[\protect\citeauthoryear{Altman}{2015}]{altman2015subgroup}
\begin{barticle}[author]
\bauthor{\bsnm{Altman},~\bfnm{Douglas~G}\binits{D.~G.}}
(\byear{2015}).
\btitle{Subgroup analyses in randomized trials—more rigour needed}.
\bjournal{Nature Reviews Clinical Oncology}
\bvolume{12}
\bpages{506--507}.
\end{barticle}
\endbibitem

\bibitem[\protect\citeauthoryear{Audibert and Tsybakov}{2007}]{audibert2007fast}
\begin{barticle}[author]
\bauthor{\bsnm{Audibert},~\bfnm{Jean-Yves}\binits{J.-Y.}} \AND \bauthor{\bsnm{Tsybakov},~\bfnm{Alexandre~B}\binits{A.~B.}}
(\byear{2007}).
\btitle{Fast learning rates for plug-in classifiers}.
\bjournal{Annals of Statistics}
\bvolume{35}
\bpages{608--633}.
\end{barticle}
\endbibitem

\bibitem[\protect\citeauthoryear{Ballarini et~al.}{2018}]{ballarini2018subgroup}
\begin{barticle}[author]
\bauthor{\bsnm{Ballarini},~\bfnm{Nicol{\'a}s~M}\binits{N.~M.}}, \bauthor{\bsnm{Rosenkranz},~\bfnm{Gerd~K}\binits{G.~K.}}, \bauthor{\bsnm{Jaki},~\bfnm{Thomas}\binits{T.}}, \bauthor{\bsnm{K{\"o}nig},~\bfnm{Franz}\binits{F.}} \AND \bauthor{\bsnm{Posch},~\bfnm{Martin}\binits{M.}}
(\byear{2018}).
\btitle{Subgroup identification in clinical trials via the predicted individual treatment effect}.
\bjournal{PloS One}
\bvolume{13}
\bpages{e0205971}.
\end{barticle}
\endbibitem

\bibitem[\protect\citeauthoryear{Brookes et~al.}{2001}]{brookes2001subgroup}
\begin{barticle}[author]
\bauthor{\bsnm{Brookes},~\bfnm{Sarah~T}\binits{S.~T.}}, \bauthor{\bsnm{Whitley},~\bfnm{Elise}\binits{E.}}, \bauthor{\bsnm{Peters},~\bfnm{Tim~J}\binits{T.~J.}}, \bauthor{\bsnm{Mulheran},~\bfnm{Paul~A}\binits{P.~A.}}, \bauthor{\bsnm{Egger},~\bfnm{Matthias}\binits{M.}} \AND \bauthor{\bsnm{Davey~Smith},~\bfnm{G}\binits{G.}}
(\byear{2001}).
\btitle{Subgroup analysis in randomised controlled trials: quantifying the risks of false-positives and false-negatives}.
\bjournal{Health Technol. Assess.}
\bvolume{5}
\bpages{1--56}.
\end{barticle}
\endbibitem

\bibitem[\protect\citeauthoryear{Brookes et~al.}{2004}]{brookes2004subgroup}
\begin{barticle}[author]
\bauthor{\bsnm{Brookes},~\bfnm{Sara~T}\binits{S.~T.}}, \bauthor{\bsnm{Whitely},~\bfnm{Elise}\binits{E.}}, \bauthor{\bsnm{Egger},~\bfnm{Matthias}\binits{M.}}, \bauthor{\bsnm{Smith},~\bfnm{George~Davey}\binits{G.~D.}}, \bauthor{\bsnm{Mulheran},~\bfnm{Paul~A}\binits{P.~A.}} \AND \bauthor{\bsnm{Peters},~\bfnm{Tim~J}\binits{T.~J.}}
(\byear{2004}).
\btitle{Subgroup analyses in randomized trials: risks of subgroup-specific analyses; power and sample size for the interaction test}.
\bjournal{Journal of Clinical Epidemiology}
\bvolume{57}
\bpages{229--236}.
\end{barticle}
\endbibitem

\bibitem[\protect\citeauthoryear{Brown and Low}{1996}]{brown1996constrained}
\begin{barticle}[author]
\bauthor{\bsnm{Brown},~\bfnm{Lawrence~D}\binits{L.~D.}} \AND \bauthor{\bsnm{Low},~\bfnm{Mark~G}\binits{M.~G.}}
(\byear{1996}).
\btitle{A constrained risk inequality with applications to nonparametric functional estimation}.
\bjournal{Annals of Statistics}
\bvolume{24}
\bpages{2524--2535}.
\end{barticle}
\endbibitem

\bibitem[\protect\citeauthoryear{Bull}{2012}]{bull2012honest}
\begin{barticle}[author]
\bauthor{\bsnm{Bull},~\bfnm{Adam~D}\binits{A.~D.}}
(\byear{2012}).
\btitle{Honest adaptive confidence bands and self-similar functions}.
\bjournal{Electronic Journal of Statistics}
\bvolume{6}
\bpages{1490--1516}.
\end{barticle}
\endbibitem

\bibitem[\protect\citeauthoryear{Cannings, Berrett and Samworth}{2020}]{cannings2020local}
\begin{barticle}[author]
\bauthor{\bsnm{Cannings},~\bfnm{Timothy~I.}\binits{T.~I.}}, \bauthor{\bsnm{Berrett},~\bfnm{Thomas~B.}\binits{T.~B.}} \AND \bauthor{\bsnm{Samworth},~\bfnm{Richard~J.}\binits{R.~J.}}
(\byear{2020}).
\btitle{Local nearest neighbour classification with applications to semi-supervised learning}.
\bjournal{Annals of Statistics}
\bvolume{48}
\bpages{1789--1814}.
\end{barticle}
\endbibitem

\bibitem[\protect\citeauthoryear{Cannon et~al.}{2002}]{cannon2002learning}
\begin{barticle}[author]
\bauthor{\bsnm{Cannon},~\bfnm{Adam}\binits{A.}}, \bauthor{\bsnm{Howse},~\bfnm{James}\binits{J.}}, \bauthor{\bsnm{Hush},~\bfnm{Don}\binits{D.}} \AND \bauthor{\bsnm{Scovel},~\bfnm{Clint}\binits{C.}}
(\byear{2002}).
\btitle{Learning with the Neyman--Pearson and min-max criteria}.
\bjournal{Los Alamos National Laboratory, Tech. Rep. LA-UR}
\bpages{02--2951}.
\end{barticle}
\endbibitem

\bibitem[\protect\citeauthoryear{Cavalier}{1997}]{cavalier1997nonparametric}
\begin{barticle}[author]
\bauthor{\bsnm{Cavalier},~\bfnm{Laurent}\binits{L.}}
(\byear{1997}).
\btitle{Nonparametric estimation of regression level sets}.
\bjournal{Statistics: A Journal of Theoretical and Applied Statistics}
\bvolume{29}
\bpages{131--160}.
\end{barticle}
\endbibitem

\bibitem[\protect\citeauthoryear{Chen, Genovese and Wasserman}{2017}]{chen2017density}
\begin{barticle}[author]
\bauthor{\bsnm{Chen},~\bfnm{Yen-Chi}\binits{Y.-C.}}, \bauthor{\bsnm{Genovese},~\bfnm{Christopher~R}\binits{C.~R.}} \AND \bauthor{\bsnm{Wasserman},~\bfnm{Larry}\binits{L.}}
(\byear{2017}).
\btitle{Density level sets: Asymptotics, inference, and visualization}.
\bjournal{Journal of the American Statistical Association}
\bvolume{112}
\bpages{1684--1696}.
\end{barticle}
\endbibitem

\bibitem[\protect\citeauthoryear{Constantine and Savits}{1996}]{constantine1996multivariate}
\begin{barticle}[author]
\bauthor{\bsnm{Constantine},~\bfnm{G}\binits{G.}} \AND \bauthor{\bsnm{Savits},~\bfnm{T}\binits{T.}}
(\byear{1996}).
\btitle{A multivariate {F}aa di {B}runo formula with applications}.
\bjournal{Transactions of the American Mathematical Society}
\bvolume{348}
\bpages{503--520}.
\end{barticle}
\endbibitem

\bibitem[\protect\citeauthoryear{Crump et~al.}{2008}]{crump2008nonparametric}
\begin{barticle}[author]
\bauthor{\bsnm{Crump},~\bfnm{Richard~K}\binits{R.~K.}}, \bauthor{\bsnm{Hotz},~\bfnm{V~Joseph}\binits{V.~J.}}, \bauthor{\bsnm{Imbens},~\bfnm{Guido~W}\binits{G.~W.}} \AND \bauthor{\bsnm{Mitnik},~\bfnm{Oscar~A}\binits{O.~A.}}
(\byear{2008}).
\btitle{Nonparametric tests for treatment effect heterogeneity}.
\bjournal{The Review of Economics and Statistics}
\bvolume{90}
\bpages{389--405}.
\end{barticle}
\endbibitem

\bibitem[\protect\citeauthoryear{Dau, Laloë and Servien}{2020}]{dau2020exact}
\begin{barticle}[author]
\bauthor{\bsnm{Dau},~\bfnm{Hai~Dang}\binits{H.~D.}}, \bauthor{\bsnm{Laloë},~\bfnm{Thomas}\binits{T.}} \AND \bauthor{\bsnm{Servien},~\bfnm{Rémi}\binits{R.}}
(\byear{2020}).
\btitle{Exact asymptotic limit for kernel estimation of regression level sets}.
\bjournal{Statistics and Probability Letters}
\bvolume{161}
\bpages{108721}.
\bdoi{https://doi.org/10.1016/j.spl.2020.108721}
\end{barticle}
\endbibitem

\bibitem[\protect\citeauthoryear{Doss and Weng}{2018}]{doss2018bandwidth}
\begin{barticle}[author]
\bauthor{\bsnm{Doss},~\bfnm{Charles~R.}\binits{C.~R.}} \AND \bauthor{\bsnm{Weng},~\bfnm{Guangwei}\binits{G.}}
(\byear{2018}).
\btitle{{Bandwidth selection for kernel density estimators of multivariate level sets and highest density regions}}.
\bjournal{Electronic Journal of Statistics}
\bvolume{12}
\bpages{4313--4376}.
\bdoi{10.1214/18-EJS1501}
\end{barticle}
\endbibitem

\bibitem[\protect\citeauthoryear{Dudley}{2018}]{dudley2018real}
\begin{bbook}[author]
\bauthor{\bsnm{Dudley},~\bfnm{Richard~M}\binits{R.~M.}}
(\byear{2018}).
\btitle{Real Analysis and Probability}.
\bpublisher{CRC Press, Cambridge}.
\end{bbook}
\endbibitem

\bibitem[\protect\citeauthoryear{Dusseldorp, Conversano and Van~Os}{2010}]{dusseldorp2010combining}
\begin{barticle}[author]
\bauthor{\bsnm{Dusseldorp},~\bfnm{Elise}\binits{E.}}, \bauthor{\bsnm{Conversano},~\bfnm{Claudio}\binits{C.}} \AND \bauthor{\bsnm{Van~Os},~\bfnm{Bart~Jan}\binits{B.~J.}}
(\byear{2010}).
\btitle{Combining an additive and tree-based regression model simultaneously: STIMA}.
\bjournal{Journal of Computational and Graphical Statistics}
\bvolume{19}
\bpages{514--530}.
\end{barticle}
\endbibitem

\bibitem[\protect\citeauthoryear{Feinstein}{1998}]{feinstein1998problem}
\begin{barticle}[author]
\bauthor{\bsnm{Feinstein},~\bfnm{Alvan~R}\binits{A.~R.}}
(\byear{1998}).
\btitle{The problem of cogent subgroups: a clinicostatistical tragedy}.
\bjournal{Journal of Clinical Epidemiology}
\bvolume{51}
\bpages{297--299}.
\end{barticle}
\endbibitem

\bibitem[\protect\citeauthoryear{Foster, Taylor and Ruberg}{2011}]{foster2011subgroup}
\begin{barticle}[author]
\bauthor{\bsnm{Foster},~\bfnm{Jared~C}\binits{J.~C.}}, \bauthor{\bsnm{Taylor},~\bfnm{Jeremy~MG}\binits{J.~M.}} \AND \bauthor{\bsnm{Ruberg},~\bfnm{Stephen~J}\binits{S.~J.}}
(\byear{2011}).
\btitle{Subgroup identification from randomized clinical trial data}.
\bjournal{Statistics in Medicine}
\bvolume{30}
\bpages{2867--2880}.
\end{barticle}
\endbibitem

\bibitem[\protect\citeauthoryear{Gabler et~al.}{2016}]{gabler2016no}
\begin{barticle}[author]
\bauthor{\bsnm{Gabler},~\bfnm{Nicole~B}\binits{N.~B.}}, \bauthor{\bsnm{Duan},~\bfnm{Naihua}\binits{N.}}, \bauthor{\bsnm{Raneses},~\bfnm{Eli}\binits{E.}}, \bauthor{\bsnm{Suttner},~\bfnm{Leah}\binits{L.}}, \bauthor{\bsnm{Ciarametaro},~\bfnm{Michael}\binits{M.}}, \bauthor{\bsnm{Cooney},~\bfnm{Elizabeth}\binits{E.}}, \bauthor{\bsnm{Dubois},~\bfnm{Robert~W}\binits{R.~W.}}, \bauthor{\bsnm{Halpern},~\bfnm{Scott~D}\binits{S.~D.}} \AND \bauthor{\bsnm{Kravitz},~\bfnm{Richard~L}\binits{R.~L.}}
(\byear{2016}).
\btitle{No improvement in the reporting of clinical trial subgroup effects in high-impact general medical journals}.
\bjournal{Trials}
\bvolume{17}
\bpages{1--12}.
\end{barticle}
\endbibitem

\bibitem[\protect\citeauthoryear{Garivier and Capp{\'e}}{2011}]{garivier2011kl}
\begin{binproceedings}[author]
\bauthor{\bsnm{Garivier},~\bfnm{Aur{\'e}lien}\binits{A.}} \AND \bauthor{\bsnm{Capp{\'e}},~\bfnm{Olivier}\binits{O.}}
(\byear{2011}).
\btitle{The {KL}-{UCB} algorithm for bounded stochastic bandits and beyond}.
In \bbooktitle{Proceedings of the 24th annual Conference On Learning Theory}
\bpages{359--376}.
\end{binproceedings}
\endbibitem

\bibitem[\protect\citeauthoryear{Gerchinovitz, M{\'e}nard and Stoltz}{2020}]{gerchinovitz2020fano}
\begin{barticle}[author]
\bauthor{\bsnm{Gerchinovitz},~\bfnm{Sebastien}\binits{S.}}, \bauthor{\bsnm{M{\'e}nard},~\bfnm{Pierre}\binits{P.}} \AND \bauthor{\bsnm{Stoltz},~\bfnm{Gilles}\binits{G.}}
(\byear{2020}).
\btitle{Fano’s inequality for random variables}.
\bjournal{Statistical Science}
\bvolume{35}
\bpages{178--201}.
\end{barticle}
\endbibitem

\bibitem[\protect\citeauthoryear{Gin{\'e} and Nickl}{2010}]{GineNicklAOS738}
\begin{barticle}[author]
\bauthor{\bsnm{Gin{\'e}},~\bfnm{Evarist}\binits{E.}} \AND \bauthor{\bsnm{Nickl},~\bfnm{Richard}\binits{R.}}
(\byear{2010}).
\btitle{{Confidence bands in density estimation}}.
\bjournal{The Annals of Statistics}
\bvolume{38}
\bpages{1122 -- 1170}.
\bdoi{10.1214/09-AOS738}
\end{barticle}
\endbibitem

\bibitem[\protect\citeauthoryear{Gotovos et~al.}{2013}]{gotovos2013active}
\begin{binproceedings}[author]
\bauthor{\bsnm{Gotovos},~\bfnm{Alkis}\binits{A.}}, \bauthor{\bsnm{Casati},~\bfnm{Nathalie}\binits{N.}}, \bauthor{\bsnm{Hitz},~\bfnm{Gregory}\binits{G.}} \AND \bauthor{\bsnm{Krause},~\bfnm{Andreas}\binits{A.}}
(\byear{2013}).
\btitle{Active learning for level set estimation}.
In \bbooktitle{Twenty-Third International Joint Conference on Artificial Intelligence}
\bpages{1344–1350}.
\end{binproceedings}
\endbibitem

\bibitem[\protect\citeauthoryear{Gur, Momeni and Wager}{2022}]{gur2022smoothness}
\begin{barticle}[author]
\bauthor{\bsnm{Gur},~\bfnm{Yonatan}\binits{Y.}}, \bauthor{\bsnm{Momeni},~\bfnm{Ahmadreza}\binits{A.}} \AND \bauthor{\bsnm{Wager},~\bfnm{Stefan}\binits{S.}}
(\byear{2022}).
\btitle{Smoothness-adaptive contextual bandits}.
\bjournal{Operations Research}
\bvolume{70}
\bpages{3198--3216}.
\end{barticle}
\endbibitem

\bibitem[\protect\citeauthoryear{Herrera et~al.}{2011}]{herrera2011overview}
\begin{barticle}[author]
\bauthor{\bsnm{Herrera},~\bfnm{Franciso}\binits{F.}}, \bauthor{\bsnm{Carmona},~\bfnm{Crist{\'o}bal~Jos{\'e}}\binits{C.~J.}}, \bauthor{\bsnm{Gonz{\'a}lez},~\bfnm{Pedro}\binits{P.}} \AND \bauthor{\bsnm{Del~Jesus},~\bfnm{Mar{\'\i}a~Jos{\'e}}\binits{M.~J.}}
(\byear{2011}).
\btitle{An overview on subgroup discovery: foundations and applications}.
\bjournal{Knowledge and Information Systems}
\bvolume{29}
\bpages{495--525}.
\end{barticle}
\endbibitem

\bibitem[\protect\citeauthoryear{Holm}{1979}]{holm1979simple}
\begin{barticle}[author]
\bauthor{\bsnm{Holm},~\bfnm{Sture}\binits{S.}}
(\byear{1979}).
\btitle{A simple sequentially rejective multiple test procedure}.
\bjournal{Scandinavian Journal of Statistics}
\bvolume{6}
\bpages{65--70}.
\end{barticle}
\endbibitem

\bibitem[\protect\citeauthoryear{Huber, Benda and Friede}{2019}]{huber2019comparison}
\begin{barticle}[author]
\bauthor{\bsnm{Huber},~\bfnm{Cynthia}\binits{C.}}, \bauthor{\bsnm{Benda},~\bfnm{Norbert}\binits{N.}} \AND \bauthor{\bsnm{Friede},~\bfnm{Tim}\binits{T.}}
(\byear{2019}).
\btitle{A comparison of subgroup identification methods in clinical drug development: Simulation study and regulatory considerations}.
\bjournal{Pharmaceutical Statistics}
\bvolume{18}
\bpages{600--626}.
\end{barticle}
\endbibitem

\bibitem[\protect\citeauthoryear{Hyndman}{1996}]{hyndman1996computing}
\begin{barticle}[author]
\bauthor{\bsnm{Hyndman},~\bfnm{Rob~J.}\binits{R.~J.}}
(\byear{1996}).
\btitle{Computing and Graphing Highest Density Regions}.
\bjournal{The American Statistician}
\bvolume{50}
\bpages{120--126}.
\end{barticle}
\endbibitem

\bibitem[\protect\citeauthoryear{Kaufman and MacLehose}{2013}]{kaufman2013these}
\begin{barticle}[author]
\bauthor{\bsnm{Kaufman},~\bfnm{Jay~S}\binits{J.~S.}} \AND \bauthor{\bsnm{MacLehose},~\bfnm{Richard~F}\binits{R.~F.}}
(\byear{2013}).
\btitle{Which of these things is not like the others?}
\bjournal{Cancer}
\bvolume{119}
\bpages{4216--4222}.
\end{barticle}
\endbibitem

\bibitem[\protect\citeauthoryear{Kehl and Ulm}{2006}]{kehl2006responder}
\begin{barticle}[author]
\bauthor{\bsnm{Kehl},~\bfnm{Victoria}\binits{V.}} \AND \bauthor{\bsnm{Ulm},~\bfnm{Kurt}\binits{K.}}
(\byear{2006}).
\btitle{Responder identification in clinical trials with censored data}.
\bjournal{Computational Statistics \& Data Analysis}
\bvolume{50}
\bpages{1338--1355}.
\end{barticle}
\endbibitem

\bibitem[\protect\citeauthoryear{Lagakos}{2006}]{lagakos2006challenge}
\begin{barticle}[author]
\bauthor{\bsnm{Lagakos},~\bfnm{Stephen~W.}\binits{S.~W.}}
(\byear{2006}).
\btitle{The challenge of subgroup analyses -- reporting without distorting}.
\bjournal{New England Journal of Medicine}
\bvolume{354}
\bpages{1667--1669}.
\end{barticle}
\endbibitem

\bibitem[\protect\citeauthoryear{Lalo{\"e} and Servien}{2013}]{laloe2013estimation}
\begin{barticle}[author]
\bauthor{\bsnm{Lalo{\"e}},~\bfnm{Thomas}\binits{T.}} \AND \bauthor{\bsnm{Servien},~\bfnm{R{\'e}mi}\binits{R.}}
(\byear{2013}).
\btitle{Nonparametric estimation of regression level sets using kernel plug-in estimator}.
\bjournal{Journal of the Korean Statistical Society}
\bvolume{42}
\bpages{301-311}.
\end{barticle}
\endbibitem

\bibitem[\protect\citeauthoryear{Lipkovich, Dmitrienko and D'Agostino~Sr}{2017}]{lipkovich2017tutorial}
\begin{barticle}[author]
\bauthor{\bsnm{Lipkovich},~\bfnm{Ilya}\binits{I.}}, \bauthor{\bsnm{Dmitrienko},~\bfnm{Alex}\binits{A.}} \AND \bauthor{\bsnm{D'Agostino~Sr},~\bfnm{Ralph~B.}\binits{R.~B.}}
(\byear{2017}).
\btitle{Tutorial in biostatistics: data-driven subgroup identification and analysis in clinical trials}.
\bjournal{Statistics in Medicine}
\bvolume{36}
\bpages{136--196}.
\end{barticle}
\endbibitem

\bibitem[\protect\citeauthoryear{Lipkovich et~al.}{2011}]{lipkovich2011subgroup}
\begin{barticle}[author]
\bauthor{\bsnm{Lipkovich},~\bfnm{Ilya}\binits{I.}}, \bauthor{\bsnm{Dmitrienko},~\bfnm{Alex}\binits{A.}}, \bauthor{\bsnm{Denne},~\bfnm{Jonathan}\binits{J.}} \AND \bauthor{\bsnm{Enas},~\bfnm{Gregory}\binits{G.}}
(\byear{2011}).
\btitle{Subgroup identification based on differential effect search—a recursive partitioning method for establishing response to treatment in patient subpopulations}.
\bjournal{Statistics in Medicine}
\bvolume{30}
\bpages{2601--2621}.
\end{barticle}
\endbibitem

\bibitem[\protect\citeauthoryear{Mammen and Polonik}{2013}]{mammen2013confidence}
\begin{barticle}[author]
\bauthor{\bsnm{Mammen},~\bfnm{Enno}\binits{E.}} \AND \bauthor{\bsnm{Polonik},~\bfnm{Wolfgang}\binits{W.}}
(\byear{2013}).
\btitle{Confidence regions for level sets}.
\bjournal{Journal of Multivariate Analysis}
\bvolume{122}
\bpages{202--214}.
\end{barticle}
\endbibitem

\bibitem[\protect\citeauthoryear{Mason and Polonik}{2009}]{mason2009asymptotic}
\begin{barticle}[author]
\bauthor{\bsnm{Mason},~\bfnm{David~M.}\binits{D.~M.}} \AND \bauthor{\bsnm{Polonik},~\bfnm{Wolfgang}\binits{W.}}
(\byear{2009}).
\btitle{{Asymptotic normality of plug-in level set estimates}}.
\bjournal{Annals of Applied Probability}
\bvolume{19}
\bpages{1108--1142}.
\bdoi{10.1214/08-AAP569}
\end{barticle}
\endbibitem

\bibitem[\protect\citeauthoryear{McDiarmid}{1998}]{mcdiarmid1998concentration}
\begin{bincollection}[author]
\bauthor{\bsnm{McDiarmid},~\bfnm{Colin}\binits{C.}}
(\byear{1998}).
\btitle{Concentration}.
In \bbooktitle{Probabilistic Methods for Algorithmic Discrete Mathematics}
\bpages{195--248}.
\bpublisher{Springer}.
\end{bincollection}
\endbibitem

\bibitem[\protect\citeauthoryear{McShane}{1934}]{mcshane1934extension}
\begin{barticle}[author]
\bauthor{\bsnm{McShane},~\bfnm{Edward~James}\binits{E.~J.}}
(\byear{1934}).
\btitle{Extension of range of functions}.
\bjournal{Bulletin of the American Mathematical Society}
\bvolume{40}
\bpages{837--842}.
\end{barticle}
\endbibitem

\bibitem[\protect\citeauthoryear{Okamoto}{1973}]{okamoto1973distinctness}
\begin{barticle}[author]
\bauthor{\bsnm{Okamoto},~\bfnm{Masashi}\binits{M.}}
(\byear{1973}).
\btitle{Distinctness of the Eigenvalues of a Quadratic form in a Multivariate Sample}.
\bjournal{Annals of Statistics}
\bvolume{1}
\bpages{763--765}.
\bdoi{10.1214/aos/1176342472}
\end{barticle}
\endbibitem

\bibitem[\protect\citeauthoryear{Patel et~al.}{2016}]{patel2016identifying}
\begin{barticle}[author]
\bauthor{\bsnm{Patel},~\bfnm{Shilpa}\binits{S.}}, \bauthor{\bsnm{Hee},~\bfnm{Siew~Wan}\binits{S.~W.}}, \bauthor{\bsnm{Mistry},~\bfnm{Dipesh}\binits{D.}}, \bauthor{\bsnm{Jordan},~\bfnm{Jake}\binits{J.}}, \bauthor{\bsnm{Brown},~\bfnm{Sally}\binits{S.}}, \bauthor{\bsnm{Dritsaki},~\bfnm{Melina}\binits{M.}}, \bauthor{\bsnm{Ellard},~\bfnm{David~R}\binits{D.~R.}}, \bauthor{\bsnm{Friede},~\bfnm{Tim}\binits{T.}}, \bauthor{\bsnm{Lamb},~\bfnm{Sarah~E}\binits{S.~E.}} \AND \bauthor{\bsnm{Lord},~\bfnm{Joanne}\binits{J.}}
(\byear{2016}).
\btitle{Identifying back pain subgroups: developing and applying approaches using individual patient data collected within clinical trials}.
\bjournal{Programme Grants for Applied Research}
\bvolume{4}.
\end{barticle}
\endbibitem

\bibitem[\protect\citeauthoryear{Picard and Tribouley}{2000}]{picard2000adaptive}
\begin{barticle}[author]
\bauthor{\bsnm{Picard},~\bfnm{Dominique}\binits{D.}} \AND \bauthor{\bsnm{Tribouley},~\bfnm{Karine}\binits{K.}}
(\byear{2000}).
\btitle{Adaptive confidence interval for pointwise curve estimation}.
\bjournal{The Annals of Statistics}
\bvolume{28}
\bpages{298--335}.
\end{barticle}
\endbibitem

\bibitem[\protect\citeauthoryear{Polonik}{1995}]{polonik1995measuring}
\begin{barticle}[author]
\bauthor{\bsnm{Polonik},~\bfnm{Wolfgang}\binits{W.}}
(\byear{1995}).
\btitle{Measuring mass concentrations and estimating density contour clusters-an excess mass approach}.
\bjournal{Annals of Statistics}
\bvolume{23}
\bpages{855--881}.
\end{barticle}
\endbibitem

\bibitem[\protect\citeauthoryear{Qiao}{2020}]{qiao2020asymptotics}
\begin{barticle}[author]
\bauthor{\bsnm{Qiao},~\bfnm{Wanli}\binits{W.}}
(\byear{2020}).
\btitle{{Asymptotics and optimal bandwidth for nonparametric estimation of density level sets}}.
\bjournal{Electronic Journal of Statistics}
\bvolume{14}
\bpages{302--344}.
\bdoi{10.1214/19-EJS1668}
\end{barticle}
\endbibitem

\bibitem[\protect\citeauthoryear{Qiao and Polonik}{2019}]{qiao2019nonparametric}
\begin{barticle}[author]
\bauthor{\bsnm{Qiao},~\bfnm{Wanli}\binits{W.}} \AND \bauthor{\bsnm{Polonik},~\bfnm{Wolfgang}\binits{W.}}
(\byear{2019}).
\btitle{Nonparametric confidence regions for level sets: Statistical properties and geometry}.
\bjournal{Electronic Journal of Statistics}
\bvolume{13}
\bpages{985--1030}.
\end{barticle}
\endbibitem

\bibitem[\protect\citeauthoryear{Reeve, Cannings and Samworth}{2021}]{reeve2021adaptive}
\begin{barticle}[author]
\bauthor{\bsnm{Reeve},~\bfnm{Henry W~J}\binits{H.~W.~J.}}, \bauthor{\bsnm{Cannings},~\bfnm{Timothy~I}\binits{T.~I.}} \AND \bauthor{\bsnm{Samworth},~\bfnm{Richard~J}\binits{R.~J.}}
(\byear{2021}).
\btitle{Adaptive Transfer Learning}.
\bjournal{Annals of Statistics}
\bvolume{49}
\bpages{3618-3649}.
\end{barticle}
\endbibitem

\bibitem[\protect\citeauthoryear{Reeve, Cannings and Samworth}{2023}]{reeve2021optimal}
\begin{barticle}[author]
\bauthor{\bsnm{Reeve},~\bfnm{Henry W.~J.}\binits{H.~W.~J.}}, \bauthor{\bsnm{Cannings},~\bfnm{Timothy~I.}\binits{T.~I.}} \AND \bauthor{\bsnm{Samworth},~\bfnm{Richard~J.}\binits{R.~J.}}
(\byear{2023}).
\btitle{Optimal subgroup selection}.
\bjournal{Submitted}.
\end{barticle}
\endbibitem

\bibitem[\protect\citeauthoryear{Rodr{\'\i}guez-Casal and Saavedra-Nieves}{2019}]{rodriguez2019minimax}
\begin{barticle}[author]
\bauthor{\bsnm{Rodr{\'\i}guez-Casal},~\bfnm{Alberto}\binits{A.}} \AND \bauthor{\bsnm{Saavedra-Nieves},~\bfnm{Paula}\binits{P.}}
(\byear{2019}).
\btitle{Minimax {H}ausdorff estimation of density level sets}.
\bjournal{arXiv preprint arXiv:1905.02897}.
\end{barticle}
\endbibitem

\bibitem[\protect\citeauthoryear{Rothwell}{2005}]{rothwell2005subgroup}
\begin{barticle}[author]
\bauthor{\bsnm{Rothwell},~\bfnm{Peter~M}\binits{P.~M.}}
(\byear{2005}).
\btitle{Subgroup analysis in randomised controlled trials: importance, indications, and interpretation}.
\bjournal{The Lancet}
\bvolume{365}
\bpages{176--186}.
\end{barticle}
\endbibitem

\bibitem[\protect\citeauthoryear{Samworth and Wand}{2010}]{samworth2010asymptotics}
\begin{barticle}[author]
\bauthor{\bsnm{Samworth},~\bfnm{R.~J.}\binits{R.~J.}} \AND \bauthor{\bsnm{Wand},~\bfnm{M.~P.}\binits{M.~P.}}
(\byear{2010}).
\btitle{{Asymptotics and optimal bandwidth selection for highest density region estimation}}.
\bjournal{Annals of Statistics}
\bvolume{38}
\bpages{1767--1792}.
\bdoi{10.1214/09-AOS766}
\end{barticle}
\endbibitem

\bibitem[\protect\citeauthoryear{Scott and Davenport}{2007}]{scott2007regression}
\begin{barticle}[author]
\bauthor{\bsnm{Scott},~\bfnm{Clayton}\binits{C.}} \AND \bauthor{\bsnm{Davenport},~\bfnm{Mark}\binits{M.}}
(\byear{2007}).
\btitle{Regression level set estimation via cost-sensitive classification}.
\bjournal{IEEE Transactions on Signal Processing}
\bvolume{55}
\bpages{2752--2757}.
\end{barticle}
\endbibitem

\bibitem[\protect\citeauthoryear{Scott and Nowak}{2005}]{scott2005neyman}
\begin{barticle}[author]
\bauthor{\bsnm{Scott},~\bfnm{Clayton}\binits{C.}} \AND \bauthor{\bsnm{Nowak},~\bfnm{Robert}\binits{R.}}
(\byear{2005}).
\btitle{A {N}eyman--{P}earson approach to statistical learning}.
\bjournal{IEEE Transactions on Information Theory}
\bvolume{51}
\bpages{3806--3819}.
\end{barticle}
\endbibitem

\bibitem[\protect\citeauthoryear{Seibold, Zeileis and Hothorn}{2016}]{seibold2016model}
\begin{barticle}[author]
\bauthor{\bsnm{Seibold},~\bfnm{Heidi}\binits{H.}}, \bauthor{\bsnm{Zeileis},~\bfnm{Achim}\binits{A.}} \AND \bauthor{\bsnm{Hothorn},~\bfnm{Torsten}\binits{T.}}
(\byear{2016}).
\btitle{Model-based recursive partitioning for subgroup analyses}.
\bjournal{The International Journal of Biostatistics}
\bvolume{12}
\bpages{45--63}.
\end{barticle}
\endbibitem

\bibitem[\protect\citeauthoryear{Senn and Harrell}{1997}]{senn1997wisdom}
\begin{barticle}[author]
\bauthor{\bsnm{Senn},~\bfnm{Stephen}\binits{S.}} \AND \bauthor{\bsnm{Harrell},~\bfnm{Frank}\binits{F.}}
(\byear{1997}).
\btitle{On wisdom after the event.}
\bjournal{Journal of Clinical Epidemiology}
\bvolume{50}
\bpages{749--751}.
\end{barticle}
\endbibitem

\bibitem[\protect\citeauthoryear{Shalev-Shwartz and Ben-David}{2014}]{shalev2014understanding}
\begin{bbook}[author]
\bauthor{\bsnm{Shalev-Shwartz},~\bfnm{Shai}\binits{S.}} \AND \bauthor{\bsnm{Ben-David},~\bfnm{Shai}\binits{S.}}
(\byear{2014}).
\btitle{Understanding Machine Learning: From Theory to Algorithms}.
\bpublisher{Cambridge University Press, Cambridge}.
\end{bbook}
\endbibitem

\bibitem[\protect\citeauthoryear{Su et~al.}{2009}]{su2009subgroup}
\begin{barticle}[author]
\bauthor{\bsnm{Su},~\bfnm{Xiaogang}\binits{X.}}, \bauthor{\bsnm{Tsai},~\bfnm{Chih-Ling}\binits{C.-L.}}, \bauthor{\bsnm{Wang},~\bfnm{Hansheng}\binits{H.}}, \bauthor{\bsnm{Nickerson},~\bfnm{David~M}\binits{D.~M.}} \AND \bauthor{\bsnm{Li},~\bfnm{Bogong}\binits{B.}}
(\byear{2009}).
\btitle{Subgroup analysis via recursive partitioning}.
\bjournal{Journal of Machine Learning Research}
\bvolume{10}
\bpages{141--158}.
\end{barticle}
\endbibitem

\bibitem[\protect\citeauthoryear{Ting et~al.}{2020}]{ting2020design}
\begin{bbook}[author]
\bauthor{\bsnm{Ting},~\bfnm{Naitee}\binits{N.}}, \bauthor{\bsnm{Cappelleri},~\bfnm{Joseph~C}\binits{J.~C.}}, \bauthor{\bsnm{Ho},~\bfnm{Shuyen}\binits{S.}} \AND \bauthor{\bsnm{Chen},~\bfnm{Ding-Geng}\binits{D.-G.}}
(\byear{2020}).
\btitle{Design and Analysis of Subgroups with Biopharmaceutical Applications}.
\bpublisher{Springer}.
\end{bbook}
\endbibitem

\bibitem[\protect\citeauthoryear{Tong, Feng and Zhao}{2016}]{tong2016survey}
\begin{barticle}[author]
\bauthor{\bsnm{Tong},~\bfnm{Xin}\binits{X.}}, \bauthor{\bsnm{Feng},~\bfnm{Yang}\binits{Y.}} \AND \bauthor{\bsnm{Zhao},~\bfnm{Anqi}\binits{A.}}
(\byear{2016}).
\btitle{A survey on {N}eyman--{P}earson classification and suggestions for future research}.
\bjournal{Wiley Interdisciplinary Reviews: Computational Statistics}
\bvolume{8}
\bpages{64--81}.
\end{barticle}
\endbibitem

\bibitem[\protect\citeauthoryear{Tropp}{2012}]{tropp2012user}
\begin{barticle}[author]
\bauthor{\bsnm{Tropp},~\bfnm{Joel~A.}\binits{J.~A.}}
(\byear{2012}).
\btitle{User-Friendly Tail Bounds for Sums of Random Matrices}.
\bjournal{Foundations of Computational Mathematics}
\bvolume{12}
\bpages{389--434}.
\bdoi{10.1007/s10208-011-9099-z}
\end{barticle}
\endbibitem

\bibitem[\protect\citeauthoryear{Tsybakov}{1997}]{tsybakov1997nonparametric}
\begin{barticle}[author]
\bauthor{\bsnm{Tsybakov},~\bfnm{A.~B.}\binits{A.~B.}}
(\byear{1997}).
\btitle{{On nonparametric estimation of density level sets}}.
\bjournal{Annals of Statistics}
\bvolume{25}
\bpages{948--969}.
\bdoi{10.1214/aos/1069362732}
\end{barticle}
\endbibitem

\bibitem[\protect\citeauthoryear{Vershynin}{2018}]{vershynin2018high}
\begin{bbook}[author]
\bauthor{\bsnm{Vershynin},~\bfnm{Roman}\binits{R.}}
(\byear{2018}).
\btitle{High-Dimensional Probability: An Introduction with Applications in Data Science}
\bvolume{47}.
\bpublisher{Cambridge University Press, Cambridge}.
\end{bbook}
\endbibitem

\bibitem[\protect\citeauthoryear{Wang et~al.}{2007}]{wang2007Statistics}
\begin{barticle}[author]
\bauthor{\bsnm{Wang},~\bfnm{Rui}\binits{R.}}, \bauthor{\bsnm{Lagakos},~\bfnm{Stephen~W.}\binits{S.~W.}}, \bauthor{\bsnm{Ware},~\bfnm{James~H.}\binits{J.~H.}}, \bauthor{\bsnm{Hunter},~\bfnm{David~J.}\binits{D.~J.}} \AND \bauthor{\bsnm{Drazen},~\bfnm{Jeffrey~M.}\binits{J.~M.}}
(\byear{2007}).
\btitle{Statistics in Medicine — Reporting of Subgroup Analyses in Clinical Trials}.
\bjournal{New England Journal of Medicine}
\bvolume{357}
\bpages{2189-2194}.
\bnote{PMID: 18032770}.
\bdoi{10.1056/NEJMsr077003}
\end{barticle}
\endbibitem

\bibitem[\protect\citeauthoryear{Watson and Holmes}{2020}]{watson2020machine}
\begin{barticle}[author]
\bauthor{\bsnm{Watson},~\bfnm{James~A}\binits{J.~A.}} \AND \bauthor{\bsnm{Holmes},~\bfnm{Chris~C}\binits{C.~C.}}
(\byear{2020}).
\btitle{Machine learning analysis plans for randomised controlled trials: detecting treatment effect heterogeneity with strict control of type {I} error}.
\bjournal{Trials}
\bvolume{21}
\bpages{1--10}.
\end{barticle}
\endbibitem

\bibitem[\protect\citeauthoryear{Willett and Nowak}{2007}]{willett2007minimax}
\begin{barticle}[author]
\bauthor{\bsnm{Willett},~\bfnm{Rebecca~M}\binits{R.~M.}} \AND \bauthor{\bsnm{Nowak},~\bfnm{Robert~D}\binits{R.~D.}}
(\byear{2007}).
\btitle{Minimax optimal level-set estimation}.
\bjournal{IEEE Transactions on Image Processing}
\bvolume{16}
\bpages{2965--2979}.
\end{barticle}
\endbibitem

\bibitem[\protect\citeauthoryear{Xia et~al.}{2021}]{xia2021intentional}
\begin{barticle}[author]
\bauthor{\bsnm{Xia},~\bfnm{Lucy}\binits{L.}}, \bauthor{\bsnm{Zhao},~\bfnm{Richard}\binits{R.}}, \bauthor{\bsnm{Wu},~\bfnm{Yanhui}\binits{Y.}} \AND \bauthor{\bsnm{Tong},~\bfnm{Xin}\binits{X.}}
(\byear{2021}).
\btitle{Intentional control of type {I} error over unconscious data distortion: A {N}eyman--{P}earson approach to text classification}.
\bjournal{Journal of the American Statistical Association}
\bvolume{116}
\bpages{68--81}.
\end{barticle}
\endbibitem

\bibitem[\protect\citeauthoryear{Zanette, Zhang and Kochenderfer}{2018}]{zanette2018robust}
\begin{binproceedings}[author]
\bauthor{\bsnm{Zanette},~\bfnm{Andrea}\binits{A.}}, \bauthor{\bsnm{Zhang},~\bfnm{Junzi}\binits{J.}} \AND \bauthor{\bsnm{Kochenderfer},~\bfnm{Mykel~J}\binits{M.~J.}}
(\byear{2018}).
\btitle{Robust super-level set estimation using gaussian processes}.
In \bbooktitle{Joint European Conference on Machine Learning and Knowledge Discovery in Databases}
\bpages{276--291}.
\bpublisher{Springer}.
\end{binproceedings}
\endbibitem

\bibitem[\protect\citeauthoryear{Zhang et~al.}{2015}]{zhang2015subgroup}
\begin{barticle}[author]
\bauthor{\bsnm{Zhang},~\bfnm{Sheng}\binits{S.}}, \bauthor{\bsnm{Liang},~\bfnm{Fei}\binits{F.}}, \bauthor{\bsnm{Li},~\bfnm{Wenfeng}\binits{W.}} \AND \bauthor{\bsnm{Hu},~\bfnm{Xichun}\binits{X.}}
(\byear{2015}).
\btitle{Subgroup analyses in reporting of phase {III} clinical trials in solid tumors}.
\bjournal{Journal of Clinical Oncology}
\bvolume{33}
\bpages{1697--1702}.
\end{barticle}
\endbibitem

\bibitem[\protect\citeauthoryear{Zhang et~al.}{2017}]{zhang2017subgroup}
\begin{barticle}[author]
\bauthor{\bsnm{Zhang},~\bfnm{Zhiwei}\binits{Z.}}, \bauthor{\bsnm{Li},~\bfnm{Meijuan}\binits{M.}}, \bauthor{\bsnm{Lin},~\bfnm{Min}\binits{M.}}, \bauthor{\bsnm{Soon},~\bfnm{Guoxing}\binits{G.}}, \bauthor{\bsnm{Greene},~\bfnm{Tom}\binits{T.}} \AND \bauthor{\bsnm{Shen},~\bfnm{Changyu}\binits{C.}}
(\byear{2017}).
\btitle{Subgroup selection in adaptive signature designs of confirmatory clinical trials}.
\bjournal{J. Roy. Stat. Soc., Ser. C}
\bvolume{66}
\bpages{345--361}.
\end{barticle}
\endbibitem

\end{thebibliography}


\begin{thebibliography}{13}

\bibitem[\protect\citeauthoryear{Brown and Low}{1996}]{brown1996constrained}
\begin{barticle}[author]
\bauthor{\bsnm{Brown},~\bfnm{Lawrence~D}\binits{L.~D.}} \AND
  \bauthor{\bsnm{Low},~\bfnm{Mark~G}\binits{M.~G.}}
(\byear{1996}).
\btitle{A constrained risk inequality with applications to nonparametric
  functional estimation}.
\bjournal{Annals of Statistics}
\bvolume{24}
\bpages{2524--2535}.
\end{barticle}
\endbibitem

\bibitem[\protect\citeauthoryear{Cannings, Berrett and
  Samworth}{2020}]{cannings2020local}
\begin{barticle}[author]
\bauthor{\bsnm{Cannings},~\bfnm{Timothy~I.}\binits{T.~I.}},
  \bauthor{\bsnm{Berrett},~\bfnm{Thomas~B.}\binits{T.~B.}} \AND
  \bauthor{\bsnm{Samworth},~\bfnm{Richard~J.}\binits{R.~J.}}
(\byear{2020}).
\btitle{Local nearest neighbour classification with applications to
  semi-supervised learning}.
\bjournal{Annals of Statistics}
\bvolume{48}
\bpages{1789--1814}.
\end{barticle}
\endbibitem

\bibitem[\protect\citeauthoryear{Constantine and
  Savits}{1996}]{constantine1996multivariate}
\begin{barticle}[author]
\bauthor{\bsnm{Constantine},~\bfnm{G}\binits{G.}} \AND
  \bauthor{\bsnm{Savits},~\bfnm{T}\binits{T.}}
(\byear{1996}).
\btitle{A multivariate {F}aa di {B}runo formula with applications}.
\bjournal{Transactions of the American Mathematical Society}
\bvolume{348}
\bpages{503--520}.
\end{barticle}
\endbibitem

\bibitem[\protect\citeauthoryear{Dudley}{2018}]{dudley2018real}
\begin{bbook}[author]
\bauthor{\bsnm{Dudley},~\bfnm{Richard~M}\binits{R.~M.}}
(\byear{2018}).
\btitle{Real Analysis and Probability}.
\bpublisher{CRC Press, Cambridge}.
\end{bbook}
\endbibitem

\bibitem[\protect\citeauthoryear{Garivier and Capp{\'e}}{2011}]{garivier2011kl}
\begin{binproceedings}[author]
\bauthor{\bsnm{Garivier},~\bfnm{Aur{\'e}lien}\binits{A.}} \AND
  \bauthor{\bsnm{Capp{\'e}},~\bfnm{Olivier}\binits{O.}}
(\byear{2011}).
\btitle{The {KL}-{UCB} algorithm for bounded stochastic bandits and beyond}.
In \bbooktitle{Proceedings of the 24th annual Conference On Learning Theory}
\bpages{359--376}.
\end{binproceedings}
\endbibitem

\bibitem[\protect\citeauthoryear{Gerchinovitz, M{\'e}nard and
  Stoltz}{2020}]{gerchinovitz2020fano}
\begin{barticle}[author]
\bauthor{\bsnm{Gerchinovitz},~\bfnm{Sebastien}\binits{S.}},
  \bauthor{\bsnm{M{\'e}nard},~\bfnm{Pierre}\binits{P.}} \AND
  \bauthor{\bsnm{Stoltz},~\bfnm{Gilles}\binits{G.}}
(\byear{2020}).
\btitle{Fano’s inequality for random variables}.
\bjournal{Statistical Science}
\bvolume{35}
\bpages{178--201}.
\end{barticle}
\endbibitem

\bibitem[\protect\citeauthoryear{McDiarmid}{1998}]{mcdiarmid1998concentration}
\begin{bincollection}[author]
\bauthor{\bsnm{McDiarmid},~\bfnm{Colin}\binits{C.}}
(\byear{1998}).
\btitle{Concentration}.
In \bbooktitle{Probabilistic Methods for Algorithmic Discrete Mathematics}
\bpages{195--248}.
\bpublisher{Springer}.
\end{bincollection}
\endbibitem

\bibitem[\protect\citeauthoryear{McShane}{1934}]{mcshane1934extension}
\begin{barticle}[author]
\bauthor{\bsnm{McShane},~\bfnm{Edward~James}\binits{E.~J.}}
(\byear{1934}).
\btitle{Extension of range of functions}.
\bjournal{Bulletin of the American Mathematical Society}
\bvolume{40}
\bpages{837--842}.
\end{barticle}
\endbibitem

\bibitem[\protect\citeauthoryear{Okamoto}{1973}]{okamoto1973distinctness}
\begin{barticle}[author]
\bauthor{\bsnm{Okamoto},~\bfnm{Masashi}\binits{M.}}
(\byear{1973}).
\btitle{Distinctness of the Eigenvalues of a Quadratic form in a Multivariate
  Sample}.
\bjournal{Annals of Statistics}
\bvolume{1}
\bpages{763--765}.
\bdoi{10.1214/aos/1176342472}
\end{barticle}
\endbibitem

\bibitem[\protect\citeauthoryear{Reeve, Cannings and
  Samworth}{2021}]{reeve2021adaptive}
\begin{barticle}[author]
\bauthor{\bsnm{Reeve},~\bfnm{Henry W~J}\binits{H.~W.~J.}},
  \bauthor{\bsnm{Cannings},~\bfnm{Timothy~I}\binits{T.~I.}} \AND
  \bauthor{\bsnm{Samworth},~\bfnm{Richard~J}\binits{R.~J.}}
(\byear{2021}).
\btitle{Adaptive Transfer Learning}.
\bjournal{Annals of Statistics}
\bvolume{49}
\bpages{3618-3649}.
\end{barticle}
\endbibitem

\bibitem[\protect\citeauthoryear{Reeve, Cannings and
  Samworth}{2023}]{reeve2021optimal}
\begin{barticle}[author]
\bauthor{\bsnm{Reeve},~\bfnm{Henry W.~J.}\binits{H.~W.~J.}},
  \bauthor{\bsnm{Cannings},~\bfnm{Timothy~I.}\binits{T.~I.}} \AND
  \bauthor{\bsnm{Samworth},~\bfnm{Richard~J.}\binits{R.~J.}}
(\byear{2023}).
\btitle{Optimal subgroup selection}.
\bjournal{Submitted}.
\end{barticle}
\endbibitem

\bibitem[\protect\citeauthoryear{Tropp}{2012}]{tropp2012user}
\begin{barticle}[author]
\bauthor{\bsnm{Tropp},~\bfnm{Joel~A.}\binits{J.~A.}}
(\byear{2012}).
\btitle{User-Friendly Tail Bounds for Sums of Random Matrices}.
\bjournal{Foundations of Computational Mathematics}
\bvolume{12}
\bpages{389--434}.
\bdoi{10.1007/s10208-011-9099-z}
\end{barticle}
\endbibitem

\bibitem[\protect\citeauthoryear{Vershynin}{2018}]{vershynin2018high}
\begin{bbook}[author]
\bauthor{\bsnm{Vershynin},~\bfnm{Roman}\binits{R.}}
(\byear{2018}).
\btitle{High-Dimensional Probability: An Introduction with Applications in Data
  Science}
\bvolume{47}.
\bpublisher{Cambridge University Press, Cambridge}.
\end{bbook}
\endbibitem

\end{thebibliography}

\newpage

\title{Supplementary material for `Optimal subgroup selection'}
\runtitle{Optimal subgroup selection}

\begin{aug}
\author[A]{\fnms{Henry} W. J. \snm{Reeve}\ead[label=e1]{henry.reeve@bristol.ac.uk}},
\author[B]{\fnms{Timothy} I. \snm{Cannings}\thanksref{t1}\ead[label=e2]{timothy.cannings@ed.ac.uk}} \\
\thankstext{t1}{Research supported by Engineering and Physical Sciences Research Council (EPSRC) New Investigator Award EP/V002694/1.}
\and
\author[C]{\fnms{Richard} J. \snm{Samworth}\thanksref{t2}\ead[label=e3]{r.samworth@statslab.cam.ac.uk}}
\thankstext{t2}{Research supported by Engineering and Physical Sciences Research Council (EPSRC) Programme grant EP/N031938/1, EPSRC Fellowship EP/P031447/1 and European Research Council Advanced Grant 101019498.}

\runauthor{H. W. J. Reeve, T. I. Cannings and R. J. Samworth}

 \address[A]{School of Mathematics, University of Bristol\\\href{mailto:henry.reeve@bristol.ac.uk}{henry.reeve@bristol.ac.uk} 
}

  \address[B]{School of Mathematics and Maxwell Institute for Mathematical Sciences, The University of Edinburgh\\\href{mailto:timothy.cannings@ed.ac.uk}{timothy.cannings@ed.ac.uk}
}
        
  \address[C]{Statistical Laboratory, University of Cambridge\\
  \href{mailto:r.samworth@statslab.cam.ac.uk}{r.samworth@statslab.cam.ac.uk}
}

\end{aug}

\begin{abstract}
This is the supplementary material for \citet{reeve2021optimal}.
\end{abstract}


\setcounter{section}{0}
\setcounter{equation}{0}
\setcounter{theorem}{0}
\setcounter{example}{0}
\def\theequation{S\arabic{equation}}
\def\thesection{S\arabic{section}}
\def\thetheorem{S\arabic{theorem}}
\def\thefigure{S\arabic{figure}}
\def\thedefinition{S\arabic{definition}}
\def\theexample{S\arabic{example}}
\def\theremark{S\arabic{remark}}

\section{Proof of the hardness result}

\begin{proof}[Proof of Proposition~\ref{prop:lipschitzNotEnoughForConsistency}]
Fix $\epsilon \in (0,1)$, and let $(P_{x})_{x \in \R^d}$ be a disintegration of $\probDistribution$ into conditional probability measures on $[0,1]$; see Section~\ref{Sec:Disintegration} and Lemma~\ref{lemma:existenceOfDisintegrationFromDudley}. Since $\mathcal{A}$ has finite VC dimension, it follows from the Vapnik--Chervonenkis concentration inequality (Lemma~\ref{lemma:vapnikChervonenkisConcentration}) that there exists a finite set $\mathbb{T} \subseteq \R^d$ for which
\begin{align}\label{eq:GCThmConsequenceInPfNegProp}
\sup_{A \in \mathcal{A}}\bigg|\frac{1}{|\mathbb{T}|}\sum_{t \in \mathbb{T}} \one_{\{t \in A\}} - \mu(A)\bigg| \leq \epsilon.
\end{align}
Since $\mathbb{T}$ is finite and $\mu$ has no atoms, we may choose a radius $r>0$ sufficiently small that $\mu\bigl(\bigcup_{t \in \mathbb{T}}B_r(t)\bigr) \leq \epsilon$.  Now define a function $\rho:\R^d \rightarrow [0,1]$ by $\rho(x):=1 \wedge \bigwedge_{t \in \mathbb{T}}\{(2/r)\cdot \|x-t\|_{\infty}\}$, noting that $\rho$ is Lipschitz.  Further, define a family of probability distributions $(Q_{x})_{x \in \R^d}$ on $[0,1]$ by
\[
 \int_{[0,1]} h(y) \, dQ_x(y) = \int_{[0,1]} h\bigl(\rho(x)\cdot y\bigr) \, dP_x(y),
\]
for all Borel functions $h: [0,1] \rightarrow [0,1]$, and define a probability distribution $Q$ on $\R^d \times [0,1]$ by $Q(A\times B) = \int_A Q_x(B)\,d\mu(x)$. It follows that $(Q_{x})_{x \in \R^d}$ is a disintegration of $Q$ into conditional probability measures on $[0,1]$. In addition, taking a random pair $(X^Q,Y^Q) \sim Q$ we see by~\eqref{eq:conditionalExpectationInTermsOfDisintegration} that for $\mu$-almost every $x \in \R^d$,
\begin{align*}
\eta_Q(x)= \E\bigl(Y^Q \mid X^Q=x\bigr) = \int_{[0,1]} y \, dQ_x(y) = \rho(x)\cdot \int_{[0,1]}  y \, dP_x(y)  = \rho(x) \cdot \eta_P(x).
\end{align*}
Hence, we may extend the definition of $\eta_Q$ to $\R^d$ in such a way that $\eta_{Q}(\cdot) = \rho(\cdot)\eta_{P}(\cdot)$, which is a product of Lipschitz functions, so is itself Lipschitz; thus, $Q \in \mathcal{P}_{\mathrm{Lip}}(\mu)$. Note also that for every $t \in \mathbb{T}$, we have $\eta_{Q}(t) = \rho(t)\eta_{P}(t) = 0<\tau$, so $\mathbb{T} \cap \mathcal{X}_\tau({\eta_{Q}})=\emptyset$. Moreover, $\eta_{Q}(x) \leq \eta_{P}(x)$ for all $x \in \R^d$, so $\mathcal{X}_{\tau}{(\eta_{Q})}\subseteq \mathcal{X}_\tau(\eta_{P})$. Hence, since $\hat{A}$ controls the Type I error at the level $\alpha$ over $\mathcal{P}_{\mathrm{Lip}}(\mu)$, we have
\begin{align}
\label{Eq:TypeIQ}
\Prob_Q\bigl\{ \hat{A} \subseteq{} (\R^d \setminus \mathbb{T}) \cap\mathcal{X}_{\tau}{(\eta_{P})}\bigr\} \geq \Prob_Q\bigl\{ \hat{A} \subseteq{} (\R^d \setminus \mathbb{T}) \cap\mathcal{X}_{\tau}{(\eta_{Q})}\bigr\} &= \Prob_Q\bigl\{ \hat{A} \subseteq{} \mathcal{X}_{\tau}{(\eta_{Q})}\bigr\} \nonumber \\
&\geq 1-\alpha.    
\end{align}
Now $Q_x=P_x$ for $x \notin \bigcup_{t \in \mathbb{T}}B_r(t)$.  Hence
\[
\hellingerDistance^2(P,Q) \leq 2\mathrm{TV}(P,Q) \leq 2\mu\biggl(\bigcup_{t \in \mathbb{T}} B_r(t)\biggr) \leq 2\epsilon,
\]
so
\begin{align}
\label{Eq:TVBound}
\totalVariationDistance^2(P^{\otimes n},Q^{\otimes n}) \leq
\hellingerDistance^2(P^{\otimes n},Q^{\otimes n}) =2\biggl\{1-\prod_{i=1}^{n}\biggl(1-\frac{\hellingerDistance^2(P,Q)}{2}\biggr)\biggr\} 
\leq 2\bigl(1-(1-\epsilon)^n\bigr).
\end{align}
Note that by \eqref{eq:GCThmConsequenceInPfNegProp} if $A \in \mathcal{A}$ satisfies $A \cap \mathbb{T}=\emptyset$, then $\mu(A)\leq \epsilon$. Hence, by~\eqref{Eq:TypeIQ} and~\eqref{Eq:TVBound}, we have 
\begin{align*}
\Prob_P\bigl(\bigl\{\mu(\hat{A}) \leq \epsilon\bigr\} \cap \bigl\{\hat{A} \subseteq{} \etaSuperLevelSet{\tau}\bigr\}\bigr) & \geq 
\Prob_P\bigl(\hat{A} \subseteq{} \bigl(\R^d \setminus \mathbb{T}\bigr) \cap\mathcal{X}_{\tau}{(\eta_{P})}\bigr) \\ & \geq 1-\alpha-\sqrt{2\bigl(1-(1-\epsilon)^n\bigr)}.
\end{align*}
Letting $\epsilon \searrow 0$ gives~\eqref{Eq:FirstConc}.
Thus,
\begin{align*}
R_\tau(\hat{A})=M_\tau-\E_P\bigl({\mu}(\hat{A}) \mid \hat{A} \subseteq{} \etaSuperLevelSet{\tau}\bigr) \geq (1-\alpha) \cdot M_\tau.
\end{align*}
Finally, note that for any $\xi>0$, we may take $A_\xi \in \mathcal{A}$ with $\mu(A_\xi)>M_\tau-\xi$ and $A_\xi \subseteq{} \etaSuperLevelSet{\tau}$. Hence, we may define $\bar{A} \in \bar{\mathcal{A}}$ that takes the value $A_\xi$ with probability~$\alpha$ and~$\emptyset$ otherwise; it has regret $R_\tau(\hat{A})< (1-\alpha) \cdot M_\tau + \alpha \cdot \xi$. Letting $\xi \searrow 0$ yields the final equality in \eqref{eq:conclusionPowerWithoutHolderKnoweldge}.
\end{proof}

\section{Proof of the upper bound in Theorem \ref{thm:minimaxRate}}
\label{Sec:ProofUpperBound}

Recall that Theorem~\ref{thm:minimaxRate}(i) will follow from Lemma~\ref{lemma:pValue}, together with Propositions~\ref{thm:typeIControl} and~\ref{thm:powerBound}.
\begin{proof}[Proof of Lemma~\ref{lemma:pValue}] We begin by showing that $\sup_{x \in B}\eta(x) \leq t:=\tau+\holderConstant\cdot \diamSup(B)^{\holderExponent}$. Indeed, suppose for a contradiction that there exists some $x_0 \in B$ with $\eta(x_0)>t$. Since $B\not \subseteq \etaSuperLevelSet{\tau}$ there also exists $x_1 \in B$ with $\regressionFunction(x_1) \leq \tau$. Since $\regressionFunction$ is continuous there exists $x_2$ on the line segment between $x_0$ and $x_1$ with $\regressionFunction(x_2)=\tau$. Thus, since $x_0$, $x_2 \in \etaSuperLevelSet{\tau}$ we have,
\begin{align*}
\regressionFunction(x_0)& \leq \regressionFunction(x_2)+|\regressionFunction(x_0)-\regressionFunction(x_2)|\leq \tau+ \holderConstant \cdot \supNorm{x_0-x_2}^{\holderExponent}\\
& \leq \tau+ \holderConstant \cdot \supNorm{x_0-x_1}^{\holderExponent} \leq \tau+\holderConstant \cdot \diamSup(B)^{\holderExponent}=t<\regressionFunction(x_0),
\end{align*}
a contradiction which proves the claim $\sup_{x \in B}\eta(x) \leq t$. Now let $m := n \cdot \empiricalMarginalDistribution(B) = \sum_{i \in [n]}\one_{\{X_i \in B\}}$.  If $m = 0$, then $\hat{p}_n(B) = 1$, so we may assume without loss of generality that $m \geq 1$.  Let $(i_j)_{j \in [m]}$ denote a strictly increasing sequence such that $X_{i_j} \in B$ for all $j \in [m]$. For each $j \in [m]$ let $Z_j := Y_{i_j}$ so that 
\begin{align*}
{\E}\big( Z_j \mid\sampleX \big)&={\E}\big( Y_{i_j} \mid\sampleX \big)= \regressionFunction(X_{i_j}) \leq  t. 
\end{align*}
Moreover, $(Z_j)_{j \in [m]}$ are conditionally independent given $\sampleX$.  Writing $\bar{Z} := m^{-1}\sum_{j \in [m]} Z_j$, we have by construction of $\pValueN(B)$ that
\begin{align*}
{\Prob}\big( \pValueN(B) \leq \alpha \mid \sampleX \bigr) &={\Prob}\Bigl[ \exp\bigl\{ -n \cdot \empiricalMarginalDistribution(B)\cdot \kl\bigl( \empiricalRegressionFunction(B),t\bigr) \bigr\} \leq \alpha \text{ and } \empiricalRegressionFunction(B)>t \Bigm| \sampleX \Bigr]\\
&={\Prob}\bigl\{ \kl\bigl( \bar{Z},t\bigr) \geq m^{-1}\cdot{\log(1/\alpha)} \text{ and } \bar{Z} > t \bigm| \sampleX \bigr\} \leq \alpha,
\end{align*}
where the final inequality follows from a Chernoff bound, stated for convenience as Lemma~\ref{consequenceOfGarivier2011kl}.
\end{proof}
The proof of Proposition~\ref{thm:typeIControl} will rely on the following lemma.
\begin{lemma}\label{lemma:typeIControl} Fix $\holderExponent \in (0,1]$, $\holderConstant \geq 1$ and $\probDistribution \in \classOfHolderDistributionsSuperLevelSet$ with $\regressionFunction \in \mathcal{F}_{\mathrm{H\ddot{o}l}}\bigl(\holderExponent,\holderConstant,\etaSuperLevelSet{\tau}\bigr)$.  Then, with $(B_{(\ell)})_{\ell\in [\numberOfConsideredHyperCubes]}$, $\ell_{\alpha}$ and $m$ as in Algorithm~\ref{subsetSelectionAlgo} (and setting $\ell_\alpha := 0$ when $\numberOfConsideredHyperCubes \cdot \hat{p}_n(B_{(1)}) > \alpha$), we have
\begin{align*}
{\Prob}\biggl( \bigcup_{\ell \in [\ell_{\alpha}]}B_{(\ell)} \not \subseteq \etaSuperLevelSet{\tau} \biggm|\sampleX \biggr) \leq \alpha.
\end{align*}
\end{lemma}
\begin{proof} Let  
$\mathcal{N}(\sampleX):=\{ B \in \setOfHypercubes(\sampleX): B \not\subseteq \etaSuperLevelSet{\tau}\}$ and $K:=|\mathcal{N}(\sampleX)|$.  Note that  
\[
\biggl\{\bigcup_{\ell \in [\ell_{\alpha}]}B_{(\ell)} \not \subseteq \etaSuperLevelSet{\tau}\biggr\} \bigcap \{K = 0\} \subseteq \biggl\{\bigcup_{\ell \in [\numberOfConsideredHyperCubes]}B_{(\ell)} \not \subseteq \etaSuperLevelSet{\tau}\biggr\} \bigcap \{K = 0\} = \emptyset.
\]
On the other hand, when $K \geq 1$, we may write 
\[
\tilde{\ell}:=\min\bigl\{\ell \in [\numberOfConsideredHyperCubes]~:~B_{(\ell)} \in \mathcal{N}(\sampleX)\bigr\},
\]
so that when $\numberOfConsideredHyperCubes \cdot \hat{p}_n(B_{(1)})\leq \alpha$, we have $\bigcup_{\ell \in [\ell_{\alpha}]}B_{(\ell)} \not \subseteq \etaSuperLevelSet{\tau}$ if and only if $\tilde{\ell}\leq \ell_{\alpha}$.  When $K \geq 1$, we have by the minimality of $\tilde{\ell}$ that $\tilde{\ell}\leq \numberOfConsideredHyperCubes+1-K$, so
\begin{align*}
{\Prob}\biggl(\bigcup_{\ell \in [\ell_{\alpha}]}B_{(\ell)} \not \subseteq \etaSuperLevelSet{\tau}&\biggm|\sampleX \biggr) \\
&= {\Prob}\biggl(\biggl\{\bigcup_{\ell \in [\ell_{\alpha}]}B_{(\ell)} \not \subseteq \etaSuperLevelSet{\tau}\biggr\} \bigcap \{\numberOfConsideredHyperCubes \cdot \hat{p}_n(B_{(1)})\leq \alpha\}\biggm|\sampleX \biggr) \\
&= {\Prob}\big( \{\tilde{\ell}\leq \ell_{\alpha}\} \cap \{\numberOfConsideredHyperCubes \cdot \hat{p}_n(B_{(1)})\leq \alpha\} \bigm|\sampleX \big) \\
&\leq{\Prob}\big( \bigl\{(\numberOfConsideredHyperCubes+1-\tilde{\ell})\cdot \hat{p}_n(B_{(\tilde{\ell})}) \leq \alpha\bigr\} \cap \{\numberOfConsideredHyperCubes \cdot \hat{p}_n(B_{(1)})\leq \alpha\}\bigm|\sampleX \big)\\
&\leq {\Prob}\biggl(K\cdot \min_{B \in \mathcal{N}(\sampleX)} \hat{p}_n(B) \leq \alpha \biggm|\sampleX \biggr)\\
&\leq \sum_{B \in \mathcal{N}(\sampleX)}{\Prob}\biggl(\hat{p}_n(B) \leq \frac{\alpha}{K} \biggm|\sampleX \biggr)\leq \alpha,
\end{align*}
where we applied Lemma~\ref{lemma:pValue} for the final inequality. 
\end{proof}
\begin{proof}[Proof of Proposition~\ref{thm:typeIControl}] By construction in Algorithm~\ref{subsetSelectionAlgo}, we have $\hat{A}_{\mathrm{OSS}}(\sample) \subseteq \bigcup_{\ell \in [\ell_{\alpha}]}B_{(\ell)}$. Hence the result follows from Lemma~\ref{lemma:typeIControl}.
\end{proof}
We now turn to the proof of Proposition~\ref{thm:powerBound}.  A key component of this result is the following proposition, which states that if a set $A \in \mathcal{A}$ may be covered with a finite collection of hyper-cubes $\{B_1,\ldots,B_L\}\subseteq \setOfHypercubes$, each with sufficiently large diameter and $\mu$-measure, in such a way that $\regressionFunction$ is well above the level $\tau$ on each $B_\ell$, then $\hat{A}_{\mathrm{OSS}}$ will return a set of $\mu$-measure comparable with $\mu(A)$.
\begin{prop}\label{prop:generalPowerBound} Take $\alpha \in (0,1)$, $n \in \N$, $\delta \in (0,1)$, $(\holderExponent,\holderConstant) \in (0,1]\times (0,\infty)$, $\probDistribution \in \classOfHolderDistributions$ and $\mathcal{A} \subseteq \borel(\R^d)$ with $\vcDim(\mathcal{A})<\infty$ and $\emptyset \in \mathcal{A}$.  Given $L \in \N$, suppose that there exist hyper-cubes $\{B_1,\ldots,B_L\} \subseteq \setOfHypercubes$ such that $\min_{q \in [L]} \mu(B_q)\geq 8\log(4L/\delta)/n$, $\min_{q \in [L]}\diamSup(B_q) \geq 1/n$ and 
\begin{align}\label{eq:conditionForGeneralPowerBound}
\min_{q \in [L]} \biggl\{\sup_{x \in B_q}\regressionFunction(x)-2\holderConstant\cdot \diamSup(B_q)^{\holderExponent}-\sqrt{\frac{2\log\bigl(2^{2+d}L \cdot n (2+\log_2 n)/(\alpha \cdot \delta)\bigr)}{n\cdot \mu(B_q)}}\biggr\}  \geq \tau.
\end{align}
Let $S^{\dagger}:=\bigcup_{q \in [L]}B_q$, and taking the universal constant $C_{\mathrm{VC}} > 0$ from Lemma~\ref{lemma:vapnikChervonenkisConcentration}, let
\begin{align*}
J_{n,\delta}(S^{\dagger}):=\sup\big\{ \mu(A): A \in \mathcal{A}\cap \powerSet( S^{\dagger})\big\}-2C_{\mathrm{VC}}~ \sqrt{\frac{\vcDim(\mathcal{A})}{n}}-\sqrt{\frac{2\log(2/\delta)}{n}} .
\end{align*}
Then 
\[
\Prob\bigl\{\mu\bigl(\hat{A}_{\mathrm{OSS}}(\sample)\bigr) < J_{n,\delta}(S^{\dagger}) \bigr\} \leq \delta.
\]
\end{prop}
Proposition \ref{prop:generalPowerBound} will be proved through a series of lemmas below. 
\begin{lemma}\label{lemma:multiplicativeChernoffOnBoxesLemma} Let $P$ be a distribution on $\R^d \times [0,1]$ having marginal $\mu$ on $\R^d$, and let $\delta \in (0,1)$, $n \in \N$ and $L \in \N$.  Suppose further that $\{B_1,\ldots,B_L\} \subseteq \setOfHypercubes$ with $\min_{q \in [L]} \mu(B_q)\geq 8\log(4L/\delta)/n$, and define the event
\[
\mathcal{E}_{1,\delta} := \biggl\{\min_{q \in [L]}\biggl(\empiricalMarginalDistribution(B_q) - \frac{\mu(B_q)}{2}\biggr) > 0\biggr\}.
\]  Then ${\Prob}(\mathcal{E}_{1,\delta}^{\mathrm{c}}) \leq \delta/4$.
\end{lemma}
\begin{proof}
By the multiplicative Chernoff bound (Lemma~\ref{lemma:multChernoff}), for each $q \in [L]$,
\begin{align*}
{\Prob}\biggl(  \subSampleEmpiricalMarginalDistribution(B_q) \leq \frac{\mu(B_q)}{2}   \biggr)  = {\Prob}\biggl( \sum_{i=1}^m \one_{\{X_i \in B_q\}} \leq \frac{m}{2} \cdot \mu(B_q)  \biggr) \leq \exp\biggl(-\frac{m}{8}\cdot \mu(B_q)\biggr)\leq \frac{\delta}{4L}.
\end{align*}
The result therefore follows by a union bound.
\end{proof}

\begin{lemma}\label{lemma:hoeffdingOnBoxesLemma} Let $P$ be a distribution on $\R^d \times [0,1]$ having regression function $\regressionFunction:\R^d \rightarrow [0,1]$, and let $\delta \in (0,1)$, $n \in \mathbb{N}$ and $L \in \mathbb{N}$.  Suppose that $\{B_1,\ldots,B_L\} \subseteq \setOfHypercubes$ and define the event
\[
\mathcal{E}_{2,\delta} := \biggl\{\max_{q \in [L]}\biggl(\inf_{x \in B_q}\regressionFunction(x) -\empiricalRegressionFunction(B_q)- \sqrt{\frac{\log(4L/\delta)}{2n \cdot \empiricalMarginalDistribution(B_q)}}\biggr) < 0\biggr\},
\]
where the empirical distribution $\hat{\mu}_n$ and empirical regression function $\hat{\eta}_n$ are defined in~\eqref{eq:empiricalMarginalDef} and~\eqref{eq:empiricalRegFuncDef} respectively.  Then ${\Prob}(\mathcal{E}_{2,\delta}^{\mathrm{c}}) \leq \delta/4$.
\end{lemma}
\begin{proof} 
By Hoeffding's inequality (Lemma~\ref{consequenceOfGarivier2011kl}), for every $q \in [L]$, we have
\begin{align*}
 \esssup~{\Prob}\biggl( \empiricalRegressionFunction(B_q) \leq \inf_{x \in B_q}\regressionFunction(x) - \sqrt{\frac{\log(4L/\delta)}{2n \cdot \empiricalMarginalDistribution(B_q)}} \biggm|\sampleX \biggr)\leq \frac{\delta}{4L}.
\end{align*}
The result now follows by the law of total expectation, combined with a union bound.
\end{proof}

\begin{lemma}\label{lemma:selectedBoxes} Let $(\holderExponent,\holderConstant) \in (0,1]\times [1,\infty)$, $\probDistribution \in \classOfHolderDistributions$, $\delta \in (0,1)$, $n \in \N$, $L \in \N$ and $\xi \in (0,1)$.  Suppose that for some $\{B_1,\ldots,B_L\} \subseteq \setOfHypercubes$ we have $\min_{q \in [L]} \mu(B_q)\geq 8\log(8L/\delta)/n$ and 
\begin{align}\label{eq:assumptionForLemmaSelectedBoxes} 
\min_{q \in [L]} \biggl\{\sup_{x \in B_q}\regressionFunction(x)-2\holderConstant\cdot \diamSup(B_q)^{\holderExponent}-\sqrt{\frac{2\log\bigl(4L/(\xi\cdot \delta)\bigr)}{n\cdot \mu(B_q)}}\biggr\} \geq \tau.
\end{align}
Then, recalling the definition of the $p$-values $\hat{p}_n$ from~\eqref{eq:pValueDef}, we have
\[
\Prob\biggl(\max_{q \in [L]}~\pValueN(B_q) \geq \xi  \biggr) \leq \frac{\delta}{2}.
\]
\end{lemma}
\begin{proof} By Lemmas~\ref{lemma:multiplicativeChernoffOnBoxesLemma} and~\ref{lemma:hoeffdingOnBoxesLemma}, we have $\Prob(\mathcal{E}_{1,\delta}^{\mathrm{c}} \cup \mathcal{E}_{2,\delta}^{\mathrm{c}}) \leq \delta/2$.  On $\mathcal{E}_{1,\delta} \cap \mathcal{E}_{2,\delta}$, we have for each $q \in [L]$ that 
\begin{align*}
\empiricalRegressionFunction(B_q)&>  \inf_{x \in B_q}\regressionFunction(x) -  \sqrt{\frac{\log(4L/\delta)}{n\cdot \mu(B_q)}}\\
&\geq \sup_{x \in B_q}\regressionFunction(x)-\holderConstant\cdot \diamSup(B_q)^{\holderExponent} -  \sqrt{\frac{\log(4L/\delta)}{n\cdot \mu(B_q)}}\\ 
&\geq \tau+\holderConstant\cdot \diamSup(B_q)^{\holderExponent}+\sqrt{\frac{\log(1/\xi)}{n\cdot \mu(B_q)}},
\end{align*}
where we used the fact that $\probDistribution \in \classOfHolderDistributions$, \eqref{eq:assumptionForLemmaSelectedBoxes} and the fact that $\sqrt{2(a+b)} \geq \sqrt{a} + \sqrt{b}$ for all $a,b \geq 0$. Thus, on $\mathcal{E}_{1,\delta} \cap \mathcal{E}_{2,\delta}$, we have for every $q \in [L]$ that  
\begin{align*}
\pValueN(B_q) \leq  \exp\biggl( -\frac{n\cdot \mu(B_q)}{2}\cdot \kl\bigl\{ \empiricalRegressionFunction(B_q),\tau+\holderConstant\cdot \diamSup(B_q)^{\holderExponent}\bigr\} \biggr)<\xi,
\end{align*}
as required, where the final bound uses Pinsker's inequality.
\end{proof}

We can now complete the proof of Proposition \ref{prop:generalPowerBound} before returning to complete the proof of Proposition~\ref{thm:powerBound}.

\begin{proof}[Proof of Proposition \ref{prop:generalPowerBound}] We begin by defining events
\begin{align*}
\mathcal{E}_{\mathrm{PV}}&:=\biggl\{\max_{q \in [L]} \pValueN(B_q) < \frac{\alpha}{2^d n(2+\log_2 n)}\biggr\}, \\
\mathcal{E}_{\mathrm{VC}}&:=\biggl\{\sup_{A\in \mathcal{A}} \big| \hat{\mu}_n(A)-\mu(A) \big| \leq C_{\mathrm{VC}}\sqrt{\frac{\vcDim(\mathcal{A})}{n}}+\sqrt{\frac{\log(2/\delta)}{2n}} \biggr\}.
\end{align*}
By Lemma~\ref{lemma:selectedBoxes}, with $\xi = \alpha/\bigl(2^d n(2+\log_2 n)\bigr) \in (0,1)$, and Lemma~\ref{lemma:vapnikChervonenkisConcentration}, we have $\Prob(\mathcal{E}_{\mathrm{PV}}^{\mathrm{c}} \cup \mathcal{E}_{\mathrm{VC}}^{\mathrm{c}} ) \leq \delta$. On $\mathcal{E}_{\mathrm{PV}}$ we have each $\pValueN(B_q)<1$, which implies $\empiricalMarginalDistribution(B_q)>0$, and  since $\diamSup(B_1) \geq 1/n$ we deduce that $B_q \in \mathcal{H}(\sampleX)$.

Now $\numberOfConsideredHyperCubes=|\setOfHypercubes(\sampleX)| \leq 2^d n(2+\log_2 n)$, so on the event $\mathcal{E}_{\mathrm{PV}}$, for each $q \in [L]$, we  have $\numberOfConsideredHyperCubes \cdot \pValueN(B_q) \leq \alpha$, and hence $B_q = B_{(\ell(q))}$ for some $\ell(q) \leq \ell_{\alpha}$. Thus, on the event $\mathcal{E}_{\mathrm{PV}}$ we have $S^{\dagger} \subseteq \bigcup_{\ell \in [\ell_{\alpha}]}B_{(\ell)}$. Now take $\zeta >0$ and choose $A_{\zeta}^* \in \mathcal{A}\cap \powerSet(S^{\dagger})$ with $\mu(A^*_{\zeta})> \sup\big\{ \mu(A): A \in \mathcal{A}\cap \powerSet( S^{\dagger})\big\} - \zeta$. It follows that $A_{\zeta}^* \in \mathcal{A}\cap \powerSet\bigl(\bigcup_{\ell \in [\ell_{\alpha}]}B_{(\ell)}\bigr)$ and hence on the event $\mathcal{E}_{\mathrm{PV}}\cap \mathcal{E}_{\mathrm{VC}}$ that
\begin{align*}
\mu(\hat{A}_{\mathrm{OSS}}) & \geq \hat{\mu}_n(\hat{A}_{\mathrm{OSS}}) - C_{\mathrm{VC}}\sqrt{\frac{\vcDim(\mathcal{A})}{n}}-\sqrt{\frac{\log(2/\delta)}{2n}} \\
& \geq \hat{\mu}_n(A^*_{\zeta}) - C_{\mathrm{VC}}\sqrt{\frac{\vcDim(\mathcal{A})}{n}}-\sqrt{\frac{\log(2/\delta)}{2n}} \\
& \geq {\mu}(A^*_{\zeta}) - 2C_{\mathrm{VC}}\sqrt{\frac{\vcDim(\mathcal{A})}{n}}-\sqrt{\frac{2\log(2/\delta)}{n}} \geq J_{n,\delta}(S^{\dagger})-\zeta.
\end{align*}
Letting $\zeta \searrow 0$, we conclude that $\mu(\hat{A}_{\mathrm{OSS}}) \geq J_{n,\delta}(S^{\dagger})$, on the event $\mathcal{E}_{\mathrm{PV}}\cap \mathcal{E}_{\mathrm{VC}}$, as required.
\end{proof}
\begin{proof}[Proof of Proposition~\ref{thm:powerBound}] We define $\rho:=\approximableDensityExponent(2\holderExponent+d)+\holderExponent\approximableMarginExponent$, 
\begin{align*}
\thetaNDelta:= \frac{8 \holderConstant^{d/\holderExponent}}{n} \log_+\biggl(\frac{4 \cdot 3^d\cdot n}{\alpha \wedge \delta}\biggr),
\end{align*}
$r_* :=\holderConstant^{-1/\holderExponent} \thetaNDelta^{\approximableDensityExponent/\rho}$, $\xi:= \thetaNDelta^{\holderExponent\approximableMarginExponent/\rho}$ and $\Delta:= 2^4 \thetaNDelta^{\holderExponent\approximableDensityExponent/\rho}$. We initially assume that $\Delta \leq 1$, so that $1/n \leq \holderConstant^{-d/\holderExponent} \cdot \thetaNDelta \leq r_* \leq 2^{-4}$.  Now choose a maximal subset $\{x_1,\ldots, x_L\} \subseteq \omegaSuperLevelSet{\xi}\cap \etaSuperLevelSet{\tau+\Delta}$ with the property that $\|x_q-x_{q'}\|_\infty >r_*$ for distinct $q,q' \in [L]$. Then $\omegaSuperLevelSet{\xi}\cap \etaSuperLevelSet{\tau+\Delta} \subseteq \bigcup_{q \in [L]} \closedMetricBallSupNorm{x_q}{r_*}$ and $\closedMetricBallSupNorm{x_{q}}{r_*/3}\cap \closedMetricBallSupNorm{x_{q'}}{r_*/3} = \emptyset$ for distinct $q, q' \in [L]$.  Now, since $\xi \leq 1$,
\begin{align*}
L \leq \sum_{q=1}^L  \frac{\mu \bigl(\closedMetricBallSupNorm{x_q}{r_*/3}\bigr)}{\xi \cdot (r_*/3)^d} &= \frac{1}{\xi \cdot (r_*/3)^d} \cdot \mu \biggl( \bigcup_{q=1}^L  \closedMetricBallSupNorm{x_q}{r_*/3}\biggr) \\
&\leq \frac{(3 \holderConstant^{1/\holderExponent})^d}{ \thetaNDelta^{(\holderExponent \approximableMarginExponent +d\approximableDensityExponent)/\rho}} \leq \frac{3^d\holderConstant^{d/\holderExponent}}{\theta} \leq 3^d n.
\end{align*}
For each $q \in [L]$ we can find $B_q \in \setOfHypercubes$ such that $\closedMetricBallSupNorm{x_q}{r_*}\subseteq B_q$ and such that $r_* \leq \diamSup(B_q) \leq 2^{-\lceil \log_2(\frac{1}{2r_*}) \rceil + 1} \leq 4r_*$,  which is possible since $r_* \leq 1/4$. We then have that for every $q \in [L]$, 
\[
\mu(B_q) \geq \mu\bigl(\closedMetricBallSupNorm{x_q}{r_*}\bigr) \geq \xi \cdot r_*^d \geq \frac{\thetaNDelta}{\holderConstant^{d/\holderExponent}} \geq \frac{8}{n}\log(4L/\delta).
\]
Hence
\begin{align*}
\min_{q \in [L]}&\biggl\{\sup_{x \in B_q}\regressionFunction(x)-2\holderConstant\cdot \diamSup(B_q)^{\holderExponent}-\sqrt{\frac{2\log\bigl(2^{2+d}L \cdot n (2+\log_2 n)/(\alpha \cdot \delta)\bigr)}{n\cdot \mu(B_q)}}\biggr\}\\
&\geq \min_{q \in [L]}\biggl\{\regressionFunction(x_q) -2^{1+2\holderExponent}\cdot\holderConstant \cdot r_*^{\holderExponent}-\sqrt{\frac{2\log\bigl(2^{3+d}3^d n^2 \log_2(n)/(\alpha \cdot \delta)\bigr)}{n\cdot \xi \cdot r_*^d}}\biggr\}\\
&\geq \tau+\Delta-2^3 \cdot \holderConstant \cdot r_*^{\holderExponent}-\sqrt{\frac{\thetaNDelta }{\xi \cdot r_*^d\cdot \holderConstant^{d/\holderExponent}}} \geq \tau,
\end{align*}
so \eqref{eq:conditionForGeneralPowerBound} holds. Thus, taking $S^{\dagger}:= \bigcup_{q \in [L]}B_q \supseteq \omegaSuperLevelSet{\xi}\cap \etaSuperLevelSet{\tau+\Delta}$, when $\Delta \leq 1$ we may apply Proposition~\ref{prop:generalPowerBound} to see that with probability at least $1-\delta$, we have
\begin{align}
\mu&(\hat{A}_{\mathrm{OSS}}) 
\geq\sup\big\{ \mu(A): A \in \mathcal{A}\cap \powerSet( S^{\dagger})\big\}-2C_{\mathrm{VC}} \sqrt{\frac{\vcDim(\mathcal{A})}{n}}-\sqrt{\frac{2\log(2/\delta)}{n}} \nonumber \\
&\geq \sup\big\{ \mu(A): A \in \mathcal{A}\cap \powerSet\bigl( \omegaSuperLevelSet{\xi}\cap \etaSuperLevelSet{\tau+\Delta}\bigr)\big\}-2C_{\mathrm{VC}} \sqrt{\frac{\vcDim(\mathcal{A})}{n}}-\sqrt{\frac{2\log(2/\delta)}{n}} \nonumber \\
&\geq  M_\tau - \approximableSetsConstant \cdot (\xi^{\approximableDensityExponent}+\Delta^{\approximableMarginExponent})-2C_{\mathrm{VC}} \sqrt{\frac{\vcDim(\mathcal{A})}{n}}-\sqrt{\frac{2\log(2/\delta)}{n}} \label{Eq:BigDisplay} \\
&=  M_\tau - \approximableSetsConstant \cdot (1+2^{5\approximableMarginExponent}) \cdot \thetaNDelta^{\holderExponent \approximableDensityExponent \approximableMarginExponent/\rho}-2C_{\mathrm{VC}} \sqrt{\frac{\vcDim(\mathcal{A})}{n}}-\sqrt{\frac{2\log(2/\delta)}{n}} \nonumber \\
& \geq M_\tau-C  \biggl\{ \biggl(\frac{\holderConstant^{d/\holderExponent}\cdot \log_+\bigl(n/(\alpha \wedge \delta)\bigr)}{n}\biggr)^{\frac{\holderExponent \approximableDensityExponent \approximableMarginExponent}{\approximableDensityExponent(2\holderExponent+d)+\holderExponent\approximableMarginExponent}}+\biggl(\frac{\log_+(1/\delta)}{n}\biggr)^{1/2}\biggr\},\label{Eq:BigDisplay2}
\end{align}
where $C\geq 1$ depends only on $d$, $\approximableDensityExponent$, $\approximableMarginExponent$, $\approximableSetsConstant$ and $\vcDim(\mathcal{A})$.  Finally, if $\Delta> 1$, then~\eqref{Eq:BigDisplay} holds because $\mu(\hat{A}_{\mathrm{OSS}}) \geq 0$, $M_\tau \leq 1$ and $\approximableSetsConstant \geq 1$, so~\eqref{Eq:BigDisplay2} holds too.  This completes the proof of the first claim of the proposition.

For the second claim, observe by Proposition~\ref{thm:typeIControl} that for $\alpha \in (0,1/2]$,
\begin{align*}
R_\tau(\hat{A}_{\mathrm{OSS}}) \leq \frac{M_\tau - \E\mu(\hat{A}_{\mathrm{OSS}})}{\Prob\bigl(\hat{A}_{\mathrm{OSS}} \subseteq{} \etaSuperLevelSet{\tau}\bigr)} &\leq \frac{M_\tau - \E\mu(\hat{A}_{\mathrm{OSS}})}{1-\alpha} \\
&\leq \tilde{C}\biggl\{ \biggl(\frac{\holderConstant^{d/\holderExponent}\cdot\log_+(n/\alpha)}{n}\biggr)^{\frac{\holderExponent \approximableDensityExponent \approximableMarginExponent}{\approximableDensityExponent(2\holderExponent+d)+\holderExponent\approximableMarginExponent}}+\frac{1}{n^{1/2}}\biggr\},
\end{align*}
where the final bound follows by integrating the tail bound in the first part of the proposition.
\end{proof}


\section{Proofs of claims in Examples~\ref{ex:1},~\ref{ex:2} and~\ref{ex:3} and a related result}
\label{Sec:Examples}

\textbf{Example~\ref{ex:1}:}
The marginal density of $X$ is convex on $(-\infty,-\nu-2]$ and on $[\nu+2,\infty)$, so writing $\phi$ for the standard normal density, we have $\omega(x) = \phi(x-\nu) + \phi(x+\nu)$ for $|x| \geq \nu+2$.  Hence there exists $\xi_0 \in \bigl(0,1/\sqrt{2\pi}\bigr]$, depending only on $\nu$, such that for $\xi \in (0,\xi_0]$ we have $\omegaSuperLevelSet{\xi} = [-x_\xi,x_\xi]$, where $x_\xi \in \bigl[\nu + \sqrt{2\log\bigl(\frac{1}{(2\pi)^{1/2}\xi}\bigr)},\nu + \sqrt{2\log\bigl(\frac{2^{1/2}}{\pi^{1/2}\xi}\bigr)}\bigr]$ satisfies $\phi(x_\xi-\nu) + \phi(x_\xi+\nu) = \xi$.  In fact, when $\nu \geq 2$, we may take $\xi_0 = 2\phi(\nu)$.  Moreover, $\eta(x) = \frac{\phi(x-\nu)}{\phi(x-\nu) + \phi(x+\nu)} = \frac{1}{1+e^{-2x\nu}}$, so $\etaSuperLevelSet{\tau+\Delta} = [x_{\nu,\tau,\Delta},\infty)$, where $x_{\nu,\tau,\Delta} := \frac{1}{2\nu}\log\bigl(\frac{\tau+\Delta}{1 - (\tau+\Delta)}\bigr)$.  By reducing $\xi_0 > 0$, depending only on $\nu$ and $\tau$, if necessary, we may assume that $-x_\xi \leq x_{\nu,\tau,\Delta} \leq x_\xi$ for $\xi \in (0,\xi_0]$ and $\Delta \in \bigl(0,(1-\tau)/2\bigr]$.  Writing $\Phi$ for the standard normal distribution function, we deduce that for $\xi \in (0,\xi_0]$,
\begin{align*}
\sup\bigl\{ \mu(A):A \in \mathcal{A}_{\mathrm{int}} &\cap \powerSet\bigl(\omegaSuperLevelSet{\xi}\cap \etaSuperLevelSet{\tau+\Delta}\bigr)  \bigr\} = \mu\bigl([x_{\nu,\tau,\Delta},x_\xi]\bigr) \\
&= \frac{1}{2}\Phi(x_\xi - \nu) - \frac{1}{2}\Phi(x_{\nu,\tau,\Delta} - \nu) + \frac{1}{2}\Phi(x_\xi + \nu) - \frac{1}{2}\Phi(x_{\nu,\tau,\Delta} + \nu).
\end{align*}
Using the Mills ratio and the mean value inequality, it follows that for $\xi \in (0,\xi_0]$,
\begin{align*}
M_\tau &- \sup\bigl\{\mu(A):A \in \mathcal{A}_{\mathrm{int}} \cap \powerSet\bigl(\omegaSuperLevelSet{\xi}\cap \etaSuperLevelSet{\tau+\Delta}\bigr)\bigr\} \\
&= 1 - \frac{1}{2}\Phi(x_\xi - \nu) - \frac{1}{2}\Phi(x_\xi + \nu) - \frac{1}{2}\Phi(x_{\nu,\tau,0} - \nu) - \frac{1}{2}\Phi(x_{\nu,\tau,0} + \nu) \\
&\hspace{7cm}+ \frac{1}{2}\Phi(x_{\nu,\tau,\Delta} - \nu) + \frac{1}{2}\Phi(x_{\nu,\tau,\Delta} + \nu) \\
&\leq \frac{\phi(x_\xi-\nu)}{2(x_{\xi} - \nu)} + \frac{\phi(x_\xi+\nu)}{2(x_{\xi} +\nu)} + \frac{\Delta}{\sqrt{2\pi}\nu\tau(1-\tau)} \\
&\leq \frac{\xi}{2\sqrt{2\log\bigl(\frac{1}{(2\pi)^{1/2}\xi_0}\bigr)}} + \frac{\Delta}{\sqrt{2\pi}\nu\tau(1-\tau)}.
\end{align*}
On the other hand, when $\xi > \xi_0$, we have
\[
M_\tau - \sup\bigl\{\mu(A):A \in \mathcal{A}_{\mathrm{int}} \cap \powerSet\bigl(\omegaSuperLevelSet{\xi}\cap \etaSuperLevelSet{\tau+\Delta}\bigr)\bigr\} \leq 1 \leq \frac{\xi}{\xi_0}.
\]
We conclude that $P \in \classOfWellApproximableSetsWithIntervals$ with $\approximableDensityExponent = \approximableMarginExponent = 1$ when we take 
\[
\approximableSetsConstant^{-1} = \min\biggl\{2\sqrt{2\log\biggl(\frac{1}{(2\pi)^{1/2}\xi_0}\biggr)},\sqrt{2\pi}\nu\tau(1-\tau),\xi_0\biggr\}.
\]

\bigskip

\noindent \textbf{Example~\ref{ex:2}:}
Given $\epsilon_0 > 0$, choose $A_0 \in \mathcal{A}_{\mathrm{hpr}}\cap \powerSet\bigl( \etaSuperLevelSet{\tau} \cap [0,1]^d\bigr)$ such that $\mu(A_0) \geq M_\tau - \epsilon_0$.  Let $\partial A_0$ and $r = (r_1,\ldots,r_d) \in [0,1]^d$ denote the boundary and vector of side-lengths of $A_0$ respectively.  Observe that for $\Delta \leq \epsilon \cdot \delta^{1/\gamma}$,
\begin{align*}
\etaSuperLevelSet{\tau+\Delta} &\supseteq \bigl\{x \in \etaSuperLevelSet{\tau} \cap [0,1]^d:\distSup(x,\mathcal{S}_\tau) \geq (\Delta/\epsilon)^\gamma\bigr\} \\
&\supseteq \bigl\{x \in A_0: \distSup(x,\partial A_0) \geq (\Delta/\epsilon)^\gamma\bigr\}.
\end{align*}
Moreover, $\omegaSuperLevelSet{\xi} = [0,1]^d$ for $\xi \leq 1$.  For $s > 0$, let $A_0(s) := \bigl\{x \in A_0: \distSup(x,\partial A_0) \geq s\bigr\}$.  Note that for $s \leq \min_j r_j/2$,
\[
\mu(A_0) - \mu(A_0(s)) \leq \prod_{j=1}^d r_j - \prod_{j=1}^d (r_j-2s) \leq 1 - (1-2s)^d \leq 2ds.
\]
On the other hand, if $s > \min_{j \in [d]} r_j/2$, then 
\[
\mu(A_0) - \mu(A_0(s)) \leq \prod_{j=1}^d r_j \leq \min_{j \in [d]} r_j < 2s.
\]
Then, for $\xi \in [0,1]$ and any $\Delta \in (0,\epsilon \cdot \delta^{1/\gamma}]$,
\begin{align*}
M_\tau - \sup\bigl\{\mu(A):A \in \mathcal{A}_{\mathrm{hpr}}  \cap \powerSet\bigl(\omegaSuperLevelSet{\xi}\cap \etaSuperLevelSet{\tau+\Delta}\bigr)\bigr\} &\leq M_\tau - \mu\bigl(A_{0,(\Delta/\epsilon)^\gamma}\bigr) \\
&\leq M_\tau - \mu(A_0) + 2d\Bigl(\frac{\Delta}{\epsilon}\Bigr)^\gamma \\
&\leq \epsilon_0 + 2d\Bigl(\frac{\Delta}{\epsilon}\Bigr)^\gamma.
\end{align*}
On the other hand, if $\xi > 1$ or $\Delta > \epsilon \cdot \delta^{1/\gamma}$, then for any $\approximableDensityExponent \in (0,\infty)$ and $\approximableSetsConstant \geq 1/(\epsilon^{\approximableMarginExponent}\delta)$, we have
\[
M_\tau - \sup\bigl\{\mu(A):A \in \mathcal{A}_{\mathrm{hpr}}\cap \powerSet\bigl(\omegaSuperLevelSet{\xi}\cap \etaSuperLevelSet{\tau+\Delta}\bigr)\bigr\} \leq 1 \leq \approximableSetsConstant \cdot (\xi^{\approximableDensityExponent}+\Delta^{\approximableMarginExponent}).
\]
Since $\epsilon_0 > 0$ was arbitrary, the conclusion follows.

\bigskip

\noindent \textbf{Example~\ref{ex:3}:} 
Writing $\omega_\kappa := \omega_{\mu_\kappa,d}$ for the lower-density of $\mu_\kappa$, we have for $x \in \R^d$ that
\begin{align*}
\omega_{\kappa}(x) &\geq \sup_{t \in [\supNorm{x},\infty)} g_{\kappa}(t)\biggl\{ \Lebesgue\bigl(\closedMetricBallSupNorm{x}{1 \wedge t} \cap \closedMetricBallSupNorm{0}{t}\bigr) \wedge \inf_{r \in (0,1 \wedge t)} \frac{\Lebesgue\bigl(\closedMetricBallSupNorm{x}{r} \cap \closedMetricBallSupNorm{0}{t}\bigr)}{r^d}\biggr\} \\
&\geq \sup_{t \in [\supNorm{x},\infty)} (1 \wedge t^d) g_{\kappa}(t) \\
&\geq \begin{cases} g_\kappa(\supNorm{x}) \wedge g_{\kappa}(1) &\text{ if } \kappa \in (0,2) \\
\Bigl(\frac{1}{2^d(\kappa-1)} \cdot g_\kappa(\supNorm{x})\Bigr) \wedge g_{\kappa}\Bigl(\frac{1}{2(\kappa-1)^{1/d}}\Bigr) &\text{ if } \kappa \in [2,\infty).
\end{cases}
\end{align*}
Now, writing $a_{d,\kappa} := \frac{1}{2(\kappa-1)^{1/d}}\cdot \one_{\{\kappa \geq 2\}}+\one_{\{\kappa<2\}}$ and $\xi_{d,\kappa} := g_{\kappa}(a_{d,\kappa})$, we have for $\xi \leq \xi_{d,\kappa}$ that $\mathcal{X}_\xi(\omega_\kappa) \supseteq \bigl\{x \in \R^d: \supNorm{x} \leq R_{\xi,d,\kappa}\bigr\}$, where
\[
R_{\xi,d,\kappa} := \begin{cases} \Bigl(\frac{(\kappa/(2^d \xi))^{1-\kappa} - 1}{1-\kappa}\Bigr)^{1/d} &\text{ if } \kappa \in (0,1)\\
\log^{1/d}\Bigl(\frac{1}{2^d\xi}\Bigr) &\text{ if } \kappa=1\\ 
\Bigl(\frac{1 - \{2^d\xi/(\kappa a_{d,\kappa}^d)\}^{\kappa-1}}{\kappa-1}\Bigr)^{1/d} &\text{ if } \kappa \in (1,\infty). 
\end{cases}
\]
We now calculate that
\begin{equation}
\label{eq:mtau3rdExample}
M_\tau \equiv M_\tau(P_{\kappa,\gamma},\mathcal{A}_{\mathrm{hpr}}) = \mu_\kappa\bigl([0,\infty) \times \R^{d-1}\bigr) = 1/2.
\end{equation}
Observe that $\mathcal{X}_{\tau+\Delta}(\eta_{\gamma}) = \{x = (x_1,\ldots,x_d)^\top \in \R^d: x_1 \geq (\Delta/\holderConstant)^\gamma\}$ for $\Delta \in (0,1 - \tau]$.  Hence, for $\Delta \in (0,1 - \tau]$ and $\xi \leq \xi_{d,\kappa}$, we have
\begin{align}
\label{Eq:suplowerbound}
    \sup\big\{ \mu_{\kappa}(A):A &\in \mathcal{A}_{\mathrm{hpr}}  \cap \powerSet\bigl(\mathcal{X}_\xi(\omega_\kappa) \cap \mathcal{X}_{\tau+\Delta}(\eta_\gamma)\bigr)  \big\} \nonumber \\
    &\geq \sup\big\{ \mu_{\kappa}(A):A \in \mathcal{A}_{\mathrm{hpr}}  \cap \powerSet\bigl(\closedMetricBallSupNorm{0}{R_{\xi,d,\kappa}} \cap \bigl([(\Delta/\holderConstant)^\approximableMarginExponent,\infty) \times \R^{d-1}\bigr)\bigr)  \big\} \nonumber \\
    &= \mu_{\kappa}\bigl(\closedMetricBallSupNorm{0}{R_{\xi,d,\kappa}} \cap \bigl([(\Delta/\holderConstant)^\approximableMarginExponent,\infty) \times \R^{d-1}\bigr)\bigr) \nonumber \\
      &\geq \frac{1}{2}\mu_{\kappa}\bigl( \closedMetricBallSupNorm{0}{R_{\xi,d,\kappa}}\bigr)-\mu_{\kappa}\bigl([0,(\Delta/\holderConstant)^\approximableMarginExponent]\times \R^{d-1}\bigr).
    \end{align}
For $x = (x_1,\ldots,x_d)^\top \in \R^d$, let $x_{-1} := (x_2,\ldots,x_d)^\top \in \R^{d-1}$.  Then
\begin{align}
\label{Eq:gamma2}
\mu_{\kappa}\bigl([0,(\Delta/\holderConstant)^\approximableMarginExponent]\times \R^{d-1}\bigr) &= \int_{[0,(\Delta/\holderConstant)^\approximableMarginExponent]\times \R^{d-1}} g_{\kappa}(\supNorm{x}) \, dx \nonumber \\
&\leq \int_{[0,(\Delta/\holderConstant)^\approximableMarginExponent]\times \R^{d-1}} g_{\kappa}(\supNorm{x_{-1}}) \, dx 
= b_{d,\kappa} \cdot \biggl(\frac{\Delta}{\holderConstant}\biggr)^\gamma,
\end{align}
where
\begin{align*}
b_{d,\kappa}:=\int_{\R^{d-1}} g_{\kappa}(\supNorm{x_{-1}}) \, dx_{-1} &= (d-1) \cdot 2^{d-1}\int_0^\infty y^{d-2} g_\kappa(y) \, dy \\
&= \begin{cases}
\frac{(1-\kappa)^{1/d}\Gamma(2-1/d)\Gamma(\frac{1 + (d-1)\kappa}{d(1-\kappa)})}{2\Gamma(\kappa/(1-\kappa))} &\text{ if } \kappa \in (0,1) \\
\Gamma(2-1/d)/2 &\text{ if } \kappa = 1 \\
\frac{(\kappa-1)^{1/d}\Gamma(2-1/d)\Gamma(2 + \frac{1}{\kappa-1})}{2\Gamma(2-\frac{1}{d} + \frac{1}{1-\kappa})}&\text{ if } \kappa \in (1,\infty).
\end{cases}
\end{align*}
Moreover, for $\xi \leq \xi_{d,\kappa} \leq \kappa/2^d$,
\begin{equation}
\label{Eq:kappa}
1 - \mu_{\kappa}\bigl( \closedMetricBallSupNorm{0}{R_{\xi,d,\kappa}}\bigr) = d \cdot 2^d \int_{R_{\xi,d,\kappa}}^\infty y^{d-1}g_{\kappa}(y) \, dy = \biggl(\frac{2^d \xi}{a_{d,\kappa}^d\kappa}\biggr)^\kappa.
\end{equation}
But for $\Delta > 1 - \tau$ or $\xi > \xi_{d,\kappa}$, we have 
\[
\sup\big\{ \mu_{\kappa}(A):A \in \mathcal{A}_{\mathrm{hpr}}  \cap \powerSet\bigl(\mathcal{X}_\xi(\omega_\kappa) \cap \mathcal{X}_{\tau+\Delta}(\eta_\gamma)\bigr)  \big\} \geq 0 \geq \frac{1}{2} -\biggl(\frac{\Delta}{1-\tau}\biggr)^{\approximableMarginExponent} - \biggl(\frac{\xi}{\xi_{d,\kappa}}\biggr)^{\kappa}.
\]
We deduce from~\eqref{eq:mtau3rdExample},~\eqref{Eq:suplowerbound},~\eqref{Eq:gamma2} and~\eqref{Eq:kappa} that $P_{\approximableDensityExponent,\approximableMarginExponent} \in \classOfWellApproximableSetsWithHyperCubes$, with 
\[
\approximableSetsConstant = \biggl(\frac{b_{d,\kappa}^{1/\gamma}}{\holderConstant} \vee \frac{1}{1-\tau}\biggr)^\gamma \vee \biggl(\frac{2^d}{a_{d,\kappa}^d\kappa}\vee \frac{1}{\xi_{d,\kappa}} \biggr)^\kappa.
\]

Given a closed set $S \subseteq \R^d$, we define the projection $\Pi_S:\R^d \rightarrow S$ by
\[
\Pi_S(x) := \sargmin_{z \in S} \|x - z\|_2,
\]
where $\sargmin$ denotes the smallest element of the $\argmin$ in the lexicographic ordering.  
\begin{prop}
\label{Prop:GeneralExample}
Let $P$ be a distribution on $\mathbb{R}^d \times [0,1]$ with marginal $\mu$ on $\R^d$ and continuous regression function $\eta$.  Suppose that $\mathcal{A} \subseteq \mathcal{B}(\R^d)$ has the properties that
\begin{itemize}
    \item there exists a bounded set $A_0 \in \mathcal{A} \cap \mathrm{Pow}\bigl(\etaSuperLevelSet{\tau}\bigr)$ such that $\mu(A_0) = M_\tau$;
    \item writing $\partial A_0$ for the topological boundary of $A_0$, there exist $s_0 > 0$ and $C_{\mathrm{App}}' > 0$ such that for every $s \in (0,s_0]$, we can find $A_0(s) \in \mathcal{A}$ with $A_0(s) \subseteq \{x \in A_0: \distEuclidean(x,\partial A_0) > s\}$, satisfying
    \[
    \mu(A_0) - \mu\bigl(A_0(s)\bigr) \leq C_{\mathrm{App}}' \cdot s.
    \]
\end{itemize}
Suppose that there exist $c_*, \kappa > 0$ such that
\[
\lowerDensity(x) \geq c_* \cdot s^{1/\kappa}
\]
for all $x \in A_0(s)$ and $s \in (0,s_0)$.  Assume further that $\mathcal{S}_{\tau} := \{x \in \R^d:\eta(x) = \tau\}$ is non-empty, and that there exists $\delta_0 > 0$ such that $\eta$ is differentiable on $\mathcal{S}_{\tau,\delta_0} := \mathcal{S}_\tau + \delta_0 \openMetricBallEuclideanNorm{0}{1}$.  Let $A_{0,\delta_0} := A_0 + \delta_0\openMetricBallEuclideanNorm{0}{1}$ and assume that $\epsilon_0 := \inf_{x \in  A_{0,\delta_0} \cap \mathcal{S}_{\tau,\delta_0}} \|\nabla \eta(x)\|_2 > 0$.  Then $\Delta_* := \inf_{x \in A_0 \setminus \mathcal{S}_{\tau,\delta_0}} \eta(x) - \tau > 0$, and $P \in \mathcal{P}_{\mathrm{App}}(\mathcal{A},\tau,\kappa,1,C_{\mathrm{App}})$ for 
\[
C_{\mathrm{App}} \geq \max\biggl\{\frac{C_{\mathrm{App}}'}{\max(c_*^\kappa,\epsilon_0)},\frac{1}{s_0c_*^\kappa},\frac{1}{\min(\epsilon_0\delta_0,\Delta_*)}\biggr\}.
\]
\end{prop}
\begin{proof}
Since $\eta$ is continuous on the intersection of the closure of $A_0$ with the complement of $\mathcal{S}_{\tau,\delta_0}$, and since this intersection is compact but does not contain any point in $\mathcal{S}_\tau$, we have that $\Delta_* > 0$.  By \citet[][Proposition~2]{cannings2020local}, we have
\begin{equation}
\label{Eq:Representation}
\mathcal{S}_{\tau,\delta_0} = \biggl\{x_0 + \frac{t\nabla \eta(x_0)}{\|\nabla \eta(x_0)\|_2}: x_0 \in \mathcal{S}_\tau, |t| < \delta_0\biggr\}.
\end{equation}
Moreover, from the proof of that result, we see that for any $x \in \mathcal{S}_{\tau,\delta_1}$, we can take $x_0 = \Pi_{\mathcal{S}_\tau}(x)$ in the representation~\eqref{Eq:Representation}.
Now suppose that $x \in A_0 \cap \mathcal{S}_{\tau,\delta_0}$, so that $x = x_0 + t\nabla \eta(x_0)/\|\nabla \eta(x_0)\|_2$ with $x_0 = \Pi_{\mathcal{S}_\tau}(x) \in A_{0,\delta_0} \cap \mathcal{S}_\tau$ and $|t| = \distEuclidean(x,\mathcal{S}_\tau) < \delta_0$.  Since the line segment joining $x_0$ and $x$ is contained in $A_{0,\delta_0} \cap \mathcal{S}_{\tau,\delta_0}$, we have  
\begin{align}
\label{Eq:etaminustaubound}
|\eta(x) - \tau| = \biggl|\eta\biggl(x_0 + \frac{t\nabla \eta(x_0)}{\|\nabla \eta(x_0)\|_2}\biggr) - \eta(x_0)\biggr| \geq |t|\epsilon_0.
\end{align}
Now observe that $\mathrm{int}(A_0) \cap \mathcal{S}_\tau = \emptyset$ because if $x_0 \in \mathrm{int}(A_0) \cap \mathcal{S}_\tau$, then for sufficiently small $t > 0$, the point $x_0 - t\nabla \eta(x_0)/\|\nabla \eta(x_0)\|_2$ would belong to $A_0$ and $\etaSuperLevelSet{\tau}^{\mathrm{c}}$, a contradiction.  Hence, for $x \in A_0$ we have $\distEuclidean(x,\mathcal{S}_\tau) \geq \distEuclidean(x,\partial A_0)$, so for $\Delta < \min(\epsilon_0\delta_0,\Delta_*)$,
\begin{align*}
\{x \in A_0:\eta(x) \in [\tau,\tau+ \Delta)\} \subseteq A_0 \cap \mathcal{S}_{\tau,\Delta/\epsilon_0} \subseteq \bigl\{x \in A_0: \distEuclidean(x,\partial A_0) \leq \Delta/\epsilon_0\bigr\}. 
\end{align*}
Then, for any $\xi \in \bigl(0,c_*s_0^{1/\kappa}\bigr)$ and $\Delta \in \bigl(0,\min(\epsilon_0\delta_0,\Delta_*)\bigr)$,
\begin{align*}
M_\tau - \sup\bigl\{\mu(A):A \in \mathcal{A} \cap \powerSet\bigl(\omegaSuperLevelSet{\xi} &\cap \etaSuperLevelSet{\tau+\Delta}\bigr)\bigr\} \leq \mu(A_0) - \mu\biggl(A_0\Bigl(\frac{\xi^\kappa}{c_*^\kappa} \wedge \frac{\Delta}{\epsilon_0}\Bigr)\biggr) \\
&\leq C_{\mathrm{App}}' \cdot \biggl(\frac{\xi^\kappa}{c_*^\kappa} \wedge \frac{\Delta}{\epsilon_0} \biggr) \leq  \approximableSetsConstant \cdot (\xi^{\kappa} + \Delta).
\end{align*}
On the other hand, if $\xi \geq c_*s_0^{1/\kappa}$ or $\Delta \geq \min(\epsilon_0\delta_0,\Delta_*)$, then provided we take $\approximableSetsConstant \geq \max\bigl\{1/(s_0c_*^\kappa),1/\min(\epsilon_0\delta_0,\Delta_*)\bigr\}$, we have
\begin{align*}
M_\tau - \sup\bigl\{\mu(A):A \in \mathcal{A} \cap \powerSet\bigl(\omegaSuperLevelSet{\xi}\cap \etaSuperLevelSet{\tau+\Delta}\bigr)\bigr\} \leq 1 &\leq \frac{1}{s_0} \cdot \Bigl(\frac{\xi}{c_*}\Bigr)^\kappa + \frac{\Delta}{\min(\epsilon_0\delta_0,\Delta_*)} \\
&\leq \approximableSetsConstant \cdot (\xi^\kappa + \Delta).
\end{align*}
The result follows.
\end{proof}

\section{Proofs from Section~\ref{Sec:Adaptation}}

\subsection{Proofs from Section~\ref{SubSec:lambda}}\label{subsec:proofOfLambdaEstResults}

\newcommand{\holderExponentMax}{\holderExponent_\ast}

In order to prove Theorem~\ref{thm:lipEst}, we first establish Proposition~\ref{prop:lipEstUniformOverHolderExponent} below, which will be useful in the sequel. Recall the definition of $\Delta_{n,\holderExponent,\holderConstant}$ from~\eqref{Eq:Deltanbeta}.

\begin{prop}\label{prop:lipEstUniformOverHolderExponent} Given $\delta \in (0,1)$ and $x_0$, $x_1 \in \R^d$, we let 
\begin{align*}
\mathcal{E}^{\mathrm{H\ddot{o}l}}_{n,\delta}(x_0,x_1):=\bigcap\biggl\{\biggl( \frac{|\eta(x_0)-\eta(x_1)|- \Delta_{n,\holderExponent,\holderConstant}(x_0)\vee \Delta_{n,\holderExponent,\holderConstant}(x_1)}{\|x_{0}-x_1\|_{\infty}^\holderExponent} \biggr) \leq \holderConstantEstimator\leq \holderConstant \biggr\},
\end{align*}
where the intersection is over all  $(\holderExponent,\holderConstant) \in (0,1] \times [1,\infty)$  for which $\probDistribution \in \mathcal{P}_{\mathrm{H\ddot{o}l}}(\holderExponent,\holderConstant,\tau)$ and $\min_{x \in \{x_0,x_1\}}\{\eta(x)-\Delta_{n,\holderExponent,\holderConstant}(x)\} \geq \tau$.  Here we adopt the convention that an intersection over an empty set is the entire probability space.  Then $\Prob_P\bigl(\mathcal{E}^{\mathrm{H\ddot{o}l}}_{n,\delta}(x_0,x_1)\bigr) \geq 1-\delta$.
\end{prop}
The proof of Proposition~\ref{prop:lipEstUniformOverHolderExponent} will make use of the event
\begin{align*}
\mathcal{E}_{\eta,\delta} :=\bigcap_{(i,k) \in [n]^2}\biggl\{ \biggl| \empiricalRegressionFunction\bigl(\closedMetricBallSupNorm{X_i}{r_{i,k}}\bigr)-\frac{\sum_{t\in [n]} \regressionFunction(X_t)\cdot\one_{{\{X_t \in \closedMetricBallSupNorm{X_i}{r_{i,k}}\}}}}{\sum_{t \in [n]} \one_{{\{X_t \in \closedMetricBallSupNorm{X_i}{r_{i,k}}\}}}} \biggr|\leq \sqrt{\frac{\log(4n^2/\delta)}{2k}} \biggr\}.
\end{align*}
\begin{proof}[Proof of Proposition~\ref{prop:lipEstUniformOverHolderExponent}] Fix $n \in \mathbb{N}$, $\delta \in (0,1)$, and $x_0,x_1 \in \R^d$.  We will assume without loss of generality that $\regressionFunction(x_0) \geq \regressionFunction(x_1)$ and $\min\{\omega(x_0),\omega(x_1)\} > 0$ since otherwise the index set in the intersection in the definition of $\mathcal{E}^{\mathrm{H\ddot{o}l}}_{n,\delta}(x_0,x_1)$ is empty.  By Hoeffding's lemma 
and a union bound, we have $\Prob\bigl(\mathcal{E}_{\eta,\delta}^{\mathrm{c}}|\sampleX\bigr) \leq \delta/2$. For $\ell \in [n]$ and $x \in \{x_0,x_1\}$ we let $s_\ell(x):= \bigl(n  \omega(x)/(2\ell)\bigr)^{-1/d}$ and define the event
\begin{align*}
\mathcal{E}_{\mu,\delta}(x_0,x_1) &:=  \bigcap_{x \in \{x_0,x_1\}} \bigcap_{\ell = \lceil 4\log(4n/\delta)\rceil}^{\lfloor n \omega(x)/2 \rfloor} \bigl\{  |\{X_i\}_{i \in [n]}\cap \closedMetricBallSupNorm{x}{s_\ell(x)}|\geq \ell \bigr\}.
\end{align*}
Observe that for $x \in \{x_0,x_1\}$ and $\ell \in \{\lceil 4\log(4n/\delta)\rceil,\ldots,\lfloor n \omega(x)/2 \rfloor\}$ we have $s_\ell(x) \in (0,1]$ and 
\begin{align}\label{eq:knLBHolderConstEst}
\mu\bigl(\closedMetricBallSupNorm{x}{s_{\ell}(x)}\bigr) \geq  \omega(x)\cdot s_\ell(x)^d = \frac{2\ell}{n}.
\end{align}
Hence, by the multiplicative Chernoff bound (Lemma~\ref{lemma:multChernoff}) we have $\Prob\bigl(\mathcal{E}_{\mu,\delta}(x_0,x_1)^{\mathrm{c}}\bigr) \leq \delta/2$. Consequently, to complete the proof it suffices to show that $ \mathcal{E}_{\eta,\delta} \cap \mathcal{E}_{\mu,\delta}(x_0,x_1)\subseteq \mathcal{E}^{\mathrm{H\ddot{o}l}}_{n,\delta}(x_0,x_1)$.  To this end, fix $(\holderExponent,\holderConstant) \in (0,1] \times [1,\infty)$  for which $\probDistribution \in \mathcal{P}_{\mathrm{H\ddot{o}l}}(\holderExponent,\holderConstant,\tau)$ and $\eta(x)-\Delta_{n}(x) \geq \tau$ for $x \in \{x_0,x_1\}$.  On the event $\mathcal{E}_{\eta,\delta}$, we have for any $(i,k) \in [n]^2$ with $\hat{\psi}_{n,\holderExponent,\delta}(i,k)> \holderConstant$ that
\begin{align*}
\frac{\sum_{t\in [n]} \regressionFunction(X_t)\cdot\one_{{\{X_t \in \closedMetricBallSupNorm{X_i}{r_{i,k}}\}}}}{\sum_{t \in [n]} \one_{{\{X_t \in \closedMetricBallSupNorm{X_i}{r_{i,k}}\}}}} \geq \empiricalRegressionFunction\bigl(\closedMetricBallSupNorm{X_i}{r_{i,k}}\bigr)-\sqrt{\frac{\log(4n^2/\delta)}{2k}} > \tau+ \holderConstant \cdot (2 r_{i,k})^\holderExponent.
\end{align*}
Hence, $\closedMetricBallSupNorm{X_i}{r_{i,k}} \subseteq \etaSuperLevelSet{\tau}$, since $\regressionFunction \in \classOfRestrictedHolderFunctions{\etaSuperLevelSet{\tau}}$. Consequently, on the event $\mathcal{E}_{\eta,\delta}$, given any $(i,j,k,\ell) \in [n]^4$ with $\min\bigl\{\hat{\psi}_{n,\holderExponent,\delta}(i,k),\hat{\psi}_{n,\holderExponent,\delta}(j,\ell) \bigr\} > \holderConstant$, we have
\begin{align*}
\empiricalRegressionFunction\bigl(\closedMetricBallSupNorm{X_i}{r_{i,k}}\bigr)&-\empiricalRegressionFunction\bigl(\closedMetricBallSupNorm{X_j} {r_{j,\ell}}\bigr)-\sqrt{\frac{2\log(4n^2/\delta)}{k\wedge \ell}}\\
& \leq \frac{\sum_{t \in [n]} \regressionFunction(X_t)\cdot\one_{\{X_t \in \closedMetricBallSupNorm{X_i}{r_{i,k}}\}}}{\sum_{t \in [n]} \one_{\{X_t \in \closedMetricBallSupNorm{X_i}{r_{i,k}}\}}}-\frac{\sum_{t\in [n]} \regressionFunction(X_t)\cdot\one_{\{X_t \in \closedMetricBallSupNorm{X_j}{r_{j,\ell}}\}}}{\sum_{t \in [n]} \one_{\{X_t \in \closedMetricBallSupNorm{X_j}{r_{j,\ell}}\}}}\\
& \leq \regressionFunction(X_i) +\holderConstant \cdot r_{i,k}^\holderExponent - \regressionFunction(X_j) +\holderConstant \cdot r_{j,\ell}^\holderExponent \leq \holderConstant \cdot \bigl(\|X_i-X_j\|_{\infty}^\holderExponent+r_{i,k}^\holderExponent+r_{j,\ell}^\holderExponent\bigr),
\end{align*}
and hence $\hat{\phi}_{n,\holderExponent,\delta}(i,j,k,\ell)\leq \holderConstant$. Thus, on the event $\mathcal{E}_{\eta,\delta}$, we have $ \holderConstantEstimator\leq \holderConstant$.  

For the lower bound, assume without loss of generality that $\Delta_{n}(x_0) \vee \Delta_{n}(x_1) \leq 1$ and let $i_0 := \sargmin_{i \in [n]} \|X_{i}-x_0\|_\infty$ and $i_1 := \sargmin_{i \in [n]} \|X_{i}-x_1\|_\infty$. For $x \in \{x_0,x_1\}$, let $\rho_n(x):= \bigl\{ \Delta_{n}(x) /(24\holderConstant)\bigr\}^{1/\holderExponent}$ and $k_{n}(x):= \lfloor \frac{n}{2}\cdot \omega(x) \cdot \rho_{n}(x)^d \rfloor$.  Since for each  $x \in \{x_0,x_1\}$, the fact that $\Delta_{n}(x) \leq 1$ ensures that $\frac{n}{2} \cdot \omega(x) \cdot \rho_{n}(x)^d \geq 4\log(4n/\delta)+2$, we have $k_{n}(x) \in \{ \lceil 4\log(4n/\delta) \rceil,\ldots, n\}$. Hence, on the event $\mathcal{E}_{\mu,\delta}(x_0,x_1)$ we have $|\{X_i\}_{i \in [n]}\cap \closedMetricBallSupNorm{x}{\rho_{n}(x)}|\geq k_{n}(x) \bigr\}$ for $x \in \{x_0,x_1\}$.

Now fix $j \in \{0,1\}$, so that $\|X_{i_j}-x_j\|_\infty \leq \rho_{n}(x_j)$ and $r_{i_j,k_{n}(x_j)}\leq 2\cdot \rho_{n}(x_j)$.  Then $\regressionFunction(x_j) \geq \tau+\Delta_{n}(x_j) = \tau + 24\holderConstant \rho_{n}(x_j)^\holderExponent$ and $\regressionFunction \in \classOfRestrictedHolderFunctions{\etaSuperLevelSet{\tau}}$, so $\closedMetricBallSupNorm{X_{i_j}}{r_{i_j,k_{n}(x_j)}} \subseteq \closedMetricBallSupNorm{x_j}{3\cdot \rho_{n}(x_j)} \subseteq \etaSuperLevelSet{\tau}$.  Since $\Delta_{n}(x_j)  \leq 1$, we also have $\sqrt{2\log(4n^2/\delta)/k_{n}(x_j)} \leq \Delta_{n}(x_j)/4$.  Hence, on the event $ \mathcal{E}_{\eta,\delta} \cap \mathcal{E}_{\mu,\delta}(x_0,x_1)$, we have
\begin{align*}
\empiricalRegressionFunction\bigl(\closedMetricBallSupNorm{X_{i_j}}{r_{i_j,k_{n}(x_j)}}\bigr) &- \tau-\sqrt{\frac{\log(4n^2/\delta)}{2k_{n}(x_j)}}\\& \geq \frac{\sum_{t \in [n]} \regressionFunction(X_t)\cdot\one_{\{X_t \in \closedMetricBallSupNorm{X_{i_j}}{r_{i_j,k_{n}(x_j)}}\}}}{\sum_{t \in [n]} \one_{\{X_t \in \closedMetricBallSupNorm{X_{i_j}}{r_{i_j},k_{n}(x_j)}}\}} - \tau -\sqrt{\frac{2\log(4n^2/\delta)}{k_{n}(x_j)}}\\
& \geq \regressionFunction(x_j)-\holderConstant \cdot \{3 \cdot \rho_{n}(x_j)\}^\holderExponent- \tau -\sqrt{\frac{2\log(4n^2/\delta)}{k_{n}(x_j)}}\\
& \geq \Delta_{n}(x_j)-\frac{3^\holderExponent}{24}\cdot \Delta_{n}(x_j)-\frac{1}{4}\cdot\Delta_{n}(x_j)\geq \frac{5}{8}\cdot \Delta_{n}(x_j) \\
&= 15 \holderConstant \cdot \rho_{n}(x_j)^\holderExponent >  \holderConstant \cdot  (2r_{i_j,k_{n}(x_j)})^\holderExponent.
\end{align*}
Thus, on the event  $ \mathcal{E}_{\eta,\delta} \cap \mathcal{E}_{\mu,\delta}(x_0,x_1)$, we have $\hat{\psi}_{n,\holderExponent,\delta}\bigl(i_j,k_{n}(x_j)\bigr)>\holderConstant$.  Hence, whenever $\Delta_{n}(x_0) \vee \Delta_{n}(x_1) \leq 1$, we have on the event $\mathcal{E}_{\eta,\delta} \cap \mathcal{E}_{\mu,\delta}(x_0,x_1)$ that
\begin{align*}
&\hat{\phi}_{n,\holderExponent,\delta}\bigl(i_0,i_1,k_{n}(x_0),k_{n}(x_1)\bigr) \\
&\hspace{0.5cm}=\frac{\empiricalRegressionFunction\bigl(\closedMetricBallSupNorm{X_{i_0}}{r_{i_0,k_{n}(x_0)}}\bigr)-\empiricalRegressionFunction\bigl(\closedMetricBallSupNorm{X_{i_1}}{r_{i_1,k_{n}(x_1)}}\bigr)-\sqrt{2\log(4n^2/\delta)/\{k_{n}(x_0) \wedge k_{n}(x_1)\}}}{\|X_{i_0}-X_{i_1}\|_{\infty}^\holderExponent+r_{i_0,k_{n}(x_0)}^\holderExponent+r_{i_1,k_{n}(x_1)}^\holderExponent}\\
 &\hspace{0.5cm} \geq \frac{\eta(x_0)-\eta(x_1)-2\holderConstant \bigl\{3\bigl(\rho_{n}(x_0) \vee \rho_{n}(x_1)\bigr)\bigr\}^\holderExponent-2\sqrt{2\log(4n^2/\delta)/\{k_{n}(x_0) \wedge k_1(x_0)\}}}{\|x_{0}-x_1\|_{\infty}^\holderExponent+3\cdot \bigl\{2\cdot \bigl(\rho_{n}(x_0) \vee \rho_{n}(x_1)\bigr)\bigr\}^\holderExponent}\\
 &\hspace{0.5cm} \geq \frac{\eta(x_0)-\eta(x_1)-\frac{3}{4}\bigl(\Delta_{n}(x_0) \vee \Delta_{n}(x_1)\bigr)}{\|x_{0}-x_1\|_{\infty}^\holderExponent+\frac{1}{4\holderConstant}\bigl(\Delta_{n}(x_0) \vee \Delta_{n}(x_1)\bigr)}\\
 &\hspace{0.5cm} \geq \frac{\eta(x_0)-\eta(x_1)-\frac{3}{4}\bigl(\Delta_{n}(x_0) \vee \Delta_{n}(x_1)\bigr)}{\|x_{0}-x_1\|_{\infty}^\holderExponent} - \frac{\Delta_{n}(x_0) \vee \Delta_{n}(x_1)}{4\holderConstant }\cdot \frac{\eta(x_0)-\eta(x_1)}{\|x_{0}-x_1\|_{\infty}^{2\holderExponent
}}\\ 
&\hspace{0.5cm} \geq \frac{\eta(x_0)-\eta(x_1)-\Delta_{n}(x_0) \vee \Delta_{n}(x_1)}{\|x_{0}-x_1\|_{\infty}^\holderExponent}.
\end{align*}
Thus,  whenever $\Delta_{n}(x_0) \vee \Delta_{n}(x_1) \leq 1$, we have on the event $ \mathcal{E}_{\eta,\delta} \cap \mathcal{E}_{\mu,\delta}(x_0,x_1)$ that
\begin{align*}
\holderConstantEstimator &\geq \min\bigl\{\hat{\psi}_{n,\holderExponent,\delta}\bigl(i_0,k_{n}(x_0)\bigr),\hat{\psi}_{n,\holderExponent,\delta}\bigl(i_1,k_{n}(x_1)\bigr),\hat{\phi}_{n,\holderExponent,\delta}\bigl(i_0,i_1,k_{n}(x_0),k_{n}(x_1)\bigr)\bigr\}\\
&\geq \min\biggl\{ \frac{\eta(x_0)-\eta(x_1)-\Delta_{n}(x_0) \vee \Delta_{n}(x_1)}{\|x_{0}-x_1\|_{\infty}^\holderExponent}, \holderConstant \biggr\}\\
&\geq  \min\biggl\{ \frac{\eta(x_0)-\eta(x_1)-\Delta_{n}(x_0) \vee \Delta_{n}(x_1)}{\|x_{0}-x_1\|_{\infty}^\holderExponent},\frac{\eta(x_0)-\eta(x_1)}{\|x_{0}-x_1\|_{\infty}^\holderExponent} \biggr\}\\
& =\frac{\eta(x_0)-\eta(x_1)-\Delta_{n}(x_0) \vee \Delta_{n}(x_1)}{\|x_{0}-x_1\|_{\infty}^\holderExponent},
\end{align*}
as required.
\end{proof}

\begin{proof}[Proof of Theorem~\ref{thm:lipEst}] Theorem~\ref{thm:lipEst} is an immediate consequence of Proposition~\ref{prop:lipEstUniformOverHolderExponent} combined with the continuity of probability from below.
\end{proof}

The following four lemmas are used in the proof of Corollary~\ref{Cor:HolderConstantEstimation}.
\begin{lemma}\label{lemma:omegaAEPositive} Let $\mu$ be a Borel probability distribution on $\R^d$ with lower density $\omega$.  Then $\mu\bigl(\{x \in \R^d:\omega(x) = 0\}\bigr) = 0$. Hence $\mathrm{supp}(\mu)=\mathrm{cl}\bigl(\{x \in \R^d:\omega(x) > 0\}\bigr)$.
\end{lemma}
\begin{proof} Let $Z:=\{x \in \R^d:\omega(x) = 0\}$ and fix $\epsilon\in (0,1)$. By the monotone convergence theorem we may choose $R>0$ sufficiently large that $\mu\bigl(Z \setminus \closedMetricBallSupNorm{0}{R}\bigr) < \epsilon/2$ and let $\xi:=\epsilon/\bigl( 2 \cdot \{5(R+1)\}^d\bigr)$.  By \citet[Lemma~S4]{reeve2021adaptive}, 
\begin{align*}
\mu\bigl(Z \cap \closedMetricBallSupNorm{0}{R}\bigr) \leq 
\mu\bigl(\bigl\{x \in \closedMetricBallSupNorm{0}{R}: \omega(x) < \xi\bigr\}\bigr) \leq  \{5(R+1)\}^d \cdot \xi = \frac{\epsilon}{2}.
\end{align*}
Hence $\mu(Z)=0$, so $\mu\bigl(\{x \in \R^d:\omega(x) > 0\}\bigr)=1$, and consequently $\mathrm{supp}(\mu)\subseteq \mathrm{cl}\bigl(\{x \in \R^d:\omega(x) > 0\}\bigr)$. Moreover, if $z \in \mathrm{cl}\bigl(\{x \in \R^d:\omega(x) > 0\}\bigr)$ then every neighbourhood of $z$ has positive $\mu$-measure, so $z \in \mathrm{supp}(\mu)$. The second conclusion therefore follows.
\end{proof}
Recall the definitions of $\underline{\holderConstant}_{\holderExponent}(P)$ and $\classOfDistributionsSatisfyingRegularityCondition$ from Section~\ref{SubSec:lambda}.
\begin{lemma}\label{lemma:LipschitzConstantReformulation} Given any distribution $P \in \classOfDistributionsSatisfyingRegularityCondition$, we have $\underline{\holderConstant}_{\holderExponent}(P)={\holderConstant_{\holderExponent,\flat}}(P)$ where 
\[
{\holderConstant_{\holderExponent,\flat}}(P): = \sup \biggl\{\frac{|\eta(x_0)-\eta(x_1)|}{\|x_{0}-x_1\|_{\infty}^\holderExponent} : x_0 \neq x_1, \regressionFunction(x_0)\wedge \regressionFunction(x_1)>\tau, \omega(x_0)\wedge \omega(x_1)>0 \biggr\} \vee 1.
\]
\end{lemma}
\begin{proof} First note that if $\regressionFunction(x)>\tau$ then $x \in \etaSuperLevelSet{\tau}$ and if $\omega(x)>0$ then $x \in \mathrm{supp}(\mu)$, so ${\holderConstant_{\holderExponent,\flat}}(P) \leq \underline{\holderConstant}_{\holderExponent}(P)$.

Since $P \in \classOfDistributionsSatisfyingRegularityCondition$, it has a continuous regression function $\regressionFunction$ such that $\eta^{-1}([\tau,\tau+\varepsilon)) \subseteq \mathrm{supp}(\mu)$ for some $\varepsilon>0$. Now take distinct points $x_0, x_1 \in \mathrm{supp}(\mu) \cap \etaSuperLevelSet{\tau}$, and suppose initially that $\regressionFunction(x_0) \wedge \regressionFunction(x_1) > \tau$.  Given $\epsilon'>0$, it follows from Lemma \ref{lemma:omegaAEPositive} that we may choose $x_0', x_1' \in \R^d$ with $\regressionFunction(x_0')\wedge \regressionFunction(x_1')>\tau$ and $\omega(x_0')\wedge \omega(x_1')>0$ such that 
\begin{align*}
\max\bigl\{\|x_0-x_0'\|_\infty,\|x_1-x_1'\|_\infty,|\regressionFunction(x_0)-\regressionFunction(x_0')|,|\regressionFunction(x_1)-\regressionFunction(x_1')|\bigr\} \leq \epsilon'.
\end{align*}
Consequently, we have 
\begin{align}
\label{Eq:etaHolderbound}
|\regressionFunction(x_0)-\regressionFunction(x_1)| \leq |\regressionFunction(x_0')-\regressionFunction(x_1')| +2\epsilon' &\leq {\holderConstant_{\holderExponent,\flat}}(P) \cdot \|x_{0}'-x_1'\|_{\infty}^\holderExponent +2\epsilon' \nonumber \\
&\leq {\holderConstant_{\holderExponent,\flat}}(P) \cdot \bigl( \|x_{0}-x_1\|_{\infty}+2\epsilon'\bigr)^\holderExponent+2\epsilon'.
\end{align}
Since $\epsilon' > 0$ was arbitrary, we deduce that $|\regressionFunction(x_0)-\regressionFunction(x_1)| \leq {\holderConstant_{\holderExponent,\flat}}(P) \cdot \|x_{0}-x_1\|_{\infty}^\holderExponent$ for all distinct $x_0, x_1 \in \mathrm{supp}(\mu) \cap  \eta^{-1}\bigl((\tau,1]\bigr)$.  Now consider $x_0,x_1 \in \etaSuperLevelSet{\tau} \cap \mathrm{supp}(\mu)$ with $\eta(x_0)=\tau$ and $\eta(x_1) >\tau$.  Given $\epsilon' \in \bigl(0,\min\{\varepsilon ,\regressionFunction(x_1)-\tau\}\bigr)$, by the intermediate value theorem we may choose $x_2$ on the line segment between $x_0$ and $x_1$ with $\eta(x_2) = \tau+\epsilon'$. It then follows that $x_2 \in \eta^{-1}\bigl([\tau,\tau+\varepsilon)\bigr) \subseteq \mathrm{supp}(\mu)$. By the consequence of~\eqref{Eq:etaHolderbound}, we deduce that $|\regressionFunction(x_2)-\regressionFunction(x_1)| \leq {\holderConstant_{\holderExponent,\flat}}(P) \cdot \|x_2-x_1\|_{\infty}^\holderExponent$, and hence
\begin{align*}
|\regressionFunction(x_0)-\regressionFunction(x_1)| & \leq |\regressionFunction(x_0)-\regressionFunction(x_2)|+|\regressionFunction(x_2)-\regressionFunction(x_1)|\\
& \leq \epsilon'+ {\holderConstant_{\holderExponent,\flat}}(P) \cdot \|x_{2}-x_1\|_{\infty}^\holderExponent \leq \epsilon'+ {\holderConstant_{\holderExponent,\flat}}(P) \cdot \|x_0-x_1\|_{\infty}^\holderExponent.
\end{align*}
Again, since  $\epsilon' \in \bigl(0,\min\{\varepsilon ,\regressionFunction(x_1)-\tau\}\bigr)$ can be taken arbitrarily small, we deduce that $|\regressionFunction(x_0)-\regressionFunction(x_1)| \leq {\holderConstant_{\holderExponent,\flat}}(P) \cdot \|x_{0}-x_1\|_{\infty}^\holderExponent$ for all $x_0, x_1 \in \etaSuperLevelSet{\tau}\cap \mathrm{supp}(\mu)$ with $\eta(x_1)>\tau$. Finally, when $\eta(x_0)=\eta(x_1)=\tau$, the inequality $|\regressionFunction(x_0)-\regressionFunction(x_1)| \leq {\holderConstant_{\holderExponent,\flat}}(P) \cdot \|x_{0}-x_1\|_{\infty}^\holderExponent$ is immediate. Thus, we have $|\regressionFunction(x_0)-\regressionFunction(x_1)| \leq {\holderConstant_{\holderExponent,\flat}}(P) \cdot \|x_{0}-x_1\|_{\infty}^\holderExponent$ for all $x_0,x_1 \in \etaSuperLevelSet{\tau}\cap \mathrm{supp}(\mu)$, so $\underline{\holderConstant}_{\holderExponent}(P) \leq {\holderConstant_{\holderExponent,\flat}}(P)$, as required.
\end{proof}
Our next lemma adapts ideas from \cite{mcshane1934extension}.
\begin{lemma}\label{lemma:belongsToCorrespondingLipschitzClass} Given $P \in \classOfDistributionsSatisfyingRegularityCondition$ with $\underline{\holderConstant}_{\holderExponent}(P)<\infty$, we have $P \in \measureClass_{\mathrm{H\ddot{o}l}}\bigl(\holderExponent,\underline{\holderConstant}_{\holderExponent}(P),\tau\bigr)$.
\end{lemma}
\begin{proof}  
Since $P \in \classOfDistributionsSatisfyingRegularityCondition$, it has a continuous regression function $\eta_0$ such that $\eta_0^{-1}\bigl([\tau,\tau+\varepsilon)\bigr) \subseteq \mathrm{supp}(\mu)$ for some $\varepsilon>0$. We construct a regression function $\eta_1:\R^d \rightarrow [0,1]$ by
\begin{align*}
\eta_1(x):=\begin{cases}1 \wedge \inf\bigl\{\eta_0(z) +\underline{\holderConstant}_{\holderExponent}(P)\cdot \supNorm{z-x}^\holderExponent : z \in \mathrm{supp}(\mu) \cap \mathcal{X}_\tau(\eta_0) \bigr\} &\text{ if }x \in \mathcal{X}_\tau(\eta_0)\\
\eta_0(x) &\text{ otherwise.}
\end{cases}
\end{align*}
First note that $\mathcal{X}_\tau(\eta_0) = \mathcal{X}_\tau(\eta_1)$.  We claim that $|\eta_1(x)-\eta_1(x')| \leq \underline{\holderConstant}_{\holderExponent}(P)\cdot \supNorm{x-x'}^\holderExponent$ for all $x$, $x' \in \mathcal{X}_\tau(\eta_0)$.  Indeed, if $\mathrm{supp}(\mu) \cap \mathcal{X}_\tau(\eta_0) = \emptyset$, then $\eta_1(x) = \eta_1(x')$, and if $z \in \mathrm{supp}(\mu) \cap \mathcal{X}_\tau(\eta_0)$, then 
\begin{align*}
\eta_0(z) +\underline{\holderConstant}_{\holderExponent}(P)\cdot \supNorm{z-x}^\holderExponent \leq \eta_0(z) +\underline{\holderConstant}_{\holderExponent}(P)\cdot \supNorm{z-x'}^\holderExponent + \underline{\holderConstant}_{\holderExponent}(P)\cdot \supNorm{x-x'}^\holderExponent,
\end{align*}
and taking an infimum over $z \in \mathrm{supp}(\mu) \cap \mathcal{X}_\tau(\eta_0) $ yields $\eta_1(x)-\eta_1(x') \leq \underline{\holderConstant}_{\holderExponent}(P)\cdot \supNorm{x-x'}^\holderExponent$.  By interchanging the roles of $x$ and $x'$, the claim follows.  Our second claim is that $\eta_1(x)=\eta_0(x)$ for all $x \in \mathrm{supp}(\mu) \cap \mathcal{X}_\tau(\eta_0)$.  To see this, first observe that $\eta_1(x) \leq \eta_0(x)$ for all $x \in \mathrm{supp}(\mu) \cap \mathcal{X}_\tau(\eta_0)$ by definition of $\eta_1$.  On the other hand, by definition of $\underline{\holderConstant}_{\holderExponent}(P)$, we have for any $z \in \mathrm{supp}(\mu) \cap \mathcal{X}_\tau(\eta_0)$ that
\[
\eta_0(x) \leq \eta_0(z) + \underline{\holderConstant}_{\holderExponent}(P) \cdot \supNorm{z-x}^\holderExponent,
\]
so taking an infimum over $z \in \mathrm{supp}(\mu) \cap \mathcal{X}_\tau(\eta_0)$ yields $\eta_0(x) \leq \eta_1(x)$.  In particular, it now follows that $\eta_0(x) = \eta_1(x)$ for all $x \in \mathrm{supp}(\mu)$.  To show that  $\regressionFunction_1 \in \mathcal{F}_{\mathrm{H\ddot{o}l}}\bigl(\holderExponent, \underline{\holderConstant}_{\holderExponent}(P),\mathcal{X}_\tau(\eta_1)\bigr)$, we must also verify that $\eta_1$ is continuous.  To this end, let $x \in \R^d$. If $\eta_0(x)>\tau$, then since $\eta_0$ is continuous, we have for all sufficiently small $\epsilon'>0$ and $z \in \closedMetricBallSupNorm{x}{\epsilon'}$ that $\eta_0(z)>\tau$.  Hence $|\eta_1(x)-\eta_1(z)| \leq \underline{\holderConstant}_{\holderExponent}(P)\cdot \supNorm{x-z}^\holderExponent$ by our first claim, so $\eta_1$ is continuous on $\regressionFunction_0^{-1}\bigl((\tau,1]\bigr)$. On the other hand, if $\eta_0(x) <\tau+\varepsilon$ then, again, since $\eta_0$ is continuous, for all sufficiently small $\epsilon'>0$ and $z \in \closedMetricBallSupNorm{x}{\epsilon'}$ we also have have $\eta_0(z)<\tau+\varepsilon$, so $x,z \in  \mathrm{supp}(\mu) \cap \mathcal{X}_\tau(\eta_0) $.  We deduce from our second claim that $\eta_1(x)=\eta_0(x)$ and $\eta_1(z)=\eta_0(z)$, so again, $\eta_1$ is continuous at~$x$.  We conclude that  $\regressionFunction_1 \in \mathcal{F}_{\mathrm{H\ddot{o}l}}\bigl(\holderExponent, \underline{\holderConstant}_{\holderExponent}(P),\mathcal{X}_\tau(\eta_1)\bigr)$. Moreover, since $\eta_0$ is a regression function for $P$, and $\eta_0(x)=\eta_1(x)$ for all $x \in \mathrm{supp}(\mu)$, we have that $\eta_1$ is also a regression function for $P$, and hence $P \in \measureClass_{\mathrm{H\ddot{o}l}}\bigl(\holderExponent, \underline{\holderConstant}_{\holderExponent}(P),\tau\bigr)$.
\end{proof}

Recall the definition of the classes $\classOfHolderDistributionsSuperLevelSetIdentifiable$ from Section~\ref{SubSec:lambda}.
\begin{lemma}
\label{Lemma:Inclusion}
Fix $\holderExponent \in (0,1]$ and $\holderConstant \geq 1$.  Then
\[
\classOfDistributionsSatisfyingRegularityCondition \cap  \classOfHolderDistributionsSuperLevelSet \subseteq  \bigcup_{\underline{\epsilon}\in (0,\infty)^3}\classOfHolderDistributionsSuperLevelSetIdentifiable.
\]
\end{lemma}
\begin{proof}
Let $P \in \classOfDistributionsSatisfyingRegularityCondition \cap  \classOfHolderDistributionsSuperLevelSet $.  Then $\underline{\holderConstant}_{\holderExponent}(P) \leq \holderConstant<\infty$. Moreover, since $P \in \classOfDistributionsSatisfyingRegularityCondition$, we have by Lemma \ref{lemma:LipschitzConstantReformulation} that $P \in \classOfHolderDistributionsSuperLevelSetIdentifiable$ for sufficiently small $\underline{\epsilon}=(\epsilon_0,\epsilon_1, \epsilon_2) \in (0,1]^3$.
\end{proof}
Finally, we are in a position to prove Corollary~\ref{Cor:HolderConstantEstimation}.
\begin{proof}[Proof of Corollary~\ref{Cor:HolderConstantEstimation}] First, since $P \in \classOfHolderDistributionsSuperLevelSetIdentifiable$ we may choose $x_0, x_1 \in \etaSuperLevelSet{\tau+\epsilon_0}$ with $\|x_0-x_1\|_{\infty} \geq \epsilon_1$, as well as $\min\{\omega(x_0),\omega(x_1)\} \geq \epsilon_2$ and 
\begin{align*}
|\eta(x_0)-\eta(x_1)| \geq \frac{3}{4}\cdot  \underline{\holderConstant}_{\holderExponent}(P) \cdot \|x_{0}-x_1\|_{\infty}^\holderExponent \cdot\one_{\{\underline{\holderConstant}_{\holderExponent}(P)>1\}}.
\end{align*}
Writing $\holderConstant=\underline{\holderConstant}_{\holderExponent}(P)$, we have by Lemma \ref{lemma:belongsToCorrespondingLipschitzClass} that $P \in \classOfHolderDistributionsSuperLevelSet$.  Define
\begin{align*}
\Delta &:=192 \cdot \holderConstant^{d/(2\holderExponent+d)}\cdot \biggl(\frac{\log(2n/\delta)}{n\{\omega(x_0)\wedge \omega(x_1)\}}\biggr)^{\holderExponent/(2\holderExponent+d)} \\
&\phantom{:}\leq 192 \cdot \holderConstant^{d/(2\holderExponent+d)}\cdot \biggl(\frac{\log(2n/\delta)}{n\cdot \epsilon_2}\biggr)^{\holderExponent/(2\holderExponent+d)}.
\end{align*}
Observe that when $n \in \mathbb{N}$ and $\delta \in (0,1)$ satisfy~\eqref{Eq:SampleSizeCondition}, we have $\Delta \leq \min\{ \epsilon_0, \holderConstant \epsilon_1^\holderExponent/4\}$. By Theorem \ref{thm:lipEst} we have $\Prob_P(\mathcal{E})\geq 1-\delta$, where 
\begin{align*}
\mathcal{E}&:=\biggl\{ \frac{|\eta(x_0)-\eta(x_1)|- \Delta}{\|x_{0}-x_1\|_{\infty}^\holderExponent}  \leq \holderConstantEstimator\leq \holderConstant \biggr\}.
\end{align*}
Note that on the event $\mathcal{E}$, we have $2\holderConstantEstimator \leq 2\lambda$, and if $\lambda = 1$, then $\lambda \leq 2\holderConstantEstimator$ is immediate. On the other hand, if $\lambda>1$ then on the event $\mathcal{E}$, we have
\begin{align*}
\frac{|\eta(x_0)-\eta(x_1)|- \Delta}{\|x_{0}-x_1\|_{\infty}^\holderExponent} & \geq \frac{3}{4} \cdot \holderConstant - \frac{\holderConstant \cdot \epsilon_1^\holderExponent}{4 \cdot \|x_{0}-x_1\|_{\infty}^\holderExponent} \geq \frac{\holderConstant}{2},
\end{align*}
as required.
\end{proof}
\begin{proof}[Proof of Theorem~\ref{Thm:BetaKnown}] First, for $\holderConstant' \in [1,\infty)$, let $S_{\holderConstant'}$ denote the union $\bigcup_{\ell \in [\ell_{\alpha_n}]}B_{(\ell)}$ appearing in line 6 of Algorithm \ref{subsetSelectionAlgo} when it is applied with $\alpha_n=(\alpha/2) \wedge (1/n)$ in place of $\alpha$ and $\holderConstant'$ in place of~$\holderConstant$, and let $\hat{A}_{\lambda'}$ denote the corresponding output set in $\mathcal{A}$.  Observe that if $\holderConstant_0' \leq \holderConstant_1'$ then  $S_{\holderConstant'_1} \subseteq S_{\holderConstant'_0}$, since the $p$-values in \eqref{eq:pValueDef} satisfy $\hat{p}_{n,\holderExponent,\holderConstant_0'}(\cdot)\leq \hat{p}_{n,\holderExponent,\holderConstant_1'}(\cdot)$. Consequently, $\empiricalMarginalDistribution(\hat{A}_{\lambda'_1})\leq \empiricalMarginalDistribution(\hat{A}_{\lambda'_0})$ by line 6 in Algorithm~\ref{subsetSelectionAlgo}.

\textbf{(i)} \sloppy Let $\probDistribution \in\classOfDistributionsSatisfyingRegularityCondition\cap  \classOfHolderDistributionsSuperLevelSetIdentifiable$. By Lemma~\ref{lemma:belongsToCorrespondingLipschitzClass} we have $P \in \measureClass_{\mathrm{H\ddot{o}l}}\bigl(\holderExponent,\underline{\holderConstant}_{\holderExponent}(P),\tau\bigr) $.  Hence, by Lemma \ref{lemma:typeIControl} and Corollary~\ref{Cor:HolderConstantEstimation} we have
\begin{align}
\label{Eq:TypeIerrorbound}
\Prob_P\bigl( \hat{A}_{\mathrm{OSS}}'(\sample)  \nsubseteq \etaSuperLevelSet{\tau}\bigr) &\leq\Prob_P\bigl( S_{2\holderConstantEstimatorArgs{\holderExponent}{\alpha_n}} \nsubseteq \etaSuperLevelSet{\tau}\bigr) \nonumber \\
& \leq \Prob_P\bigl( S_{\underline{\lambda}_{\holderExponent}(P)} \nsubseteq \etaSuperLevelSet{\tau}\bigr)+\Prob_P\bigl(S_{2\holderConstantEstimatorArgs{\holderExponent}{\alpha_n}} \nsubseteq S_{\underline{\lambda}_{\holderExponent}(P)} \bigr) \nonumber \\
& \leq \Prob_P\bigl( S_{\underline{\lambda}_{\holderExponent}(P)} \nsubseteq \etaSuperLevelSet{\tau}\bigr)+\Prob_P\bigl(2\holderConstantEstimatorArgs{\holderExponent}{\alpha_n} < \underline{\lambda}_{\holderExponent}(P) \bigr) \leq 2\alpha_n \leq \alpha.
\end{align}

\textbf{(ii)} Now suppose that $P \in \classOfDistributionsSatisfyingRegularityCondition\cap  \classOfHolderDistributionsSuperLevelSetIdentifiable\cap \classOfWellApproximableSets$. On the event $\{  2\holderConstantEstimatorArgs{\holderExponent}{\alpha_n} \leq 2\underline{\lambda}_{\holderExponent}(P)\}$, we have $\empiricalMarginalDistribution(\hat{A}_{\mathrm{OSS}}')=\empiricalMarginalDistribution(\hat{A}_{2\holderConstantEstimatorArgs{\holderExponent}{\alpha_n}}) \geq \empiricalMarginalDistribution(\hat{A}_{2\underline{\lambda}_{\holderExponent}(P)})$. Hence, by Lemma \ref{lemma:vapnikChervonenkisConcentration}, we have
\begin{align}
\label{Eq:EmpiricalRiskBound}
\E_P\bigl\{ \mu(\hat{A}_{2\underline{\lambda}_{\holderExponent}(P)})-\mu(\hat{A}_{\mathrm{OSS}}')\bigr\} &= \E_P\bigl\{ \mu(\hat{A}_{2\underline{\lambda}_{\holderExponent}(P)}) - \empiricalMarginalDistribution(\hat{A}_{2\underline{\lambda}_{\holderExponent}(P)}) \nonumber \\
&\hspace{1cm}+ \empiricalMarginalDistribution(\hat{A}_{2\underline{\lambda}_{\holderExponent}(P)}) - \empiricalMarginalDistribution(\hat{A}_{\mathrm{OSS}}') + \empiricalMarginalDistribution(\hat{A}_{\mathrm{OSS}}')
-\mu(\hat{A}_{\mathrm{OSS}}')\bigr\} \nonumber \\
&\leq 2C_{\mathrm{VC}}~ \sqrt{\frac{\vcDim(\mathcal{A})}{n}} =: \epsilon^{\mathrm{VC}}_n.
\end{align}
By Lemmas~\ref{lemma:belongsToCorrespondingLipschitzClass} and~\ref{lemma:typeIControl} we have $
\Prob_P\bigl( \hat{A}_{2\underline{\lambda}_{\holderExponent}(P)}  \nsubseteq \etaSuperLevelSet{\tau}\bigr) \leq \Prob_P\bigl( S_{2\underline{\lambda}_{\holderExponent}(P)} \nsubseteq \etaSuperLevelSet{\tau}\bigr) \leq \alpha_n$.  It follows by~\eqref{Eq:TypeIerrorbound} and~\eqref{Eq:EmpiricalRiskBound} that
\begin{align*}
 \E_P&\bigl\{\bigl(M_\tau - \mu(\hat{A}_{\mathrm{OSS}}')\bigr) \cdot \one_{\{\hat{A}_{\mathrm{OSS}}' \subseteq \etaSuperLevelSet{\tau}\}}\bigr\} \leq \E_P\big\{M_\tau - \mu\bigl(\hat{A}_{\mathrm{OSS}}'\bigr)\big\}+\Prob_P\bigl( \hat{A}_{\mathrm{OSS}}'(\sample)  \nsubseteq \etaSuperLevelSet{\tau}\bigr)\\
& \leq \E_P\big\{M_\tau - \mu\bigl(\hat{A}_{2\underline{\lambda}_{\holderExponent}(P)}\bigr)\big\}+\epsilon^{\mathrm{VC}}_n+2\alpha_n\\
& \leq \E_P\bigl\{\bigl(M_\tau - \mu(\hat{A}_{2\underline{\lambda}_{\holderExponent}(P)})\bigr) \cdot \one_{\{\hat{A}_{2\underline{\lambda}_{\holderExponent}(P)} \subseteq \etaSuperLevelSet{\tau}\}}\bigr\} +  
\Prob_P\bigl( \hat{A}_{2\underline{\lambda}_{\holderExponent}(P)}  \nsubseteq \etaSuperLevelSet{\tau}\bigr)+\epsilon^{\mathrm{VC}}_n+2\alpha_n\\
& \leq \E_P\bigl(M_\tau - \mu(\hat{A}_{2\underline{\lambda}_{\holderExponent}(P)})\bigm|\hat{A}_{2\underline{\lambda}_{\holderExponent}(P)} \subseteq \etaSuperLevelSet{\tau}\bigr) +3\alpha_n+\epsilon^{\mathrm{VC}}_n\\
&= R_\tau(\hat{A}_{2\underline{\lambda}_{\holderExponent}(P)})+3\alpha_n+\epsilon^{\mathrm{VC}}_n
\\
&\leq C' \biggl\{\biggl({\frac{\bigl(2\underline{\lambda}_{\holderExponent}(P)\bigr)^{d/\holderExponent}}{n} \cdot \log_+\Bigl(\frac{n}{(\alpha/2) \wedge (1/n)}\Bigr)\biggr)^{\frac{\holderExponent \approximableDensityExponent \approximableMarginExponent}{\approximableDensityExponent(2\holderExponent+d) +\holderExponent\approximableMarginExponent}}}+\frac{1}{n^{1/2}} \biggr\}+3\alpha_n+\epsilon^{\mathrm{VC}}_n,
\end{align*}
where we have used Proposition~\ref{thm:powerBound} with $C'$ in place of $C$ for the final inequality. The regret bound on $\hat{A}'_{\mathrm{OSS}}$ follows by once again using $\Prob_P\bigl( \hat{A}_{\mathrm{OSS}}'(\sample)  \subseteq \etaSuperLevelSet{\tau}\bigr) \geq 1-\alpha_n\geq 1/2$.
\end{proof}

\subsection{Proofs from Section~\ref{SubSec:beta}}

Recall the definition of the event $\mathcal{E}_{\eta,\delta}$ from Section \ref{subsec:proofOfLambdaEstResults}.

\begin{lemma}\label{lemma:uniformLowerBoundRatio}  Let $\probDistribution$ be a distribution with regression function $\regressionFunction \in \classOfRestrictedHolderFunctions{\R^d}$. On the event $\mathcal{E}_{\eta,\delta}$, we have $v \geq \holderExponent \cdot u - \log \holderConstant$  for all $(u,v) \in \hat{\Gamma}_{n,\delta}^\dagger$.
\end{lemma}
\begin{proof} 
If $(u,v) \in \hat{\Gamma}_{n,\delta}^\dagger$, then we can find $(i,j,k,\ell) \in [n]^4$ with $\hat{\varphi}^\dagger_{n,\holderExponent,\delta}(i,j,k,\ell)<\infty$ such that $u = \hat{\varepsilon}^\dagger_{n,\holderExponent,\delta}(i,j,k,\ell)$ and $v = \hat{\varphi}^\dagger_{n,\holderExponent,\delta}(i,j,k,\ell)$.  Thus, on the event $\mathcal{E}_{\eta,\delta}$, we have
\begin{align*}
v &\geq -\log\Biggl(\Biggl|\frac{\sum_{t=1}^n \regressionFunction(X_t)\cdot\one_{{\{X_t \in \closedMetricBallSupNorm{X_i}{r_{i,k}}\}}}}{\sum_{t=1}^n \one_{{\{X_t \in \closedMetricBallSupNorm{X_i}{r_{i,k}}\}}}} - \frac{\sum_{t=1}^n \regressionFunction(X_t)\cdot\one_{{\{X_t \in \closedMetricBallSupNorm{X_j}{r_{j,\ell}}\}}}}{\sum_{t=1}^n \one_{{\{X_t \in \closedMetricBallSupNorm{X_j}{r_{j,\ell}}\}}}}\Biggr|\Biggr) \\
&\geq -\log \bigl( \holderConstant ( \|X_i-X_j\|_{\infty}+r_{i,k}+r_{j,\ell})^\holderExponent \bigr) =\holderExponent \cdot u-\log \holderConstant,
\end{align*}
where we have applied $\regressionFunction \in \classOfRestrictedHolderFunctions{\R^d}$ in the final inequality.
\end{proof}
Now for $x \in \R^d$ and $s>0$ we define the event $\mathcal{E}_{\mu}(x,s):=\bigl\{ \empiricalMarginalDistribution\bigl(\closedMetricBallSupNorm{x}{s}\bigr) \geq \frac{1}{2} \cdot \mu\bigl(\closedMetricBallSupNorm{x}{s}\bigr) \bigr\}$. 

\begin{lemma}\label{lemma:nearAttainmentOfBoundRatioGeneral} Let $\probDistribution$ be a distribution with regression function $\regressionFunction \in \classOfRestrictedHolderFunctions{\R^d}$. Suppose that $x_0$, $x_1 \in \regularSetSelfSimilar$ and $s \in (0,1]$ satisfy
\begin{align*}
| \regressionFunction(x_0)-\regressionFunction(x_1)| \geq 4\sqrt{\frac{\log(4n^2/\delta)}{\lowerDensityConstantSelfSimilar n s^d}}+6\holderConstant s^\holderExponent.
\end{align*}
Then, on the event $\mathcal{E}_{\eta,\delta} \cap \mathcal{E}_\mu(x_0,s) \cap \mathcal{E}_\mu(x_1,s)$, there exists $(u,v) \in \hat{\Gamma}_{n,\delta}^\dagger$ with 
\begin{align*}
u & \geq - \log \bigl( \supNorm{x_0-x_1}+6s\bigr)\\
v &\leq - \log\biggl( | \regressionFunction(x_0)-\regressionFunction(x_1)| - 4 \sqrt{\frac{\log(4n^2/\delta)}{\lowerDensityConstantSelfSimilar n s^d}} - 6\holderConstant s^\holderExponent\biggr).
\end{align*}
\end{lemma}
\begin{proof} Suppose that the event $\mathcal{E}_{\eta,\delta} \cap \mathcal{E}_\mu(x_0,s) \cap \mathcal{E}_\mu(x_1,s)$ holds.  Then in particular,
\begin{align}\label{eq:consequenceOfRegularSetSelfSimilarCondition}
\min\bigl\{\empiricalMarginalDistribution\bigl(\closedMetricBallSupNorm{x_0}{s}\bigr),\empiricalMarginalDistribution\bigl(\closedMetricBallSupNorm{x_1}{s}\bigr)\bigr\} & \geq \frac{1}{2} \cdot \lowerDensityConstantSelfSimilar \cdot s^d>0.
\end{align}
As such, we may choose $X_i \in \closedMetricBallSupNorm{x_0}{s}$ and $X_j \in \closedMetricBallSupNorm{x_1}{s}$. Moreover, letting $k=\ell= \lceil \frac{n}{2} \cdot \lowerDensityConstantSelfSimilar \cdot s^d\rceil$, it follows from \eqref{eq:consequenceOfRegularSetSelfSimilarCondition} that $\max(r_{i,k},r_{j,\ell}) \leq 2s$. Now take $u = \hat{\varepsilon}^\dagger_{n,\holderExponent,\delta}(i,j,k,\ell)$ and $v = \hat{\varphi}^\dagger_{n,\holderExponent,\delta}(i,j,k,\ell)$ so that $(u,v) \in \hat{\Gamma}_{n,\delta}^\dagger$. The lower bound on $u$ follows. The upper bound on $v$ then follows from our event combined with the facts that $\closedMetricBallSupNorm{X_i}{r_{i,k}} \subseteq \closedMetricBallSupNorm{x_0}{3s}$ and $\closedMetricBallSupNorm{X_j}{r_{j,k}} \subseteq \closedMetricBallSupNorm{x_1}{3s}$.
\end{proof}

\begin{lemma}\label{lemma:nearAttainmentOfBoundRatioSelfSimilarTwoPoint} Let $P \in \classOfSelfSimilarHolderDistributions$ with $\lowerHolderConstantSelfSimilar \leq \holderConstant$ and $0 < r \leq r' \leq \minRadiusSelfSimilar \leq 1$. Let $\mathcal{E}_{\delta}^\dagger(r,r')$ denote the event that there exist $(u,v),(u',v') \in \hat{\Gamma}_{n,\delta}^\dagger$ with $u \geq -\log(7r)$, $u' \geq -\log(7r')$, $v \leq -\log\bigl((1/2)\lowerHolderConstantSelfSimilar r^\holderExponent\bigr)$ and $v' \leq -\log\bigl((1/2)\lowerHolderConstantSelfSimilar (r')^\holderExponent\bigr)$.  Then  provided that
\begin{align}\label{eq:sampleSizeConditionFromLemmaNearAttainmentOfBoundRatioSelfSimilarTwoPoint}
\frac{n}{\log(4n^2/\delta)} \geq  \frac{\bigl(16\holderConstant/\lowerHolderConstantSelfSimilar\bigr)^{d/\holderExponent}}{\lowerDensityConstantSelfSimilar} \cdot \max\biggl\{\frac{ 2^{10} }{ \lowerHolderConstantSelfSimilar^2 \cdot r^{2\holderExponent+d}}, \frac{ 8 }{  r^{d}} \biggr\},
\end{align}
we have $\Prob\bigl(\mathcal{E}_{\eta,\delta}\cap \mathcal{E}_{\delta}^\dagger(r,r')\bigr) \geq 1-\delta$.
\end{lemma}
\begin{proof} Since $P \in \classOfSelfSimilarHolderDistributions$, there exist $x_0, x_1, x_0',x_1' \in \regularSetSelfSimilar$ with $\supNorm{x_0-x_1} \leq r$, $\supNorm{x_0'-x_1'} \leq r'$, $|\regressionFunction(x_0)-\regressionFunction(x_1)| \geq \lowerHolderConstantSelfSimilar \cdot r^\holderExponent$ and $|\regressionFunction(x_0')-\regressionFunction(x_1')| \geq \lowerHolderConstantSelfSimilar \cdot r'^\holderExponent$. Now let $s:=\{\lowerHolderConstantSelfSimilar/(16\holderConstant)\}^{1/\holderExponent} \cdot r \in (0,1]$ and introduce the event
\begin{align*}
\tilde{\mathcal{E}}_{\delta}(r,r'):= 
\mathcal{E}_{\eta,\delta}\cap \bigcap_{z \in \{x_0,x_1,x_0',x_1'\}}\mathcal{E}_\mu(z,s).
\end{align*}
When~\eqref{eq:sampleSizeConditionFromLemmaNearAttainmentOfBoundRatioSelfSimilarTwoPoint} holds, we have
\begin{align*}
4\sqrt{\frac{\log(4n^2/\delta)}{\lowerDensityConstantSelfSimilar n s^d}}+6\holderConstant s^\holderExponent & = 4\sqrt{\frac{(16 \holderConstant/\lowerHolderConstantSelfSimilar)^{d/\holderExponent} \log(4n^2/\delta)}{\lowerDensityConstantSelfSimilar n  r^d}}+\frac{3}{8}\cdot \lowerHolderConstantSelfSimilar r^\holderExponent \leq \frac{1}{2}\cdot \lowerHolderConstantSelfSimilar r^\holderExponent. 
\end{align*}
Hence, by Lemma \ref{lemma:nearAttainmentOfBoundRatioGeneral}, on the event $\tilde{\mathcal{E}}_{\delta}(r,r')$ and when ~\eqref{eq:sampleSizeConditionFromLemmaNearAttainmentOfBoundRatioSelfSimilarTwoPoint} holds, there exist pairs $(u,v), (u',v') \in \hat{\Gamma}_{n,\delta}^\dagger$ with $u \geq -\log(7r)$, $u' \geq -\log(7r')$, $v \leq -\log\bigl((1/2)\lowerHolderConstantSelfSimilar r^\holderExponent\bigr)$ and $v' \leq -\log\bigl((1/2)\lowerHolderConstantSelfSimilar (r')^\holderExponent\bigr)$.  Thus, when~\eqref{eq:sampleSizeConditionFromLemmaNearAttainmentOfBoundRatioSelfSimilarTwoPoint} holds, we have
\[
\Prob\bigl(\mathcal{E}_{\eta,\delta}\cap \mathcal{E}_{\delta}^\dagger(r,r')\bigr) \geq \Prob\bigl( \tilde{\mathcal{E}}_{\delta}(r,r')\bigr) \geq 1-\delta
\]
by Hoeffding's inequality (Lemma~\ref{consequenceOfGarivier2011kl}) and the multiplicative Chernoff bound (Lemma~\ref{lemma:multChernoff}).
\end{proof}

We are now in a position to prove Theorem \ref{thm:holderExponentEstimationThm}.

\begin{proof}[Proof of Theorem \ref{thm:holderExponentEstimationThm}] First, for $n \in \mathbb{N}$ satisfying~\eqref{eq:sampleSizeConditionFromThmholderExponentEstimationThm}, let
\begin{align*}
r_n:=\biggl\{ \frac{2^{10}\bigl(16\holderConstant/\lowerHolderConstantSelfSimilar\bigr)^{d/\holderExponent}}{\lowerDensityConstantSelfSimilar\cdot \lowerHolderConstantSelfSimilar^2} \cdot \frac{\log(4n^2/\delta)}{n}\biggr\}^{1/(2\holderExponent+d)}.
\end{align*}
and $r_n':= n^{-1/(7+2d)}$. The sample size condition \eqref{eq:sampleSizeConditionFromThmholderExponentEstimationThm} ensures that \eqref{eq:sampleSizeConditionFromLemmaNearAttainmentOfBoundRatioSelfSimilarTwoPoint} holds and that $r_n \leq r_n' \leq r_0$, so we may apply Lemma~\ref{lemma:nearAttainmentOfBoundRatioSelfSimilarTwoPoint} to obtain $\Prob\bigl(\mathcal{E}_{\eta,\delta}\cap \mathcal{E}_{\delta}^\dagger(r_n,r_n')\bigr) \geq 1-\delta$.

Next, we show that on the event $\mathcal{E}_{\eta,\delta}\cap \mathcal{E}_{\delta}^\dagger(r_n,r_n')$ we have $\holderExponentEstJoint{\delta}\leq \holderExponent$. Indeed, suppose that $\mathcal{E}_{\eta,\delta}\cap \mathcal{E}_{\delta}^\dagger(r_n,r_n')\cap\{\holderExponentEstJoint{\delta}>0\}$ holds and let $(u_0,v_0) \in \hat{\Gamma}_{n,\delta}^\dagger$ be such that $u_0 \leq \frac{\log n}{6+2d}$. By Lemma \ref{lemma:uniformLowerBoundRatio} we have $v_0 \geq \holderExponent \cdot u_0 - \log \holderConstant$.  Since $\mathcal{E}_{\delta}^\dagger(r_n,r_n')$ holds, there exists $(u_1,v_1) \in \hat{\Gamma}_{n,\delta}^\dagger$ with $u_1 \geq -\log(7r_n)$ and $v_1 \leq -\log\bigl((1/2)\lowerHolderConstantSelfSimilar r_n^\holderExponent\bigr) \leq \holderExponent u_1+ \log(14/\lowerHolderConstantSelfSimilar)$. Moreover, by~\eqref{eq:sampleSizeConditionFromThmholderExponentEstimationThm} we have $u_1 \geq \frac{\log n}{3+d} \geq 2u_0$. Hence,
\begin{align*}
\holderExponentEstJoint{\delta} \leq \frac{v_1-v_0-\log \slowlyIncreasingFunctionHolderExponent(n)}{u_1-u_0} \leq \holderExponent+\frac{ \log(14\holderConstant/\lowerHolderConstantSelfSimilar)-\log \slowlyIncreasingFunctionHolderExponent(n)}{u_1-u_0} \leq \holderExponent,
\end{align*}
where we have again applied~\eqref{eq:sampleSizeConditionFromThmholderExponentEstimationThm} to ensure that $\slowlyIncreasingFunctionHolderExponent(n) \geq 14\holderConstant/\lowerHolderConstantSelfSimilar$. 

Finally, we show that on the event $\mathcal{E}_{\eta,\delta}\cap \mathcal{E}_{\delta}^\dagger(r_n,r_n')$ we have $\holderExponentEstJoint{\delta}\geq \holderExponent- \frac{2(7+2d)\log \slowlyIncreasingFunctionHolderExponent(n)}{\log n}$. Indeed, on $\mathcal{E}_{\delta}^\dagger(r_n,r_n') \cap \{\holderExponentEstJoint{\delta} <\holderExponent \}$, there exists $(u_0,v_0) \in \hat{\Gamma}_{n,\delta}^\dagger$ with $u_0 \geq -\log(7r_n')$ and $v_0 \leq -\log\bigl((1/2)\lowerHolderConstantSelfSimilar (r'_n)^\holderExponent\bigr) \leq \holderExponent \cdot u_0+ \log\bigl(14/\lowerHolderConstantSelfSimilar\bigr)$.  Now fix any $(u_1,v_1) \in \hat{\Gamma}_{n,\delta}^\dagger$ satisfying $u_1 \geq 2u_0$. By 
Lemma~\ref{lemma:uniformLowerBoundRatio} we have $v_1 \geq \holderExponent \cdot u_1 - \log \holderConstant$ and consequently,
\begin{align*}
\holderExponentEstJoint{\delta} \geq \frac{v_1-v_0- \log\slowlyIncreasingFunctionHolderExponent(n)}{u_1-u_0} \geq \holderExponent-\frac{ \log(14\holderConstant/\lowerHolderConstantSelfSimilar)+\log \slowlyIncreasingFunctionHolderExponent(n)}{u_1-u_0} \geq \holderExponent- \frac{2(7+2d)\log\slowlyIncreasingFunctionHolderExponent(n)}{\log n},
\end{align*}
as required.
\end{proof}

The following lemma is analogous to Corollary \ref{Cor:HolderConstantEstimation} but with an estimated H\"{o}lder exponent.


\begin{lemma}
\label{lemma:HolderConstantEstimationWithEstimatedHolderExponent}
 Fix $\holderExponent \in (0,1]$, $d \in \N$, $\holderConstant \in [1,\infty)$, $\lowerHolderConstantSelfSimilar \in (0,\holderConstant]$, $\lowerDensityConstantSelfSimilar, \minRadiusSelfSimilar \in (0,1]$ and $\boldsymbol{\epsilon}=(\epsilon_0,\epsilon_1, \epsilon_2) \in (0,1]^3$ and take $P \in \classOfDistributionsSatisfyingRegularityCondition\cap  \classOfHolderDistributionsSuperLevelSetIdentifiable \cap \classOfSelfSimilarHolderDistributions$.  Let $n \in \N$ and $\delta \in (0,1)$ be such that $ 14\holderConstant/\lowerHolderConstantSelfSimilar \leq \slowlyIncreasingFunctionHolderExponent(n) \leq n^{\frac{\log(9/7)}{2(7+2d)\log(1/\epsilon_2)}}$ the sample size condition
\eqref{eq:sampleSizeConditionFromThmholderExponentEstimationThm} holds and
\begin{equation}
\label{Eq:SampleSizeConditionHolderConstantWithHolderExponent}
\frac{n}{\log(2n/\delta)} \geq \frac{1}{\epsilon_2}\cdot \biggl\{  \slowlyIncreasingFunctionHolderExponent(n)^{2(7+2d)}\cdot \Bigl( 192 \cdot \max\Bigl\{ \frac{\holderConstant}{\epsilon_0},\frac{12}{\epsilon_1}\Bigr\}\Bigr)^{2+d}\biggr\}^{1/\holderExponent}.
\end{equation}

Then
\[
\Prob_P\biggl(\biggl\{  \holderExponent -\frac{2(7+2d)\log \slowlyIncreasingFunctionHolderExponent(n)}{\log n}\leq \holderExponentEst{\delta} \leq \holderExponent \biggr\} \cap \biggl\{ \frac{\underline{\holderConstant}_{\holderExponent}(P)}{2} \leq \holderConstantEstJoint{\delta}\leq \underline{\holderConstant}_{\holderExponent}(P)\biggr\}\biggr) \geq 1 - 2\delta.
\]
\end{lemma}
\begin{proof}[Proof of Lemma~\ref{lemma:HolderConstantEstimationWithEstimatedHolderExponent}] For brevity we write $\underline{\holderConstant}_{\holderExponent}=\underline{\holderConstant}_{\holderExponent}(P)$.
Since $P \in \classOfHolderDistributionsSuperLevelSetIdentifiable$, we may choose $x_0, x_1 \in \etaSuperLevelSet{\tau+\epsilon_0}$ with $\|x_0-x_1\|_{\infty} \geq \epsilon_1$, as well as $\min\{\omega(x_0),\omega(x_1)\} \geq \epsilon_2$ and 
\begin{align*}
|\eta(x_0)-\eta(x_1)| \geq \frac{3}{4}\cdot  \underline{\holderConstant}_{\holderExponent} \cdot \|x_{0}-x_1\|_{\infty}^\holderExponent \cdot\one_{\{\underline{\holderConstant}_{\holderExponent}>1\}}.
\end{align*}
By Lemma~\ref{lemma:belongsToCorrespondingLipschitzClass}, we have $P \in \mathcal{P}_{\mathrm{H\ddot{o}l}}(\holderExponent,\underline{\holderConstant}_\holderExponent,\tau)$. Hence, on the event $\{\holderExponentEstJoint{\delta}\leq \holderExponent\}$ we have $P \in \mathcal{P}_{\mathrm{H\ddot{o}l}}(\holderExponentEstJoint{\delta},\underline{\holderConstant}_\holderExponent,\tau)$.  Letting $\holderExponent_{n}^{\flat}:= \max\{ \holderExponent -2(7+2d)\log \slowlyIncreasingFunctionHolderExponent(n)/\log n,0\}$, it follows from Theorem \ref{thm:holderExponentEstimationThm}, combined with the sample size condition~\eqref{eq:sampleSizeConditionFromThmholderExponentEstimationThm} and the lower bound on $f(n)$, that  $\Prob\bigl( \holderExponent_{n}^{\flat} \leq \holderExponentEst{\delta} \leq \holderExponent \bigr) \geq 1-\delta$. In addition, define $\hat{\Delta}_n:=  \Delta_{n,\holderExponentEstJoint{\delta},\underline{\holderConstant}_\holderExponent}(x_0)\vee \Delta_{n,\holderExponentEstJoint{\delta},\underline{\holderConstant}_\holderExponent}(x_1)$, so that provided $\holderExponentEst{\delta} \geq \holderExponent_{n}^{\flat}$ we have
\begin{align*}
\hat{\Delta}_n &=192 \cdot \underline{\holderConstant}_\holderExponent^{d/(2\holderExponentEstJoint{\delta}+d)}\cdot \biggl(\frac{\log(2n/\delta)}{n\{\omega(x_0)\wedge \omega(x_1)\}}\biggr)^{\holderExponentEstJoint{\delta}/(2\holderExponentEstJoint{\delta}+d)}\\
& \leq 192 \cdot \underline{\holderConstant}_\holderExponent \cdot \biggl(\frac{\log(2n/\delta)}{ n \cdot \epsilon_2}\biggr)^{\holderExponent_{n}^{\flat}/(2+d)} \leq  \min\Bigl\{ \epsilon_0, \frac{\underline{\holderConstant}_{\holderExponent} \epsilon_1}{12} \Bigr\},
\end{align*}
where we have applied the sample size condition~\eqref{Eq:SampleSizeConditionHolderConstantWithHolderExponent} in both inequalities. In particular, we have $\min_{x \in \{x_0,x_1\}}\{\eta(x)-\Delta_{n,\holderExponent,\holderConstant}(x)\} \geq \tau$ on the event $\bigl\{  \holderExponent_{n}^{\flat} \leq \holderExponentEst{\delta} \bigr\}$ and under our sample size conditions. Consequently, if we let $\tilde{\mathcal{E}}_{\holderExponent,\underline{\holderConstant}_\holderExponent,\delta}$ denote the event
\begin{align*}
\tilde{\mathcal{E}}_{\holderExponent,\underline{\holderConstant}_\holderExponent,\delta}&:=\bigl\{  \holderExponent_{n}^{\flat} \leq \holderExponentEst{\delta} \leq \holderExponent \bigr\}\cap\biggl\{ \frac{|\eta(x_0)-\eta(x_1)|-\hat{\Delta}_n}{\|x_{0}-x_1\|_{\infty}^{\holderExponentEst{\delta}}} \leq \holderConstantEstJoint{\delta}\leq \underline{\holderConstant}_\holderExponent\biggr\},
\end{align*}
then it follows from Proposition \ref{prop:lipEstUniformOverHolderExponent} that $\Prob\bigl( \tilde{\mathcal{E}}_{\holderExponent,\underline{\holderConstant}_\holderExponent,\delta}\bigr) \geq 1-2\delta$. Thus, to complete the proof it suffices to show that on the event $\tilde{\mathcal{E}}_{\holderExponent,\underline{\holderConstant}_\holderExponent,\delta}$ we have $\underline{\lambda}_{\holderExponent} \leq 2\holderConstantEstJoint{\delta}$. If $\underline{\lambda}_{\holderExponent} = 1$, then the required bound is immediate. On the other hand, if $\underline{\lambda}_{\holderExponent} >1$, then on the event $\tilde{\mathcal{E}}_{\holderExponent,\underline{\holderConstant}_\holderExponent,\delta}$, we have 
\begin{align*}
\holderConstantEstJoint{\delta} \geq \frac{|\eta(x_0)-\eta(x_1)|- \hat{\Delta}_n}{\|x_0 - x_1\|_{\infty}^{\holderExponentEstJoint{\delta}}} & \geq \frac{3}{4} \cdot \underline{\holderConstant}_\holderExponent \cdot \epsilon_1^{\holderExponent-\holderExponent_n^\flat}- \frac{1}{12} \cdot \underline{\holderConstant}_{\holderExponent} \cdot  \epsilon_1^{1-\holderExponent} \geq \frac{\underline{\holderConstant}_\holderExponent}{2},
\end{align*}
since $\epsilon_1^{\holderExponent-\holderExponent_n^\flat} \geq 7 /9$ by the upper bound on $\slowlyIncreasingFunctionHolderExponent(n)$.  The result follows.
\end{proof}
We conclude this section by applying Lemma~\ref{lemma:HolderConstantEstimationWithEstimatedHolderExponent} to prove Theorem~\ref{Thm:BetaUnknown}.
\begin{proof}[Proof of Theorem~\ref{Thm:BetaUnknown}] We proceed similarly to the proof of Theorem \ref{Thm:BetaKnown}. For $\holderExponent' \in (0,1]$, $\holderConstant' \in [1,\infty)$ we let $S_{\holderExponent',\holderConstant'}^\circ$ denote the union $\bigcup_{\ell \in [\ell_{\tilde{\alpha}_n}]}B_{(\ell)}$ appearing in line 6 of Algorithm \ref{subsetSelectionAlgo} when it is applied with $\tilde{\alpha}_n$ in place of $\alpha$, $\holderExponent'$ in place of~$\holderExponent$ and $\holderConstant'$ in place of $\holderConstant$. Furthermore, we let $\hat{A}_{\holderExponent',\holderConstant'}^\circ$ denote the corresponding output set in $\mathcal{A}$.  If $\holderExponent_0' \leq \holderExponent_1'$ and $\holderConstant_0' \geq \holderConstant_1'$, then the $p$-values in~\eqref{eq:pValueDef} satisfy $\hat{p}_{n,\holderExponent_1',\holderConstant_1'}(\cdot)\leq \hat{p}_{n,\holderExponent_0',\holderConstant_0'}(\cdot)$, and so  $S_{\holderExponent'_0,\holderConstant'_0}^\circ \subseteq S_{\holderExponent'_1,\holderConstant'_1}^\circ$. Consequently, $\empiricalMarginalDistribution(\hat{A}_{\holderExponent'_0,\holderConstant'_0}^\circ)\leq \empiricalMarginalDistribution(\hat{A}_{\holderExponent'_1,\holderConstant'_1}^\circ)$ by line 6 in Algorithm~\ref{subsetSelectionAlgo}.

\textbf{(i)} Suppose that $P \in \classOfDistributionsSatisfyingRegularityCondition\cap  \classOfHolderDistributionsSuperLevelSetIdentifiable \cap \classOfSelfSimilarHolderDistributions$.  Then by Lemmas~\ref{lemma:typeIControl},~\ref{lemma:belongsToCorrespondingLipschitzClass} and ~\ref{lemma:HolderConstantEstimationWithEstimatedHolderExponent}, we have
\begin{align*}
\Prob_P\bigl( \hat{A}_{\mathrm{OSS}}''(\sample)  \nsubseteq \etaSuperLevelSet{\tau}\bigr) &\leq\Prob_P\bigl( S_{\holderExponentEstJoint{\tilde{\alpha}_n},2\holderConstantEstJoint{\tilde{\alpha}_n}}^\circ \nsubseteq \etaSuperLevelSet{\tau}\bigr) \nonumber \\
& \leq \Prob_P\bigl( S_{\holderExponent,\underline{\holderConstant}_{\holderExponent}}^\circ \nsubseteq \etaSuperLevelSet{\tau}\bigr)+\Prob_P\bigl(S^\circ_{\holderExponentEstJoint{\tilde{\alpha}_n},2\holderConstantEstJoint{\tilde{\alpha}_n}} \nsubseteq S_{\holderExponent,\underline{\holderConstant}_{\holderExponent}}^\circ \bigr) \nonumber \\
& \leq  \Prob_P\bigl( S_{\holderExponent,\underline{\holderConstant}_{\holderExponent}}^\circ \nsubseteq \etaSuperLevelSet{\tau}\bigr)+\Prob_P\bigl( \holderExponentEstJoint{\tilde{\alpha}_n} > \holderExponent\bigr)+\Prob_P\bigl( 2\holderConstantEstJoint{\tilde{\alpha}_n} < \underline{\holderConstant}_{\holderExponent}\bigr) \\&\leq 3\tilde{\alpha}_n \leq \alpha,
\end{align*}
as required.

\textbf{(ii)} Now suppose further that $P \in  \classOfDistributionsSatisfyingRegularityCondition\cap  \classOfHolderDistributionsSuperLevelSetIdentifiable \cap \classOfSelfSimilarHolderDistributions\cap \classOfWellApproximableSets$ so that by Theorem~\ref{thm:holderExponentEstimationThm} we have $\Prob_P\bigl(\{\holderExponentEstJoint{\tilde{\alpha}_n}<\holderExponent^\flat_n\} \cup \{\holderConstantEstJoint{\tilde{\alpha}_n} > \underline{\holderConstant}_{\holderExponent}\}\bigr) \leq 2\tilde{\alpha}_n$, where $\holderExponent^\flat_n$ is defined as in the proof of Lemma \ref{lemma:HolderConstantEstimationWithEstimatedHolderExponent}. Moreover, on the event $\{\holderExponentEstJoint{\tilde{\alpha}_n}\geq \holderExponent^\flat_n\} \cap \{\holderConstantEstJoint{\tilde{\alpha}_n} \leq \underline{\holderConstant}_{\holderExponent}\}$, we have 
\[
\empiricalMarginalDistribution(\hat{A}_{\mathrm{OSS}}'')= \empiricalMarginalDistribution(\hat{A}_{\holderExponentEstJoint{\tilde{\alpha}_n},2\holderConstantEstJoint{\tilde{\alpha}_n}}^\circ) \geq \empiricalMarginalDistribution(\hat{A}_{\holderExponent_n^{\flat},2\underline{\holderConstant}_{\holderExponent}}^\circ).
\]
Hence by Lemma~\ref{lemma:vapnikChervonenkisConcentration},
\begin{align*}
\E_P\bigl\{ \mu(\hat{A}_{\holderExponent_n^{\flat},2\underline{\holderConstant}_{\holderExponent}}^\circ)-\mu(\hat{A}_{\mathrm{OSS}}'')\bigr\} \leq 2C_{\mathrm{VC}}~ \sqrt{\frac{\vcDim(\mathcal{A})}{n}} =: \epsilon^{\mathrm{VC}}_n.
\end{align*}
By Lemma \ref{lemma:typeIControl}, we have $\Prob_P\bigl( \hat{A}_{\holderExponent_n^{\flat},2\underline{\holderConstant}_{\holderExponent}}^\circ \nsubseteq \etaSuperLevelSet{\tau}\bigr) \leq \Prob_P\bigl(S_{\holderExponent_n^{\flat},2\underline{\holderConstant}_{\holderExponent}}^\circ \nsubseteq \etaSuperLevelSet{\tau}\bigr) \leq \tilde{\alpha}_n$.  Observe that we may assume without loss of generality that $(2\underline{\holderConstant}_{\holderExponent})^{d/\holderExponent}\leq n$, because otherwise the regret bound is vacuous.  Thus, combining with the derivation in \textbf{(i)}, we have
\begin{align*}
 \E_P&\bigl\{\bigl(M_\tau - \mu(\hat{A}_{\mathrm{OSS}}'')\bigr) \cdot \one_{\{\hat{A}_{\mathrm{OSS}}'' \subseteq \etaSuperLevelSet{\tau}\}}\bigr\} \leq \E_P\big\{M_\tau - \mu\bigl(\hat{A}_{\mathrm{OSS}}''\bigr)\big\}+\Prob_P\bigl( \hat{A}_{\mathrm{OSS}}''(\sample)  \nsubseteq \etaSuperLevelSet{\tau}\bigr)\\
& \leq \E_P\big\{M_\tau - \mu\bigl(\hat{A}_{\holderExponent_n^{\flat},2\underline{\holderConstant}_{\holderExponent}}^\circ\bigr)\big\}+\epsilon^{\mathrm{VC}}_n+3\tilde{\alpha}_n\\
& \leq \E_P\bigl\{\bigl(M_\tau - \mu(\hat{A}_{\holderExponent_n^{\flat},2\underline{\holderConstant}_{\holderExponent}}^\circ)\bigr) \cdot \one_{\{\hat{A}_{\holderExponent_n^{\flat},2\underline{\holderConstant}_{\holderExponent}}^\circ \subseteq \etaSuperLevelSet{\tau}\}}\bigr\} +  
\Prob_P\bigl( \hat{A}_{\holderExponent_n^{\flat},2\underline{\holderConstant}_{\holderExponent}}^\circ  \nsubseteq \etaSuperLevelSet{\tau}\bigr)+\epsilon^{\mathrm{VC}}_n+3\tilde{\alpha}_n\\
& \leq \E_P\bigl(M_\tau - \mu(\hat{A}_{\holderExponent_n^{\flat},2\underline{\holderConstant}_{\holderExponent}}^\circ)\bigm|\hat{A}_{\holderExponent_n^{\flat},2\underline{\holderConstant}_{\holderExponent}}^\circ \subseteq \etaSuperLevelSet{\tau}\bigr) + 4\tilde{\alpha}_n+\epsilon^{\mathrm{VC}}_n\\
&= R_\tau(\hat{A}_{\holderExponent_n^{\flat},2\underline{\holderConstant}_{\holderExponent}}^\circ)+4\tilde{\alpha}_n+\epsilon^{\mathrm{VC}}_n
\\
&\leq C' \biggl\{\biggl({\frac{(2\underline{\holderConstant}_{\holderExponent})^{d/\holderExponent_n^{\flat}}}{n} \cdot \log_+\Bigl(\frac{n}{\tilde{\alpha}_n}\Bigr)\biggr)^{\frac{\holderExponent_n^{\flat} \approximableDensityExponent \approximableMarginExponent}{\approximableDensityExponent(2\holderExponent_n^{\flat}+d) +\holderExponent_n^{\flat}\approximableMarginExponent}}}+\frac{1}{n^{1/2}} \biggr\}+4\tilde{\alpha}_n+\epsilon^{\mathrm{VC}}_n\\
&\leq C'  \biggl\{n^{\frac{2(\holderExponent-\holderExponent^\flat_n) \approximableDensityExponent \approximableMarginExponent}{\approximableDensityExponent(2\holderExponent_n^{\flat}+d) +\holderExponent_n^{\flat}\approximableMarginExponent}}  \biggl( {\frac{(2\underline{\holderConstant}_{\holderExponent})^{d/\holderExponent} }{n}  \log_+\Bigl(\frac{3n^2}{\alpha}\Bigr)\biggr)^{\frac{\holderExponent \approximableDensityExponent \approximableMarginExponent}{\approximableDensityExponent(2\holderExponent+d) +\holderExponent\approximableMarginExponent}}} \! \! +\frac{1}{n^{1/2}} \biggr\}+4\tilde{\alpha}_n+\epsilon^{\mathrm{VC}}_n\\
&\leq C' \biggl\{ \slowlyIncreasingFunctionHolderExponent(n)^{\frac{4\approximableMarginExponent(7+2d)}{d}}  \biggl( {\frac{(2\underline{\holderConstant}_{\holderExponent})^{d/\holderExponent} }{n}  \log_+\Bigl(\frac{3n^2}{\alpha}\Bigr)\biggr)^{\frac{\holderExponent \approximableDensityExponent \approximableMarginExponent}{\approximableDensityExponent(2\holderExponent+d) +\holderExponent\approximableMarginExponent}}}+\frac{1}{n^{1/2}} \biggr\}+4\tilde{\alpha}_n+\epsilon^{\mathrm{VC}}_n,
\end{align*}
where in the third to last inequality we applied Proposition~\ref{thm:powerBound} with $C'$ in place of $C$ by noting that $P \in \mathcal{P}_{\mathrm{H\ddot{o}l}}(\holderExponent,\underline{\holderConstant}_{\holderExponent},\tau) \subseteq \mathcal{P}_{\mathrm{H\ddot{o}l}}(\holderExponent_n^{\flat},2\underline{\holderConstant}_{\holderExponent},\tau)$.  The regret bound on $\hat{A}''_{\mathrm{OSS}}$ follows by using \textbf{(i)}  once again. 
\end{proof}


\section{Proof of the upper bound in Theorem~\ref{Thm:minimaxRateHOS}}
\label{Sec:minimaxRateHOSProof}

Recall that the upper bound in Theorem~\ref{Thm:minimaxRateHOS} will follow from Proposition~\ref{thm:powerBoundHO}.  First we state the following consequence of Hoeffding's inequality.
\begin{lemma}\label{lemma:hoeffdingConsequenceHOSmoothness} Fix $(\holderExponent,\holderConstant) \in (0,\infty)\times [1,\infty)$ and let $\probDistribution \in \classOfHolderDistributions$.  Suppose that $\sample= \bigl((X_i,Y_i)\bigr)_{i \in [n]}\sim P^{\otimes n}$ and let $\sampleX = (X_i)_{i \in [n]}$. Given $x \in \R^d$, $h \in [0,1]$, and $\alpha \in (0,1)$ define 
\begin{align*}
\hat{\Delta}_{x,h}(\alpha):=\begin{cases} \sqrt{e_0^\top \bigl(Q_{x,h}^{\holderExponent}\bigr)^{-1} e_0 }  \cdot \biggl(  \holderConstant \cdot h^\holderExponent |\mathcal{N}_{x,h}|^{1/2} + \sqrt{\frac{\log(1/\alpha)}{2}} \biggr) &\text{ if } Q_{x,h}^{\holderExponent} \text{ is invertible, }\\
1 &\text{ otherwise.}
\end{cases}
\end{align*}
Then 
\begin{align*}
\max\bigl\{\Prob\bigl( \hat{\regressionFunction}(x)-\regressionFunction(x) \geq \hat{\Delta}_{x,h}(\alpha) \bigm| \sampleX\bigr),\Prob\bigl( \regressionFunction(x)-\hat{\regressionFunction}(x) \geq \hat{\Delta}_{x,h}(\alpha) \bigm| \sampleX\bigr)\bigr\} \leq \alpha.
\end{align*}
\end{lemma}
\begin{proof} Fix a realisation of $\sampleX = (X_i)_{i \in [n]}$. It suffices to restrict our attention to the case where $Q_{x,h}^{\holderExponent}$ is invertible.  Writing $u_i:= \big\langle e_0, \bigl(Q_{x,h}^{\holderExponent}\bigr)^{-1} \Phi_{x,h}^\holderExponent(X_i)\big\rangle$ for $i \in [n]$, we have
\begin{align*}
\sum_{i \in \mathcal{N}_{x,h}} u_i \cdot \langle w_{x,h}^\holderExponent, \Phi_{x,h}^\holderExponent(X_i)\rangle &=   e_0^\top \bigl(Q_{x,h}^{\holderExponent}\bigr)^{-1}\sum_{i \in \mathcal{N}_{x,h}} \Phi_{x,h}^\holderExponent(X_i) ( w_{x,h}^\holderExponent)^\top \Phi_{x,h}^\holderExponent(X_i) &= e_0^\top w_{x,h}^\holderExponent \\
&= \regressionFunction(x).
\end{align*}
Hence,
\begin{align}\label{eq:etaEstHigherorderDecomp}
\langle e_0, \hat{w}_{x,h}^\holderExponent \rangle &=e_0^\top\bigl(Q_{x,h}^{\holderExponent}\bigr)^{-1} V_{x,h}^{\holderExponent} = \sum_{i \in \mathcal{N}_{x,h}}  Y_i  \cdot e_0^\top \bigl(Q_{x,h}^{\holderExponent}\bigr)^{-1} \Phi_{x,h}^\holderExponent(X_i) \nonumber \\ 
& = \sum_{i \in \mathcal{N}_{x,h}} u_i \cdot \Bigl( \bigl\{ Y_i-\regressionFunction(X_i)\bigr\}  +  \bigl\{\regressionFunction(X_i) - \taylorSeries_{x}^{\holderExponent}[\regressionFunction](X_i) \bigr\}   +\big\langle w_{x,h}^\holderExponent, \Phi_{x,h}^\holderExponent(X_i)\big\rangle \Bigr) \nonumber \\
& = \sum_{i \in \mathcal{N}_{x,h}} u_i \cdot \Bigl( \bigl\{ Y_i-\regressionFunction(X_i)\bigr\}  +  \bigl\{\regressionFunction(X_i) - \taylorSeries_{x}^{\holderExponent}[\regressionFunction](X_i) \bigr\}   \Bigr)+\regressionFunction(x).
\end{align}
Note also that 
\begin{align}\label{eq:sumOfUisSquared}
\sum_{i \in \mathcal{N}_{x,h}} u_i^2 = \sum_{i \in \mathcal{N}_{x,h}} e_0^\top \bigl(Q_{x,h}^{\holderExponent}\bigr)^{-1} \Phi_{x,h}^\holderExponent(X_i)  \Phi_{x,h}^\holderExponent(X_i)^\top \bigl(Q_{x,h}^{\holderExponent}\bigr)^{-1} e_0 = e_0^\top \bigl(Q_{x,h}^{\holderExponent}\bigr)^{-1} e_0.
\end{align}
In addition, since $\probDistribution \in \classOfHolderDistributions$, for each $i \in  \mathcal{N}_{x,h}$, we have
\begin{equation}
\label{Eq:HolderApproximation}
|\regressionFunction(X_i) - \taylorSeries_{x}^{\holderExponent}[\regressionFunction](X_i)| \leq \holderConstant \cdot \supNorm{X_i-x}^\holderExponent \leq \holderConstant \cdot h^\holderExponent,
\end{equation}
and so by~\eqref{Eq:HolderApproximation}, the Cauchy--Schwarz inequality and~\eqref{eq:sumOfUisSquared}, we have
\begin{align}
\biggl| \sum_{i \in \mathcal{N}_{x,h}} u_i \cdot \bigl\{\regressionFunction(X_i) - \taylorSeries_{x}^{\holderExponent}[\regressionFunction](X_i) \bigr\} \biggr|   & \leq  \holderConstant \cdot h^\holderExponent \cdot \sum_{i \in \mathcal{N}_{x,h}} |u_i| \nonumber \\
&\leq  \holderConstant \cdot h^\holderExponent \cdot \sqrt{|\mathcal{N}_{x,h}| \cdot e_0^\top \bigl(Q_{x,h}^{\holderExponent}\bigr)^{-1} e_0 }. \label{Eq:etaPapprox}
\end{align}
We conclude, by the definition of $\hat{\regressionFunction}(x)$, \eqref{eq:etaEstHigherorderDecomp}, \eqref{eq:sumOfUisSquared}, \eqref{Eq:etaPapprox} and Hoeffding's inequality, that
\begin{align*}
\Prob\biggl( \hat{\regressionFunction}(x)- \eta(x) \geq \hat{\Delta}_{x,h}(\alpha) \biggr) 
&= \Prob\Biggl( \sum_{i \in \mathcal{N}_{x,h}} u_i \cdot \bigl\{ Y_i-\regressionFunction(X_i)\bigr\}  \geq \sqrt{\frac{\log(1/\alpha)}{2}\sum_{i \in \mathcal{N}_{x,h}}u_i^2} \Biggr)\leq \alpha.
\end{align*}
The other inequality follows similarly.
\end{proof}
\begin{lemma}
\label{lemma:pValueHigherOrderSmoothness} 
Fix $(\holderExponent,\holderConstant) \in (0,\infty)\times [1,\infty)$ and let $\probDistribution \in \classOfHolderDistributions$. Suppose that $\sample= \bigl((X_i,Y_i)\bigr)_{i \in [n]}\sim P^{\otimes n}$ and let $\sampleX = (X_i)_{i \in [n]}$.  Then for any closed hyper-cube $B \subseteq \R^d$ with $\diamSup(B) \leq 1$ and $\inf_{x' \in B} \regressionFunction(x') \leq \tau$, and any $\alpha \in (0,1)$, we have
\begin{align*}
{\Prob}\big( \pValueNHO(B) \leq \alpha \mid \sampleX \big) \leq \alpha.
\end{align*}
\end{lemma}
\begin{proof}
Recall that $x \in \R^d$ and $r \in [0,1/2]$ denote the centre and $\ell_\infty$-radius of $B$, and that $h =(2r)^{1 \wedge \frac{1}{\holderExponent}} \in [0,1]$.  Again, it restrict our attention to the case where $Q_{x,h}^{\holderExponent}$ is invertible. Since $\inf_{x' \in B} \regressionFunction(x') \leq \tau$ and $\probDistribution \in \classOfHolderDistributions$, we have $\regressionFunction(x) \leq \tau + \holderConstant \cdot r^{\holderExponent \wedge 1}$, and hence the lemma follows from Lemma \ref{lemma:hoeffdingConsequenceHOSmoothness}.
\end{proof}
\begin{proof}[Proof of Proposition~\ref{thm:typeIControlHigherOrderSmoothness}]
This follows from Lemma~\ref{lemma:pValueHigherOrderSmoothness} in the same way as Proposition~\ref{thm:typeIControl} followed from Lemma~\ref{lemma:pValue}.
\end{proof}

\newcommand{\lbSmallSetMeasure}{a}

We now turn to the proof of Proposition~\ref{prop:generalPowerBound}, which will rely on several lemmas.  For $\lbSmallSetMeasure > 0$, let $\mathcal{K}(\lbSmallSetMeasure)$ denote the set of measurable sets $K \subseteq \closedMetricBallSupNorm{0}{1}$ with $\Lebesgue(K) \geq \lbSmallSetMeasure$.
\begin{lemma}\label{lemma:compactnessLB} Given $d \in \N$, $\holderExponent \in (0,\infty)$ and $\lbSmallSetMeasure \in (0,1)$, we have 
\begin{align*}
c_{\min}(d,\holderExponent,\lbSmallSetMeasure):=1 \wedge\inf_{K \in \mathcal{K}(\lbSmallSetMeasure)} \biggl \{\eigenValueMinimal\biggl(\int_K \Phi_{0,1}^\holderExponent(z) \Phi_{0,1}^\holderExponent(z)^\top \, d\Lebesgue(z)\biggr) \biggr\}>0.
\end{align*}
\end{lemma}
\begin{proof} Suppose, for a contradiction, that $c_{\min}(d,\holderExponent,\lbSmallSetMeasure)=0$.  Then we can find a sequence $(K^{(\ell)})_{\ell \in \N}$ in $\mathcal{K}(\lbSmallSetMeasure)$, along with a sequence $(w^{(\ell)})_{\ell \in \N}$ with $w^{(\ell)} \in \R^{\mathcal{V}(\holderExponent)}$, $\|w^{(\ell)}\|_2=1$ and  
\begin{align}\label{eq:toCondradict}
 \lim_{\ell \rightarrow \infty}\int_{K^{(\ell)}} \big\langle w^{(\ell)}, \Phi_{0,1}^\holderExponent(z)\big\rangle^2 \, d\Lebesgue(z) \nonumber &=  \lim_{\ell \rightarrow \infty}\bigl(w^{(\ell)}\bigr)^\top \biggl(\int_{K^{(\ell)}}  \Phi_{0,1}^\holderExponent(z)\Phi_{0,1}^\holderExponent(z)^\top \, d\Lebesgue(z)\biggr)w^{(\ell)} \\
 &=0.
\end{align}
By moving to a subsequence if necessary,  we may assume that  $\lim_{\ell \rightarrow \infty} w^{(\ell)}=w^*$ for some $w^* \in \R^{\mathcal{V}(\holderExponent)}$ with $\|w^*\|_2=1$. Now since $z \mapsto \langle w^*, \Phi_{0,1}^\holderExponent(z)\rangle$ is a non-zero polynomial, the zero-set $\mathcal{Z}_{w^*}:=\bigl\{ z \in \closedMetricBallSupNorm{0}{1}:\langle w^*, \Phi_{0,1}^\holderExponent(z)\rangle=0\bigr\}$ satisfies $\Lebesgue(\mathcal{Z}_{w^*})=0$ \citep[e.g.][Lemma~1]{okamoto1973distinctness}. In addition, by the continuity of $z\mapsto \langle w^*, \Phi_{0,1}^\holderExponent(z)\rangle$, the set $\mathcal{Z}_{w^*}$ is closed. By countable additivity of the finite measure $\Lebesgue|_{\closedMetricBallSupNorm{0}{1}}$, there exists $\epsilon_{\lbSmallSetMeasure}>0$ such that $\Lebesgue\bigl(\mathcal{Z}_{w^*}^ {\epsilon_\lbSmallSetMeasure}\bigr)\leq \lbSmallSetMeasure/2$ where $\mathcal{Z}_{w^*}^{\epsilon_{\lbSmallSetMeasure}}:= \bigcup_{z \in \mathcal{Z}_{w^*}^{\epsilon_{\lbSmallSetMeasure}}} \openMetricBallSupNorm{z}{\epsilon} = \mathcal{Z}_{w^*} + \openMetricBallSupNorm{z}{\epsilon_{\lbSmallSetMeasure}}$. By continuity again, $\delta_\lbSmallSetMeasure:= \inf_{z \in \closedMetricBallSupNorm{0}{1}\setminus \mathcal{Z}_{w^*}^ {\epsilon_\lbSmallSetMeasure}} \bigl|\big\langle w^*, \Phi_{0,1}^\holderExponent(z)\big\rangle\bigr| >0$. Now choose $\ell_0 \in \mathbb{N}$ sufficiently large that 
\[
\sup_{\ell\geq \ell_0} \bigl\|w^{(\ell)}-w^*\bigr\|_2 \leq \frac{\delta_\lbSmallSetMeasure}{2\sqrt{|\mathcal{V}(\holderExponent)|}},
\]
so that, by Cauchy--Schwarz,
\begin{align*}
\bigl|\langle w^{(\ell)}, \Phi_{0,1}^\holderExponent(z)\rangle\bigr| \geq \bigl|\langle w^*, \Phi_{0,1}^\holderExponent(z)\rangle\bigr| - \bigl|\langle w^{(\ell)} - w^*, \Phi_{0,1}^\holderExponent(z)\rangle\bigr| &\geq \delta_\lbSmallSetMeasure - \frac{\delta_\lbSmallSetMeasure}{2\sqrt{|\mathcal{V}(\holderExponent)|}} \cdot \bigl\|\Phi_{0,1}^\holderExponent(z)\bigr\|_2 \\
&\geq \frac{\delta_\lbSmallSetMeasure}{2}
\end{align*}
for all $\ell \geq \ell_0$ and $z \in \closedMetricBallSupNorm{0}{1}\setminus \mathcal{Z}_{w^*}^ {\epsilon_\lbSmallSetMeasure}$. Hence, for all $\ell \geq \ell_0$,
\begin{align*}
\int_{K^{(\ell)}} \big\langle w^{(\ell)}, \Phi_{0,1}^\holderExponent(z)\big\rangle^2 \, d\Lebesgue(z) \geq \int_{K^{(\ell)}\setminus \mathcal{Z}_{w^*}^ {\epsilon_\lbSmallSetMeasure}} \big\langle w^{(\ell)}, \Phi_{0,1}^\holderExponent(z)\big\rangle^2 \, d\Lebesgue(z) \geq \frac{\lbSmallSetMeasure \cdot \delta_\lbSmallSetMeasure^2}{8}>0,
\end{align*}
which contradicts \eqref{eq:toCondradict}, and completes the proof of the lemma.
\end{proof}

\begin{lemma}\label{lemma:leastSingularValuePopulation} Suppose that $\regularityConstant \in (0,1)$, $\xi \in (0,\infty)$, $\holderExponent \in (0,\infty)$, $x \in \R^d$ and $r \in (0,1/2]$ satisfy $\closedMetricBallSupNorm{x}{r}\cap \muRegularSet \cap \densitySuperLevelSet{\xi} \neq \emptyset$. Given any $h \in [2r, 1]$, we have $\mu\bigl(\closedMetricBallSupNorm{x}{h}\bigr) \geq \regularityConstant \cdot \xi \cdot (h/2)^d$. In addition, if either $\holderExponent \in (0,1]$ or $3r \leq \regularityConstant h$, then 
\begin{align}\label{eq:singularValueLBLemmaLeastSingularValuePopulation}
\eigenValueMinimal\biggl(\int_{\closedMetricBallSupNorm{x}{h}} \Phi_{x,h}^\holderExponent(z) \Phi_{x,h}^\holderExponent(z)^\top \, d\mu(z)\biggr)\geq  2^{-(3d+1)} \cdot \regularityConstant 
\cdot c_{\min}^0 \cdot \mu\bigl(\closedMetricBallSupNorm{x}{h}\bigr),
\end{align}
where $c_{\min}^0 \equiv c_{\min}(d,\holderExponent,2^{-(d+1)}{\regularityConstant}) \in (0,1]$ is taken from Lemma~\ref{lemma:compactnessLB}.
\end{lemma}
\begin{proof} First choose $x_0 \in \closedMetricBallSupNorm{x}{r}\cap \muRegularSet \cap \densitySuperLevelSet{\xi}$, noting that $\closedMetricBallSupNorm{x_0}{h/2}\subseteq \closedMetricBallSupNorm{x_0}{h-r}\subseteq  \closedMetricBallSupNorm{x}{h}$. Hence, since $x_0 \in \muRegularSet \cap \densitySuperLevelSet{\xi}$, we deduce
\begin{align*}
\mu\bigl(\closedMetricBallSupNorm{x}{h}\bigr) \geq \mu\bigl(\closedMetricBallSupNorm{x_0}{h-r}\bigr) \geq \mu\bigl(\closedMetricBallSupNorm{x_0}{h/2}\bigr) \geq \regularityConstant \cdot \biggl(\frac{h}{2}\biggr)^d \cdot \xi.
\end{align*}
For $\holderExponent \in (0,1]$, we have $\Phi_{x,h}^\holderExponent(\cdot) \equiv 1$, so \eqref{eq:singularValueLBLemmaLeastSingularValuePopulation} is immediate. Suppose now that $3r \leq \regularityConstant h$, so that $\openMetricBallSupNorm{x_0}{(1+\regularityConstant)(h-r)} \supseteq \closedMetricBallSupNorm{x_0}{h+r} \supseteq \closedMetricBallSupNorm{x}{h}$.  Thus, since $x_0 \in \muRegularSet$, we infer that with $M_{x,h}:=\sup_{x' \in \closedMetricBallSupNorm{x}{h}}f_\mu(x')$,
\begin{align}\label{eq:boundingBelowByMaxDensityIneq}
\mu\bigl(\closedMetricBallSupNorm{x}{h}\bigr) \geq \mu\bigl(\openMetricBallSupNorm{x_0}{h-r}\bigr) \geq \regularityConstant  (h-r)^d \cdot \sup_{x' \in \openMetricBallSupNorm{x_0}{(1+\regularityConstant)(h-r)}}f_\mu(x')\geq  \regularityConstant \cdot  \biggl(\frac{h}{2}\biggr)^d  \cdot M_{x,h}.
\end{align}
Moreover, if we take $J_{x,h}:= \bigl\{ x' \in \closedMetricBallSupNorm{x}{h}: f_\mu(x') \geq 2^{-(2d+1)} \cdot \regularityConstant \cdot M_{x,h}\bigr\}$, then
\begin{align*}
 \mu\bigl(\closedMetricBallSupNorm{x}{h}\bigr) & \leq \Lebesgue(J_{x,h}) \cdot M_{x,h}+\Lebesgue\bigl(\closedMetricBallSupNorm{x}{h} \setminus J_{x,h}\bigr) \cdot 2^{-(2d+1)} \cdot \regularityConstant \cdot M_{x,h} \\ &\leq \Lebesgue(J_{x,h}) \cdot M_{x,h} +\frac{\regularityConstant}{2} \cdot  \biggl(\frac{h}{2}\biggr)^d  \cdot M_{x,h}, 
\end{align*}
so by \eqref{eq:boundingBelowByMaxDensityIneq} we have $\Lebesgue(J_{x,h}) \geq 2^{-(d+1)} \cdot \regularityConstant \cdot  h^d $. Taking $K_{x,h}:=h^{-1}\cdot (J_{x,h}-x)\subseteq \closedMetricBallSupNorm{0}{1}$, we have $\Lebesgue(K_{x,h}) \geq 2^{-(d+1)}\cdot \regularityConstant$. Given any $w \in  \R^{\mathcal{V}(\holderExponent)}$ with $\|w\|_2=1$, it follows from Lemma \ref{lemma:compactnessLB} that
\begin{align*}
\int_{\closedMetricBallSupNorm{x}{h}} \big\langle w, \Phi_{x,h}^\holderExponent(z)\big\rangle^2 \, d\mu(z) &\geq  2^{-(2d+1)} \cdot \regularityConstant \cdot M_{x,h}  \cdot \int_{J_{x,h}} \big\langle w, \Phi_{x,h}^\holderExponent(z)\big\rangle^2 \, d\Lebesgue(z) \\
&\geq  2^{-(2d+1)} \cdot \regularityConstant \cdot M_{x,h}  \cdot h^{d} \cdot \int_{K_{x,h}} \big\langle w, \Phi_{0,1}^\holderExponent(z')\big\rangle^2 \, d\Lebesgue(z')\\
&\geq  2^{-(3d+1)} \cdot \regularityConstant   \cdot c_{\min}^0 \cdot \mu\bigl(\closedMetricBallSupNorm{x}{h}\bigr).
\end{align*}
The result follows.
\end{proof}

\begin{lemma}\label{lemma:chernoffApplications} Suppose that $\regularityConstant \in (0,1)$, $\xi \in (0,\infty)$, $\holderExponent \in (0,\infty)$, $x \in \R^d$, $r \in (0,1/2]$ and $h \in [2r,1]$ satisfy $\closedMetricBallSupNorm{x}{r}\cap \muRegularSet \cap \densitySuperLevelSet{\xi} \neq \emptyset$, and choose $\delta \in (0,1)$. Suppose also that
\begin{align}\label{eq:hLBFromLemmaChernoffApplications}
h\geq \biggl\{ \frac{2^{4(d+1)} \cdot |\mathcal{V}(\holderExponent)|}{c_{\min}^0 \cdot \regularityConstant^2 \cdot \xi \cdot n}\cdot  \log\biggl(\frac{2|\mathcal{V}(\holderExponent)|}{\delta}\biggr) \biggr\}^{1/d},
\end{align}
and that either $\holderExponent \in (0,1]$ or $3r \leq \regularityConstant h$.  Then
\begin{align*}
\Prob\Bigl(\bigl\{ |\mathcal{N}_{x,h}| \leq 2n \cdot \mu\bigl(\closedMetricBallSupNorm{x}{h}\bigr)\bigr\} \cap \bigl\{ \eigenValueMinimal\bigl(Q_{x,h}^{\holderExponent} \bigr)\geq 2^{-(3d+2)} \cdot n\cdot c_{\min}^0 \cdot \regularityConstant \cdot & \mu\bigl(\closedMetricBallSupNorm{x}{h}\bigr) \bigr\} \Bigr) \\
&\geq 1-\delta.
\end{align*}
\end{lemma}
\begin{proof} By Lemma~\ref{lemma:leastSingularValuePopulation} and~\eqref{eq:hLBFromLemmaChernoffApplications}, we have $\mu\bigl(\closedMetricBallSupNorm{x}{h}\bigr) \geq \regularityConstant \cdot \xi \cdot (h/2)^d \geq (8/3) \log(2/\delta)/n$. Hence, by the multiplicative Chernoff bound (Lemma~\ref{lemma:multChernoff}),  
\begin{align}\label{eq:firstPartUnionLemmaChernoffApp}
\Prob\bigl\{|\mathcal{N}_{x,h}|  > 2n \cdot \mu\bigl(\closedMetricBallSupNorm{x}{h}\bigr)\bigr\} \leq \frac{\delta}{2}.
\end{align}
In addition, if either $\holderExponent \in (0,1]$ or $3r \leq \regularityConstant h$, then by Lemma~\ref{lemma:leastSingularValuePopulation} again,
\begin{align*}
\eigenValueMinimal\biggl(\int_{\closedMetricBallSupNorm{x}{h}} \Phi_{x,h}^\holderExponent(z) \Phi_{x,h}^\holderExponent(z)^\top  \, d\mu(z)\biggr) & \geq 2^{-(3d+1)} \cdot \regularityConstant \cdot c_{\min}^0
\cdot  \mu\bigl(\closedMetricBallSupNorm{x}{h}\bigr) \\
& \geq 2^{-(4d+1)} \cdot \regularityConstant^2 \cdot c_{\min}^0 \cdot \xi \cdot h^d\\
& \geq  \frac{8|\mathcal{V}(\holderExponent)|}{n} \cdot \log\biggl(\frac{2|\mathcal{V}(\holderExponent)|}{\delta}\biggr).
\end{align*}
Note also that $\eigenValueMaximal\bigl(\Phi_{x,h}^\holderExponent(X_1) \Phi_{x,h}^\holderExponent(X_1)^\top \cdot \one_{\{X_1 \in \closedMetricBallSupNorm{x}{h}\}} \bigr)\leq |\mathcal{V}(\holderExponent)|$. Hence, by a matrix multiplicative Chernoff bound (Lemma~\ref{Lemma:Tropp}) applied with $m=n$, $\mathbf{Z}_i = \Phi_{x,h}^\holderExponent(X_i) \Phi_{x,h}^\holderExponent(X_i)^\top \cdot \one_{\{X_i \in \closedMetricBallSupNorm{x}{h}\}}$ and $q = |\mathcal{V}(\holderExponent)|$, we have
\begin{align}\label{eq:secondPartUnionLemmaChernoffApp}
\Prob\bigl\{ \eigenValueMinimal\bigl(Q_{x,h}^{\holderExponent} \bigr) &<  2^{-(3d+2)} \cdot n\cdot c_{\min}^0 \cdot \regularityConstant 
\cdot  \mu\bigl(\closedMetricBallSupNorm{x}{h}\bigr)  \bigr\} \nonumber \\  &\leq  \Prob\biggl\{ \eigenValueMinimal\bigl(Q_{x,h}^{\holderExponent} \bigr) < \frac{n}{2} \cdot \eigenValueMinimal\biggl(\int_{\closedMetricBallSupNorm{x}{h}} \Phi_{x,h}^\holderExponent(z) \Phi_{x,h}^\holderExponent(z)^\top d\mu(z)\biggr)  \biggr\} \leq \frac{\delta}{2}.
\end{align}
The result now follows by combining~\eqref{eq:firstPartUnionLemmaChernoffApp} and~\eqref{eq:secondPartUnionLemmaChernoffApp} with a union bound.
\end{proof}

\begin{lemma}\label{lemma:powerForSingleHOPValue} Suppose that $\alpha \in (0,1)$, $\holderExponent > 0$, $\holderConstant \geq 1$, $\approximableDensityExponent, \approximableMarginExponent > 0$, $\regularityConstant \in (0,1)$, $\approximableSetsConstant \geq 1$, take $P \in \classOfHolderDistributions \cap \classOfWellApproximableSetsHO$ and let ${C}_{\mathrm{pv}}:=2^{2d+5} \cdot \bigl(\regularityConstant \cdot \sqrt{c_{\min}^0}\bigr)^{-1}$. Suppose further that $\xi, \Delta \in (0,\infty)$, $x \in \R^d$ and $r \in (0,1/2]$ satisfy $\closedMetricBallSupNorm{x}{r}\cap \muRegularSet \cap \densitySuperLevelSet{\xi} \cap \etaSuperLevelSet{\tau+\Delta}\neq \emptyset$, where $\mu$ is the marginal distribution of $P$ on $\R^d$, and $\regressionFunction:\R^d \rightarrow [0,1]$ is the regression function. Given any $\delta \in (0,1)$ with 
\begin{align*}
r \geq \biggl\{ \frac{{C}_{\mathrm{pv}}^2  \cdot |\mathcal{V}(\holderExponent)|}{ \xi \cdot n}\cdot  \log\biggl(\frac{4|\mathcal{V}(\holderExponent)|}{\delta}\biggr) \biggr\}^{\frac{\holderExponent}{d(\holderExponent \wedge 1)}}, \ \
\Delta \geq  C_{\mathrm{pv}} \Biggl( \holderConstant r^{\holderExponent\wedge 1}+\sqrt{\frac{ \log\bigl(2/(\alpha \wedge \delta)\bigr)}{ \xi \cdot n \cdot r^{d  (\holderExponent\wedge 1)/\holderExponent}}}\Biggr),
\end{align*}
and either $\holderExponent \in (0,1]$ or $r \leq \{2(\regularityConstant/3)^\holderExponent\}^{\frac{1}{{\holderExponent- 1}}}$, we have $\Prob\bigl\{\pValueNHO\bigl(\closedMetricBallSupNorm{x}{r}\bigr) \leq \alpha\bigr\} \geq 1-\delta$.
\end{lemma}

\begin{proof} First recall that in the construction of our $p$-values $\pValueNHO(\cdot)$ in \eqref{eq:pValueDefHigherOrderSmoothness} we take $h=(2r)^{1\wedge \frac{1}{\holderExponent}}$. To prove the lemma we define events 
\begin{align*}
\mathcal{E}_{\delta}^\eta &:=\biggl\{ \hat{\regressionFunction}(x) > \regressionFunction(x) -\sqrt{e_0^\top \bigl(Q_{x,h}^{\holderExponent}\bigr)^{+} e_0 }  \cdot \biggl(  \holderConstant \cdot h^\holderExponent \cdot |\mathcal{N}_{x,h}|^{1/2} + \sqrt{\frac{\log(2/\delta)}{2}} \biggr) \biggr\}, \\
\mathcal{E}_{\delta}^\mu &:=\biggl\{  \eigenValueMinimal\bigl(Q_{x,h}^{\holderExponent} \bigr)\geq 2^{-(3d+2)} \cdot  c_{\min}^0 \cdot \regularityConstant^2 \cdot \max\biggl\{  \xi \cdot n \cdot \biggl(\frac{h}{2}\biggr)^d,\frac{|\mathcal{N}_{x,h}|}{2\regularityConstant }\biggr\}\biggr\}.
\end{align*}
Note that since $P \in \classOfHolderDistributions$ and $\closedMetricBallSupNorm{x}{r}\cap \etaSuperLevelSet{\tau+\Delta}\neq \emptyset$, we have $\regressionFunction(x) \geq \tau+\Delta-\holderConstant \cdot r^{\holderExponent \wedge 1}$. Hence, on the event $\mathcal{E}_{\delta}^\eta\cap \mathcal{E}_{\delta}^\mu$ we have
\begin{align*}
 \hat{\regressionFunction}&(x) - \tau -  \holderConstant \Bigl(1+2\sqrt{e_0^\top \bigl(Q_{x,h}^{\holderExponent}\bigr)^{-1} e_0 \cdot |\mathcal{N}_{x,h}|}\Bigr)  r^{\holderExponent \wedge 1} \\
 &> \regressionFunction(x)-\tau - \holderConstant \Bigl(1+4\sqrt{e_0^\top \bigl(Q_{x,h}^{\holderExponent}\bigr)^{-1} e_0 \cdot |\mathcal{N}_{x,h}|}\Bigr)  r^{\holderExponent \wedge 1} -\sqrt{\frac{1}{2}\cdot e_0^\top \bigl(Q_{x,h}^{\holderExponent}\bigr)^{-1} e_0 \cdot \log(2/\delta)} \\
 & > \Delta - 2\holderConstant \Bigl(1+2\sqrt{e_0^\top \bigl(Q_{x,h}^{\holderExponent}\bigr)^{-1} e_0 \cdot |\mathcal{N}_{x,h}|}\Bigr)  r^{\holderExponent \wedge 1} -\sqrt{\frac{1}{2}\cdot e_0^\top \bigl(Q_{x,h}^{\holderExponent}\bigr)^{-1} e_0 \cdot \log(2/\delta)} 
 \\& \geq \Delta - 2\holderConstant \biggl(1+\sqrt{\frac{2^{3d+5}}{ c^0_{\min} \cdot \regularityConstant}}\biggr)  r^{\holderExponent \wedge 1} - \sqrt{\frac{2^{4d+1}\cdot  \log(2/\delta)}{c^0_{\min} \cdot \regularityConstant^2 \cdot \xi \cdot n \cdot r^{d  (\holderExponent\wedge 1)/\holderExponent}}} \\
 & \geq \sqrt{\frac{2^{4d+1}\cdot  \log(2/\alpha)}{c^0_{\min} \cdot \regularityConstant^2 \cdot \xi \cdot n \cdot r^{d  (\holderExponent\wedge 1)/\holderExponent}}}\geq \sqrt{\frac{1}{2}\cdot e_0^\top \bigl(Q_{x,h}^{\holderExponent}\bigr)^{-1} e_0 \cdot \log(1/\alpha)}. 
\end{align*}
Hence, on the event $\mathcal{E}_{\delta}^\eta\cap \mathcal{E}_{\delta}^\mu$ we have $\pValueNHO\bigl(\closedMetricBallSupNorm{x}{r}\bigr) \leq \alpha$. Now by Lemma \ref{lemma:hoeffdingConsequenceHOSmoothness} we have $\Prob\bigl((\mathcal{E}_{\delta}^\regressionFunction)^{\mathrm{c}} \cap E_\delta^\mu\bigr) \leq \delta/2$. Moreover, by Lemma~\ref{lemma:leastSingularValuePopulation} we have $\mu\bigl(\closedMetricBallSupNorm{x}{h}\bigr) \geq \regularityConstant \cdot \xi \cdot (h/2)^d$. Hence, by Lemma~\ref{lemma:chernoffApplications} we have $\Prob\bigl((\mathcal{E}_{\delta}^\mu)^{\mathrm{c}}\bigr) \leq \delta/2$.  Thus $\Prob\bigl((\mathcal{E}_{\delta}^\regressionFunction)^{\mathrm{c}} \cup (E_\delta^\mu)^{\mathrm{c}}\bigr) \leq \delta$, and the conclusion follows.
\end{proof}
Given any $\regularityConstant \in (0,1)$, $\xi$, $\Delta \in (0,\infty)$, $r \in (0,1/2]$, we let
\[
{\mathcal{H}}'_{\regularityConstant}(\xi,\Delta,r) :=\bigl\{B \in \setOfHypercubesHO: \diamSup(B) = 2r \text{ and } B \cap \muRegularSet \cap \densitySuperLevelSet{\xi} \cap \etaSuperLevelSet{\tau+\Delta}\neq \emptyset\bigr\}.
\]

\begin{lemma}\label{lemma:countingCubesElementary} We have $|{\mathcal{H}}'_{\regularityConstant}(\xi,\Delta,r)| \leq (2/r)^d\cdot (\regularityConstant \cdot \xi)^{-1}$ for every $\regularityConstant \in (0,1)$, $\xi$, $\Delta \in (0,\infty)$ and $r \in (0,1/2]$.
\end{lemma}
\begin{proof} Given $B= 2r\prod_{j \in [d]}[a_j,a_j+1] \in {\mathcal{H}}'_{\regularityConstant}(\xi,\Delta,r)$, for some $(a_j)_{j \in [d]} \in \Z^d$, we write $\phi(B) := (a_j \mod 2)_{j \in [d]}\in \{0,1\}^d$ and $\psi(B):=r\prod_{j \in [d]}[2a_j-1,2a_j+3]$. Note that if $\phi(B_0)=\phi(B_1)$ for distinct $B_0$, $B_1 \in {\mathcal{H}}'_{\regularityConstant}(\xi,\Delta,r)$ then $\mu\bigl(\psi(B_0) \cap \psi(B_1)\bigr)=0$, since $\mu$ is absolutely continuous with respect to Lebesgue measure on $\R^d$. Moreover, by Lemma~\ref{lemma:leastSingularValuePopulation} we have $\mu\bigl(\psi(B)\bigr)\geq \regularityConstant \cdot \xi \cdot r^d$, so
\begin{align*}
|{\mathcal{H}}'_{\regularityConstant}(\xi,\Delta,r)|\cdot \regularityConstant \xi r^d \leq \sum_{B \in {\mathcal{H}}'_{\regularityConstant}(\xi,\Delta,r)}\mu\bigl(\psi(B)\bigr) &= \sum_{z \in \{0,1\}^d}~\sum_{B \in {\mathcal{H}}'_{\regularityConstant}(\xi,\Delta,r)\cap\phi^{-1}\{z\}}\mu\bigl(\psi(B)\bigr) \leq 2^d,
\end{align*}
as required.
\end{proof}

\begin{proof}[Proof of Proposition~\ref{thm:powerBoundHO}] Let 
\begin{align*}
\rho &:= \approximableDensityExponent(2\holderExponent+d)+\holderExponent\approximableMarginExponent, \quad \theta := \frac{\holderConstant^{d/\holderExponent}}{n}\log_+\biggl(\frac{8n^{\frac{\holderExponent}{\holderExponent \wedge 1}}|\mathcal{V}(\holderExponent)|\log_2 n}{\alpha \wedge \delta}\biggr), \\
\xi &:= \thetaNDelta^{\holderExponent\approximableMarginExponent/\rho}, \quad r_* := 2^{-\big\lfloor \frac{1}{\holderExponent \wedge 1}\{ \frac{\holderExponent\approximableDensityExponent}{\rho}\log_2(1/\theta)+\log_2 \holderConstant\}\big\rfloor},
\end{align*}
and, recalling the definition of ${C}_{\mathrm{pv}}:=2^{2d+5} \cdot \bigl(\regularityConstant \cdot \sqrt{c_{\min}^0}\bigr)^{-1}$ from Lemma~\ref{lemma:powerForSingleHOPValue}, define
\begin{align*}
A_0:=\begin{cases} 2 \vee C_{\mathrm{pv}} |\mathcal{V}(\holderExponent)|^{1/2} \vee 2^{\frac{\holderExponent-2}{\holderExponent-1}}(3/\regularityConstant)^{\frac{\holderExponent}{\holderExponent-1}}   &\text{ if } \holderExponent>1\\
2^\holderExponent \vee C_{\mathrm{pv}} |\mathcal{V}(\holderExponent)|^{1/2}   &\text{ if } \holderExponent \leq 1.
\end{cases}
\end{align*}
Then for $\thetaNDelta^\holderExponent  \leq A_0^{-\frac{\rho}{ \approximableDensityExponent}}$ we have 
\[
\biggl( \frac{C_{\mathrm{pv}}^2 \cdot |\mathcal{V}(\holderExponent)|}{ \xi }\cdot \thetaNDelta \holderConstant^{-d/\holderExponent}\biggr)^{\frac{\holderExponent}{d(\holderExponent \wedge 1)}} \leq r_* \leq \frac{1}{2},
\]
and $r_* \leq  \{2(\regularityConstant/3)^\holderExponent\}^{\frac{1}{{\holderExponent- 1}}}$ if $\holderExponent>1$.  Now let 
\[
\Delta := C_{\mathrm{pv}} \Biggl( \holderConstant r_*^{\holderExponent\wedge 1}+\sqrt{\frac{ \log\bigl(2/(\alpha \wedge \delta)\bigr)}{ \xi \cdot n \cdot r_*^{d  (\holderExponent\wedge 1)/\holderExponent}}}\Biggr),\]
so that $\Delta \leq 3 \cdot C_{\mathrm{pv}}\thetaNDelta^{\frac{\holderExponent\approximableDensityExponent}{\rho}}$ when $\thetaNDelta \leq A_0^{-\frac{\rho}{\holderExponent  \approximableDensityExponent}}$.  By Lemma~\ref{lemma:countingCubesElementary}, we have $|{\mathcal{H}}'_{\regularityConstant}(\xi,\Delta,r_*)| \leq (2/r_*)^d\cdot(\regularityConstant \cdot \xi)^{-1} \leq (\holderConstant^d\theta^{-\holderExponent})^\frac{1}{\holderExponent \wedge 1} \leq n^{\frac{\holderExponent}{\holderExponent \wedge 1}}$ when  $\thetaNDelta \leq A_0^{-\frac{\rho}{\holderExponent  \approximableDensityExponent}}$.  Hence we may apply a union bound and Lemma~\ref{lemma:powerForSingleHOPValue} with $\delta/(2n^{\frac{\holderExponent}{\holderExponent \wedge 1}})$ in place of $\delta$ and $\alpha /(n \log_2 n)$ in place of $\alpha$ to deduce that whenever  $\theta \leq A_0^{-\frac{\rho}{\holderExponent \approximableDensityExponent}}$, we have
\[
\Prob\biggl(\bigcup_{B \in {\mathcal{H}}'_{\regularityConstant}(\xi,\Delta,r_*)} \biggl\{\pValueNHO(B) > \frac{\alpha}{|\setOfHypercubesHO(\sampleX)|}\biggr\}\biggr) \leq \sum_{B \in {\mathcal{H}}'_{\regularityConstant}(\xi,\Delta,r_*)}\Prob \biggl(\pValueNHO(B) > \frac{\alpha}{n \log_2 n}\biggr) \leq \frac{\delta}{2}.
\]
Hence, whenever $\theta \leq A_0^{-\frac{\rho}{\holderExponent \approximableDensityExponent}}$, we have 
\[
\muRegularSet \cap \densitySuperLevelSet{\xi} \cap \etaSuperLevelSet{\tau+\Delta} \subseteq \bigcup_{B \in {\mathcal{H}}'_{\regularityConstant}(\xi,\Delta,r_*)}B \subseteq \bigcup_{\ell \in [\ell_{\alpha}]}B_{(\ell)},
\]
with probability at least $1-\delta/2$. Thus, with probability at least $1 - \delta/2$,
\begin{align}
M_\tau-\sup\biggl\{ \mu(A):A\in &\mathcal{A}\cap \powerSet\biggl(\bigcup_{\ell \in [\ell_{\alpha}]}B_{(\ell)}\biggr) \biggr\} \nonumber \\
&\leq M_\tau-\sup\Bigl\{ \mu(A):A\in \mathcal{A}\cap \powerSet\bigl(\muRegularSet \cap \densitySuperLevelSet{\xi} \cap \etaSuperLevelSet{\tau+\Delta}\bigr) \Bigr\} \nonumber \\
&\leq \approximableSetsConstant \cdot (\xi^{\kappa} + \Delta^{\gamma}) \mathbbm{1}_{\{\theta \leq A_0^{-\frac{\rho}{\holderExponent \approximableDensityExponent}} \}} + \mathbbm{1}_{\{\theta > A_0^{-\frac{\rho}{\holderExponent \approximableDensityExponent}}\}} \nonumber \\ 
& \leq \approximableSetsConstant \cdot \bigl\{1 + (3C_{\mathrm{pv}})^\gamma + A_0^{\gamma} \bigr\} \cdot \theta^{\beta\kappa\gamma/\rho}.
\end{align}
Finally, since $\hat{A}_{\mathrm{OSS}}^+$ is chosen from $ \mathcal{A}\cap \powerSet\bigl(\bigcup_{\ell \in [\ell_{\alpha}]}B_{(\ell)}\bigr)$ with maximal empirical measure, it follows from Lemma \ref{lemma:vapnikChervonenkisConcentration} as in the proof of Proposition~\ref{prop:generalPowerBound} that with probability at least $1-\delta$,
\begin{align*}
M_\tau-\mu(\hat{A}_{\mathrm{OSS}}^+) \leq \approximableSetsConstant &\bigl\{1 + (3C_{\mathrm{pv}})^\gamma + A_0^{\gamma} \bigr\} \cdot \theta^{\holderExponent\kappa\gamma/\rho} \\
&+ 2C_{\mathrm{VC}} \sqrt{\frac{\vcDim(\mathcal{A})}{n}}+\sqrt{\frac{2\log(2/\delta)}{n}}. 
\end{align*}
The second part of Proposition~\ref{thm:powerBoundHO} follows integrating the tail bound and applying Proposition~\ref{thm:typeIControlHigherOrderSmoothness}, as at the end of the proof of Theorem~\ref{thm:powerBound}.
\end{proof}

\section{Proofs of the lower bounds in Theorems~\ref{thm:minimaxRate} and~\ref{Thm:minimaxRateHOS}}

Recall the construction of the probability distributions $\{P_{L,r,w,s,\theta}^\ell:\ell \in [L]\}$ on $\R^d \times \{0,1\}$ from Section~\ref{Sec:LowerBound}, with corresponding regression functions $\{\eta_{L,r,w,s,\theta}^\ell:\ell \in [L]\}$ and common marginal distribution $\mu_{L,r,w}$ on $\R^d$.  Recall also the definition of $\muRegularSetArg{\cdot}$ from~\eqref{eq:defRegularSet}.  Our initial goal is to prove that $\{P_{L,r,w,s,\theta}^\ell:\ell \in [L]\}$ is a subset of $\classOfHolderDistributions$ (see Lemma~\ref{lemma:LB1HolderRegFunctions}) and $\classOfWellApproximableSets \cap \classOfWellApproximableSetsHO$ (see Lemma~\ref{lemma:approximableConditionForLB1}) for suitable $L$, $r$, $w$, $s$ and~$\theta$.  The first of these lemmas will rely on several auxiliary results, 

Given two multi-indices $\nu = (\nu_1,\ldots,\nu_d)^\top, \nu' = (\nu_1',\ldots,\nu_d')^\top \in \N_0^d$, we write $\nu \prec \nu'$ if either $\|\nu\|_1 < \|\nu'\|_1$ or both $\|\nu\|_1 = \|\nu'\|_1$ and there exists $j \in \{0,1,\ldots,d-1\}$ such that $\nu_1=\nu'_1,\ldots,\nu_j = \nu'_j$ and $\nu_{j+1} < \nu'_{j+1}$.  Now, given $m \in \N$ and $j \in [m]$, we write
\begin{align*}
\mathcal{Q}_j&(\nu,m) \\
&:= \biggl\{\! (k_1,\ldots,k_j,\ell_1,\ldots,\ell_j) \in \N^j \times (\N_0^d)^j \! : 0 \! \prec \! \ell_1 \! \prec \! \ldots \! \prec \! \ell_j, \sum_{q=1}^j k_q = m, \sum_{q=1}^j k_q \ell_q = \nu\biggr\}.
\end{align*}
In addition, for multi-indices $\nu = (\nu_1,\ldots,\nu_d)^\top \in \N_0^d$, we let $\nu!:= \prod_{m=1}^{\|\nu\|_1}\nu_j!$.  The following lemma is a version of the Fa\`a di Bruno formula.

\begin{lemma}[Corollary~2.10 of \citet{constantine1996multivariate}]\label{lemma:faaDiBruno} Let $x \in \R^d$ and $\nu = (\nu_1,\ldots,\nu_d)^\top \in \N_0^d$.  Suppose that all partial derivative of order $\|\nu\|_1$ of $f:\R^d \rightarrow \R$ exist and are continuous in a neighbourhood of $x$, and that $g:\R \rightarrow \R$ is $\|\nu\|_1$-times continuously differentiable in a neighbourhood of $f(x)$. Then 
\begin{align*}
\partial^\nu_x(g\circ f)=\nu!\cdot \sum_{m=1}^{\|\nu\|_1}g^{(m)}\bigl(f(x)\bigr) \sum_{j =1}^{\|\nu\|_1} \sum_{\substack{(k_1,\ldots,k_j,\ell_1,\ldots,\ell_j) \\ \in \mathcal{Q}_j(\nu,m)}}\prod_{q=1}^j\frac{\{\partial^{\ell_q}_x(f)\}^{k_q}}{k_q! \cdot (\ell_q!)^{k_q}}.
\end{align*}
\end{lemma}

\begin{lemma}\label{lemma:smoothnessNorm} Given $\nu = (\nu_1,\ldots,\nu_d)^\top \in \N_0^d$, we have $\sup_{\|x\|_2 \geq 1} |\partial^\nu_x (\|\cdot \|_2)| < \infty$.
\end{lemma}
\begin{proof}
For $t \in [d]$, write $e_t = (0,\ldots,0,1,0,\ldots,0)^\top \in \mathbb{R}^d$ for the $t^{\mathrm{th}}$ standard basis vector in~$\mathbb{R}^d$.  By Lemma~\ref{lemma:faaDiBruno} with $f(x)=\|x\|_2^2$ and $g(z) = \sqrt{z}$, we have for $x=(x_1,\ldots,x_d)^\top \in \R^d$ that
\begin{align*}
&\partial^\nu_x (\|\cdot \|_2) \\
&= \nu!\sum_{m=1}^{\|\nu\|_1} \frac{(-1)^{m+1} (2m\!-\!3)!!}{2^{m}\|x\|_2^{2m-1}} \sum_{j =1}^{\|\nu\|_1}\! \sum_{\substack{(k_1,\ldots,k_j,\ell_1,\ldots,\ell_j) \\ \in \mathcal{Q}_j(\nu,m)}}\!\prod_{q=1}^j\frac{2^{k_q}\bigl\{\sum_{t=1}^d (x_t \mathbbm{1}_{\{\ell_q = e_t\}}\!+\!\mathbbm{1}_{\{\ell_q = 2e_t\}})\bigr\}^{k_q}}{k_q! \cdot(\ell_q!)^{k_q}}.
\end{align*}
It follows that for all ${x \in \R^d}$ with $\|x\|_2 \geq 1$ we have
\begin{align*}
|\partial^\nu_x (\|\cdot \|_2)|& \leq \nu! \sum_{m=1}^{\|\nu\|_1}  \frac{(2m\!-\!3)!!}{\|x\|_2^{2m-1}}  \sum_{j =1}^{\|\nu\|_1} \sum_{\substack{(k_1,\ldots,k_j,\ell_1,\ldots,\ell_j) \\ \in \mathcal{Q}_j(\nu,m)}}\prod_{q=1}^j\frac{\|x\|_2^{k_q}}{k_q! \cdot (\ell_q!)^{k_q}}\\
&\leq \nu! \sum_{m=1}^{\|\nu\|_1} (2m\!-\!3)!!\cdot  \sum_{j =1}^{\|\nu\|_1} \sum_{\substack{(k_1,\ldots,k_j,\ell_1,\ldots,\ell_j) \\ \in \mathcal{Q}_j(\nu,m)}}\prod_{q=1}^j\frac{1}{k_q! \cdot (\ell_q!)^{k_q}}<\infty,
\end{align*}
as required.
\end{proof}

\begin{lemma}\label{lemma:smoothnesFunctionsOfNorm} For each $m,d \in \N$, there exists $C_{m,d}>0$, depending only on $m$ and $d$, such that for any infinitely differentiable function $g:[0,\infty) \rightarrow [0,\infty)$ with $g'(z)=0$ for all $z \in [0,1]$, and any $\nu=(\nu_1,\ldots,\nu_d)^\top \in \N_0^d$ with $\|\nu\|_1=m$, we have
\begin{align*}
\bigl|\partial_x^\nu (g \circ \|\cdot \|_2)\bigr| \leq C_{m,d} \cdot \max_{k \in [m]} \sup_{z \in [0,\infty)} |g^{(k)}(z)|
\end{align*}
for all $x \in \R^d$. 
\end{lemma}
\begin{proof} The lemma follows from combining Lemmas \ref{lemma:faaDiBruno} and \ref{lemma:smoothnessNorm}, and by considering the cases $\|x\|_2 <1$ and $\|x\|_2 \geq 1$ separately.
\end{proof}

\begin{lemma}\label{lemma:LB1DerivsBound} Let $L,d \in \N$, $r \in (0,\infty)$, $s \in (0,1\wedge(r/2)]$, $w \in \bigl(0,(2r)^{-d} \wedge 1\bigr)$ and $\theta \in (0,\epsilon_0/2]$. Then for $\ell \in [L]$,   $\nu=(\nu_1,\ldots,\nu_d)^\top \in \N_0^d$ with $\|\nu\|_1=m$ and $x \in \R^d$, we have
\begin{align}\label{eq:LB1DerivsBoundClaim1}
\bigl|\partial_x^\nu (\eta_{L,r,w,s,\theta}^{\ell})\bigr| &\leq 2A_{m} C_{m,d} \cdot \frac{\theta}{s^{m}},
\end{align}
where $A_m$ is taken from~\eqref{Eq:Am} and $C_{m,d}$ is taken from Lemma~\ref{lemma:smoothnesFunctionsOfNorm}.
Hence, given any $\xi \in [0,1]$ and $x$, $x' \in \R^d$, we have
\begin{align}\label{eq:LB1DerivsBoundClaim2}
\bigl|\partial_x^\nu (\eta_{L,r,w,s,\theta}^{\ell})-\partial_{x'}^\nu (\eta_{L,r,w,s,\theta}^{\ell})\bigr| &\leq 2A_{m+1} (2C_{m,d}\vee d C_{m+1,d})\cdot \frac{\theta}{s^{m+\xi}}\cdot \supNorm{x-x'}^\xi.
\end{align}
\end{lemma}

\begin{proof} To prove~\eqref{eq:LB1DerivsBoundClaim1}, we construct an open cover of $\R^d$ by $\{U_1,\ldots,U_{L+1}\}$ where  $U_{\ell'}:=\openMetricBallSupNorm{{z}_{\ell'}}{r_\sharp(w)}$ for $\ell' \in [L]$ and $U_{L+1}:= \R^d \setminus \bigcup_{\ell' \in [L]}\closedMetricBallSupNorm{{z}_{\ell'}}{2d^{1/2}r}$. First suppose that $\ell' \in [L] \backslash \{\ell\}$ and consider the function $g_0:[0,\infty)\rightarrow [0,\infty)$ defined by 
\begin{align*}
g_0(t):= \begin{cases} \tau-\theta  &\text{if } t \leq 1   \\
\tau + \theta - 2\theta \cdot h\bigl(t \!-\! 1\bigr) &\text{if } 1 < t \leq 2 \\
\tau + \theta &\text{if } 2 < t \leq \frac{d^{1/2}r}{s} \\
\tau - \theta + 2\theta \cdot h\bigl(\frac{s \cdot t }{d^{1/2}r}\!-\!1\bigr) &\text{if } \frac{d^{1/2}r}{s} < t < \frac{2d^{1/2}r}{s}\\
\tau-\theta &\text{otherwise.}
\end{cases}
\end{align*}
By Lemma~\ref{lemma:smoothnesFunctionsOfNorm}, together with $s \leq r/2 \leq d^{1/2}r$, we have
\[\sup_{x \in \R^d}\bigl|\partial_x^\nu (g_0 \circ \|\cdot \|_2)\bigr| \leq C_{m,d} \cdot \max_{k \in [m]} \sup_{z \in [0,\infty)} |g_0^{(k)}(z)| \leq 2A_mC_{m,d} \theta.
\]
Moreover, for all $x \in U_{\ell'}$ we have $\eta_{L,r,w,s,\theta}^{\ell}(x)= g_0 (\|s^{-1}\cdot (x-{z}_{\ell'}) \|_2)$. Hence, for all ${x \in U_{\ell'}}$ we have $\bigl|\partial_x^\nu (\eta_{L,r,w,s,\theta}^{\ell})\bigr| \leq 2A_mC_{m,d} {\theta} s^{-m}$ since $\|\nu\|_1 =m$. Next, consider the open set $U_\ell$ and define a function $g_1:[0,\infty)\rightarrow [0,\infty)$ by 
\begin{align*}
g_1(t)&:= \begin{cases} \tau+\theta  &\text{if } t \leq 1  \\
\tau - \theta + 2\theta \cdot h\bigl(t\!-\!1\bigr) &\text{if } 1 < t < 2\\
\tau-\theta &\text{otherwise.}
\end{cases}
\end{align*}
By applying Lemma \ref{lemma:smoothnesFunctionsOfNorm} again, we have $\bigl|\partial_x^\nu (g_1 \circ \|\cdot \|_2)\bigr|  \leq 2A_mC_{m,d} {\theta}$ for all $x \in \R^d$. Moreover, for $x \in U_{\ell}$ we have $\eta_{L,r,w,s,\theta}^{\ell}(x)= g_1 (\|(d^{1/2}r)^{-1}\cdot (x-{z}_{\ell}) \|_2)$. Hence, for all ${x \in U_{\ell}}$, we have $\bigl|\partial_x^\nu (\eta_{L,r,w,s,\theta}^{\ell})\bigr| \leq 2A_mC_{m,d} {\theta} (d^{1/2}r)^{-m}\leq 2A_mC_{m,d} {\theta} s^{-m}$. Finally we note that $\eta_{L,r,w,s,\theta}^{\ell}\big|_{U_{L+1}}\equiv\tau-\theta$, so $\sup_{x \in U_{L+1}}\bigl|\partial_x^\nu (\eta_{L,r,w,s,\theta}^{\ell})\bigr| =0 \leq 2A_mC_{m,d} {\theta} s^{-m}$.  The claim \eqref{eq:LB1DerivsBoundClaim1} follows.

To prove \eqref{eq:LB1DerivsBoundClaim2}, we first consider the case where $\supNorm{x-x'} \leq s$, in which case, we may apply the mean value theorem combined with \eqref{eq:LB1DerivsBoundClaim1} and H\"older's inequality to obtain
\begin{align*}
\bigl|\partial_x^\nu (\eta_{L,r,w,s,\theta})-\partial_{x'}^\nu (\eta_{L,r,w,s,\theta})\bigr| 
& \leq d A_{m+1} C_{m+1,d} \cdot \frac{\theta}{s^{m+1}}\cdot \|{x-x'}\|_\infty\\
& \leq d A_{m+1} C_{m+1,d} \cdot \frac{\theta}{s^{m+\xi}}\cdot \|{x-x'}\|_\infty^{\xi}.
\end{align*}
Moreover, when $\supNorm{x-x'}> s$, \eqref{eq:LB1DerivsBoundClaim2} follows immediately from \eqref{eq:LB1DerivsBoundClaim1} and the triangle inequality.
\end{proof}
\begin{lemma}\label{lemma:TaylorFromHolderDerivs} Take $\holderExponent>0$, $C_f>0$ and let $f:\R^d \rightarrow \R$ be a $\lceil \holderExponent \rceil$-times differentiable function such that for all  $\nu=(\nu_1,\ldots,\nu_d)^\top \in \N_0^d$ with $\|\nu\|_1= \lceil \holderExponent \rceil-1 =: m$, and $x$, $x' \in \R^d$, we have
\begin{align*}
\bigl|\partial_{x'}^\nu (f)-\partial_{x}^\nu (f)\bigr| \leq C_f\cdot \supNorm{x'-x}^{\holderExponent-m}.
\end{align*}
Then for all $x, x'\in \R^d$ we have
\begin{align*}
\bigl|f(x') - \taylorSeries_{x}^{\holderExponent}[f](x')\bigr|\leq  C_f \cdot \binom{m+d-1}{d-1} \cdot \supNorm{x'-x}^{\holderExponent}.
\end{align*}
\end{lemma}
\begin{proof} By Taylor's theorem, there exists $t \in (0,1)$ such that 
\begin{align*}
f(x') = \sum_{\nu \in \N_0^d:\|\nu\|_1<m}\frac{(x'-x)^\nu}{\nu!}\cdot \partial^\nu_x(f)+\sum_{\nu \in \N_0^d:\|\nu\|_1=m}\frac{(x'-x)^\nu}{\nu!}\cdot \partial^\nu_{x+t\cdot (x'-x)}(f).
\end{align*}
Hence, 
\begin{align*}
\bigl|f(x') - \taylorSeries_{x}^{\holderExponent}[f](x')\bigr|&=\biggl| \sum_{\nu \in \N_0^d:\|\nu\|_1=m}\frac{(x'-x)^\nu}{\nu!}\cdot \bigl(\partial^\nu_{x+t\cdot(x'-x)}(f)-\partial^\nu_{x}(f)\bigr)\biggr| \\
&\leq  C_f \cdot \binom{m+d-1}{d-1} \cdot \supNorm{x'-x}^{\holderExponent},
\end{align*}
as required.
\end{proof}
\begin{lemma}\label{lemma:LB1HolderRegFunctions} Let $\holderExponent >0$, $\holderConstant >1$, $L, d \in \N$, $r \in (0,\infty)$, $s \in (0,1\wedge(r/2) ]$, $w \in \bigl(0,(2r)^{-d} \wedge 1\bigr)$ and $\theta \in (0,\epsilon_0/2]$. There exists $c^{\flat}_{\holderExponent,d} >0$, depending only on $\holderExponent$ and $d$, such that whenever $\theta \leq c^{\flat}_{\holderExponent,d}\cdot \holderConstant \cdot s^{\holderExponent}$, we have that for each $\ell \in [L]$, the function $\eta_{L,r,w,s,\theta}^{\ell}$ is $(\holderExponent,\holderConstant)$-H\"{o}lder on~$\R^d$; i.e.~$\probDistribution_{L,r,w,s,\theta}^{\ell} \in \classOfHolderDistributions$.
\end{lemma}
\begin{proof} By taking 
\[
c^{\flat}_{\holderExponent,d}:=\min_{q \in \N_0: q \leq \lceil \holderExponent\rceil -1}\biggl\{2A_{q+1} (2C_{q,d}\vee d C_{q+1,d})\cdot \binom{q+d-1}{d-1} \biggr\}^{-1},
\]
the result follows from Lemmas~\ref{lemma:LB1DerivsBound} and~\ref{lemma:TaylorFromHolderDerivs}.
\end{proof}
Lemma~\ref{lemma:approximableConditionForLB1} also requires one auxiliary lemma.
\begin{lemma}
\label{lemma:LB1densityLB} Given $L,d \in \N$, $r > 0$ and $w \in \bigl(0,(2r)^{-d} \wedge 1\bigr)$, we have $\lowerDensity_{\mu_{L,r,w},d}(x)\geq \frac{w}{L \cdot (4d^{1/2})^d}$ for all $x \in \bigcup_{\ell \in [L]} K^0_r(\ell)$.  Moreover, $\bigcup_{\ell \in [L]} K^0_r(\ell) \subseteq \mathcal{R}_{\regularityConstant}(\mu_{L,r,w})$ for every $\regularityConstant\leq (4d^{1/2})^{-d}$.
\end{lemma}
\begin{proof} Let $\ell \in [L]$, let $x \in K^0_r(\ell)=\closedMetricBallSupNorm{z_{\ell}}{r}$ and let $\tilde{r} \in (0,1)$.  If $\tilde{r} \in (0,2r]$, then $\closedMetricBallSupNorm{x}{\tilde{r}}\cap K^0_r(\ell)$ contains a hyper-cube of radius $\tilde{r}/2$, so $\mu_{L,r,w}\bigl(\closedMetricBallSupNorm{x}{\tilde{r}}\bigr)\geq w \cdot L^{-1} \cdot \tilde{r}^d$. Consequently, when $\tilde{r} \in \bigl(0,8d^{1/2}r\bigr]$, we have 
\[
\mu_{L,r,w}\bigl(\closedMetricBallSupNorm{x}{\tilde{r}}\bigr)\geq \mu_{L,r,w}\bigl(\closedMetricBallSupNorm{x}{\tilde{r}/(4d^{1/2})}\bigr) \geq  \frac{w}{L \cdot (4d^{1/2})^d} \cdot \tilde{r}^d. 
\]
Note also that since $x \in K^0_r(\ell) \subseteq \closedMetricBallSupNorm{z_{\ell}}{2d^{1/2}r}$ there exists $\sigma_x \in \{-1,1\}^d$ and $\tilde{x} =z_\ell+\sigma_x \cdot 2d^{1/2}r \in \R^d$ with $\supNorm{\tilde{x}-x}\leq 2d^{1/2}r$. Hence, if $\tilde{r} \in (4d^{1/2}r, r_\sharp(w)]$, then $\closedMetricBallSupNorm{{z}_\ell+\sigma_x \cdot ( 2d^{1/2}r+\tilde{r}/4)}{\tilde{r}/4} \subseteq \closedMetricBallSupNorm{x}{\tilde{r}} \cap K^1_r(\ell)$. Thus, we have $\mu_{L,r,w}\bigl(\closedMetricBallSupNorm{x}{\tilde{r}}\bigr)\geq (w/L) \cdot (\tilde{r}/2)^d$ for $\tilde{r} \in (4d^{1/2}r, r_\sharp(w)]$, and consequently, $\mu_{L,r,w}\bigl(\closedMetricBallSupNorm{x}{\tilde{r}}\bigr)\geq (w/L) \cdot (\tilde{r}/4)^d$ for $\tilde{r} \in (8d^{1/2}r,2r_\sharp(w)]$. Finally, if $\tilde{r} \in (2r_\sharp(w),1)$, then $K^0_r(\ell)\cup K_r^1(\ell) \subseteq \closedMetricBallSupNorm{{z}_\ell}{r_\sharp(w)} \subseteq \closedMetricBallSupNorm{x}{\tilde{r}}$ so $\mu_{L,r,w}\bigl(\closedMetricBallSupNorm{x}{\tilde{r}}\bigr)\geq 1/L > (w/L) \cdot \tilde{r}^d$.  The first conclusion of the lemma therefore follows.  The second part then follows from the fact that the Lebesgue density of $\mu_{L,r,w}$ is at most $w/L$ on $\R^d$.
\end{proof}

\begin{lemma}\label{lemma:approximableConditionForLB1} Let $\holderExponent >0$, $\holderConstant >1$, $L, d \in \N$, $r \in (0,\infty)$, $s \in (0,1\wedge(r/2)]$, $w \in \bigl(0,(2r)^{-d} \wedge 1\bigr)$ and $\theta \in (0,\epsilon_0/2]$, $\ell \in [L]$, $\regularityConstant\leq (4d^{1/2})^{-d}$ and let $\eta = \eta_{L,r,w,s,\theta}^{\ell}$, $\mu=\mu_{L,r,w}$ and $\probDistribution=\probDistribution_{L,r,w,s,\theta}^{\ell}$. Suppose also that $\mathcal{A}_{\mathrm{hpr}}  \subseteq \mathcal{A} \subseteq \mathcal{A}_{\mathrm{conv}}$. Given any $A \in \mathcal{A} \cap \powerSet\bigl(\etaSuperLevelSet{\tau}\bigr)$ and $\ell' \in [L]$ with $A\cap K^0_r(\ell')\neq \emptyset$ and ${z}_{\ell'} \notin A$ for some $\ell' \in [L]$, we have $\mu(A) \leq (w/L)\cdot (2r)^d/2$. In particular, $\mu(A) \leq (w/L)\cdot (2r)^d/2$ whenever $A\cap K^0_r(\ell')\neq \emptyset$ for some $A \in \mathcal{A} \cap \powerSet\bigl(\etaSuperLevelSet{\tau}\bigr)$ and $\ell' \in [L]\setminus \{\ell\}$. Moreover, $M_\tau(\probDistribution,\mathcal{A})= \mu\bigl(K^0_r(\ell)\bigr)= (w/L) \cdot (2r)^{d}$.  Finally, if $(w/L) \cdot (2r)^d \leq  \approximableSetsConstant \cdot \min\big\{ (w/\{(4d^{1/2})^d\cdot L\})^{\approximableDensityExponent}, \theta^{\approximableMarginExponent}\big\}$, then $\probDistribution \in \classOfWellApproximableSets\cap \classOfWellApproximableSetsHO$.
\end{lemma}

\begin{proof} First take $A \in \mathcal{A} \cap \powerSet\bigl(\etaSuperLevelSet{\tau}\bigr)$ and $\ell' \in [L]$ with $A\cap  K^0_r(\ell')\neq \emptyset$ and ${z}_{\ell'} \notin A$.  Since $\eta(x)=\tau-\theta$ for all $x \in K^1_r(\ell')$, we must have $A \cap  \bigl(K^1_r(\ell')\cup \{{z}_{\ell'}\}\bigr) = \emptyset$. Moreover, $A$ is convex with $A\cap K^0_r(\ell')\neq \emptyset$, and it follows that $A \subseteq \{x \in \R^d~:~\supNorm{x-{z}_{\ell'}}<2d^{1/2}r\}$. Thus $A \cap \support(\mu)=A \cap \{x \in \R^d~:~\supNorm{x-{z}_{\ell'}}<2d^{1/2}r\} \cap \support(\mu)=A \cap K^0_r(\ell')$ is the intersection of two axis-aligned hyper-rectangles, so is itself an axis-aligned hyper-rectangle. Since $A \cap \support(\mu) \subseteq K^0_r(\ell')\setminus\{{z}_{\ell'}\}=\closedMetricBallSupNorm{{z}_{\ell'}}{r}\setminus\{{z}_{\ell'}\}$, we deduce that $\mu(A) \leq (w/L)\cdot (2r)^d/2$.  In particular, if $\ell' \in [L] \setminus \{\ell\}$, then $\eta({z}_{\ell'})=\tau-\theta$, so ${z}_{\ell'} \notin \etaSuperLevelSet{\tau}$ and the conclusion $\mu(A) \leq (w/L)\cdot (2r)^d/2$ holds.

For the next part, note that $K^0_r(\ell)=\closedMetricBallSupNorm{{z}_\ell}{r} \in \mathcal{A}_{\mathrm{hpr}}  \cap \powerSet\bigl(\etaSuperLevelSet{\tau}\bigr)\subseteq \mathcal{A} \cap \powerSet\bigl(\etaSuperLevelSet{\tau}\bigr)$ since $\eta(x) =\tau+\theta$ for all $x \in K^0_r(\ell)$. Hence, $M_\tau(\probDistribution,\mathcal{A}) \geq \mu\big(K^0_r(\ell)\big)= (w/L) \cdot (2r)^{d}$.  On the other hand, given $A \in \mathcal{A} \cap \powerSet\bigl(\etaSuperLevelSet{\tau}\bigr)$, we have either $A \cap \support(\mu) \subseteq K^0_r(\ell)$, in which case $\mu(A) \leq \mu\big(K^0_r(\ell)\big)= (w/L) \cdot (2r)^{d}$, or $A \cap \support(\mu) \cap K^0_r(\ell') \neq \emptyset$ for some $\ell' \in [L]\setminus \{\ell\}$, since $\support(\mu) \cap \etaSuperLevelSet{\tau} \subseteq \bigcup_{\ell \in [L]}K^0_r(\ell)$, in which case $\mu(A) \leq (w/L)\cdot (2r)^d/2$. Hence $M_\tau(\probDistribution,\mathcal{A}) = (w/L) \cdot (2r)^{d}$.

For the final part, assume that $(w/L) \cdot (2r)^d \leq  \approximableSetsConstant \cdot \min\big\{ \{w/(4^{d}d^{1/2}\cdot L)\}^{\approximableDensityExponent}, \theta^{\approximableMarginExponent}\big\}$, and fix $(\xi,\Delta) \in (0,\infty)^2$.  We consider two cases: first suppose that $\xi \leq {w}/\{L \cdot (4d^{1/2})^d\}$ and $\Delta \leq \theta$.  By Lemma~\ref{lemma:LB1densityLB}, we have $K^0_r(\ell) \subseteq  \mathcal{X}_\xi({\lowerDensity_{\mu,d}})$. Moreover, it follows from the construction of $\eta$ that $K^0_r(\ell) \subseteq \mathcal{X}_{\tau+\theta}(\eta) \subseteq \mathcal{X}_{\tau+\Delta}(\eta)$. Thus, with $A_{\xi,\Delta} = K^0_r(\ell) \in \mathcal{A}\cap \powerSet\bigl(\mathcal{X}_{\xi}(\lowerDensity_{\mu,d}) \cap \mathcal{X}_{\tau+\Delta}(\eta)\bigr)$, we have
\begin{align*}
 \mu(A_{\xi,\Delta}) = \frac{w}{L} \cdot (2r)^d \geq \frac{w}{L} \cdot (2r)^d- \approximableSetsConstant \cdot (\xi^{\approximableDensityExponent}+\Delta^{\approximableMarginExponent}) = M_\tau - \approximableSetsConstant \cdot (\xi^{\approximableDensityExponent}+\Delta^{\approximableMarginExponent}).
\end{align*}
On the other hand, if $\xi > {w}/\{L \cdot (4d^{1/2})^d\}$ or $\Delta > \theta$, then with $A_{\xi,\Delta}=\emptyset \in \mathcal{A}\cap \powerSet\bigl(\mathcal{X}_{\xi}(\lowerDensity_{\mu,d}) \cap \mathcal{X}_{\tau+\Delta}(\eta)\bigr)$, we have
\begin{align*}
\mu(A_{\xi,\Delta}) = 0 \geq \frac{w}{L} \cdot (2r)^d  - \approximableSetsConstant \cdot \min\biggl\{ \biggl( \frac{w}{(4d^{1/2})^{d}\cdot L}\biggr)^{\approximableDensityExponent}\!\!, \theta^{\approximableMarginExponent}\biggr\} \geq M_\tau - \approximableSetsConstant \cdot (\xi^{\approximableDensityExponent}+\Delta^{\approximableMarginExponent}).
\end{align*}
We conclude that $\probDistribution \in \classOfWellApproximableSets$. To prove $\probDistribution \in \classOfWellApproximableSetsHO$, we proceed similarly using the facts that $K^0_r(\ell) \subseteq \mathcal{R}_{\regularityConstant}(\mu)$ by Lemma~\ref{lemma:LB1densityLB} and that $\mu$ has Lebesgue density at least $w/L$ on $K^0_r(\ell)$.
\end{proof}
Lemma~\ref{lemma:chiSqrDistanceFirstPartNegative} below bounds the $\chi^2$-divergence between pairs of  distributions in our class $\{P_{L,r,w,s,\theta}^\ell:\ell \in [L]\}$.
\begin{lemma}\label{lemma:chiSqrDistanceFirstPartNegative} Suppose that  $\epsilon_0 \in (0,1/2)$, $\tau \in [\epsilon_0,1-\epsilon_0]$, $L,d \in \N$, $r \in (0,\infty)$, $s \in (0,1 \wedge (r/2)]$, $w \in \bigl(0,(2r)^{-d} \wedge 1\bigr)$ and $\theta \in (0,\epsilon_0/2]$. Then 
\[
\chi^2\big(\probDistribution^\ell_{L,r,w,s,\theta} ,\probDistribution^{\ell'}_{L,r,w,s,\theta}  \big)\leq  \frac{2^{5+2d} \cdot w  \cdot \theta^2 \cdot s^d}{\epsilon_0 \cdot L}
\]
for all $\ell,\ell' \in [L]$. 
\end{lemma}
\begin{proof} Let $Q_{L,r,w}:=\mu_{L,r,w} \times m_{\mathrm{ct}}$ where $m_{\mathrm{ct}}$ denotes the counting measure on $\{0,1\}$. Note that $\probDistribution^{\ell}_{L,r,w,s,\theta}$ is absolutely continuous with respect to $Q_{L,r,w}$, for all $\ell \in [L]$. Given $\ell \in [L]$, define $p_\ell:\R^d \times \{0,1\} \rightarrow \R$ by
\begin{align*}
p_{\ell}(x,y):&=\frac{d\probDistribution^{\ell}_{L,r,w,s,\theta}}{dQ_{L,r,w}}(x,y)=  (1-y) \cdot \bigl(1-\regressionFunction_{L,r,\theta}^{\ell}(x)\bigr)+y\cdot \regressionFunction_{L,r,\theta}^{\ell}(x).
\end{align*}
Take $\ell,~\ell' \in [L]$ with $\ell\neq \ell'$ and observe that $\eta^{\ell}_{L,r,w,s,\theta}(x)= \eta_{L,r,w,s,\theta}^{\ell'}(x)$ for all $x \in J_{L,r,w} \setminus \big(\closedMetricBallSupNorm{{z}_\ell}{2s}\cup \closedMetricBallSupNorm{{z}_{\ell'}}{2s}\big)$. Note also that $\mu_{L,r,w}\big(\closedMetricBallSupNorm{{z}_\ell}{2s}\cup \closedMetricBallSupNorm{{z}_{\ell'}}{2s}\big) = (2w/L) \cdot (4s)^d$; moreover, $\regressionFunction_{L,r,\theta}^{\ell}(x) \in [\tau-\theta,\tau+\theta] \subseteq [\epsilon_0/2,1-\epsilon_0/2]$ for all $x \in J_{L,r,w}$ and $\ell \in [L]$. Hence, 
\begin{align*}
\chi^2&(\probDistribution^\ell_{L,r,w,s,\theta},\probDistribution^{\ell'}_{L,r,w,s,\theta})\\ &=\int_{\R^d\times \{0,1\}} \biggl(\frac{d\probDistribution^{\ell}_{L,r,w,s,\theta}}{d \probDistribution^{\ell'}_{L,r,\theta}}-1\biggr)^2d\probDistribution^{\ell'}_{L,r,\theta}\\
&=\int_{\R^d\times \{0,1\}} \biggl(\frac{p_{\ell}(x,y)}{p_{\ell'}(x,y)}-1\biggr)^2 p_{\ell'}(x,y) \, dQ_{L,r,w}(x,y)\\
&=\int_{\R^d\times \{0,1\}} \frac{\{p_{\ell}(x,y)-p_{\ell'}(x,y)\}^2}{p_{\ell'}(x,y)} \, dQ_{L,r,w}(x,y)\\
&=\int_{\R^d} \biggl( \frac{\{\eta^{\ell}_{L,r,w,s,\theta}(x)- \eta_{L,r,w,s,\theta}^{\ell'}(x)\}^2}{1-\eta_{L,r,w,s,\theta}^{\ell'}(x)}+\frac{\{\eta^{\ell}_{L,r,w,s,\theta}(x)- \eta_{L,r,w,s,\theta}^{\ell'}(x)\}^2}{\eta_{L,r,w,s,\theta}^{\ell'}(x)} \biggr) \, d\mu_{L,r,w}(x)\\
&\leq \frac{4}{\epsilon_0} \int_{\closedMetricBallSupNorm{{z}_\ell}{2s}\cup \closedMetricBallSupNorm{{z}_{\ell'}}{2s}}  \bigl\{\eta^{\ell}_{L,r,w,s,\theta}(x)- \eta_{L,r,w,s,\theta}^{\ell'}(x)\bigr\}^2 \, d\mu_{L,r,w}(x)\\
&\leq \frac{4}{\epsilon_0} \cdot (2\theta)^2 \cdot \mu_{L,r,w}\big(\closedMetricBallSupNorm{{z}_\ell}{2s}\cup \closedMetricBallSupNorm{{z}_{\ell'}}{2s}\big) = \frac{2^{5+2d} \cdot w \cdot \theta^2 \cdot s^d}{ \epsilon_0 \cdot L} ,
\end{align*}
as required.
\end{proof}
\newcommand{\xiForProofNZeta}{\xi_{n,\zeta,\holderConstant}}
We are now in a position to state the crucial proposition for the proof of Proposition~\ref{lemma:constrainedRiskApplicationfanoApplication}.
\begin{prop}\label{lemma:completingFirstLBConstructions}  Assume that $\mathcal{A}_{\mathrm{hpr}}  \subseteq \mathcal{A} \subseteq \mathcal{A}_{\mathrm{conv}}$. Fix $(\holderExponent,\approximableMarginExponent,\approximableDensityExponent,\holderConstant,\approximableSetsConstant)\in (0,\infty)^3\times [1,\infty)^2$ with $ \holderExponent\approximableMarginExponent(\approximableDensityExponent-1) < d \approximableDensityExponent$, $\epsilon_0 \in (0,1/2)$, $\tau \in [\epsilon_0,1-\epsilon_0]$,  $\regularityConstant\leq (4d^{1/2})^{-d}$ and $\zeta >0$. There exist $C_0\ge 1$, $c_0, c_1 > 0$, depending only on $d$, $\holderExponent$, $\approximableMarginExponent$, $\approximableDensityExponent$, $\approximableSetsConstant$ and $\epsilon_0$, such that for any $n \geq C_0 \holderConstant^{d/\holderExponent} \log(1+ \zeta)$ there exists a family of $L \geq \bigl\{ c_0 \bigl({n}/{\{\holderConstant^{d/\holderExponent}\log(1+\zeta)}\}\bigr)^{\frac{\holderExponent \approximableMarginExponent(\approximableDensityExponent \wedge 1)}{\approximableDensityExponent(2\holderExponent+d)+\holderExponent\approximableMarginExponent}} \bigr\}\vee 4$ distributions 
\[
\{P_1,\ldots, P_L \}\subseteq \classOfLBDistributions
\] 
with regression functions $\eta_1,\ldots,\eta_L$ and common marginal distribution $\mu$ on $\R^d$, such that
\begin{enumerate}
\item [(a)] $\chi^2\bigl(P_\ell^{\otimes n},P_{\ell'}^{\otimes n}\bigr) \leq \zeta$ for all $\ell$, $\ell' \in [L]$;
\item[(b)] if $A \in \mathcal{A} \cap \powerSet\bigl( \mathcal{X}_{\tau}(\eta_\ell)\cap \mathcal{X}_{\tau}(\eta_{\ell'})\bigr)$ for some $\ell$,$\ell'\in [L]$ with $\ell \neq \ell'$, then
\begin{align}\label{eq:conclusionCompletingFirstLBConstructions}
M_\tau\bigl(P_\ell,\mathcal{A}\bigr)-\mu(A) \geq c_1 \cdot \biggl( \frac{\holderConstant^{d/\holderExponent} \cdot \log(1+\zeta)}{n}\biggr)^{\frac{\holderExponent \approximableMarginExponent \approximableDensityExponent}{\approximableDensityExponent(2\holderExponent+d)+\holderExponent\approximableMarginExponent}}.
\end{align}
\end{enumerate}
\end{prop}

\begin{proof} We first define some quantities for our construction.  Let $\rho:={\approximableDensityExponent(2\holderExponent+d)+\holderExponent\approximableMarginExponent}$, 
\begin{align*}
\xiForProofNZeta &:= \frac{\epsilon_0 \cdot ( c^\flat_{\holderExponent,d})^{d/\holderExponent} }{2^{5+4d}d^{d/2}} \cdot \frac{ \holderConstant^{d/\holderExponent}\cdot \log(1+\zeta)}{n}, \ \ \theta:= \xiForProofNZeta^{\approximableDensityExponent \holderExponent/\rho}, \ \ r:=\approximableSetsConstant^{1/d}\cdot ({8d^{1/2}})^{-1}\cdot{\theta^{\frac{\approximableMarginExponent(\approximableDensityExponent-1)}{d\approximableDensityExponent}}}, \\
 L &:= \big\lfloor (8d^{1/2})^{-1} \theta^{-\approximableMarginExponent/\approximableDensityExponent} \bigl\{(2r)^{-d} \wedge 1\bigr\} \big\rfloor, \ \ w := (4d^{1/2})^d \cdot L \cdot  \theta^{{\approximableMarginExponent}/{\approximableDensityExponent}} \ \ \text{and} \ \  s:=\Bigl(\frac{\theta}{c^\flat_{\holderExponent,d} \holderConstant}\Bigr)^{1/\holderExponent}.
 \end{align*}
Finally, let $C_0 := C_0^0 \cdot (C_0^1 \vee C_0^2 \vee C_0^3 \vee C_0^4 \vee C_0^5)$, where
\begin{align*}
C_0^0 &:= \frac{\epsilon_0 \cdot (c^\flat_{\holderExponent,d})^{d/\holderExponent}}{2^{5+4d} \cdot d^{d/2}}, \ \ C_0^1 := \biggl( \frac{16d^{1/2}}{ (c^\flat_{\holderExponent,d})^{1/\holderExponent}\cdot \approximableSetsConstant^{1/d}}\biggr)^{\frac{d\rho}{d\approximableDensityExponent - \holderExponent \approximableMarginExponent(\approximableDensityExponent-1)}}, \ \ C_0^2 := \frac{1}{(c^\flat_{\holderExponent,d} )^{\rho/(\approximableDensityExponent \holderExponent)}}, \\
C_0^3 &:=\Bigl(\frac{2}{\epsilon_0}\Bigr)^{\frac{\rho}{\approximableDensityExponent\holderExponent}}, \ \ C_0^4 :=  \biggl(\frac{\approximableSetsConstant}{2^{2d-5}d^{(d-1)/2}}\biggr)^{\frac{\rho}{\approximableDensityExponent\holderExponent\approximableMarginExponent}} \quad  \text{and} \quad  C_0^5 := (32d^{1/2})^{\frac{\rho}{\holderExponent\approximableMarginExponent}}.
\end{align*}
Observe that when $n \geq C_0  \cdot \holderConstant^{d/\holderExponent} \cdot \log(1+\zeta)$, we have $\xiForProofNZeta \leq 1/(C_0^1 \vee C_0^2 \vee C_0^3 \vee C_0^4 \vee C_0^5)$.  Hence, the choice of $C_0^1$ ensures that $s \leq r/2$, the choice of $C_0^2$ guarantees that $s \leq 1$, the choice of $C_0^3$ ensures that $\theta \leq \epsilon_0/2$, and $C_0^4$ and $C_0^5$ are chosen to guarantee that 
\begin{align*}
L &= \bigg\lfloor 4 \cdot \min\biggl\{ \biggl(\frac{n}{C_0^0C_0^4\cdot \holderConstant^{d/\holderExponent} \cdot\log(1+\zeta)}\biggr)^{\frac{\holderExponent \approximableMarginExponent\approximableDensityExponent}{\rho}} ,  \biggl(\frac{n}{C_0^0C_0^5\cdot \holderConstant^{d/\holderExponent} \cdot\log(1+\zeta)}\biggr)^{\frac{\holderExponent \approximableMarginExponent}{\rho}}\biggr\} \bigg\rfloor \\ &\geq \biggl\{ c_0 \cdot  \biggl(\frac{n}{ \holderConstant^{d/\holderExponent} \cdot\log(1+\zeta)}\biggr)^{\frac{\holderExponent \approximableMarginExponent(\approximableDensityExponent \wedge 1)}{\approximableDensityExponent(2\holderExponent+d)+\holderExponent\approximableMarginExponent}}\biggr\}\vee 4,
\end{align*}
where $c_0:=2\min\bigl\{(C_0^0C_0^4)^{-\frac{\holderExponent \approximableMarginExponent\approximableDensityExponent}{\rho}}, (C_0^0C_0^5)^{-\frac{\holderExponent \approximableMarginExponent}{\rho}} \bigr\}$.  Note also that $w \leq \frac{1}{2}\big\{(2r)^{-d} \wedge 1\bigr\}$.  We may therefore apply the construction following~\eqref{def:etaFuncFirstPartLBExtension} to define distributions  $\probDistribution_\ell:=\probDistribution^\ell_{L,r,w,s,\theta}$ for $\ell \in [L]$ when $n \geq C_0 \cdot  \holderConstant^{d/\holderExponent} \cdot\log(1+\zeta)$.  We write $\mu = \mu_{L,r,w}$ and $\eta_\ell = \eta_{L,r,w,s,\theta}^\ell$ in this construction.

Our choice of $s$ ensures that $\theta= c^{\flat}_{\holderExponent,d}\cdot \holderConstant \cdot s^{\holderExponent}$, so we may apply Lemma~\ref{lemma:LB1HolderRegFunctions} to deduce that $\probDistribution_\ell \in \classOfHolderDistributions$ for all $\ell \in [L]$.  Moreover, our choice of $w$ and $r$ guarantee that $(w/L) \cdot (2r)^d \leq  \approximableSetsConstant \cdot \min\big\{ (w/\{(4d^{1/2})^d\cdot L\})^{\approximableDensityExponent}, \theta^{\approximableMarginExponent}\big\}$, so we may apply Lemma~\ref{lemma:approximableConditionForLB1} to conclude that $\probDistribution_{\ell} \in \classOfWellApproximableSets \cap \classOfWellApproximableSetsHO$ for all $\ell \in [L]$.  Next, by Lemma~\ref{lemma:chiSqrDistanceFirstPartNegative}, for each $\ell$, $\ell' \in [L]$,
\[
\chi^2\big(\probDistribution_\ell ,\probDistribution_{\ell'}\big)\leq  \frac{2^{5+2d} \cdot w  \cdot \theta^2 \cdot s^d}{\epsilon_0 \cdot L} = \frac{\log(1+\zeta)}{n},
\]
by our choice of $w$, $\theta$ and $s$, so 
\begin{align*}
\chi^2(\probDistribution_\ell^{\otimes n},\probDistribution_{\ell'}^{\otimes n})&\leq \bigl\{1 + \chi^2\big(\probDistribution_\ell ,\probDistribution_{\ell'}\big)\bigr\}^n - 1 \leq \biggl( 1+ \frac{\log(1+\zeta)}{n}\biggr)^n-1 \leq \zeta.
\end{align*}
This proves \emph{(a)}.  To prove \emph{(b)}, we take $\ell$,$\ell'\in [L]$ with $\ell \neq \ell'$ and $A \in \mathcal{A} \cap \powerSet\bigl( \mathcal{X}_{\tau}(\eta_\ell)\cap \mathcal{X}_{\tau}(\eta_{\ell'})\bigr)$. By Lemma~\ref{lemma:approximableConditionForLB1}, we have $M_\tau(\probDistribution_{\ell},\mathcal{A})= \mu\big(K^0_r(\ell)\big)= (w/L) \cdot (2r)^{d}$. On the other hand, if $\mu(A)>0$, then since $\support(\mu)=\bigcup_{(\ell'',j) \in [L]\times \{0,1\}} K^j_r(\ell'')$ and $\bigcup_{\ell'' \in [L]} K^1_r(\ell'')\subseteq \R^d \setminus \mathcal{X}_{\tau}(\eta_{\ell''})$, we must have $A \cap K^0_r(\tilde{\ell})\neq \emptyset$ for some $\tilde{\ell} \in [L]$. Since at least one of $\tilde{\ell} \neq \ell$ or $\tilde{\ell} \neq \ell'$ must hold, it follows from Lemma \ref{lemma:approximableConditionForLB1} that $\mu(A) \leq (w/L) \cdot (2r)^d/2$. Hence
\[
M_\tau\bigl(P_\ell,\mathcal{A}\bigr)-\mu(A)\geq \frac{w}{L} \cdot 2^{d-1} \cdot r^d = \frac{\approximableSetsConstant}{2}\cdot \thetaNDelta^{\approximableMarginExponent}= \frac{\approximableSetsConstant}{2}  \biggl( \frac{C_0^0 \cdot \holderConstant^{d/\holderExponent} \cdot \log(1+\zeta)}{n}\biggr)^{\frac{\holderExponent \approximableMarginExponent \approximableDensityExponent}{\approximableDensityExponent(2\holderExponent+d)+\holderExponent\approximableMarginExponent}},
\]
so \eqref{eq:conclusionCompletingFirstLBConstructions} holds with $c_1:=\frac{\approximableSetsConstant}{2}\cdot (C_0^0)^{\frac{\holderExponent\approximableDensityExponent\approximableMarginExponent}{\rho}}$.
\end{proof}

\begin{proof}[Proof of Proposition~\ref{lemma:constrainedRiskApplicationfanoApplication}] Part~(i): We initially assume that $n \geq C_0 
 \holderConstant^{d/\holderExponent} \log\bigl(1/(4\alpha)\bigr)$.  By Proposition~\ref{lemma:completingFirstLBConstructions}, with $\zeta = 1/(4\alpha)-1>0$, there exists a pair of distributions $P_1$, $P_2 \in \classOfLBDistributions$ with common marginal distribution $\mu$ on $\R^d$ and corresponding regression functions $\eta_1$,$\eta_2$ such that $\chi^2\bigl(P_{2}^{\otimes n},P_{1}^{\otimes n}\bigr) +1\leq 1/(4\alpha)$ and if $A \in \mathcal{A} \cap \powerSet\bigl( \mathcal{X}_{\tau}(\eta_1)\cap \mathcal{X}_{\tau}(\eta_2)\bigr)$, then
\begin{align}\label{eq:firstAppLemmaCompletingFirstLBConstructions}
M_\tau\bigl(P_2,\mathcal{A} \bigr)-\mu(A)\geq c_1 \cdot \biggl( \frac{\holderConstant^{d/\holderExponent} \cdot\log\bigl(1/(4\alpha)\bigr)}{n}\biggr)^{\frac{\holderExponent \approximableMarginExponent \approximableDensityExponent}{\approximableDensityExponent(2\holderExponent+d)+\holderExponent\approximableMarginExponent}}.
\end{align}
We now define a test $\varphi:(\R^d\times [0,1])^n \rightarrow \{1,2\}$ by 
\begin{align*}
\varphi(D):=\begin{cases} 1 &\text{ if } \hat{A}(D) \subseteq \mathcal{X}_\tau(\eta_1) \\ 2 &\text{ otherwise.}\end{cases}    
\end{align*}
Since $\hat{A}$ controls the Type I error at level $\alpha$ over $\classOfLBDistributions$, we have
\begin{align*}
P_1^{\otimes n} \bigl( \{D \in (\R^d\times [0,1])^n:\varphi(D)=2\}\bigr)
=\Prob_{P_1}\bigl(\hat{A}(\sample) \nsubseteq \mathcal{X}_\tau({\eta_1})\bigr) \leq \alpha.
\end{align*}
Hence, by an immediate consequence of \citet[][Theorem~1]{brown1996constrained}, which we restate as Lemma~\ref{lemma:BrownLowInequality} for convenience, with $\epsilon=\sqrt{\alpha}$ we have 
\begin{align*}
\Prob_{P_2}\bigl(\hat{A}(\sample) \subseteq \mathcal{X}_\tau({\eta_1})\bigr) &= P_2^{\otimes n} \bigl( \{D \in (\R^d\times [0,1])^n:\varphi(D)=1\}\bigr) \\
&\geq \Bigl\{ 1-\epsilon \sqrt{\chi^2\bigl(P_{2}^{\otimes n},P_{1}^{\otimes n}\bigr) +1}\Bigr\}^2 \geq \frac{1}{4}.
\end{align*}
Thus, by \eqref{eq:firstAppLemmaCompletingFirstLBConstructions} we have
\begin{align}
\label{Eq:Conclusion}
\E_{\probDistribution_2} &\bigl[\bigl\{ M_\tau\bigl(P,\mathcal{A}\bigr)-\mu(\hat{A})\bigr\} \cdot \one_{\{\hat{A} \subseteq  \mathcal{X}_{\tau}(\eta_2) \}}\bigr] \nonumber \\
& \geq  \Prob_{P_2} \bigl(\bigl\{ \hat{A}(\sample) \subseteq \mathcal{X}_\tau({\eta_1})\bigr\} \cap \bigl\{ \hat{A}(\sample) \subseteq \mathcal{X}_\tau({\eta_2})\bigr\} \bigr)\cdot c_1 \biggl( \frac{\holderConstant^{d/\holderExponent} \log\bigl(1/(4\alpha)\bigr)}{n}\biggr)^{\frac{\holderExponent \approximableMarginExponent \approximableDensityExponent}{\approximableDensityExponent(2\holderExponent+d)+\holderExponent\approximableMarginExponent}} \nonumber \\
& \geq \bigl\{\Prob_{P_2} \bigl(\hat{A}(\sample) \subseteq \mathcal{X}_\tau({\eta_1})\bigr) - \Prob_{P_2}\bigl( \hat{A}(\sample) \nsubseteq \mathcal{X}_\tau({\eta_2})\bigr) \bigr\}\cdot c_1 \biggl( \frac{\holderConstant^{d/\holderExponent} \log\bigl(1/(4\alpha)\bigr)}{n}\biggr)^{\frac{\holderExponent \approximableMarginExponent \approximableDensityExponent}{\approximableDensityExponent(2\holderExponent+d)+\holderExponent\approximableMarginExponent}} \nonumber \\
& \geq \biggl(\frac{1}{4}-\alpha\biggr)\cdot c_1 \biggl( \frac{\log\bigl(1/(4\alpha)\bigr)}{n}\biggr)^{\frac{\holderExponent \approximableMarginExponent \approximableDensityExponent}{\approximableDensityExponent(2\holderExponent+d)+\holderExponent\approximableMarginExponent}} \geq \frac{c_1}{8} \cdot \biggl( \frac{\holderConstant^{d/\holderExponent} \cdot \log\bigl(1/(4\alpha)\bigr)}{n}\biggr)^{\frac{\holderExponent \approximableMarginExponent \approximableDensityExponent}{\approximableDensityExponent(2\holderExponent+d)+\holderExponent\approximableMarginExponent}},
\end{align}
as required.  Moreover, if $1 \leq n < C_0 \cdot \holderConstant^{d/\holderExponent} \cdot \log\bigl(1/(4\alpha)\bigr)$, then by~\eqref{Eq:Conclusion} with $n \!=\! \big\lceil C_0 \cdot \holderConstant^{d/\holderExponent} \cdot \log\bigl(1/(4\alpha)\bigr) \big\rceil$, we have
\[
\E_{\probDistribution_2} \bigl[\bigl\{ M_\tau\bigl(P,\mathcal{A}\bigr)-\mu(\hat{A})\bigr\} \cdot \one_{\{\hat{A} \subseteq \mathcal{X}_{\tau}(\eta_2) \}}\bigr] \geq \frac{c_1}{8} \cdot C_0^{-\frac{\holderExponent \approximableMarginExponent \approximableDensityExponent}{\approximableDensityExponent(2\holderExponent+d)+\holderExponent\approximableMarginExponent}},
\]
again.  This completes the proof of Part~\emph{(i)}.

\medskip

\noindent Part~(ii): Fix $\rho:={\approximableDensityExponent(2\holderExponent+d)+\holderExponent\approximableMarginExponent}$ and let $C_0\ge 1$ and $c_0,c_1>0$ be as in Proposition~\ref{lemma:completingFirstLBConstructions}.  Let $C_1 \equiv C_1(d, \holderExponent,  \approximableMarginExponent,\approximableDensityExponent,  \approximableSetsConstant,\epsilon_0) \geq e^{\frac{2\rho}{\holderExponent\approximableMarginExponent}} - 1$ be large enough that for all $n/\holderConstant^{d/\holderExponent} \geq C_1$, we have both $n/\holderConstant^{d/\holderExponent} \geq C_0 \log\bigl(1+(n/\holderConstant^{d/\holderExponent})^{\frac{\holderExponent \approximableMarginExponent(\approximableDensityExponent \wedge 1\}}{2\rho}}\bigr)$ and 
\begin{align*}
c_0 \epsilon_0^2 \biggl(\frac{(n/\holderConstant^{d/\holderExponent})}{\log\bigl(1+(n/\holderConstant^{d/\holderExponent})^{\frac{\holderExponent \approximableMarginExponent(\approximableDensityExponent \wedge 1)}{2\rho}}\bigr)}\biggr)^{\frac{\holderExponent \approximableMarginExponent(\approximableDensityExponent \wedge 1)}{\rho}}\geq 2^5(n/\holderConstant^{d/\holderExponent})^{\frac{\holderExponent \approximableMarginExponent(\approximableDensityExponent \wedge 1)}{2\rho}}.
\end{align*}
Next, for $n \geq C_1$, we apply Proposition~\ref{lemma:completingFirstLBConstructions} with $\zeta = (n/\holderConstant^{d/\holderExponent})^{\frac{\holderExponent \approximableMarginExponent(\approximableDensityExponent \wedge 1)}{2\rho}}$ to obtain a family of $L \geq c_0 \bigl\{{(n/\holderConstant^{d/\holderExponent})}/{\log\bigl(1+(n/\holderConstant^{d/\holderExponent})^{\frac{\holderExponent \approximableMarginExponent(\approximableDensityExponent \wedge 1)}{2\rho}}\bigr)}\bigr\}^{\frac{\holderExponent \approximableMarginExponent(\approximableDensityExponent \wedge 1)}{\rho}} \geq 2^5(n/\holderConstant^{d/\holderExponent})^{\frac{\holderExponent \approximableMarginExponent(\approximableDensityExponent \wedge 1)}{2\rho}}/\epsilon_0^2$ distributions $\{P_1,\ldots, P_L \}\subseteq \classOfHolderDistributions \cap \classOfWellApproximableSets \cap \classOfWellApproximableSetsHO$ with common marginal distribution $\mu$ on $\R^d$ and corresponding regression functions $\eta_1,\ldots,\eta_L$, such that
\begin{enumerate}
\item [\emph{(a)}] $\chi^2\bigl(P_\ell^{\otimes n},P_{\ell'}^{\otimes n}\bigr) \leq (n/\holderConstant^{d/\holderExponent})^{\frac{\holderExponent \approximableMarginExponent(\approximableDensityExponent \wedge 1)}{2\rho}} \leq ({\epsilon_0}/4)^2\cdot (L-1)$ for all $\ell$, $\ell' \in [L]$;
\item[\emph{(b)}] if $A \in \mathcal{A} \cap \powerSet\bigl( \mathcal{X}_{\tau}(\eta_\ell)\cap \mathcal{X}_{\tau}(\eta_{\ell'})\bigr)$ for some $\ell$,$\ell'\in [L]$ with $\ell \neq \ell'$, then
\begin{align}\label{eq:fanoApplicationCompletingFirstLBConstructions}
M_\tau\bigl(P_\ell,\mathcal{A}\bigr)-\mu(A)\geq c_1 \cdot \biggl( \frac{\holderConstant^{d/\holderExponent}\log\{1+(n/\holderConstant^{d/\holderExponent})^{\frac{\holderExponent \approximableMarginExponent(\approximableDensityExponent \wedge 1)}{2\rho}} \}}{n}\biggr)^{\holderExponent \approximableMarginExponent \approximableDensityExponent/\rho}.
\end{align} 
\end{enumerate}
Now define a test function $\varphi: (\R^d \times [0,1])^n \rightarrow [L]$ by
\begin{align*}
\varphi(D):=\begin{cases} \min\{ \ell \in [L]:\hat{A}(D) \subseteq \mathcal{X}_\tau(\eta_\ell)\bigr\} &\text{ if } \hat{A}(D) \subseteq \mathcal{X}_\tau(\eta_\ell) \text{ for some } \ell \in [L]\\ L &\text{ otherwise.}\end{cases}    
\end{align*}
By Lemma~\ref{lemma:chiSqrFano}, we have
\begin{align}
\label{Eq:FanoApp}
\max_{\ell \in [L]}P_\ell^{\otimes n}\bigl(\bigl\{D \in &(\R^d \times [0,1])^n: \varphi(D) \neq \ell\bigr\}\bigr) \nonumber \\
&\geq \frac{1}{L-1}\sum_{\ell=2}^L P_\ell^{\otimes n}\bigl(\bigl\{D \in (\R^d \times [0,1])^n: \varphi(D) \neq \ell\bigr\}\bigr) \nonumber \\
&\geq 1 - \frac{1}{L-1} - \sqrt{\frac{1}{L-1} \sum_{\ell=2}^L \chi^2(P_\ell^{\otimes n},P_1^{\otimes n}) \cdot \frac{1}{L-1}\biggl(1-\frac{1}{L-1}\biggr)} \nonumber \\
&\geq 1 - \Bigl(\frac{\epsilon_0}{4}\Bigr)^2 - \frac{\epsilon_0}{4} \geq 1 - \frac{\epsilon_0}{2}.
\end{align}
Now choose $\ell_0 \in [L]$ with $P_{\ell_0}^{\otimes n}\bigl(\bigl\{D \in (\R^d \times [0,1])^n: \varphi(D) \neq \ell_0\bigr\}\bigr) \geq 1-\epsilon_0/2$, and observe that if $\varphi(D) \neq \ell_0$ and $\hat{A}(D) \subseteq \mathcal{X}_\tau(\eta_{\ell_0})$ then we must also have $\hat{A}(D) \subseteq \mathcal{X}_\tau(\eta_{\ell_1})$ for some $\ell_1 \in [\ell_0-1]$.  It follows from this and~\eqref{Eq:FanoApp} that
\begin{align*}
\Prob_{P_{\ell_0}}\biggl(\bigl\{\hat{A}(\sample) \subseteq \mathcal{X}_\tau(\eta_{\ell_0}) \bigr\} &\cap \bigcup_{\ell_1=1}^{\ell_0-1} \bigl\{ \hat{A}(\sample) \subseteq \mathcal{X}_\tau(\eta_{\ell_1}) \bigr\} \biggr) \\
&\geq \Prob_{P_{\ell_0}}\bigl(\bigl\{\hat{A}(\sample) \subseteq \mathcal{X}_\tau(\eta_{\ell_0}) \bigr\} \cap \bigl\{ \varphi(\mathcal{D}) \neq \ell_0\bigr\}\bigr) \\
& \geq P_{\ell_0}^{\otimes n}\bigl(\bigl\{D \in (\R^d \!\times \! [0,1])^n: \varphi(D) \neq \ell_0\bigr\}\bigr) - \alpha \geq \frac{\epsilon_0}{2}.
\end{align*}
Thus, by~\eqref{eq:fanoApplicationCompletingFirstLBConstructions}, we have for all $n \geq C_1 \cdot \holderConstant^{d/\holderExponent}\geq e^{\frac{2\rho}{\holderExponent\approximableMarginExponent}} - 1$ that
\begin{align*}
&\E_{\probDistribution_{\ell_0}} \bigl[\bigl\{ M_\tau\bigl(P_{\ell_0}, \mathcal{A}\bigr) -\mu(\hat{A})\bigr\} \cdot \one_{\{\hat{A} \subseteq  \mathcal{X}_{\tau}(\eta_{\ell_0}) \}}\bigr] \\
&\hspace{0.2cm} \geq  \Prob_{P_{\ell_0}}\biggl(\bigl\{\hat{A}(\sample) \subseteq \mathcal{X}_\tau(\eta_{\ell_0}) \bigr\} \cap \bigcup_{\ell_1=1}^{\ell_0-1} \bigl\{ \hat{A}(\sample) \subseteq \mathcal{X}_\tau(\eta_{\ell_1}) \bigr\} \biggr) \\ & \hspace{1cm} \cdot  c_1 \cdot \biggl( \frac{\holderConstant^{d/\holderExponent}\log\{1+(n/\holderConstant^{d/\holderExponent})^{\frac{\holderExponent \approximableMarginExponent(\approximableDensityExponent \wedge 1)}{2\rho}} \}}{n}\biggr)^{\holderExponent \approximableMarginExponent \approximableDensityExponent/\rho}\\ &\hspace{0.2cm} \geq \frac{c_1\epsilon_0}{2} \cdot \biggl( \frac{{\holderExponent \approximableMarginExponent}(\approximableDensityExponent \wedge 1) \cdot \holderConstant^{d/\holderExponent} \cdot \log_+ (n/\holderConstant^{d/\holderExponent})}{{2\rho}n}\biggr)^{\holderExponent \approximableMarginExponent \approximableDensityExponent/\rho}.
\end{align*}
We extend the bound to $n < C_1 \cdot \holderConstant^{d/\holderExponent}$ by monotonicity as at the end of the proof of Proposition~\ref{lemma:constrainedRiskApplicationfanoApplication}\emph{(i)}, with 
\[
c_2 := \frac{c_1\epsilon_0}{2} \cdot \biggl( \frac{{\holderExponent \approximableMarginExponent}(\approximableDensityExponent \wedge 1)\log_+\lceil C_1\rceil}{2\rho \lceil C_1 \rceil }\biggr)^{\holderExponent \approximableMarginExponent \approximableDensityExponent/\rho},
\]
which completes the proof.
\end{proof}
Finally, we prove the parametric lower bounds in Theorems~\ref{thm:minimaxRate} and~\ref{Thm:minimaxRateHOS}.  Some care is required here to show that our constructed distributions belong to the relevant distributional classes.
\begin{proof}[Proof of Proposition~\ref{lemma:paramatricLB}] Observe that there exists $c_{\mathrm{H}} \equiv c_{\mathrm{H}}(\holderExponent) \in (0,1]$ such that when $\theta \leq c_{\mathrm{H}} \cdot \holderConstant \cdot s^{\holderExponent}$, we have that $\eta$ is $(\holderExponent,\holderConstant)$-H\"{o}lder, so $\{P_{\zeta}^{\ell}\}_{\ell \in \{-1,1\}} \subseteq \classOfHolderDistributions$. In addition, $\support(\mu_{\zeta}^{\ell}) \cap \mathcal{X}_\tau(\eta) = A_{-1}\cup A_1$ for $\ell \in \{-1,1\}$. Since $\mathcal{A} \subseteq \mathcal{A}_{\mathrm{conv}}$ and $A_0 \subseteq \R^d \setminus \mathcal{X}_\tau(\eta)$, it follows that  for $\ell \in \{-1,1\}$,
\[
M_\tau(P_{\zeta}^{\ell},\mathcal{A}) = \max_{j \in \{-1,1\}}\mu_{\zeta}^{\ell}(A_j)= \frac{s^d}{(2t)^d+2s^d}+\zeta.
\]  
Observe also that for any $\ell \in \{-1,1\}$, $x \in A_\ell$ and $r \in (0,4s]$, we have
\[
\mu_{\zeta}^{\ell}\bigl(\closedMetricBallSupNorm{x}{r}\bigr) \geq \biggl(\frac{1}{(2t)^{d}+2s^d}+\frac{\zeta }{s^d}\biggr) \cdot \Lebesgue\bigl(\closedMetricBallSupNorm{x}{r/4} \cap A_{\ell}\bigr) \geq \biggl(\frac{1}{(2t)^{d}+2s^d}+\frac{\zeta }{s^d}\biggr)\cdot \biggl(\frac{r}{4}\biggr)^d.
\]
Moreover, for any $\ell \in \{-1,1\}$, $x \in A_\ell$ and $r \in (4s,1]$, we have
\begin{align*}
\mu_{\zeta}^{\ell}\bigl(\closedMetricBallSupNorm{x}{r}\bigr) &\geq  \frac{\Lebesgue\bigl(\closedMetricBallSupNorm{x}{r} \cap A_0\bigr)}{(2t)^{d}+2s^d} \geq  \frac{r^{d-1}\cdot (r-2s)}{(2t)^{d}+2s^d} \\ &\geq \frac{r^d}{2\{(2t)^{d}+2s^d\}} \geq  \biggl(\frac{1}{(2t)^{d}+2s^d}+\frac{\zeta }{s^d}\biggr)\cdot \frac{r^d}{3}.
\end{align*}
Hence, $\lowerDensity_{\mu_{\zeta}^{\ell},d}(x) \geq \{(8t)^{d}+2(4s)^d\}^{-1}$ for all $x \in A_\ell$, and moreover $A_\ell \subseteq \mathcal{R}_{\regularityConstant}(\mu_{\zeta}^{\ell})$. Thus, for any $(\xi,\Delta) \in \bigl(0,\{(8t)^{d}+2(4s)^d\}^{-1}\bigr]\times (0,\theta]$, we have
\[
\sup\bigl\{ \mu_{\zeta}^{\ell}(A):A \in \mathcal{A}\cap \powerSet\bigl(\mathcal{X}_{\xi}(\lowerDensity_{\mu_{\zeta}^{\ell},d}) \cap \mathcal{X}_{\tau+\Delta}(\eta)\bigr)  \bigr\} = \mu_{\zeta}^\ell(A_\ell) =  M_\tau(P_{\zeta}^{\ell},\mathcal{A}),
\]
and similarly,
\[
\sup\bigl\{ \mu_{\zeta}^{\ell}(A):A \in \mathcal{A}\cap \powerSet\bigl(\mathcal{R}_{\regularityConstant}(\mu^{\ell}_\zeta)\cap \mathcal{X}_{\xi}(f_{\mu_{\zeta}^{\ell}}) \cap \mathcal{X}_{\tau+\Delta}(\eta)\bigr)  \bigr\} = \mu_{\zeta}^\ell(A_\ell) =  M_\tau(P_{\zeta}^{\ell},\mathcal{A}).
\]
On the other hand, if either $\xi> \{(8t)^{d}+2(4s)^d\}^{-1}$ or $\Delta>\theta$, then provided that $\frac{3s^d}{2\{(2t)^d+2s^d\}} \leq \approximableSetsConstant \cdot \bigl[ \{(8t)^{d}+2(4s)^d\}^{-\approximableDensityExponent}\wedge \theta^{\approximableMarginExponent}\bigr]$, we have
\begin{align*}\sup\bigl\{ &\mu_{\zeta}^{\ell}(A):A \in \mathcal{A}\cap \powerSet\bigl(\mathcal{R}_{\regularityConstant}(\mu^{\ell}_\zeta)\cap \mathcal{X}_{\xi}(f_{\mu_{\zeta}^{\ell}}) \cap \mathcal{X}_{\tau+\Delta}(\eta)\bigr)  \bigr\}\\&\wedge \sup\bigl\{ \mu_{\zeta}^{\ell}(A):A \in \mathcal{A}\cap \powerSet\bigl(\mathcal{X}_{\xi}(\lowerDensity_{\mu_{\zeta}^{\ell},d}) \cap \mathcal{X}_{\tau+\Delta}(\eta)\bigr)  \bigr\} \\ &\geq 0 \geq M_\tau(P_{\zeta}^{\ell},\mathcal{A}) - \approximableSetsConstant \cdot (\xi^{\approximableDensityExponent}+\Delta^{\approximableMarginExponent}).
\end{align*}
It follows that $\{P_{\zeta}^{\ell}\}_{\ell \in \{-1,1\}}\subseteq \classOfLBDistributions$ whenever $\frac{3s^d}{2\{(2t)^d+2s^d\}} \leq \approximableSetsConstant \cdot  \bigl[\{(8t)^{d}+2(4s)^d\}^{-\approximableDensityExponent}\wedge \theta^{\approximableMarginExponent}\bigr]$.  

In addition, recalling that $m_{\mathrm{ct}}$ denotes the counting measure, we have
\begin{align*}
\mathrm{KL}\bigl(P_{\zeta}^{-1},P_{\zeta}^{1}\bigr) &\leq \chi^2\bigl(P_{\zeta}^{-1},P_{\zeta}^{1}\bigr) \\
&= \int_{\R^d \times \{0,1\}} \frac{\bigl\{f_{\mu_{\zeta}^{-1}}(x) - f_{\mu_{\zeta}^{1}}(x)\bigr\}^2\eta(x)^y\{1 - \eta(x)\}^{1-y}}{f_{\mu_{\zeta}^{1}}(x)} \, d(\Lebesgue \times m_{\mathrm{ct}})(x,y) \\
&= \frac{(2\zeta/s^d)^2}{\frac{1}{(2t)^d + 2s^d} - \frac{\zeta}{s^d}} + \frac{(2\zeta/s^d)^2}{\frac{1}{(2t)^d + 2s^d} + \frac{\zeta}{s^d}} \leq \frac{32}{3} \cdot \frac{(2t)^d + 2s^d}{s^{2d}} \cdot \zeta^2.
\end{align*}
We deduce by Pinsker's inequality that 
\begin{align*}
\totalVariationDistance\bigl((P_{\zeta}^{-1})^{\otimes n},(P_{\zeta}^{1})^{\otimes n}\bigr) \leq  \sqrt{\mathrm{KL}\bigl((P_{\zeta}^{-1})^{\otimes n},(P_{\zeta}^{+1})^{\otimes n}\bigr)/2} &\leq 4\zeta\sqrt{n} \cdot \sqrt{\frac{(2t)^d + 2s^d}{3s^{2d}}} \\
&= 2a\cdot \zeta\sqrt{n},
\end{align*}
where $a\equiv a_{d,s,t}:=2 \cdot \sqrt{\frac{(2t)^d + 2s^d}{3s^{2d}}}$.  Thus, 
\begin{align*}
1&= (P_{\zeta}^{1})^{\otimes n} \bigl(\{D \in (\R^d \times \{0,1\})^n : A_1 \cap \hat{A}(D) = \emptyset\}\bigr) \\
&\hspace{5cm}+(P_{\zeta}^{1})^{\otimes n} \bigl(\{D \in (\R^d \times \{0,1\})^n : A_1 \cap \hat{A}(D) \neq \emptyset\}\bigr)\\ 
&\leq  (P_{\zeta}^{1})^{\otimes n} \bigl(\{D:A_1 \cap \hat{A}(D) = \emptyset\}\bigr)+(P_{\zeta}^{-1})^{\otimes n} \bigl(\{D:A_1 \cap \hat{A}(D) \neq \emptyset\}\bigr) + 2a \cdot \zeta\sqrt{n}.
\end{align*}
Hence, since at most one of $A_{-1}$ and $A_1$ can have non-empty intersection with $\hat{A}$ whenever $\hat{A} \subseteq \etaSuperLevelSet{\tau}$, we have
\begin{align*}
\max_{\ell \in \{-1,1\}}(P_{\zeta}^{\ell})^{\otimes n} \bigl(\{D:A_\ell \cap\hat{A}(D) = \emptyset\}\cap \bigl\{D:\hat{A}(D) \subseteq \etaSuperLevelSet{\tau}\bigr\}  \bigr) &\geq \frac{1}{2}-a \cdot \zeta\sqrt{n}-\alpha \\
&\geq \epsilon_0 - a \cdot \zeta\sqrt{n}.
\end{align*}
We conclude that 
\begin{align*}
\max_{\ell \in \{-1,1\}} \E_{P_{\zeta}^{\ell}} \bigl[\bigl\{ M_\tau\bigl(P_{\zeta}^\ell,\mathcal{A}\bigr)-\mu_{\zeta}^\ell(\hat{A})\bigr\} \cdot \one_{\{\hat{A} \subseteq \etaSuperLevelSet{\tau}\}}\bigr] \geq 2\zeta \bigl(\epsilon_0 - a \cdot \zeta\sqrt{n}\bigr).
\end{align*}
Taking $t:=\bigl\{\bigl(2^{\approximableDensityExponent(3d+1)} \vee (c_{\mathrm{H}} \cdot \holderConstant)^{\approximableMarginExponent}\bigr)/\approximableSetsConstant\bigr\}^{\frac{\holderExponent\approximableMarginExponent}{d(d\approximableDensityExponent-\holderExponent\approximableMarginExponent(\approximableDensityExponent-1))}}\vee 1$,  $s:=t^{-\frac{d\approximableDensityExponent}{\holderExponent\approximableMarginExponent}}$ and 
$\theta := c_{\mathrm{H}} \cdot \holderConstant \cdot s^{\holderExponent}$ ensures that the conditions $\frac{3s^d}{2\{(2t)^d+2s^d\}} \leq \approximableSetsConstant \cdot \bigl[ \{(8t)^{d}+2(4s)^d\}^{-\approximableDensityExponent}\wedge \bigl(c_{\mathrm{H}} \cdot \holderConstant \cdot s^{\holderExponent}\bigr)^{\approximableMarginExponent}\bigr]$ and $\theta \leq c_{\mathrm{H}} \cdot \holderConstant \cdot s^{\holderExponent}$ hold.  Hence, taking $\zeta := \epsilon_0/(2a\sqrt{n})$ yields the required lower bound.
\end{proof}

\begin{proof}[Proof of Theorems~\ref{thm:minimaxRate}\emph{(ii)} and~\ref{Thm:minimaxRateHOS}\emph{(ii)}] In light of the remarks following Proposition~\ref{lemma:constrainedRiskApplicationfanoApplication}, these results follow from Propositions~\ref{lemma:constrainedRiskApplicationfanoApplication} and~\ref{lemma:paramatricLB}.
\end{proof}

\section{Parameter constraints}

The following lemma reveals natural constraints satisfied by the parameters $\holderExponent$, $\approximableDensityExponent$ and $\approximableMarginExponent$.
\begin{lemma}\label{lemma:parameterConstraints} Take $\tau \in (0,1)$, $(\holderExponent, \approximableMarginExponent,\approximableDensityExponent,\regularityConstant, \holderConstant,\approximableSetsConstant) \in (0,1]\times (0,\infty)^2\times (0,1) \times [1,\infty)^2$. Let $\probDistribution \in \classOfHolderDistributions \cap \classOfWellApproximableSetsWithHyperCubes$ be a distribution on $\R^d \times [0,1]$ with regression function $\eta:\R^d \rightarrow [0,1]$ and with a Lebesgue absolutely continuous marginal $\mu$ on $\R^d$ with continuous density $f_\mu$. Suppose that $\etaSuperLevelSet{\tau} \subseteq \muRegularSet$, that  $\mu\bigl(\eta^{-1}((\tau,1])\bigr)>0$ and that $\eta^{-1}(\{\tau\}) \neq \emptyset$.  Then $\holderExponent\approximableMarginExponent (\approximableDensityExponent-1) \leq d \approximableDensityExponent$.
\end{lemma}
\begin{proof} Note that since $\mu\bigl(\eta^{-1}((\tau,1])\bigr)>0$ and $\eta$ is continuous, we must have $M_\tau>0$. Take $\Delta \in \bigl(0, \{M_{\tau}/(2\approximableSetsConstant)\}^{1/\approximableMarginExponent}\bigr)$, and write $\omega:= \omega_{\mu,d}$ for the lower density of $\mu$. Since $\probDistribution \in  \classOfWellApproximableSetsWithHyperCubes$, we may take $A_{\Delta}\in \mathcal{A}_{\mathrm{hpr}}\cap \powerSet\bigl(\omegaSuperLevelSet{\Delta^{\approximableMarginExponent/\approximableDensityExponent}}\cap \etaSuperLevelSet{\tau+\Delta}\bigr)$ with $\mu(A_{\Delta})\geq M_{\tau}-2\approximableSetsConstant \cdot \Delta^{\approximableMarginExponent}>0$.  Now $A_{\Delta}$ is a non-empty, compact subset of $\R^d$ and $\eta^{-1}(\{\tau\})$ is a non-empty closed subset of $\R^d$, so we may choose $x_0 \in \eta^{-1}(\{\tau\})$ and $z_0\in A_{\Delta}$ such that 
\begin{align*}
\supNorm{x_0-z_0}=\inf_{x \in \eta^{-1}(\{\tau\})}\inf_{z \in A_{\Delta}}\supNorm{x-z}.
\end{align*}
Let $A_{\Delta}^{\sharp} := \{ x \in \R^d: \supNorm{x-z} \leq \supNorm{x_0-z_0} \text{ for some } z \in A_\Delta\}$, and note that $A_{\Delta}^{\sharp} \in \mathcal{A}_{\mathrm{hpr}}\cap \powerSet\bigl(\etaSuperLevelSet{\tau}\bigr)$, so $\mu(A_{\Delta}^{\sharp}) \leq M_{\tau}$ and $\mu(A_{\Delta}^{\sharp}\setminus A_{\Delta}) \leq 2\approximableSetsConstant \cdot \Delta^{\approximableMarginExponent}$. In addition, since $\probDistribution \in \classOfHolderDistributions$, we have $\supNorm{x_0-z_0} \geq (\Delta/\holderConstant)^{1/\holderExponent} =: 3r_{\Delta}$. Hence, if we take 
\[
w_0:= z_0+\biggl(1+\frac{\regularityConstant}{2}\biggr)\cdot r_\Delta\cdot \frac{x_0-z_0}{\supNorm{x_0-z_0}},
\]
we have $\closedMetricBallSupNorm{w_0}{r_{\Delta}} \subseteq \etaSuperLevelSet{\tau}\cap (A_{\Delta}^{\sharp}\setminus A_{\Delta})$ and $z_0 \in \openMetricBallSupNorm{w_0}{(1+\regularityConstant)\cdot r_{\Delta}}$. Thus, as $f_{\mu}(z_0) \geq 2^{-d} \cdot  \omega(z_0) \geq 2^{-d} \cdot  \Delta^{\approximableMarginExponent/\approximableDensityExponent}$ and $w_0 \in \etaSuperLevelSet{\tau}\subseteq \muRegularSet$ and $r_\Delta \in (0,1)$, we have
\begin{align*}
2\approximableSetsConstant \cdot \Delta^{\approximableMarginExponent} \geq \mu(A_{\Delta}^{\sharp}\setminus A_{\Delta}) \geq \mu\bigl(\closedMetricBallSupNorm{w_0}{r_{\Delta}}\bigr) &\geq \regularityConstant \cdot r_{\Delta}^d \cdot \sup_{x' \in \openMetricBallSupNorm{w_0}{(1+\regularityConstant)r_{\Delta}}}f_\mu(x') \\
&\geq \regularityConstant \cdot r_{\Delta}^d \cdot f_\mu(z_0) \geq \frac{\regularityConstant}{6^d} \cdot \biggl(\frac{\Delta}{\holderConstant }\biggr)^{d/\holderExponent} \cdot \Delta^{\approximableMarginExponent/\approximableDensityExponent}.
\end{align*}
Letting $\Delta \searrow 0$ we deduce that $\holderExponent\approximableMarginExponent (\approximableDensityExponent-1) \leq d \approximableDensityExponent$.
\end{proof}

\section{Auxiliary results}
\label{Sec:AuxiliaryResults}

\subsection{Disintegration and measure-theoretic preliminaries}\label{Sec:Disintegration}

Suppose we have a pair of measurable spaces $(\mathcal{X},\mathcal{G}_X)$ and $(\mathcal{Y},\mathcal{G}_Y)$ along with a probability distribution $P$ on the product space $(\mathcal{X}\times \mathcal{Y}, \mathcal{G}_X \otimes \mathcal{G}_Y)$. Let $\mu$ denote the marginal distribution of $P$ on $(\mathcal{X},\mathcal{G}_X)$. We say that $(P_x)_{x \in \mathcal{X}}$ is a \emph{disintegration} of $P$ into conditional distributions on $\mathcal{Y}$ if
\begin{enumerate}
    \item $P_x$ is a probability measure on $(\mathcal{Y},\mathcal{G}_Y)$, for each $x \in \mathcal{X}$;
    \item $x \mapsto P_x(B)$ is a $\mathcal{G}_X$-measurable function, for every $B \in \mathcal{G}_Y$;
    \item $P(A\times B) = \int_A P_x(B) \, d\mu(x)$ for all $A \in \mathcal{G}_X$ and $B \in \mathcal{G}_Y$.
\end{enumerate}
We will make use of the following existence result: recall that a topological space $(\mathcal{X},\mathcal{T}_X)$ is said to be \emph{Polish} if there exists a metric $d_X$ on $\mathcal{X}$ that induces the topology $\mathcal{T}_X$ and for which $(\mathcal{X},d_X)$ is a complete, separable metric space.

\begin{lemma}\label{lemma:existenceOfDisintegrationFromDudley} Suppose that $(\mathcal{X},\mathcal{G}_X)$ and $(\mathcal{Y},\mathcal{G}_Y)$ are Polish spaces with their corresponding Borel $\sigma$-algebras.  Let $P$ be a probability distribution on $(\mathcal{X}\times \mathcal{Y}, \mathcal{G}_X \otimes \mathcal{G}_Y)$, with $\mu$ denoting the marginal distribution of $P$ on $(\mathcal{X},\mathcal{G}_X)$. Then there exists a disintegration $(P_x)_{x \in \mathcal{X}}$ of $P$ into conditional distributions on $\mathcal{Y}$ with the property that 
\begin{align}\label{eq:disintegrationProperty}
\int_{\mathcal{X}\times \mathcal{Y}}g(x,y)\,dP(x,y) = \int_{\mathcal{X}} \biggl(\int_{\mathcal{Y}} g(x,y) \,dP_x(y)\biggr) \,d\mu(x),
\end{align}
for every $P$-integrable function $g: \mathcal{X}\times \mathcal{Y} \rightarrow \R$. Moreover, the disintegration $(P_x)_{x \in \mathcal{X}}$ of $P$  is unique in the sense that if there exists another disintegration $(\tilde{P}_x)_{x \in \mathcal{X}}$ of $P$ into conditional distributions on $\mathcal{Y}$, then $\tilde{P}_x = P_x$ for $\mu$-almost every $x \in \mathcal{X}$.
\end{lemma}
\begin{proof}
This follows by combining Theorems 10.2.1 and 10.2.2 of \citet{dudley2018real}.
\end{proof}

A disintegration has the following useful interpretation. Suppose we have a pair of random variables $(X,Y)$ taking values in $\mathcal{X}\times \mathcal{Y}$ with joint distribution $P$ on $\mathcal{G}_X\times \mathcal{G}_Y$, where $(\mathcal{X},\mathcal{G}_X)$ and $(\mathcal{Y},\mathcal{G}_Y)$ are Polish spaces with their corresponding Borel $\sigma$-algebras. Let $\mu$ be the marginal on $\mathcal{X}$ and $(P_x)_{x \in \mathcal{X}}$ be a disintegration of $P$ into conditional distributions. Then for all $P$-integrable functions $g: \mathcal{X}\times \mathcal{Y} \rightarrow \R$ we have
\begin{align}\label{eq:conditionalExpectationInTermsOfDisintegration}
\E\big( g(X,Y) \mid X=x\big)=\int_{\mathcal{Y}} g(x,y)\,dP_x(y),
\end{align}
for $\mu$ almost every $x \in \mathcal{X}$. Indeed, by Lemma~\ref{lemma:existenceOfDisintegrationFromDudley} we see that $x \mapsto \int_{\mathcal{Y}} g(x,y)\,dP_x(y)$ is a $\mu$-integrable function, and hence $\mathcal{G}_X$-measurable. Moreover, given any $A \in \mathcal{G}_X$, we have
\begin{align*}
\int_A \biggl(\int_{\mathcal{Y}} g(x,y)\,dP_x(y) \biggr) \,d\mu(x) &= 
\int_{\mathcal{X}} \biggl(\int_{\mathcal{Y}} \one_{A}(x) \cdot g(x,y)\, dP_x(y) \biggr) \,d\mu(x) \\
& = 
\int_{\mathcal{X}\times \mathcal{Y}}  \one_{A}(x) \cdot g(x,y)\, dP(x,y)  = 
\int_{A\times \mathcal{Y}}  g(x,y)\, dP(x,y),
\end{align*}
where the second equality follows from \eqref{eq:disintegrationProperty} with $\one_{A}(x) \cdot g(x,y)$ in place of $g(x,y)$.

Recall that for Borel subsets $B_0$, $B_1 \subseteq \R^d$ and a Borel measure $\mu$ on $\R^d$, we write $B_0 \subseteq B_1$ if $\mu(B_0 \setminus B_1) = 0$ and $B_0 \not\subseteq B_1$ if $\mu(B_0 \setminus B_1) > 0$.

\begin{lemma}\label{lemma:tauSuperLevelSetConditionWellDefined} Suppose that $\eta_0$ and $\eta_1:\R^d \rightarrow [0,1]$ are regression functions for a Borel probability distribution $P$ on $\R^d \times [0,1]$ with marginal probability distribution $\mu$. Then, $\marginalDistribution(\{x \in \R^d : \eta_0(x)\neq \eta_1(x)\})=0$. Hence, given $A \in \borel(\R^d)$, we have $A \subseteq \mathcal{X}_\tau(\eta_0)$ if and only if $A \subseteq \mathcal{X}_\tau(\eta_1)$.
\end{lemma}
\begin{proof} Given $\epsilon>0$, let $B_\epsilon:=\{x \in \R^d: \eta_0(x) \geq \eta_1(x)+\epsilon\}$. Then, 
\begin{align*}
\epsilon \cdot \mu(B_\epsilon) & \leq \int_{B_\epsilon} \{ \eta_0(x)-\eta_1(x)\} \, d\marginalDistribution = \int_{B_\epsilon\times [0,1]}y \, dP(x,y) - \int_{B_\epsilon\times [0,1]} y \, dP(x,y) = 0,
\end{align*}
so $\mu(B_\epsilon)=0$. By taking a countable union we see that $\marginalDistribution\bigl(\{x \in \R^d : \eta_0(x)> \eta_1(x)\}\bigr)=0$, so by symmetry we have $\marginalDistribution(\{x \in \R^d : \eta_0(x)\neq \eta_1(x)\})=0$. Thus,  given $A \in \borel(\R^d)$ we have $\mu(A\setminus \mathcal{X}_\tau(\eta_0))=\mu(A\setminus \mathcal{X}_\tau(\eta_1))$, so $A \subseteq  \mathcal{X}_\tau(\eta_0)$ if and only if $A \subseteq \mathcal{X}_\tau(\eta_1)$.
\end{proof}

\subsection{Concentration results}

We will require the following classic result that gives a uniform concentration inequality over classes of finite Vapnik--Chervonenkis dimension; we state it for distributions on $\R^d$ for simplicity.
\begin{lemma}[Vapnik--Chervonenkis concentration]\label{lemma:vapnikChervonenkisConcentration} Let $\mu$ be a probability distribution on~$\R^d$, and let $X_1,\ldots,X_n \stackrel{\mathrm{iid}}{\sim} \mu$, with corresponding empirical distribution $\hat{\mu}_n$. There exists a universal constant $C_{\mathrm{VC}} > 0$ such that for any collection of sets $\mathcal{S} \subseteq \borel(\R^d)$ with $1 \leq \vcDim(\mathcal{S}) <\infty$, we have
\begin{align*}
\E\biggl( \sup_{S \in \mathcal{S}} \big| \hat{\mu}_n(S)-\mu(S) \big|\biggr) \leq C_{\mathrm{VC}}~ \sqrt{\frac{\vcDim(\mathcal{S})}{n}}.
\end{align*}
Moreover, for all $\delta \in (0,1)$ we have
\begin{align*}
\Prob\biggl( \sup_{S \in \mathcal{S}} \big| \hat{\mu}_n(S)-\mu(S) \big| > C_{\mathrm{VC}}~ \sqrt{\frac{\vcDim(\mathcal{S})}{n}}+\sqrt{\frac{\log(1/\delta)}{2n}}  \biggr) \leq \delta.
\end{align*}
\end{lemma}
\begin{proof}
For the expectation bound, see \citet[Theorem 8.3.23]{vershynin2018high}. The high-probability bound follows by McDiarmid's inequality \cite[Theorem 2.9.1]{vershynin2018high}.
\end{proof}
The following lemma is used in the proof of Lemma~\ref{consequenceOfGarivier2011kl}.
 \begin{lemma}[\cite{garivier2011kl}]\label{lemma:basicKLConcentrationInd} Let $(Z_j)_{j\in [m]}$ be independent random variables taking values in $[0,1]$ with $\max_{j \in [m]}\E[Z_j] \leq t$ for some $t \in (0,1)$.  Writing $\bar{Z}:= m^{-1} \sum_{j \in [m]}Z_j$, we have for $\kappa \in (t,1)$ that
 \begin{align*} 
 \Prob(\bar{Z}\geq \kappa)\leq e^{-m \cdot \kl(\kappa,t)}. 
 \end{align*}
 \end{lemma}
\begin{proof} 
By Jensen's inequality, for $\theta > 0$ and $j \in [m]$,
\[
\E(e^{\theta \cdot Z_j}) \leq 1-\E(Z_j) + e^{\theta}\cdot \E(Z_j) \leq 1 + t(e^{\theta}-1).
\]
Hence, by Markov's inequality, 
\begin{align*}
 \Prob(\bar{Z}\geq \kappa)&\leq e^{-m\theta\kappa} \prod_{j=1}^m  \E (e^{\theta \cdot Z_j}) \leq  \bigl[ e^{-\theta\kappa}\bigl\{1+t(e^\theta-1)\bigr\}\bigr]^m.
\end{align*}
The lemma follows on taking $\theta=\log\bigl(\frac{\kappa (1-{t})}{{t}(1-\kappa)}\bigr)>0$.
\end{proof}

\begin{lemma}\label{consequenceOfGarivier2011kl} Let $(Z_j)_{j\in [m]}$ be  independent random variables taking values in $[0,1]$ with $\max_{j \in [m]}\E(Z_j) \leq t$ for some $t \in (0,1)$. Let $\bar{Z}:=m^{-1}\sum_{j \in [m]}Z_j$.  Then for every $\alpha \in (0,1)$, we have
 \begin{align*}
  \Prob\biggl(\bar{Z}\geq t+\sqrt{\frac{\log(1/\alpha)}{2m}}\biggr) \leq \Prob\biggl(\biggl\{\kl(\bar{Z},t) \geq \frac{\log(1/\alpha)}{m}\biggr\} \bigcap \{\bar{Z}>t\} \biggr)\leq \alpha. 
 \end{align*}
\end{lemma}
\begin{proof} 
The first inequality follows from the fact that 
\[
2(\bar Z-t)^2 = 2\totalVariationDistance^2\bigl(\mathrm{Bern}(\bar{Z}),\mathrm{Bern}(t)\bigr) \leq \kl(\bar Z,t),
\]
by Pinsker's inequality.  To prove the second inequality we begin by noting that $w \mapsto \kl(w,t)$ is continuous and strictly increasing on the interval $[t,1]$, and consider two cases. If $\alpha \in \bigl(0,e^{-m \cdot \kl(1,t)}\bigr)$ and $\bar{Z}>t$, then
\[
\kl(\bar{Z},t) \leq \kl(1,t) < \frac{\log(1/\alpha)}{m}.
\]
On the other hand, if $\alpha \in \bigl[e^{-m \cdot \kl(1,t)},1\bigr)$, then by the intermediate value theorem we can find $\kappa_{\alpha} \in [t,1]$ such that $\kl(\kappa_{\alpha},t) =m^{-1}\cdot \log(1/\alpha)$. Then by Lemma~\ref{lemma:basicKLConcentrationInd},
 \begin{align*}
 \Prob\biggl(\biggl\{\kl(\bar{Z},t) \geq \frac{\log(1/\alpha)}{m}\biggr\} \bigcap \{\bar{Z}>t\} \biggr) =\Prob\big( \bar{Z} \geq \kappa_{\alpha}\bigr) \leq e^{-m \cdot \kl(\kappa_\alpha,t)} = \alpha,
 \end{align*}
 as required.
\end{proof}
In addition we shall make use of the following Chernoff bounds. 
\begin{lemma}[Multiplicative Chernoff --- Theorem~2.3(b,c) of \citet{mcdiarmid1998concentration}]\label{lemma:multChernoff} Let $(Z_j)_{j\in [m]}$ be a sequence of independent random variables taking values in $[0,1]$. Then given any $\theta > 0$, \begin{align*}
\Prob\biggl(\sum_{j=1}^m Z_j\leq (1-\theta)\cdot \sum_{j=1}^m \E(Z_j)\biggr) &\leq \exp\biggl( - \frac{\theta^2}{2} \cdot \sum_{j=1}^m\E(Z_j)\biggr)\\
\Prob\biggl(\sum_{j=1}^m Z_j\geq (1+\theta)\cdot \sum_{j=1}^m \E(Z_j)\biggr) &\leq \exp\biggl( - \frac{\theta^2}{2(1+\theta/3)} \cdot \sum_{j=1}^m\E(Z_j)\biggr).
\end{align*}
\end{lemma}

\begin{lemma}[Multiplicative matrix Chernoff -- Theorem~1.1 of \citet{tropp2012user}]
\label{Lemma:Tropp}
Let $(\mathbf{Z}_j)_{j \in [m]}$ be independent, non-negative definite $q \times q$ matrices with $\eigenValueMaximal(\mathbf{Z}_j) \leq a_{\max}$ almost surely, for every $j \in [m]$.  Then, writing $a_{\min} := m^{-1}\cdot\eigenValueMinimal\bigl(\sum_{j \in [m]} \mathbb{E}\mathbf{Z}_j\bigr)$, we have for every $\theta \in [0,1]$ that
\[
\mathbb{P}\biggl\{\eigenValueMinimal\biggl(\sum_{j=1}^m \mathbf{Z}_j\biggr) \leq m(1-\theta)a_{\min}\biggr\} \leq q \cdot \biggl(\frac{e^{-\theta}}{(1-\theta)^{1-\theta}}\biggr)^{ma_{\min}/a_{\max}} \leq q \cdot e^{-\frac{\theta^2 m a_{\min}}{2a_{\max}}}.
\]

\end{lemma}

\subsection{Useful lemmas for the lower bounds}

We shall make use of the following result from \cite{brown1996constrained}.

\begin{lemma}[Brown--Low constrained risk inequality]\label{lemma:BrownLowInequality} Let $\mathbb{Q}_1$, $\mathbb{Q}_2$ be probability measures on a measurable space $(\Omega,\mathcal{F})$ such that $\mathbb{Q}_2$ is absolutely continuous with respect to $\mathbb{Q}_1$, and assume that 
\begin{align*}
I:= \chi^2(\mathbb{Q}_2,\mathbb{Q}_1)+1 <\infty.   
\end{align*}
Let $\epsilon \in  \bigl(0,I^{-1/2}\bigr)$ and let $Z:\Omega \rightarrow \{1,2\}$ be a $\mathcal{F}$-measurable random variable with $\mathbb{Q}_1(Z=2) \leq \epsilon^2$. Then $\mathbb{Q}_2(Z=1) \geq \bigl( 1-\epsilon\sqrt{I}\bigr)^2$.
\end{lemma}

The following version of Fano's lemma is a minor variant of \citet[][Lemma~4.3]{gerchinovitz2020fano}.
\begin{lemma}[Fano's lemma for $\chi^2$ divergences]\label{lemma:chiSqrFano}
Let $\Prob_1,\ldots,\Prob_M,\mathbb{Q}_1,\ldots,\mathbb{Q}_M$ denote probability measures on
$(\Omega,\mathcal{F})$, and let $A_1,\ldots,A_M \in
\mathcal{A}$.  Write $\bar{p} := M^{-1}\sum_{j=1}^M \Prob_j(A_j)$ and $\bar{q} := M^{-1}\sum_{j=1}^M \mathbb{Q}_j(A_j)$.  If $\bar{q} \in (0,1)$, then
\[
\bar{p} \leq \bar{q} + \sqrt{\frac{1}{M}\sum_{j=1}^M\chi^2(\Prob_j,\mathbb{Q}_j) \cdot \bar{q}(1-\bar{q})}.
\]
In particular, if $M \geq 2$ and $A_1,\ldots,A_M$ form a partition of
$\Omega$, then
\[
\frac{1}{M}\sum_{j=1}^M \Prob_j(A_j) \leq \frac{1}{M} + \sqrt{\inf_{\mathbb{Q} \in \mathcal{Q}} \frac{1}{M} \sum_{j=1}^M\chi^2(\Prob_j,\mathbb{Q}) \cdot \frac{1}{M}\biggl(1-\frac{1}{M}\biggr)},
\]
where $\mathcal{Q}$ denotes the set of all probability distributions on $\Omega$.
\end{lemma}
\begin{proof}
By the joint convexity of $\chi^2$ divergence, together with the data processing inequality \citep[e.g.][Lemma~2.1]{gerchinovitz2020fano}, we have
\[
\frac{(\bar{p}-\bar{q})^2}{\bar{q}(1-\bar{q})} = \chi^2\bigl(\mathrm{Bern}(\bar{p}),\mathrm{Bern}(\bar{q})\bigr) \leq \frac{1}{M}\sum_{j=1}^M
\chi^2\bigl(\Prob_j(A_j),\mathbb{Q}_j(A_j)\bigr) \leq \frac{1}{M}\sum_{j=1}^M \chi^2(\Prob_j,\mathbb{Q}_j).
\]
The first result follows on rearranging this inequality, and the second follows by taking $\mathbb{Q}_1 = \cdots = \mathbb{Q}_M = \mathbb{Q}$ and then taking an infimum over $\mathbb{Q} \in \mathcal{Q}$. 
\end{proof}

\end{document}